\documentclass[12pt]{amsart}
\usepackage{currvita, amssymb, amsmath, amsthm, 
graphicx, subfigure, pifont, wrapfig, pinlabel, amscd,
latexsym, exscale, enumerate, amsfonts, times, color, hyphenat, pinlabel, 
mathtools, a4wide, fullpage, array, bbm, url}
\usepackage{float}
\usepackage[colorlinks=false, pdfborder={0.9 0.9 0.9}]{hyperref}
\usepackage{stmaryrd}
\usepackage{multirow}
\usepackage[all]{xy}
\usepackage{young}
\usepackage[vcentermath]{youngtab}

\newcommand{\foamt}{\textbf{Foam}_{3}}

\newcommand{\bN}{\mathbb{N}}
\newcommand{\bZ}{\mathbb{Z}}
\newcommand{\bQ}{\mathbb{Q}}

\newcommand{\bC}{\mathbb{C}}

\newcommand{\onel}{{\textbf 1}_{\lambda}}

\def\ui{{\text{$\underline{i}$}}}
\def\idm{\text{\bfseries 1}}

\newcommand{\refequal}[1]{\xy {\ar@{=}^{#1}
(-1,0)*{};(1,0)*{}};
\endxy}

\newcommand{\U}{\dot{{\bf U}}(\mathfrak{sl}_n)}

\newcommand{\Ucat}{\mathcal{U}(\mathfrak{sl}_n)}

\newtheorem{prop}{Proposition}[section]
\newtheorem{thm}[prop]{Theorem}
\newtheorem{lem}[prop]{Lemma}
\newtheorem{cor}[prop]{Corollary}

\theoremstyle{remark}
\newtheorem{rem}[prop]{Remark}

\theoremstyle{remark}
\newtheorem{ex}[prop]{\textbf{Example}}

\theoremstyle{remark}

\theoremstyle{definition}
\newtheorem{defn}[prop]{Definition}

\numberwithin{equation}{subsection}

\newcommand{\figins}[3] 
{\raisebox{#1pt}{\includegraphics[height=#2 in]{res/figs/basicsa/#3}}}

\newcommand{\figwhins}[4] 
{\raisebox{#1pt}{\includegraphics[height=#2 in, width=#3 in]{res/figs/basicsa/#4}}}

\newcommand{\F}{\mathcal{F}}

 \newcounter{In}

\title{$\mathfrak{sl}_3$-web bases, intermediate crystal bases and categorification}
\author{Daniel Tubbenhauer}
\thanks{The author was supported by the Courant Research Center ``Higher Order Structures'' in G\"{o}ttingen, the Graduiertenkolleg 1493 in G\"{o}ttingen and the Mathematisches Institut, Georg-August-Universit\"{a}t G\"{o}ttingen.}
\date{Last compiled \today}

\begin{document}
\begin{abstract}
We give an explicit graded cellular basis of the $\mathfrak{sl}_3$-web algebra $K_S$.

In order to do this, we identify Kuperberg's basis for the $\mathfrak{sl}_3$-web space $W_S$ with a version of Leclerc-Toffin's intermediate crystal basis and we identify Brundan, Kleshchev and Wang's degree of tableaux with the weight of flows on webs and the $q$-degree of foams.

We use these observations to give a ``foamy'' version of Hu and Mathas graded cellular basis of the cyclotomic Hecke algebra which turns out to be a graded cellular basis of the $\mathfrak{sl}_3$-web algebra.

We restrict ourselves to the $\mathfrak{sl}_3$ case over $\bC$ here, but our approach should, up to the combinatorics of $\mathfrak{sl}_N$-webs, work for all $N>1$ or over $\bZ$.     
\end{abstract}

\maketitle
\paragraph*{Acknowledgements} I like to thank Bernard Leclerc, Lukas Lewark, Marco Mackaay, Weiwei Pan, Hoel Queffelec, A Referee, Louis-Hadrien Robert, Antonio Sartori and Catharina Stroppel for helpful discussions and comments about (intermediate) crystal bases, $q$-skew Howe duality, cyclotomic Hecke/KL-R algebras, $\mathfrak{sl}_N$-web algebras and HM-bases.

Special thanks to Marco Mackaay for a lot of helpful discussions.

The author thanks the FCT - Funda\c c\~{a}o para a 
Ci\^{e}ncia e a Tecnologia, through project number PTDC/MAT/101503/2008, 
New Geometry and Topology for sponsoring one research visit during this project.

\tableofcontents

\section{Introduction}\label{sec-intro}
After Khovanov published his groundbreaking work~\cite{kh4} on the so-called \textit{arc algebra $H_n$}, which was inspired by his categorification of the Jones polynomial~\cite{kh1}, researchers started to study these diagrammatic algebras and generalizations of it from different viewpoints and it turns out that these algebras have a beautiful combinatorial structure and representation theory. Moreover, they are related to algebraic geometry and knot theory.
\vspace*{0.25cm}

The literature about this subject is wide nowadays and many results are known. To list a few of them, the generalizations of the arc algebra in type $A_1$ are for example studied in~\cite{bs1},~\cite{bs2},~\cite{bs3}, \cite{bs4},~\cite{bs5},~\cite{chkh},~\cite{kh5},~\cite{st2} and~\cite{sw}, in type $A_2$ for example in~\cite{mpt} and~\cite{rob2} or~\cite{rob1}, in type $A_N$ case for example in~\cite{mack1} and~\cite{my}. There is also an arc algebra in type $D$ studied in~\cite{ehst1},~\cite{ehst2} and~\cite{ehst3} and one in the $\mathfrak{gl}(1|1)$ case studied in~\cite{sar}.
\vspace*{0.25cm}

These algebras in type $A$ can be seen as a \textit{categorification} of the underlying \textit{$\mathfrak{sl}_N$-web space}. In the $\mathfrak{sl}_2$ case (if one ignores orientations) one has the so-called Temperley-Lieb category, which gives a diagrammatic presentation of the representation theory of ${\mathbf U}_q(\mathfrak{sl}_2)$ (an interesting historical remark is that a first diagrammatic approach already arose in a paper of Rumer, Teller, and Weyl~\cite{rtw}). In the $\mathfrak{sl}_3$ case the underlying space consists of Kuperberg's $\mathfrak{sl}_3$-webs, introduced in~\cite{kup}. They give a diagrammatic presentation of the representation theory of ${\mathbf U}_q(\mathfrak{sl}_3)$. In the $\mathfrak{sl}_N$ case the underlying space consists of $\mathfrak{sl}_N$-webs. They give a diagrammatic presentation of the representation theory of ${\mathbf U}_q(\mathfrak{sl}_N)$, as was recently proven by Cautis, Kamnitzer and Morrison~\cite{ckm}.
\vspace*{0.15cm}

In this paper we continue the study of the so-called \textit{$\mathfrak{sl}_3$-web algebras} introduced in~\cite{mpt}. We denote them by $K_S$, where $S$ is a \textit{sign string} (string of $+$ and $-$ signs). To be more precise, we show that $K_S$ is a \textit{graded cellular algebra} in the sense of~\cite{grle} Graham and Lehrer (who introduced the notion in the ungraded setting) and Hu and Mathas~\cite{hm} (who extended Graham and Lehrer's notion to the graded setting) by giving an \textit{explicit} graded cellular basis for the $\mathfrak{sl}_3$-web algebra $K_S$. We follow a different approach than Brundan and Stroppel used in the $\mathfrak{sl}_2$ case~\cite{bs1}., since their arguments does not seem to generalize in a straightforward way to $N>2$. In fact, we claim that our approach in this paper will, up to some combinatorics of $\mathfrak{sl}_N$-webs, generalize to all $N>1$.
\vspace*{0.25cm}

It should be noted that it was known before, as the author showed together with Mackaay and Pan in~\cite{mpt} in the $\mathfrak{sl}_3$ case and Mackaay and Yonezawa~\cite{my} showed in the $\mathfrak{sl}_N$ case, that all the $\mathfrak{sl}_N$-web algebras are graded cellular algebras. The way this was proven is by an \textit{abstract Morita equivalence} - no other proof is known at the moment. An explicit cellular basis in only known in the $\mathfrak{sl}_2$ case. As mentioned, the construction is due to Brundan and Stroppel~\cite{bs1}.
\vspace*{0.25cm}  

Our approach in this paper is as follows. One main ingredient is the usage of \textit{categorified, diagrammatic quantum skew Howe duality} studied recently independently by varies authors in this framework~\cite{ckm},~\cite{lqr} and~\cite{mpt} to cite a few (the first appearance in the context of $\mathfrak{sl}_2$-webs seems to be the paper of Huerfano and Khovanov~\cite{hukh}, although they never used the notion of skew Howe duality). In the $\mathfrak{sl}_3$ case this means that there is a strong $2$-representation $\Psi\colon\Ucat\to\foamt$, called \textit{foamation}, of Khovanov-Lauda's~\cite{kl5} categorification of $\U$, denoted by $\Ucat$, to the category of $\mathfrak{sl}_3$-foams $\foamt$. This foamation functor was used in~\cite{mpt} to show that $K_S$ is Morita equivalent to a certain block of $R_{\Lambda}$, where $R_{\Lambda}$ denotes Khovanov-Lauda's~\cite{kl1} and~\cite{kl3} and Rouquier's~\cite{rou} cyclotomic quotient, called the \textit{cyclotomic KL-R algebra}.
\newpage

In a remarkable paper~\cite{bk1} Brundan and Kleshchev showed that the cyclotomic KL-R algebra is isomorphic to the so-called \textit{cyclotomic Hecke algebra}. Their isomorphism was used by them to define a \textit{non-trivial grading} on the cyclotomic Hecke algebra. Using this isomorphism, Hu and Mathas gave in~\cite{hm} a graded cellular basis for the cyclotomic KL-R algebra based on earlier work of Dipper, James and Mathas~\cite{djm} who gave a (ungraded) cellular basis of the cyclotomic Hecke algebra and Brundan and Kleshchev's (and Wang's) work on graded Specht modules for these algebras, see~\cite{bk2},~\cite{bk3} and~\cite{bkw}.
\vspace*{0.25cm}

The approach we follow here is that we give a \textit{``foamy''} version of the Hu and Mathas basis (short HM-basis) using foamation and quantum skew Howe duality. It turns out that the combinatorial structure can be easier seen within the cyclotomic Hecke algebra, while the topological structure can be easier seen within the foam setting.
\vspace*{0.25cm}

We note that the explicit construction is a non-trivial task, since bases behave really \textit{badly} under Morita equivalence and some non-trivial translation from the cyclotomic Hecke algebra to the foam framework was needed, that is, even the answers to basic questions were unknown before. It turns out, as we explain below, that this non-trivial combinatorics is in fact quite nice itself, since it gives a new perspective on the $\mathfrak{sl}_3$-web spaces.
\vspace*{0.25cm}

The second main ingredient is that we use a special basis for our underlying $\mathfrak{sl}_3$-web space $W_S$, the so-called \textit{intermediate crystal basis} $B_{\Lambda}$ of Leclerc-Toffin~\cite{leto} (short LT-basis) of the highest weight ${\mathbf U}_q(\mathfrak{sl}_n)$-module $V_{\Lambda}$, which can be translated to the $\mathfrak{sl}_3$-web setting using quantum skew Howe duality. It turns out the the LT-basis is easy to work with if one wants to categorify the results from the level of $\mathfrak{sl}_N$-webs.
\vspace*{0.25cm}

To be a little bit more precise, we \textit{``categorify''} the LT-algorithm for $\mathfrak{sl}_3$-webs by giving a \textit{growth algorithm for foams} producing a graded cellular basis.
\vspace*{0.25cm}

That is, the LT-algorithm, under $q$-skew Howe duality, turns given semi-standard tableaux into $\mathfrak{sl}_3$-webs by applying a sequence of so-called \textit{$\mathfrak{sl}_3$-ladder operators} obtained from an algorithm on semi-standard tableaux to a certain highest weight vector. On the other hand, the growth algorithm for foams, under categorified $q$-skew Howe duality, turns standard $3$-multitableaux into $\mathfrak{sl}_3$-foams by applying a sequence of certain \textit{$\mathfrak{sl}_3$-foams} (that \textit{categorify} the ladder operators) obtained from an algorithm on standard $3$-multitableaux to a certain highest weight object.
\vspace*{0.25cm}

It should be noted that, in order to formulate the growth algorithm for foams, we relate standard $3$-multitableaux to \textit{Khovanov-Kuperberg flows} (see~\cite{kk}) on $\mathfrak{sl}_3$-webs. Recall that these flows are a combinatorial way to answer the important question what a web, seen as an \textit{invariant tensor} $\mathrm{Inv}_{{\mathbf U}_q(\mathfrak{sl}_3)}(\bigotimes_kV_{s_k})$, is \textit{explicitly} in terms of the elementary tensors of $\bigotimes_kV_{s_k}$. 
\vspace*{0.25cm}

In order to prove both observations highly \textit{non-trivial} combinatorics on the $\mathfrak{sl}_3$-webs is needed (in fact this seems to be the only problem for generalizing everything to $N>3$). But this combinatorics turns out to be quite beautiful itself, i.e. we identify Kuperberg's basis of non-elliptic webs as an intermediate crystal basis (which shows ``immediately'' that it is related by an \textit{unitriangular} change-of-base matrix to the dual canonical and again demonstrates that it is somehow ``the'' basis of the $\mathfrak{sl}_3$-web space), we give a growth algorithm for $\mathfrak{sl}_3$-webs \textit{with flows} and we relate webs with flows to standard fillings of $3$-multitableaux. We use latter to identify Brundan, Kleshchev and Wang's notion of degree~\cite{bkw} in a very natural way, i.e. as the weight of flows on webs~\cite{kk} and the $q$-degree of foams, where latter is just a slight modification of the \textit{geometrical Euler characteristic} (hence, Brundan, Kleshchev and Wang's degree is an isotopy invariant).
\vspace*{0.25cm}

The reason why we think this approach should ``easily'' (up to combinatorics of $\mathfrak{sl}_N$-webs) generalize is that both of our main ingredients are known in general for $N>1$, see~\cite{mack1} and~\cite{my} (and conjecturally $\mathfrak{sl}_N$-foams~\cite{msv}). Note that the results from Section~\ref{sec-compare} can be easily generalized to $\mathfrak{sl}_N$-webs to give a \textit{growth algorithm} for such $\mathfrak{sl}_N$-webs. In fact, we tend to argue that this basis forms somehow a ``natural'' basis of the $\mathfrak{sl}_N$-web space, which is \textit{neither} the Satake basis \textit{nor} Fontaine's basis. For details about the Satake basis and about Fontaine's basis see~\cite{fon} and~\cite{fkk}. 
\vspace*{0.25cm}

This paper is organized as follows.
\begin{enumerate}
\item We start by giving our notation for $3$-multipartitions in Section~\ref{sec-combi}. Most notions from this section can be done in more generality, but we only need the case of $3$-multipartitions in this paper.
\item In Section~\ref{sec-quantum} we recall Leclerc-Toffin's basis and its relation to Kashiwara-Lusztig's crystal bases, Khovanov-Lauda's categorification $\Ucat$, the cyclotomic KL-R algebra and the HM-basis for it.
\item We recall $\mathfrak{sl}_3$-webs/foams and their connection to ${\mathbf U}_q(\mathfrak{sl}_3)$ in Section~\ref{sec-webfoam}.
\item In Section~\ref{sec-webalg} we recall the $\mathfrak{sl}_3$-web algebra $K_S$ (note that we use the conventions from~\cite{mpt}) and the foamation functor $\Psi\colon\Ucat\to\foamt$.
\item In the Section~\ref{sec-compare} we show that Kuperberg's basis is indeed a monomial basis, that is, we show that it is an intermediate crystal basis in the sense of Leclerc and Toffin~\cite{leto}. This includes that there is an unitriangular algorithm to compute the dual canonical basis of the web space $W_S$ using Kuperberg's $\mathfrak{sl}_3$-web basis and the Kuperberg bracket.
\item We show how one can relate flows on $\mathfrak{sl}_3$-webs to fillings of $3$-multipartitions. This gives rise to a \textit{growth algorithm for $\mathfrak{sl}_3$-webs with flows} in the two Sections~\ref{sec-flow} and~\ref{sec-exgrowth}. These two combinatorial sections are crucial for the rest of the paper.
\item In Section~\ref{sec-degflow} we show that Brundan, Kleshchev and Wang's degree of a tableaux~\cite{bkw} has a very natural interpretation as the weight of flows on $\mathfrak{sl}_3$-webs.
\item We give in Section~\ref{sec-catpart1} a \textit{growth algorithm for foams} that produces a homogeneous basis of the foam space based on a categorified version of Leclerc-Toffin's algorithm to compute the intermediate crystal basis.
\item We show that the growth algorithm for foams gives a homogeneous basis of the foam space by showing that it can be seen as a ``foamy'' version of Hu and Mathas graded cellular basis. Note that this includes that Brundan, Kleshchev and Wang's degree has a very natural interpretation as the $q$-degree of foams (which is just a slightly modified Euler characteristic). This is done in Section~\ref{sec-basis}.
\item And finally, in Section~\ref{sec-cell}, we show that the growth algorithm for foams produces a graded cellular basis of $K_S$.
\item It follows (almost) as a direct application (see Remark~\ref{rem-basisgrothen}) that the set of simple heads of the cell (or Specht) modules $\mathcal S$ of $K_S$, denoted by $\mathcal D$, and the set of their indecomposable, projective covers of, denoted by $\mathcal P$, give rise to the canonical and dual canonical bases of the $\mathfrak{sl}_3$-web space $W^{(*)}_S$.
\end{enumerate}
\vspace*{0.10cm}

It is worth noting that we give lots of examples in each section to show that everything can be done explicitly.
\section{Basic notions}\label{sec-basics}
\subsection{Combinatoric, partitions and tableaux}\label{sec-combi}
In this section we define/recall some combinatorial notions that we use in this paper. Note that we recall most notions only for the $\mathfrak{sl}_3$ case, although they can be done in more generality without difficulties, see for example~\cite{hm}.
\vspace*{0.25cm}

Choose arbitrary but fixed non-negative integers 
$n\geq 2$ and $k\leq n$. Let 
\[
\Lambda(n,d)=\left\{\lambda\in \mathbb{N}^n\mid \sum_{i=1}^n\lambda_i=d\right\}
\]
be the set of \textit{compositions} of $d$ of length $n$. 
By $\Lambda^+(n,d)\subset\Lambda(n,d)$ we denote the 
subset of \textit{partitions}, i.e. all $\lambda\in\Lambda(n,d)$ such that 
\[
\lambda_1\geq\lambda_2\geq \ldots\geq \lambda_n\geq 0.
\]

Let $\Lambda^{+}(n,d)_{1,2}\subset \Lambda^{+}(n,d)$ be the subset of partitions
whose entries are all $1$ or $2$. We use similar notions for $\Lambda^{+}(n,d)_{N}$ for some $N\in\bN$, that is
\[
\Lambda^{+}(n,d)_N=\left\{\lambda\in \mathbb{N}^n\mid \sum_{i=1}^n\lambda_i=d,\;\lambda_i\in\{0,\dots,N\} \right\}.
\]

Recall that we can associate to each $\lambda\in\Lambda^+(n,d)$ a \textit{diagram for $\lambda$}
\[
\lambda=\{(r,c)\mid 0\leq c\leq \lambda_r, 1\leq r\leq n\},
\]
which we, by a slight abuse of notion, denote by the same symbol $\lambda$. The elements of a diagram are called \textit{nodes $N$}. For example, if $\lambda=(3,2,1)$, that is $d=6,n=3$, then
\[
\lambda=\xy(0,0)*{\begin{Young} & & \cr & \cr \cr\end{Young}}\endxy.
\]
Hence, we use the English notation to denote our partitions/diagrams.

A \textit{tableau $T$ of shape $\lambda$} is a filling of $\lambda$ with (possible repeating) numbers from a chosen, fixed set $\{1,\dots,k\}$. Such a tableau $T$ is said to be \textit{semi-standard}, if its entries increase along its rows (weekly) and columns (strictly), and \textit{column-strict}, if its entries increase along its columns (no restriction on rows). For example
\[
T_1=\xy(0,0)*{\begin{Young} 1& 2& 3\cr 2& 3\cr 4\cr\end{Young}}\endxy\;\;\;\;T_2=\xy(0,0)*{\begin{Young} 1& 2& 2\cr 4& 3\cr 5\cr\end{Young}}\endxy
\]
the tableau $T_1$ is semi-standard, but $T_2$ is only column-strict. We denote the set of all \textit{semi-standard tableau of shape $\lambda$} by $\mathrm{Std}^s(\lambda)$ and the set of all \textit{column-strict tableau of shape $\lambda$} by $\mathrm{Col}(\lambda)$.

Moreover, we stress that we do not assume that our tableaux have \textit{only non-repeating} entries. In fact, we only assume that the number of times that an entry appears is between $0$ and $3$. For our $3$-multitableaux $\vec{\lambda}$ (see below) we assume the same with the extra restriction that repeating entries are \textit{never} in the same vector entry of $\vec{\lambda}$ and all repeating entries are of the \textit{same} residue. We note that this strange looking condition is due to the fact that we use \textit{divided powers} which do not appear in the language of cyclotomic Hecke or KL-R algebras. 
\vspace*{0.25cm}

In the same vein, a \textit{$3$-multipartition} $\vec{\lambda}\in\Lambda^+(n,d,3)$ of $d$ with length $n$ is a triple of partitions $\vec{\lambda}=(\lambda^1,\lambda^2,\lambda^3)$ with $\lambda^l=(\lambda^l_1,\dots,\lambda^l_{n_l})$ such that their total length is $n$ and their total sum is $d$. As before, we can associate to each $\vec{\lambda}\in\Lambda^+(n,d,3)$ a \textit{diagram for $\vec{\lambda}$}
\[
\vec{\lambda}=\{(r,c,l)\mid 0\leq c\leq \lambda^l_r, 1\leq r\leq n_l,l=1,\dots,3\},
\]
which we, by a slight abuse of notion, denote by the same symbol $\vec{\lambda}$. For example, if we have $\vec{\lambda}=((3,2,1),(0),(4))$, that is $d=10$, $n=4$, $\lambda^1=(3,2,1)$, $\lambda^2=(0)$ and $\lambda^3=(4)$, then
\[
\lambda=\left(\;\xy(0,0)*{\begin{Young} & & \cr & \cr \cr\end{Young}}\endxy\;,\;\emptyset\;,\;\xy(0,0)*{\begin{Young} & & &\cr\end{Young}}\endxy\;\right).
\]
At this point it is worthwhile to say that since both notations appear repeatedly in the literature: To \textit{distinguish} between the notion of partition $\lambda$ and components of multipartition $\vec{\lambda}=(\lambda^1,\lambda^2,\lambda^3)$, we write latter using \textit{superscripts} (and similar for multitableaux).

As before, a \textit{$3$-multitableau $\vec{T}$ of shape $\vec{\lambda}$} is a filling of $\vec{\lambda}$ with (possible repeating) numbers from a chosen, fixed set $\{1,\dots,k\}$. Such a tableau $\vec{T}$ is said to be \textit{standard}, if its entries increase along its rows and columns (both strictly).

We denote the set of all \textit{standard tableau $\vec{T}$ of shape $\vec{\lambda}$} by $\mathrm{Std}(\vec{\lambda})$.
\vspace*{0.25cm}

As we mentioned above, we are mostly interested in $3$-multipartitions $\Lambda^+(n,d,3)$ here. But under skew Howe-duality it is necessary to consider them as $l$-partitions, where the $l>0$ depends on the context. There are two natural embeddings $\iota^l_3,\kappa^l_3\colon\Lambda^+(n,d,3)\to\Lambda^+(n,d,l)$ for $l>2$, i.e.
\[
\iota_3^l(\vec{\lambda})=(\underbrace{(0),\dots,(0)}_{l-3},\lambda^1,\lambda^2,\lambda^3)\;\text{and}\;\kappa_3^l(\vec{\lambda})=(\lambda^1,\lambda^2,\lambda^3,\underbrace{(0),\dots,(0)}_{l-3}).
\]
We always use the first one $\iota^l_3$, since the first one fits to our other conventions. But, with a slight abuse of notation, we always think of $\iota_3^l(\vec{\lambda})$ as a $3$-multipartition $\vec{\lambda}$.

\begin{defn}\label{defn-tabcomb}
Let $\lambda\in\Lambda^+(n,d)$ be a partition. Then we associate to each node $N=(r,c)\in\lambda$ of $\lambda$ a \textit{residue} $r(N)$ by the rule $r(N)=r-c+m$ where $m$ is the number of non-zero entries of $\lambda$. It should be noted that we see $m$ as being fixed by $\lambda$, even if we speak about addable or removable nodes. 

If $\vec{\lambda}=\{(r,c,l)\mid 0\leq c\leq \lambda^l_r, 1\leq r\leq n_l, l=1,2,3\}$ is a $3$-multipartition, then we can use the same notions for each of its nodes $N=(r,n,l)\in\vec{\lambda}$. This time $m$ is the maximal number of non-zero entries of the components of $\vec{\lambda}$.

An \textit{addable node $N$ of residue $r(N)=k$} is a node $N$ that can be added to the diagram of $\lambda$ such that the new diagram is still the diagram of a partition and the residue is $r(N)=k$. We denote the \textit{set of addable nodes of residue $k$ of $\lambda$} by $\mathsf{A}^{k}(\lambda)$.

Similar, a \textit{removable node $N$ of residue $r(N)=k$} is a node that can be removed from the diagram of $\lambda$ such that the new diagram is still the diagram of a partition and the residue of $N$ is $r(N)=k$. We denote the \textit{set of removable nodes of residue $k$ of $\lambda$} by $\mathsf{R}^{k}(\lambda)$.

Again, we can use the same notions for a $3$-multipartition $\vec{\lambda}\in\Lambda^+(n,d,3)$.
\vspace*{0.25cm}

Moreover, we say a node $N=(r,c,l)$ of $\vec{\lambda}=(\lambda^l)_{l=1}^3$ comes \textit{before (or after)} another node $N^{\prime}=(r^{\prime},c^{\prime},l^{\prime})$ of $\vec{\lambda}$, denoted by $N\preceq N^{\prime}$ (or $N\succeq N^{\prime}$), iff $l<l^{\prime}$ or $l=l^{\prime}$ and $r\leq r^{\prime}$ (or $l>l^{\prime}$ or $l=l^{\prime}$ and $r\geq r^{\prime}$). We use the obvious definitions for the notions \textit{strictly} before $\prec$ and \textit{strictly} after $\succ$.

For a fixed node $N$, we denote the \textit{set of addable nodes of $\lambda$ before $N$} with the same residue $r(N)=k$ by $\mathsf{A}^{k\prec N}(\lambda)$ and we denote the \textit{set of addable nodes of $\lambda$ after $N$} with the same residue $r(N)=k$ by $\mathsf{A}^{k\succ N}(\lambda)$. In the same vein, for a fixed node $N$, we denote the \textit{set of removable nodes of $\lambda$ before $N$} with the same residue $r(N)=k$ by $\mathsf{R}^{k\prec N}(\lambda)$ and we denote the \textit{set of removable nodes of $\lambda$ after $N$} with the same residue $r(N)=k$ by $\mathsf{R}^{k\succ N}(\lambda)$.

We are mostly interested in the nodes after a given node $N$. One would have to use the nodes before if one wants to construct a ``dual'' basis of the HM-basis, see~\cite{hm}.
\end{defn}
\begin{ex}\label{ex-resi}
Let $\vec{\lambda}=(\lambda^1,\lambda^2,\lambda^3)$ be the following $3$-multipartition (we have $m=3$).
\[
\lambda^1=\xy(0,0)*{\begin{Young} 3&  4&  5&6\cr  2& 3\cr 1\cr\end{Young}}\endxy\;,\;\;\;\;\lambda^2=\xy(0,0)*{\begin{Young} 3&  4\cr  2\cr\end{Young}}\endxy\;,\;\;\;\;\lambda^3=\xy(0,0)*{\begin{Young} 3&  4&  5&6\cr  2& 3&4&5\cr 1&2 &3\cr\end{Young}}\endxy\;.
\]
We have filled the nodes of $\lambda^{1,2,3}$ with the corresponding residues. Note that the residue is constant along the diagonals.

The set of addable nodes of residue $4$ for $\vec{\lambda}$ and the set of removable nodes of residue $2$ for $\vec{\lambda}$ are given by
\[
\lambda^1=\xy(0,0)*{\begin{Young} &  &  &\cr  & &$\cdot$\cr \cr\end{Young}}\endxy\;,\;\;\;\;\lambda^2=\xy(0,0)*{\begin{Young} &  \cr  $\times$\cr\end{Young}}\endxy\;,\;\;\;\;\lambda^3=\xy(0,0)*{\begin{Young} &  &  &\cr  & & &\cr & & &$\cdot$\cr\end{Young}}\endxy\;,
\]
where we have indicated the addable nodes with a $\cdot$ and the removable with a $\times$. The removable node is after the first addable and before the second addable node. Moreover, in the following we demonstrate all nodes strictly before $\prec$ and strictly after $\succ$ a fixed node marked $-$.
\[
\lambda^1=\xy(0,0)*{\begin{Young} $\prec$&  $\prec$&  $\prec$&$\prec$\cr $\prec$ &$\prec$\cr $\prec$\cr\end{Young}}\endxy\;,\;\;\;\;\lambda^2=\xy(0,0)*{\begin{Young} $\prec$&  $\prec$\cr  $\prec$\cr\end{Young}}\endxy\;,\;\;\;\;\lambda^3=\xy(0,0)*{\begin{Young} & $-$ &  &\cr  $\succ$& $\succ$& $\succ$& $\succ$\cr $\succ$& $\succ$& $\succ$\cr\end{Young}}\endxy\;.
\]
\end{ex}
\vspace*{0.1cm}
\begin{defn}\label{defn-dominnance}
Let $\vec{\lambda}=(\lambda^1,\lambda^2,\lambda^3)$ and $\vec{\mu}=(\mu^1,\mu^2,\mu^3)$ be two $3$-multipartitions in $\Lambda^+(n,d,3)$. Recall that $\lambda_l=(\lambda^l_1,\lambda^l_2,\dots)$ and $\mu_l=(\mu^l_1,\mu^l_2,\dots)$ for $l\in\{1,2,3\}$.

We say \textit{$\vec{\lambda}$ dominates $\vec{\mu}$,} denoted by $\vec{\mu}\trianglelefteq\vec{\lambda}$, if
\[
\sum_{i=1}^{l-1}|\mu^l|+\sum_{j=1}^{s}\mu^l_j\leq \sum_{i=1}^{l-1}|\lambda^l|+\sum_{j=1}^{s}\lambda^l_j
\]
for all $1\leq l\leq 3$ and $1\leq s$. We write $\vec{\mu}\lhd\vec{\lambda}$, if $\vec{\mu}\unlhd\vec{\lambda}$ and $\vec{\mu}\neq\vec{\lambda}$. It is easy to check that $\unlhd$ is a partial ordering of the set of all $3$-multipartitions $\Lambda^+(n,d,3)$, called the \textit{dominance order}.

This order can be \textit{extended} to $3$-multitableaux in the following way. Suppose we have two standard $3$-multitableaux $\vec{T}_1\in\mathrm{Std}(\vec{\lambda})$ and $\vec{T}_2\in\mathrm{Std}(\vec{\mu})$ with $k$ nodes filled with numbers from $\{1,\dots,k\}$ (no repetitions). As in Definition~\ref{defn-combinatorics1}, we denote the corresponding $3$-multipartitions after removing all nodes with entries higher than $j\in\{1,\dots,k\}$ by $\vec{\mu}^j$ and $\vec{\lambda}^j$. Then
\[
\vec{T}_1\unlhd\vec{T}_2\Longleftrightarrow \vec{\mu}^j\unlhd\vec{\lambda}^j\;\;\text{ for all }\;\;j\in\{1,\dots,k\}.
\]
To extend it to all, i.e. allowing repetitions, $3$-multitableaux we use a special $3$-multitableaux $\vec{T}^{\prime}$ associated to $\vec{T}$ by inductively replacing repeating entries from left to right, e.g.
\begin{equation}\label{eq-divtab}
\vec{T}=\left(\;\xy(0,0)*{\begin{Young}1&2\cr 3\cr\end{Young}}\endxy\;,\;\xy(0,0)*{\begin{Young}4\cr\end{Young}}\endxy\;,\;\xy(0,0)*{\begin{Young}1&2\cr 5\cr\end{Young}}\endxy\;\right)\mapsto \vec{T}^{\prime}=\left(\;\xy(0,0)*{\begin{Young}1&3\cr 5\cr\end{Young}}\endxy\;,\;\xy(0,0)*{\begin{Young}6\cr\end{Young}}\endxy\;,\;\xy(0,0)*{\begin{Young}2&4\cr 7\cr\end{Young}}\endxy\;\right).
\end{equation}
Then we can use the same rules as before.

Given $\vec{\lambda}\in\Lambda^+(n,d,3)$ we can associate to it a \textit{unique standard $3$-multitableau $T_{\vec{\lambda}}\in\mathrm{Std}(\vec{\lambda})$} with the property
\[
\vec{T}\in\mathrm{Std}(\vec{\lambda})\Rightarrow \vec{T}\unlhd T_{\vec{\lambda}}.
\]
Note that $T_{\vec{\lambda}}$ is easily seen to be the tableau with all entries in order from top to bottom and left to right.
\end{defn}
\begin{ex}\label{ex-morecomb1}
Given the same $3$-multipartition as later in Example~\ref{ex-tabl2} part (b), i.e.
\[
\vec{\lambda}=\left(\;\xy(0,0)*{\begin{Young}&\cr\cr\end{Young}}\endxy\;,\;\xy(0,0)*{\begin{Young}\cr\end{Young}}\endxy\;,\;\xy(0,0)*{\begin{Young}&\cr\cr\end{Young}}\endxy\;\right),
\]
we see that
\[
T_{\vec{\lambda}}=\left(\;\xy(0,0)*{\begin{Young}1&2\cr3\cr\end{Young}}\endxy\;,\;\xy(0,0)*{\begin{Young}4\cr\end{Young}}\endxy\;,\;\xy(0,0)*{\begin{Young}5&6\cr7\cr\end{Young}}\endxy\;\right)
\]
It will dominate all $\vec{T}\in\mathrm{Std}(\vec{\lambda})$. For example
\[
\vec{T}=\left(\;\xy(0,0)*{\begin{Young}1&2\cr3\cr\end{Young}}\endxy\;,\;\xy(0,0)*{\begin{Young}5\cr\end{Young}}\endxy\;,\;\xy(0,0)*{\begin{Young}4&6\cr7\cr\end{Young}}\endxy\;\right)
\]
will be dominated, since
\[
\vec{T}^4=\left(\;\xy(0,0)*{\begin{Young}1&2\cr3\cr\end{Young}}\endxy\;,\;\emptyset,\;\xy(0,0)*{\begin{Young}4\cr\end{Young}}\endxy\;\right)\unlhd T_{\vec{\lambda}}^4=\left(\;\xy(0,0)*{\begin{Young}1&2\cr3\cr\end{Young}}\endxy\;,\;\xy(0,0)*{\begin{Young}4\cr\end{Young}}\endxy\;,\;\emptyset\;\right).
\]
\end{ex}
\begin{defn}\label{defn-rsequence}
Let $\vec{T}\in\mathrm{Std}(\vec{\lambda})$ be a $3$-multitableau. The \textit{residue sequence} of $\vec{T}$, denoted by $r(\vec{T})$ is the $k$-tuple whose $j\in\{1,\dots,k\}$ entry is the residues of the node with number $j$. Moreover, the \textit{residue sequence} of a $3$-multitableau $\vec{\lambda}$, denoted by $r(\vec{\lambda})$, is defined to be $r(\vec{\lambda})=r(T_{\vec{\lambda}})$. Note that, since we only use $3$-multitableaux whose repeating numbers are of the same residue, the notion makes sense for or notion of $3$-multitableaux.
\end{defn}
\subsection{Intermediate crystal bases and categorified quantum algebras}\label{sec-quantum}
In this section we shortly recall Leclerc-Toffin's definition of an \textit{intermediate} crystal basis, of which we think as a basis sitting in between the \textit{lower global} (in the sense of Kashiwara) crystal bases $b_T$ (also called \textit{canonical} basis in the sense of Lusztig) and the standard basis $e_{T^{\prime}}$. Then we recall the notion of Khovanov and Lauda's \textit{categorification} of $\dot{{\mathbf U}}_q(\mathfrak{sl}_n)$ (recall that this is Beilinson, Lusztig and MacPherson idempotented form~\cite{blm}), denoted by $\Ucat$, and the notion of \textit{Khovanov-Lauda and Rouquier algebras} (KL-R for short) and Hu-Mathas \textit{graded cellular basis} for latter.

We stress that we really only recall everything \textit{very} briefly. Much more details can be found for example in the references we give below.

Moreover, we \textit{do not} recall $q$-skew Howe duality or the representation theory of ${\mathbf U}_q(\mathfrak{gl}_n)$ and ${\mathbf U}_q(\mathfrak{sl}_n)$-tensors, because we want to keep the length of this paper in a reasonable boundary. Details concerning $q$-skew Howe duality (or even more, i.e. a nice survey about it) can for example be found in the paper by Cautis, Kamnitzer and Morrison~\cite{ckm} and, for example in Mackaay's paper~\cite{mack1}, the reader can find a good discussion how tensor products and certain bases behave under $q$-skew Howe duality. We note that  we use the conventions of~\cite{mpt} (with $E_{-i}=F_i$) and we hope that the reader does not get confused, since some authors use different conventions, e.g. a different quantum parameter $v=-q^{-1}$.
\subsubsection*{Intermediate crystal bases}
Note that a similar section can be found in~\cite{mack1}, since Mackaay also uses in his paper~\cite{mack1} the intermediate crystal basis under $q$-skew Howe duality.

Let us denote by $V_{\Lambda}$ the irreducible $\dot{{\mathbf U}}_q(\mathfrak{sl}_n)$-representation of highest weight $\Lambda=(3^\ell)$. There is a particular nice basis of $V_{\Lambda}$ called the \textit{lower global crystal basis} (or \textit{canonical basis}). Since we do not need it explicitly here, we do not recall the definition. Details can be found in~\cite{hoka} or~\cite{lu} for example. We denote the canonical basis of $V_{\Lambda}$, which is parametrized by $\mathrm{Std}^s(3^{\ell})$\footnote{Note that we write $3^{\ell}$ as a short hand for the partition $(3^{\ell})$.}, by
\[
\mathrm{can}(V_{\Lambda})=\{b_T\mid T\in\mathrm{Std}^s(3^{\ell})\}.
\]
In contrast to the standard basis $\{e_{T^{\prime}}\in \Lambda_q^{\ell}(\bC_q^n)^{\otimes 3}\mid T^{\prime}\in\mathrm{Col}(3^{\ell})\}$, which is easy to write down, but has not a very good behavior under the action, the lower global crystal basis is very hard to find, but has a very good behavior. Note that, as pointed out in~\cite{mack1}, $q$-skew Howe duality turns $b_T\mapsto b^*_T$ and vice versa, that is, the canonical basis of $V_{\Lambda}$ turns under $q$-skew Howe duality to the $\dot{{\mathbf U}}_q(\mathfrak{sl}_3)$-\textit{dual} canonical basis $\mathrm{dcan}(W_{\Lambda})=\{b^*_T\mid T\in\mathrm{Std}^s(3^{\ell})\}$ of the $\mathfrak{sl}_3$-web space $W_{\Lambda}$.
\vspace*{0.25cm} 

Leclerc and Toffin~\cite{leto} defined a different basis of $V_{\Lambda}$, denoted by $B_{\Lambda}$, parametrized by the elements in $\mathrm{Std}^s(3^{\ell})$ (as Kashiwara-Lusztig's lower global crystal basis). It is also called the \textit{intermediate crystal basis}. It is much easier to write down than the lower global crystal basis and Leclerc and Toffin~\cite{leto} gave an inductive algorithm how to compute $\mathrm{can}(V_{\Lambda})$ from $B_{\Lambda}$. We recall their definition of $B_{\Lambda}$ now very briefly (we should admit that we do not recall the details about the notation here, but they are not important for us in this paper. A good discussion can be found in~\cite{mack1}).

Given $T\in \mathrm{Std}^s(3^{\ell})$, let us recall how to obtain 
the \textit{Leclerc-Toffin (or short LT\footnote{We always mean Toffin and \textit{not} Thibon. We thank Catharina Stroppel for recognizing this confusing point.}) basis} element $A_T\in B_{\Lambda}$. 
Let $1\leq i_1\leq \ell$ be the smallest number such that the rows of 
$T=T_1$ with row number $\leq i_1$ (where we number the rows starting with $1$ from top to bottom in our convention) contain numbers equal to $i_1+1$. We denote 
the total number of such entries by $j_1>0$. Lower these entries to 
$i_1$ and denote the new tableau by $T_2\in\mathrm{Std}^s(3^{\ell})$ (this tableau will still be semi-standard). Continue until one obtains after $s$-steps a tableau $T_s$ whose entries are exactly the row numbers. The element $A_T$ is defined by ($F_{b}^{(a)}=\frac{F^a_b}{[a]!}$ is the divided power with the quantum integer $[a]=\frac{q^{a}+q^{-a}}{q+q^{-1}}$) 
\begin{equation}
\label{eq:LTdef}
A_T=F_{i_1}^{(j_1)}\cdots F_{i_s}^{(j_s)}v_{\Lambda}\;(\text{here }v_{\Lambda}\text{ is the vector of highest weight }\Lambda).
\end{equation}
These basis elements are fixed under the bar involution.

One can easily work out the expansion of 
$A_T$ on the standard basis $e_{T^{\prime}}$ of $\Lambda_q^{\ell}(\bC_q^n)^{\otimes 3}$ (the quantum exterior power - details can be found in~\cite{ckm} or~\cite{mack1}). Leclerc and Toffin showed (i.e. Lemma 9 in~\cite{leto}) that 
\begin{equation}
\label{eq:LT}
A_T=e_T+\sum_{T^{\prime}\prec T} \alpha_{T^{\prime}}(v) e_{T^{\prime}},
\end{equation}
with $T^{\prime}\in \mathrm{Col}(3^{\ell})$ and certain $\alpha_{T^{\prime}}(v)\in \bN[v,v^{-1}]$ (with $v=-q^{-1}$). Here $\prec$ denotes the lexicographical ordering on column-strict tableaux: For a column-strict tableaux $T$ we define the \textit{column-word} $co(T)=(c_1,\dots,c_{3\ell})$ to be a sequence of the entries of the columns of $T$ read from top to bottom and then from left to right. Note that this sequence has length $m\ell$. Then the set $\mathrm{Col}(3^{\ell})$ is partial order by
\[
T\leq T^{\prime}\Leftrightarrow c(T^{\prime})-c(T)\in\bN^{3\ell}\;\text{ with }\;c(T^{(\prime)})=(c^{(\prime)}_1,c^{(\prime)}_1+c^{(\prime)}_2,\dots,c^{(\prime)}_1+\dots+c^{(\prime)}_{3\ell}).
\]
Since we tend to use $3$-multipartitions and $3$-multitableaux instead let us state what this means in our notation. A column-strict tableaux $T$ of shape $(3^{\ell})$ corresponds to a $3$-multipartition $\vec{\lambda}$ by subtracting from each row its number and obtain a new column-strict tableaux $\tilde T$. Read the $k$-th column from bottom to top to obtain in this way the $k$-th partition $\lambda^k$ of $\vec{\lambda}=(\lambda^1,\lambda^2,\lambda^3)$. It is easy to see that this process is in fact invertible.

Write $\vec{\lambda}_T$ for the corresponding $3$-multipartition. Then $T\leq T^{\prime}$ iff $\vec{\lambda}_{T}\trianglelefteq\vec{\lambda}_{T^{\prime}}$, where $\trianglelefteq$ is the dominance order from Definition~\ref{defn-dominnance}. As a small example consider the following.
\[
\xy (0,0)*{\begin{Young}1& 3\cr 2 & 4 \cr\end{Young}}\endxy\leq\xy(0,0)*{\begin{Young}1& 2\cr 3 & 4\cr\end{Young}}\endxy\;\;\text{ and }\;\;
\left(\;\emptyset\;,\;\xy (0,0)*{\begin{Young}&\cr &\cr\end{Young}}\endxy\;\right)\trianglelefteq\left(\;\xy (0,0)*{\begin{Young}\cr\end{Young}}\endxy\;,\;\xy (0,0)*{\begin{Young}&\cr \cr\end{Young}}\endxy\;\right).
\]
\vspace*{0.1cm}
  
Moreover, there is an algorithm to obtain the canonical basis $\mathrm{can}(V_{\Lambda})$ which uses $B_{\Lambda}$ as an intermediate basis. Leclerc and Toffin showed (i.e. Section 
4.2 in~\cite{leto}) that 
\begin{equation}\label{eq:LT2}
b_T=A_T + \sum_{T^{\prime\prime} \prec T } \beta_{T^{\prime\prime}T}(v) A_{T^{\prime\prime}},
\end{equation}
with $T^{\prime\prime}\in \mathrm{Std}^s(3^{\ell})$ for certain bar-invariant (!) coefficients $\beta_{T^{\prime\prime}T}(v)\in \mathbb{Z}[v,v^{-1}]$.
\vspace*{0.25cm}

It is worth noting that a slight change in the definition of $B_{\Lambda}$, i.e. changing the rules which entries are to replaced, one get similar results as above. In fact, in this paper we use such a slight change. We call such these basis, by a slight abuse of notation, still \textit{LT-bases} or of \textit{LT-type}. Under $q$-skew Howe duality, as explained in Section~\ref{sec-compare}, this basis turns out to be Kuperberg's basis of $\mathfrak{sl}_3$-webs. Moreover, it changes the underlying space $\Lambda_q^{\ell}(\bC_q^n)^{\otimes 3}$ to $\Lambda_q^{\bullet}(\bC_q^3)^{\otimes n}$.
  
\subsubsection*{The special quantum 2-algebras}

Khovanov and Lauda (Rouquier introduced independently similar notions~\cite{rou}) introduced diagrammatic $2$-categories ${\mathcal{U}}(\mathfrak{g})$ which categorify the integral version of the corresponding idempotented quantum groups~\cite{kl5}.

Cautis and Lauda~\cite{cala} defined diagrammatic $2$-categories ${\mathcal{U}}_Q(\mathfrak{g})$ with implicit scalers $Q$ consisting of $t_{ij}$, $r_i$ and $s_{ij}^{pq}$ which determine certain signs in the definition of the categorified quantum groups.
\vspace*{0.25cm}

In this section, we very briefly recall ${\mathcal{U}}(\mathfrak{sl}_n)={\mathcal{U}}_Q(\mathfrak{sl}_n)$. Much more can be found in the papers cited above. The scalars $Q$ are given by 
$t_{ij}=-1$ if $j=i+1$, $t_{ij}=1$ otherwise, $r_i=1$ and $s_{ij}^{pq}=0$. This corresponds precisely to the signed version in~\cite{kl5} and~\cite{kl4}. For simplicity we work with $\bC$ as an underlying field. 

\begin{defn}(\textbf{Khovanov-Lauda})\label{defn-KL}
The $2$-category $\Ucat$ is defined as follows.
\begin{itemize}
\item The objects in $\Ucat$ are the weights $\lambda \in  \bZ^{m-1} $.
\end{itemize}
For any pair of objects $\lambda$ and $\lambda^{\prime}$ in $\Ucat$, the hom category 
$\Ucat(\lambda,\lambda^{\prime})$ is the $\bZ$-graded, additive $\bC$-linear category consisting of the following data.
\begin{itemize}
\item Objects (or $1$-morphisms in $\Ucat$), 
that is finite formal sums of the form 
$\mathcal{E}_{\ui}{\idm}_{\lambda}\{t\}$ where $t\in \bZ$ is the grading shift 
and $\ui$ is a signed sequence such that 
$\lambda^{\prime}=\lambda+\sum_{a=1}^{l}\epsilon_ai_{a}^{\prime}$.
\item The morphisms or $2$-cells are graded, $\bC$-vector spaces generated by compositions of diagrams shown below. Here $\{k\}$ denotes a degree shift by $k$. Moreover, we use the two short hand notations $\alpha^{ij}=(\alpha_i,\alpha_j)$ and $\alpha^{\lambda i}=2\frac{(\lambda,\alpha_i)}{(\alpha_i,\alpha_i)}$.
\[
\phi_1=\xy
 (0,1)*{\includegraphics[width=09px]{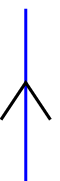}};
 (1.5,-4)*{\scriptstyle i};
 (3,1)*{\scriptstyle\lambda};
 (-5,1)*{\scriptstyle\lambda+\alpha_i};
 \endxy\,\hspace{8mm}
\,
\phi_2=\xy
 (0,1)*{\includegraphics[width=09px]{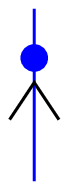}};
 (1.5,-4)*{\scriptstyle i};
 (3,1)*{\scriptstyle\lambda};
 (-5,1)*{\scriptstyle\lambda+\alpha_i};
 \endxy\,\hspace{8mm}
\,\phi_3=\xy
 (0,1)*{\includegraphics[width=25px]{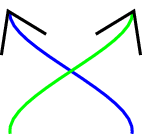}};
 (-5,-3)*{\scriptstyle i};
 (5,-3)*{\scriptstyle j};
 (5.5,0)*{\scriptstyle\lambda};
 \endxy\,\hspace{8mm}\,
 \phi_4=\xy
 (0,0)*{\includegraphics[width=25px]{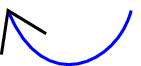}};
 (0,-3)*{\scriptstyle i};
 (5,0)*{\scriptstyle\lambda};
 \endxy\hspace{8mm}\,
 \phi_5=\xy
 (0,0)*{\includegraphics[width=25px]{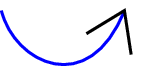}};
 (0,-3)*{\scriptstyle i};
 (5,0)*{\scriptstyle\lambda};
 \endxy
\]
with $\phi_1=\mathrm{id}_{\mathcal E_{i}\onel}$, $\phi_2\colon\mathcal E_{i}\onel\Rightarrow\mathcal E_{i}\onel\{\alpha^{ii}\}$, $\phi_3\colon\mathcal E_{i}\mathcal E_{j}\onel\Rightarrow\mathcal E_{j}\mathcal E_{i}\onel\{\alpha^{ij}\}$ and cups and caps $\phi_4\colon\onel\{\frac{1}{2}\alpha^{ii}+\alpha^{\lambda i}\}\Rightarrow\mathcal E_{i}\mathcal F_{i}\onel$ and $\phi_5\colon\onel\{\frac{1}{2}\alpha^{ii}-\alpha^{\lambda i}\}\Rightarrow\mathcal F_{i}\mathcal E_{i}\onel$. Moreover, we have diagrams of the form
\[
\psi_1=\xy
 (0,1)*{\includegraphics[width=09px]{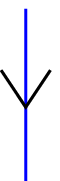}};
 (1.5,-4)*{\scriptstyle i};
 (3,1)*{\scriptstyle\lambda};
 (-5,1)*{\scriptstyle\lambda-\alpha_i};
 \endxy\,\hspace{8mm}
\,
\psi_2=\xy
 (0,1)*{\includegraphics[width=09px]{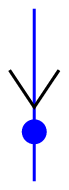}};
 (1.5,-4)*{\scriptstyle i};
 (3,1)*{\scriptstyle\lambda};
 (-5,1)*{\scriptstyle\lambda-\alpha_i};
 \endxy\,\hspace{8mm}
\,\psi_3=\xy
 (0,1)*{\includegraphics[width=25px]{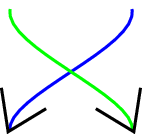}};
 (-5,-3)*{\scriptstyle i};
 (5,-3)*{\scriptstyle j};
 (5.5,0)*{\scriptstyle\lambda};
 \endxy\,\hspace{8mm}\,
 \psi_4=\xy
 (0,0)*{\includegraphics[width=25px]{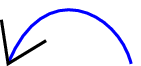}};
 (0,-3)*{\scriptstyle i};
 (5,0)*{\scriptstyle\lambda};
 \endxy\hspace{8mm}\,
 \psi_5=\xy
 (0,0)*{\includegraphics[width=25px]{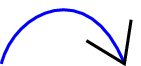}};
 (0,-3)*{\scriptstyle i};
 (5,0)*{\scriptstyle\lambda};
 \endxy
\]
with $\psi_1=\mathrm{id}_{\mathcal F_{i}\onel}$, $\psi_2\colon\mathcal F_{i}\onel\Rightarrow\mathcal F_{i}\onel\{\alpha^{ii}\}$, $\psi_3\colon\mathcal F_{i}\mathcal F_{j}\onel\Rightarrow\mathcal F_{j}\mathcal F_{i}\onel\{\alpha^{ij}\}$ and cups and caps $\psi_4\colon\mathcal F_{i}\mathcal E_{i}\onel\Rightarrow\onel\{\frac{1}{2}\alpha^{ii}+\alpha^{\lambda i}\}$ and $\psi_5\colon\mathcal E_{i}\mathcal F_{i}\onel\Rightarrow\onel\{\frac{1}{2}\alpha^{ii}-\alpha^{\lambda i}\}$.
\end{itemize}
The relations on the $2$-morphisms are 
those of the signed version in~\cite{kl5} and~\cite{kl4}. The convention for reading these diagrams is from right to left and bottom to top. The $2$-cells should satisfy several relation which we will not recall here. Details can be for example found in Cautis and Lauda's paper~\cite{cala}.
\end{defn}

\subsubsection{The cyclotomic KL-R algebras}
In this subsection, we recall the definition of the cyclotomic KL-R algebras, 
due to Khovanov and Lauda~\cite{kl1},~\cite{kl3} and, independently, 
to Rouquier~\cite{rou}. We also very shortly recall Hu and Mathas graded cellular basis~\cite{hm} for these algebras of type $A$. 

Let $\Lambda$ be a dominant $\mathfrak{sl}_n$-weight, 
$V_{\Lambda}$ the irreducible $\U$-module of highest 
weight $\Lambda$ and $P_{\Lambda}$ the set of weights in $V_{\Lambda}$.
 
\begin{defn}(\textbf{Khovanov-Lauda, Rouquier})\label{defn-cyclKLR}
The \textit{cyclotomic KL-R algebra} $R_{\Lambda}$ is 
the subquotient of $\Ucat$ 
defined by the subalgebra of all diagrams with only downward 
oriented strands and right-most region labeled $\Lambda$ and
modded out by the ideal generated by all diagrams of the form
\begin{align}\label{eq-cyclorel}
\xy
(0,0)*{\includegraphics[width=75px]{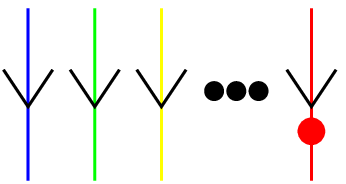}};
(-11,-8.5)*{\scriptstyle i_1};
(-5.5,-8.5)*{\scriptstyle i_2};
(0,-8.5)*{\scriptstyle i_3};
(12,-8.5)*{\scriptstyle i_n};
(14.5,-5)*{\scriptstyle \overline{\lambda}_n};
(14,0)*{\scriptstyle\lambda};
\endxy.
\end{align}
This relation is known as the \textit{cyclotomic relation}.
\end{defn}
Note that 
\[
R_{\Lambda}=\bigoplus_{\mu\in P_{\Lambda}} R_{\Lambda}(\mu),
\]
where $R_{\Lambda}(\mu)$ is the subalgebra generated by all diagrams 
whose left-most region is labeled $\mu$. It is not clear from the definition what the dimension of $R_{\Lambda}$ is. Moreover, it is not clear that $R_{\Lambda}$ is finite dimensional, but Brundan and Kleshchev proved that $R_{\Lambda}$ is indeed finite dimensional~\cite{bk1}.

Note that we mod out by relations involving dots on the last strand, rather 
than the first strand as in~\cite{kl1} to make it consistent with our other conventions in this paper.

It is worth noting that if we draw pictures for the KL-R algebra, then we do not need orientations anymore, that is pictures will look like
\[
\xy(0,0)*{\includegraphics[scale=0.75]{res/figs/basis/HM-strings}}\endxy\;\;\text{or}\;\;\xy(0,0)*{\includegraphics[scale=0.75]{res/figs/basis/HM-strings1}}\endxy
\]
\vspace*{0.15cm}

In~\cite{hm} Hu and Mathas gave a graded cellular basis of the KL-R algebra $R_{\Lambda}$ (or $R^n_{c(S)}$ in their notation). We do not recall their definition here, since it is not short and we give an alternative definition using foams later. The reader is encouraged to take a look in their great paper. We call their basis \textit{HM-basis}. We only mention that their basis is parametrized by $\vec{\lambda}\in\Lambda^+(n,c(S),l)$, i.e. all $l$-multipartitions of $c(S)$ for all suitable $n,l$, and $\vec{T},\vec{T}^{\prime}\in\mathrm{Std}(\vec{\lambda})$, i.e. standard $l$-multitableaux in the KL-R sense \textit{without} repeating entries. They denote their basis by
\[
\{\psi_{\vec{T},\vec{T}^{\prime}}\mid \vec{\lambda}\in \mathfrak{P}_{c(S)}\text{ and }\vec{T},\vec{T}^{\prime}\in\mathrm{Std}(\vec{\lambda})\},
\]
where $\mathfrak{P}_{c(S)}$ is the set of all multipartitions of $c(S)$. Moreover, it is graded by
\[
\mathrm{deg}_{\mathrm{BKW}}(\psi_{\vec{T},\vec{T}^{\prime}})=\mathrm{deg}_{\mathrm{BKW}}(\vec{T})+\mathrm{deg}_{\mathrm{BKW}}(\vec{T}^{\prime}),
\]
where the degree is Brundan, Kleshchev and Wang's degree given in~\cite{bkw}, which we recall in~\ref{defn-degpart1}. We note that $n$ in our context will be the number of strands of a given web (see next section) and $c(S)$ is as in~\ref{eq-const}. It is worth noting that we restrict to the easiest case, i.e. the linear quiver over $\bC$, but much more is known about HM-basis, see for example~\cite{hm1} or~\cite{li1} for a version over $\bZ$.
\subsection{Webs and foams}\label{sec-webfoam}
We are going to recall the notions of $\mathfrak{sl}_3$-webs and foams in this section. Nothing here is new, i.e. the whole section is literally copied from~\cite{mpt} (up to some small changes) and the results are mostly from either~\cite{kk} or~\cite{kup} in the case of webs or~\cite{kh3} and~\cite{mv1} for the foams. As before, we only briefly recall the different notions and much more can be found in the papers mentioned above.
\vspace*{0.25cm}

In~\cite{kup}, Kuperberg describes the representation theory of 
${\mathbf U}_q(\mathfrak{sl}_3)$ using oriented trivalent graphs, possibly with 
boundary, called \textit{webs} or \textit{$\mathfrak{sl}_3$-webs}. Boundaries of webs
consist of univalent vertices (the ends of oriented edges), which we will 
usually put on a horizontal line (or various horizontal lines), called the \textit{cut-line}, and that we usually picture by a dotted line, e.g. such a web is shown below.
\begin{align}
\xy
   (0,0)*{\includegraphics[width=140px]{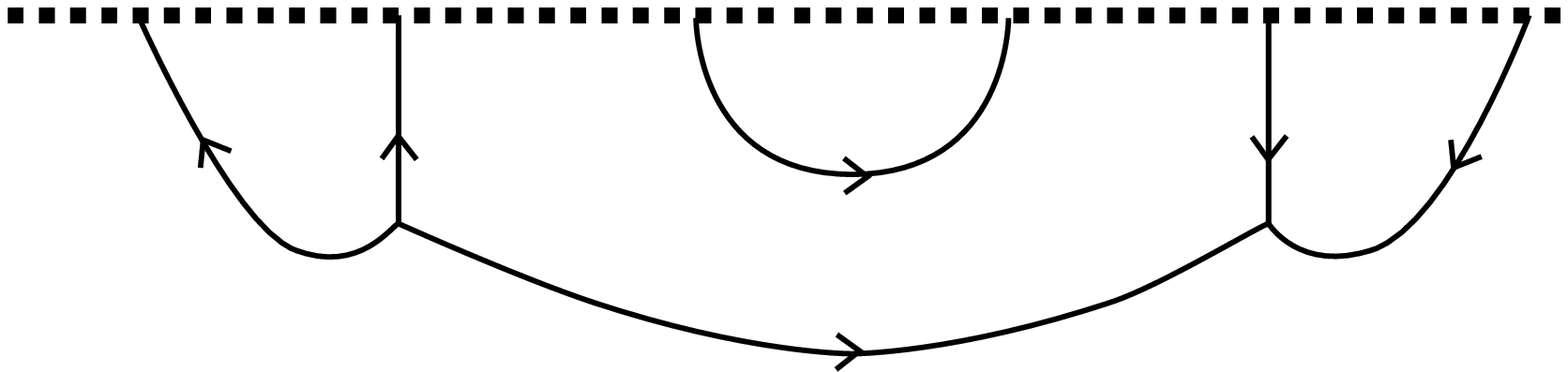}};
\endxy
\end{align}
We say that a web has $n$ free strands if the number of non-trivalent vertices 
is exactly $n$. In this way, the boundary of a web can be 
identified with a \textit{sign string} $S=(s_1,\ldots,s_n)$, with $s_i=\pm$, 
such that upward oriented boundary edges get a ``$+$'' and downward oriented 
boundary edges a ``$-$'' sign. Webs without boundary are called 
\textit{closed} webs. 

Fixing a boundary $S$, we can form the 
$\mathbb{C}(q)$-vector space $W_S$, spanned by all webs with boundary $S$, 
modulo the following set of local relations or Kuperberg relations~\cite{kup}.
\begin{align}
\label{eq:circle}
\xy(0,0)*{\includegraphics[width=20px]{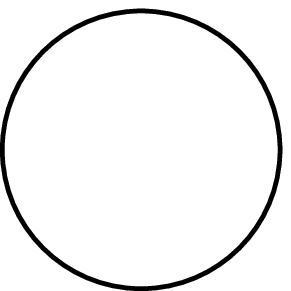}}\endxy\;\; &=\;\; [3]\\
\label{eq:digon}
\xy(0,0)*{\includegraphics[width=70px]{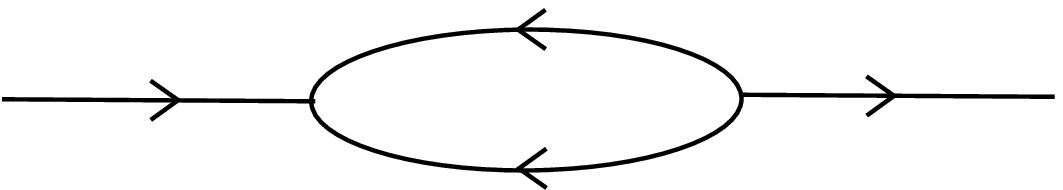}}\endxy\;\; &=\;\; [2]\;\; \xy(0,0)*{\includegraphics[width=50px]{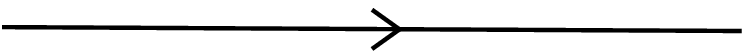}}\endxy\\
\label{eq:square}
\xy(0,0)*{\includegraphics[width=50px]{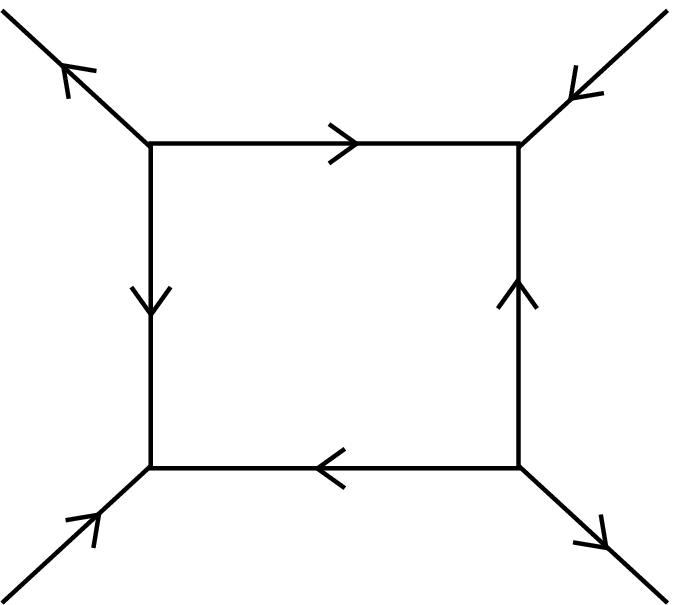}}\endxy\;\; &=\;\;\xy(0,0)*{\includegraphics[height=45px]{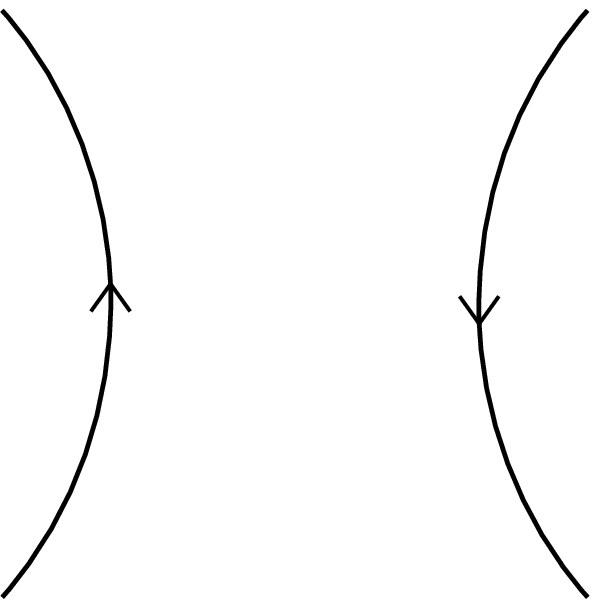}}\endxy + \;\;\xy(0,0)*{\reflectbox{\rotatebox{90}{\includegraphics[width=45px]{res/figs/basicsa/vert}}}}\endxy
\end{align}
Here 
\[
[a]=\frac{q^a-q^{-a}}{q-q^{-1}}=q^{a-1}+q^{a-3}+\cdots+q^{-(a-1)}\in\mathbb{N}[q,q^{-1}]
\]
denotes the \textit{quantum integer}. Note that we sometimes do not orient our webs. In all those cases the orientation does not matter and is therefore not pictured.

By abuse of notation, we will call all elements of $W_S$ webs. 
From relations~\eqref{eq:circle},~\eqref{eq:digon} and~\eqref{eq:square} it 
follows that any element in $W_S$ is a linear 
combination of webs with the same boundary and without circles, digons or 
squares. These are called \textit{non-elliptic webs}. 
As a matter of fact, the non-elliptic webs form a basis of $W_S$, which 
we call $B_S$. 
Therefore, we will simply call them \textit{basis webs} or \textit{Kuperberg's basis webs} or \textit{LT-basis webs} (we explain the name in Section~\ref{sec-compare}).
\vspace*{0.25cm}

Following Brundan and Stroppel's~\cite{bs1} notation for arc diagrams, 
we will write $w^*$ to denote the web obtained by reflecting a given 
web $w$ horizontally and reversing all 
orientations. Moreover, by $uv^*$, we mean the planar diagram containing the disjoint union of $u$ and 
$v^*$, where $u$ lies vertically above $v^*$. By $v^*u$, we shall mean the closed web obtained by glueing 
$v^*$ on top of $u$, when such a construction is possible (i.e. the number of free strands and orientations on the strands match). 
\begin{align}
\xy
 (0,0)*{\includegraphics[width=140px]{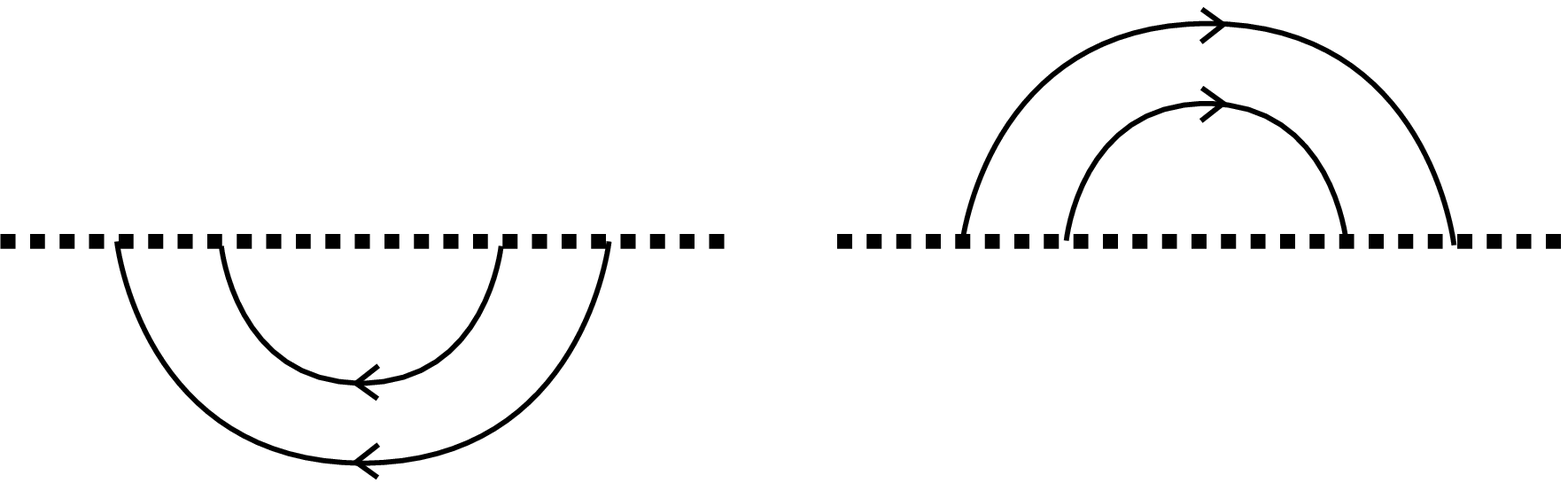}};
 (-13.5,2)*{w};
 (14,-2)*{w^*};
\endxy\;\;\;\;
\xy
 (0,0)*{\includegraphics[width=70px]{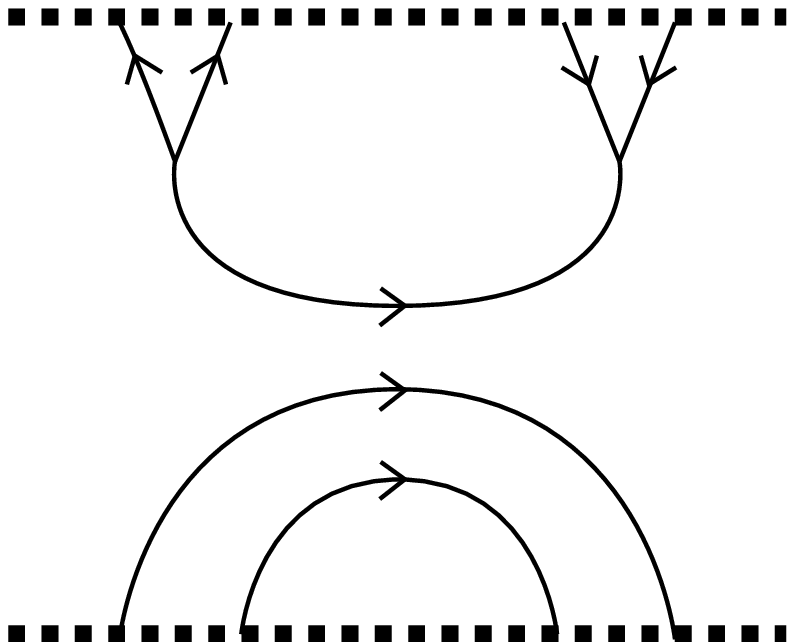}};
 (12,5)*{u};
 (12.5,-5)*{v^*};
\endxy\;\;\;\;\xy
 (0,0)*{\includegraphics[width=70px]{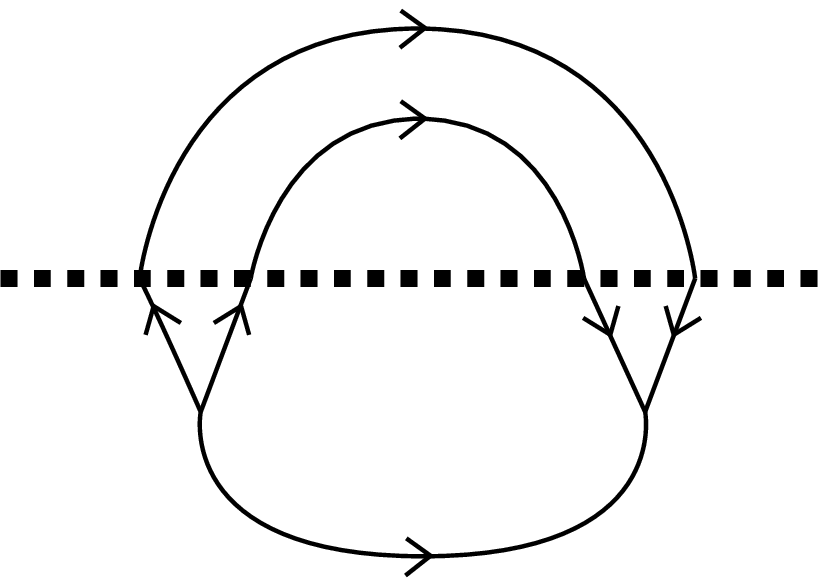}};
 (12,-4)*{u};
 (12.5,5)*{v^*};
\endxy
\end{align}
It should be noted that we usually use the symbol $w$ for any web, closed or with boundary, while we use the symbols $u,v$ for half-webs, that is $w=u^*v$ for suitable webs $u,v\in B_S$.
\vspace*{0.25cm}

To make the connection with the representation theory of 
${\mathbf U}_q(\mathfrak{sl}_3)$, we recall that 
a sign string $S=(s_1,\ldots,s_n)$ corresponds to 
\[
V_S=V_{s_1}\otimes \cdots\otimes V_{s_n},
\]
where $V_{+}$ is the fundamental representation and $V_{-}$ its dual. Both $V_+$ and $V_-$ have dimension three. In this interpretation, 
webs correspond to intertwiners and 
\[
W_S\cong \mathrm{Inv}_{{\mathbf U}_q(\mathfrak{sl}_3)}(V_S)\cong\mathrm{Hom}_{{\mathbf U}_q(\mathfrak{sl}_3)}(\bC,V_S).
\]
Therefore, the elements of $B_S$ give a basis of $\mathrm{Inv}_{{\mathbf U}_q(\mathfrak{sl}_3)}(V_S)$. 
However, this basis is not equal to the usual tensor basis nor to the dual canonical basis, see~\cite{kk}. Moreover, in~\cite{mpt} it was proved that Kuperberg's web basis and the dual canonical basis are related by a unitriangular change of basis matrix. The proof follows from categorification. We will reproduce this result without using categorification in Corollary~\ref{cor-uptri}.
\vspace*{0.25cm}

Kuperberg showed in~\cite{kup} (see also~\cite{kk}) that basis webs are indexed 
by closed weight lattice paths in the dominant Weyl chamber of 
$\mathfrak{sl}_3$. It is well-known that any path in the 
$\mathfrak{sl}_3$-weight lattice can be presented by a pair consisting of a 
sign string $S=(s_1,\ldots,s_n)$ and a 
\textit{state string} $J=(j_1,\ldots,j_n)$, with $j_i\in \{-1,0,1\}$ for all 
$1\leq i\leq n$. Given a pair $(S,J)$ representing a closed dominant path, 
a unique basis web (up to isotopy) is determined by a set of 
inductive rules called 
the \textit{growth algorithm}. We do not need it here explicitly, but it should be noted that Khovanov and Kuperberg showed in~\cite{kk}, that the growth algorithm is independent of the involved choices. A result that we need later in Section~\ref{sec-compare}. The growth algorithm gives the basis $B_S$.

Following Khovanov and Kuperberg in~\cite{kk}, we define a \textit{flow} or \textit{flow line} $f$ on a web $w$ to be an oriented subgraph that contains exactly two of the three edges 
incident to each trivalent vertex. At the boundary, the flow lines 
can be represented by a state string $J$. By convention, at the $i$-th 
boundary edge, we set $j_i=+1$ if the flow line is oriented downward, $j_i=-1$ if 
the flow line is oriented upward and $j_i=0$ there is no flow line. The same 
convention determines a state for each edge of $w$. We will also say that any flow $f$ that is compatible with a given state string $J$ on the boundary of $w$ \textit{extends} $J$. 

Given a web with a flow, denoted $w_f$, Khovanov and Kuperberg~\cite{kk} 
attribute a \textit{weight} to each trivalent vertex and each arc in $w_f$, 
as in Figures~\ref{weights} and~\ref{weights2}. 
The total weight of $w_f$ is by definition the sum of the 
weights at all trivalent vertices and arcs.
\begin{align}\label{weights}
  & \xy
 (0,0)*{\includegraphics[width=310px]{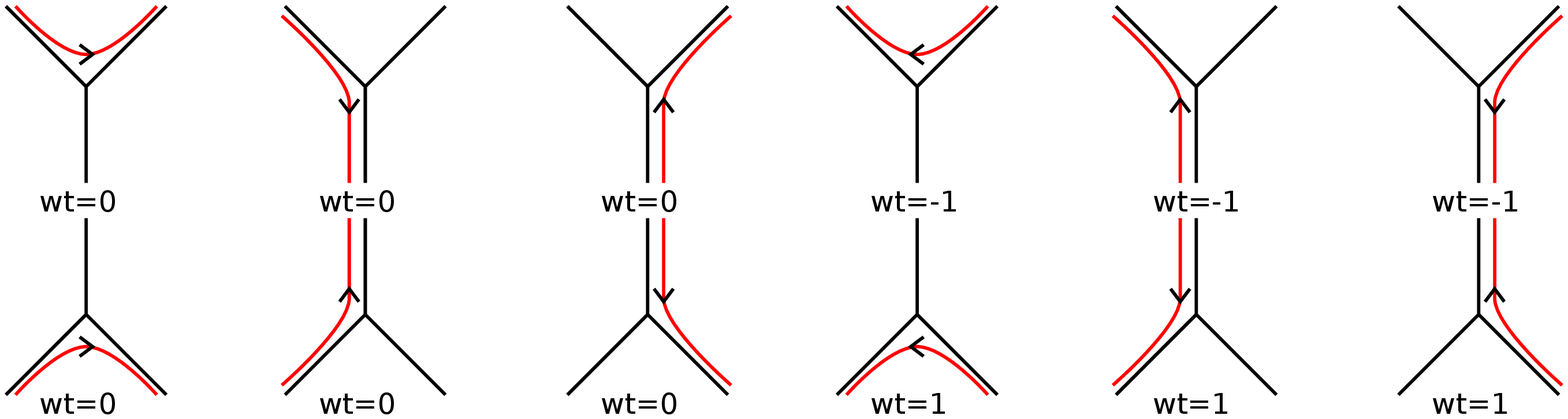}};
\endxy\\
   \label{weights2}
  & \xy
 (0,0)*{\includegraphics[width=310px]{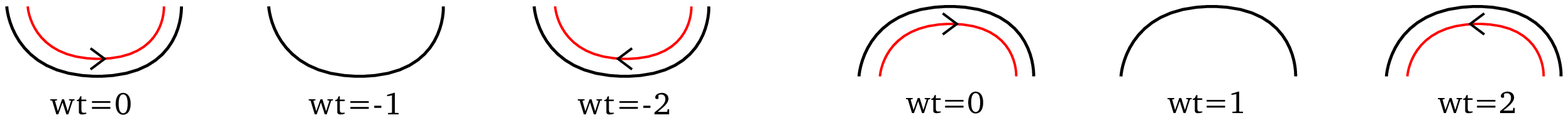}};
\endxy
\end{align}
For example, the following web has weight $-4$.
\begin{align}
\xy
 (0,0)*{\includegraphics[width=150px]{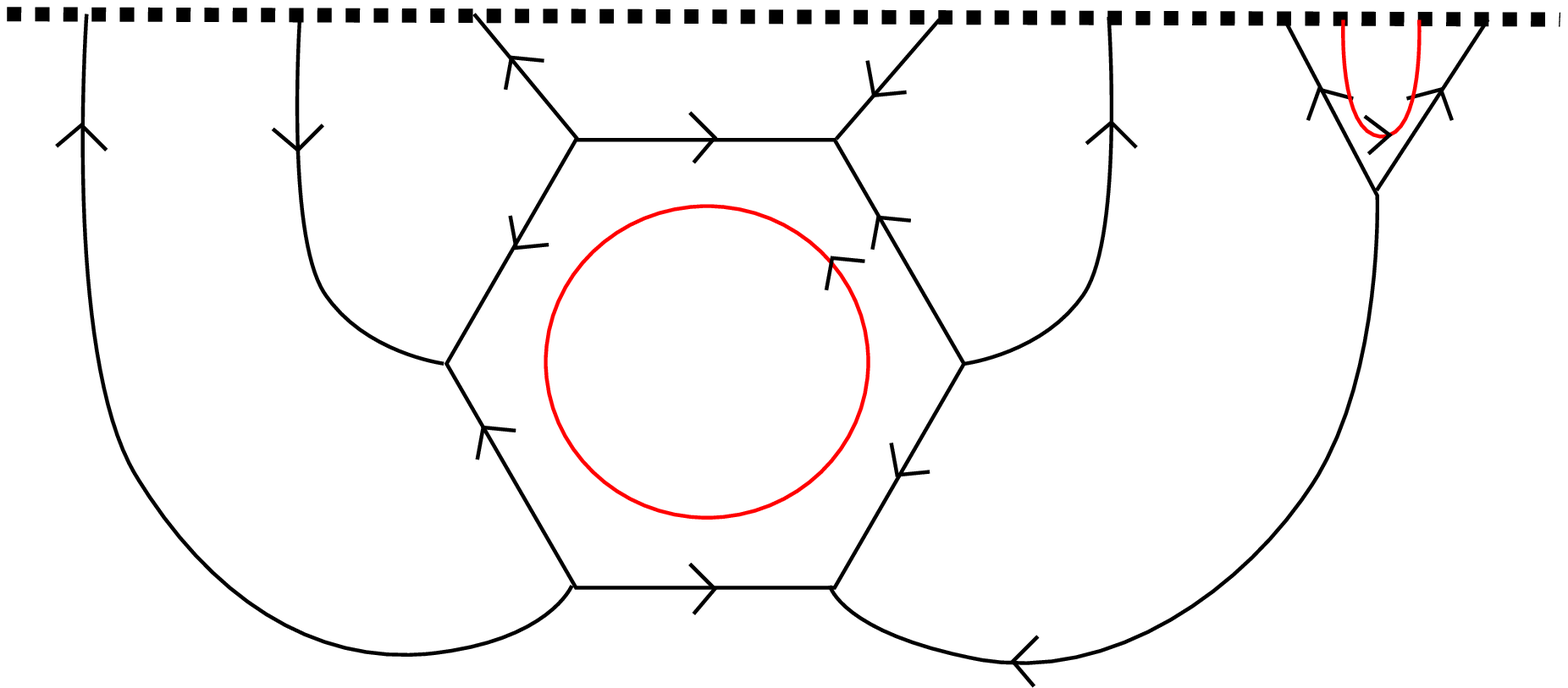}};
\endxy
\end{align}
We will show later in Section~\ref{sec-degflow} that Brundan, Kleshchev and Wang's degree of a multitableau has, after translating it using $q$-skew Howe duality, a very natural interpretation as (minus) the weight of a flow $f$ on a fixed web $w$.
\vspace*{0.25cm}

We choose arbitrary but fixed non-negative integers 
$n\geq 2$ and $\ell\leq n$, such that $d=3\ell\geq n$. Recall $\Lambda(n,d)=\left\{\lambda\in \mathbb{N}^n\mid \sum_{i=1}^n\lambda_i=d\right\}$.

Recall that a flow on the line on the boundary of a web, i.e. a pair of a sign string $S$ and a state string $J$, can be encoded using \textit{column-strict tableaux} $\mathrm{Col}(\lambda)$. Here $\lambda=(3^{\ell})\in \Lambda(n,d)$ and for any sign string $S=(s_1,\dots,s_n)$ the number $s_k$ appears with multiplicity one or two depending on the sign string $S$, see~\cite{mpt}. To be precise, in~\cite{mpt} we showed the following. We note that the restriction to the canonical flow gives semi-standard tableaux. Therefore, non-elliptic webs can be associated $1:1$ with semi-standard tableaux, a fact that was already known before (we note that some authors, e.g. Russell~\cite{ru}, use different reading conventions).
\begin{prop}\label{prop-tableauxflows}
There is a bijection between $\mathrm{Col}(\lambda)$ and the 
set of state strings $J$ such that there exists a web $w\in B_S$ and 
a flow $f$ on $w$ which extends $J$. 
\end{prop}
\begin{ex}\label{ex-flowline}
All of the three webs with flow
\[
\xy
(0,0)*{\includegraphics[height=.075\textheight]{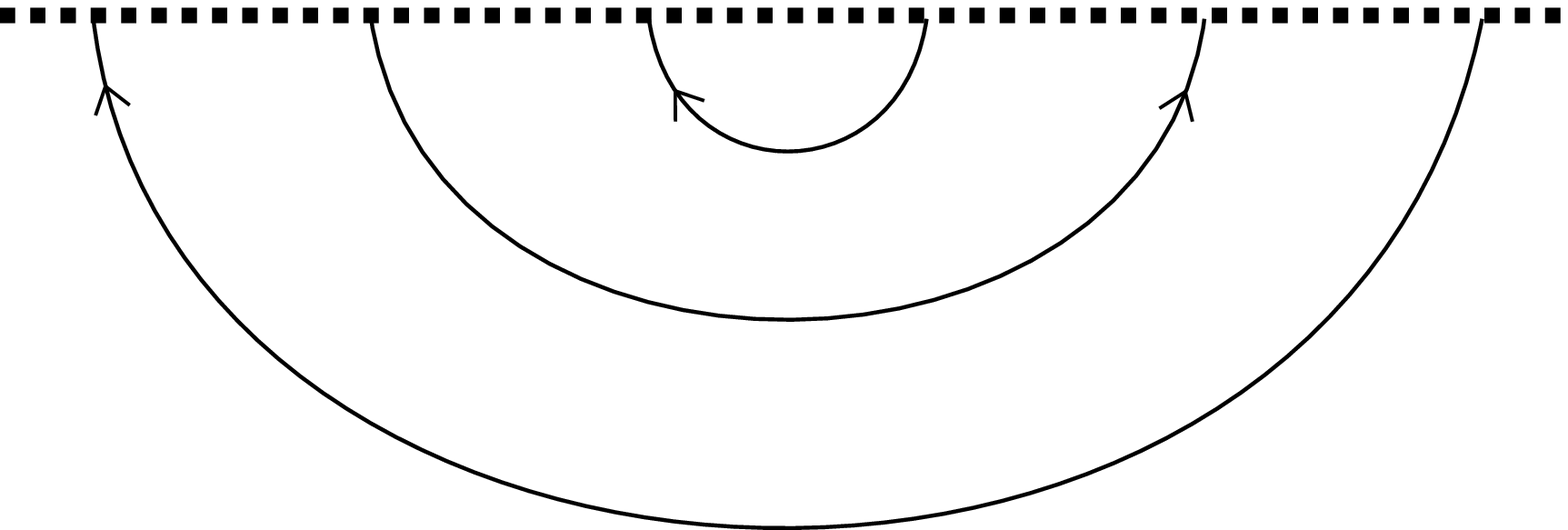}}
\endxy\;\;\;\;
\xy
(0,0)*{\includegraphics[height=.075\textheight]{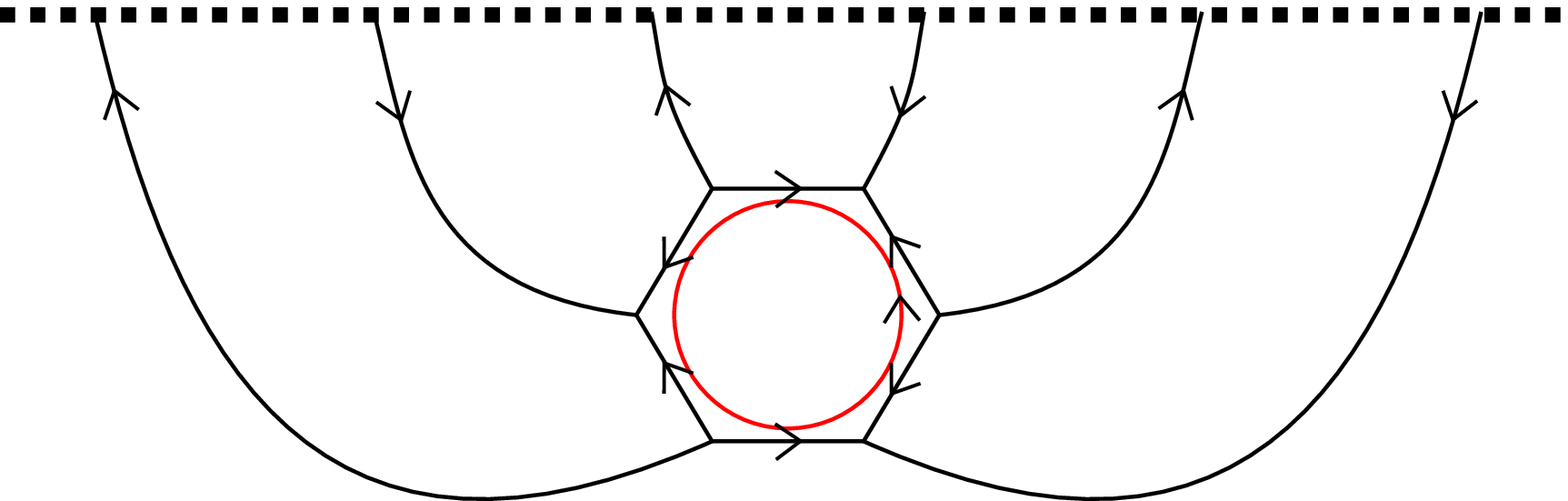}}
\endxy\;\;\;\;
\xy
(0,0)*{\includegraphics[height=.075\textheight]{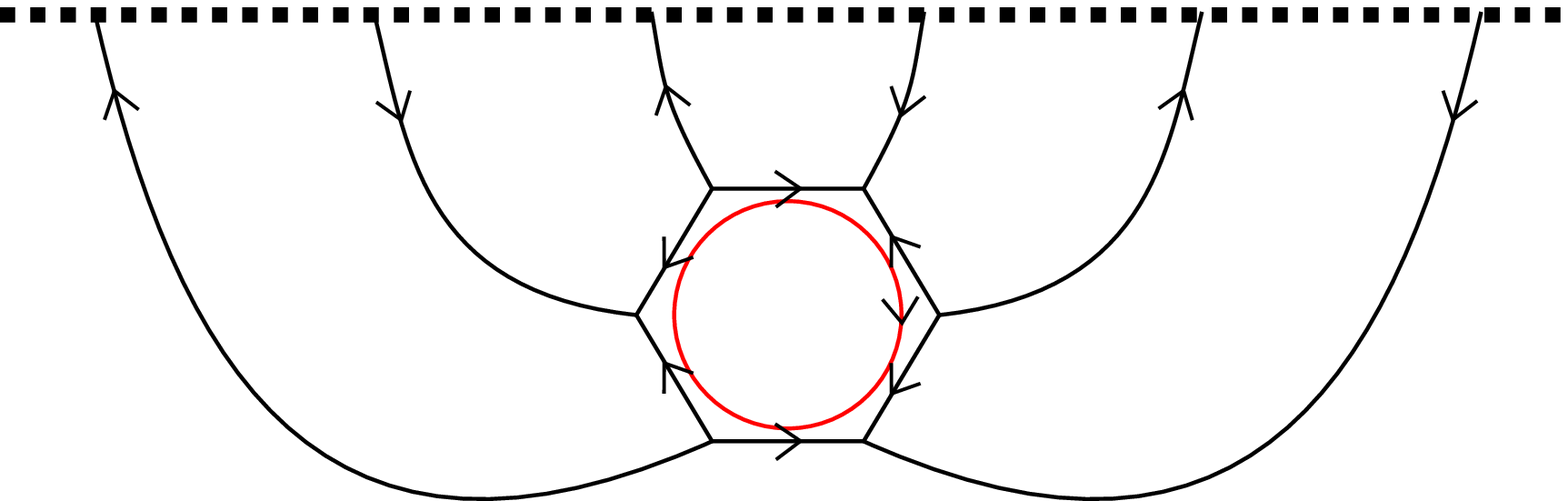}}
\endxy
\]
belong to the same tableau, i.e. the unique one for the sign string and state string pair $(S,J)$ with $S=(+,-,+,-,+,-)$ and $J=(0,0,0,0,0,0)$, that is
\[
\xy
 (0,11)*{
\begin{Young}
2&1&2 \cr
4&3&4 \cr
6&5&6 \cr
\end{Young}
}
\endxy
\]
\end{ex}
We need the following constant. For each $S=(s_1,\dots,s_n)$ with $3\ell=n_++2n_-$ define
\begin{equation}\label{eq-const}
c(s_k)=\begin{cases}k,&\text{if }s_k=+,\\2k,&\text{if }s_k=-,\end{cases}\;\text{ and }c(S)=c(s_1)+\dots+c(s_k)-\frac{3}{2}\ell(\ell+1).
\end{equation}
Moreover, Khovanov and Kuperberg defined a special flow on basis webs, called the \textit{canonical flow}. See~\cite{kk}. We do not need it here in much detail, but it is important to note that the pair of a sign string and state string $(S,J)$ for a web $w$ under the identification with its corresponding tableau gives a semi-standard tableau. See~\cite{mpt} for more details. We denote usually a web $w$ with its unique canonical flow by $w_c$.
\vspace*{0.25cm}

For another connection to representation theory, recall the following. Let $e^{\pm}_{-1,0,+1}$ be the standard basis of $V_{\pm}.$ Given $(S,J)$, let 
\[
e^S_J=e^{s_1}_{j_1}\otimes\cdots\otimes e^{s_n}_{j_n}
\] 
be the elementary tensor. Khovanov and Kuperberg proved the following  
result (Theorem 2 in~\cite{kk}) which we will reprove later. Note that we work with $q$ instead of $v=-q^{-1}$ as in~\cite{kk}.
\begin{thm}\label{thm:upptriang}(\textbf{Khovanov-Kuperberg})
Given $(S,J)$, we have 
\[ 
w^S_J=e^S_J+\sum_{J^{\prime}<J}c(S,J,J^{\prime})e^S_{J^{\prime}}
\]
for some coefficients $c(S,J,J^{\prime})\in\mathbb{N}[v,v^{-1}]$, where the 
state strings $J$ and $J^{\prime}$ are ordered lexicographically.
\end{thm}
\vspace*{0.15cm}

We shortly review the category called $\foamt$ of $\mathfrak{sl}_3$-foams 
introduced by Khovanov in~\cite{kh3} (these can be seen as a $\mathfrak{sl}_3$ version of Bar-Natan's $\mathfrak{sl}_2$-cobordism~\cite{bn2}). It is worth noting that Blanchet has proposed in~\cite{bla1} a slightly different foam category and it seems to be easier to work out the signs following his approach (as for example in~\cite{lqr}). But we, old-school as we are, do not use his setting here. Moreover, we only need the graded version in this paper. So we do not recall for example Gornik's filtered version here. For more details see~\cite{gornik}.

We recall the following definitions as they appear in~\cite{mv1}. We note that the diagrams accompanying these definitions are taken, also, from~\cite{mv1}.

A \textit{pre-foam} is a cobordism with singular arcs between two webs. 
A singular arc in a pre-foam $U$ is the set of points of $U$ which have a neighborhood 
homeomorphic to the letter Y 
times an interval. Note that singular arcs are disjoint. 
Interpreted as morphisms, we read pre-foams from bottom to top by convention. Thus, pre-foam composition consists of placing one 
pre-foam on top of the other. The orientation of the singular arcs is, by convention, as in 
the diagrams below (called the \textit{zip} and the 
\textit{unzip} respectively).
\begin{equation*}
\figins{-20}{0.5}{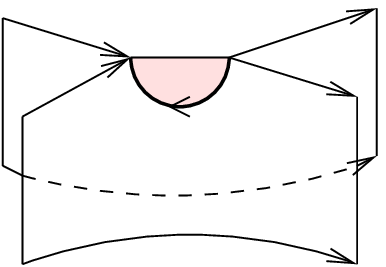}
\mspace{35mu}\mspace{35mu}
\figins{-20}{0.5}{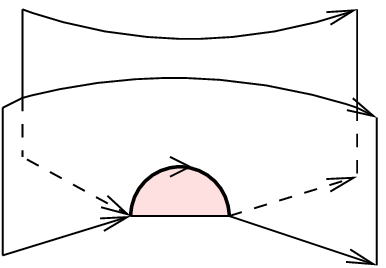}\ 
\end{equation*}
We allow pre-foams to have dots that can move freely about the facet on which they 
belong, but we do not allow a dot to cross singular arcs. 

By a \textit{foam}, we mean a formal $\mathbb{C}$-linear combination of 
isotopy classes of pre-foams modulo 
the ideal generated by the set of relations $\ell=(3D,NC,S,\Theta)$ and 
the \textit{closure relation}, 
as described below. 
\begin{gather*}
\tag{3D}\label{eq:3d}
\figins{-7}{0.25}{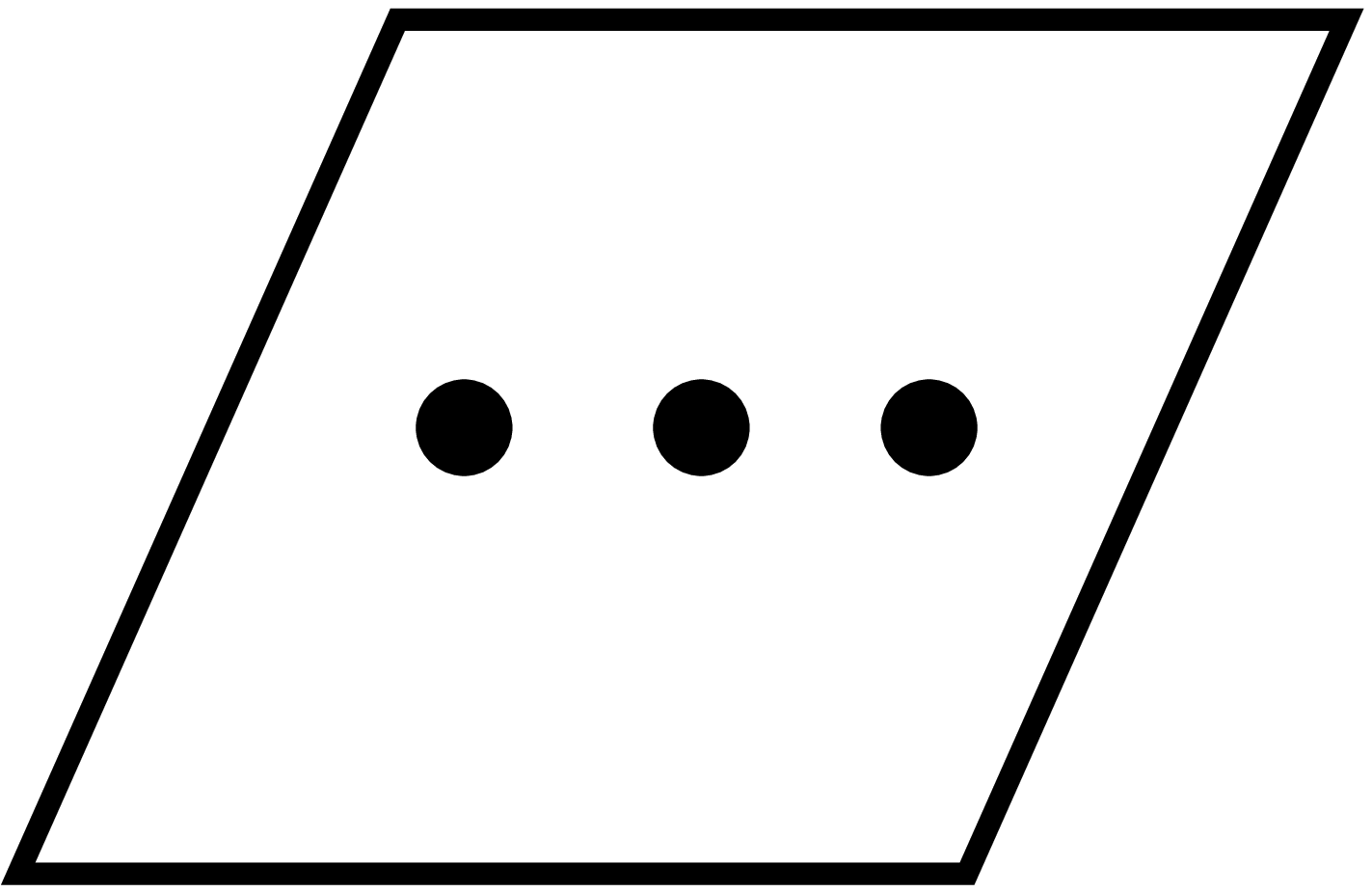}
= 0
\\[1.5ex]\displaybreak[0]
 \figwhins{-22}{0.65}{0.26}{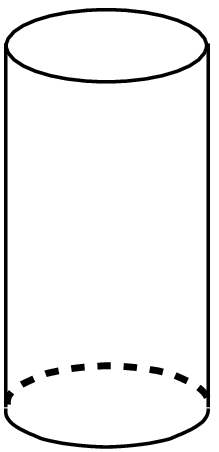}=
-\figwhins{-22}{0.65}{0.26}{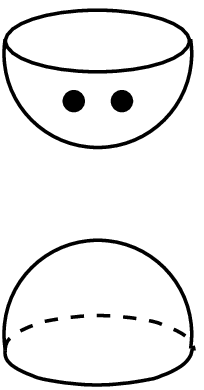}
-\figwhins{-22}{0.65}{0.26}{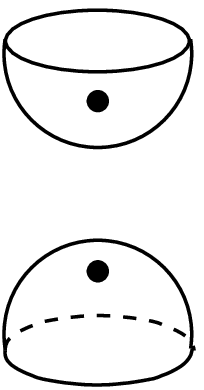}
-\figwhins{-22}{0.65}{0.26}{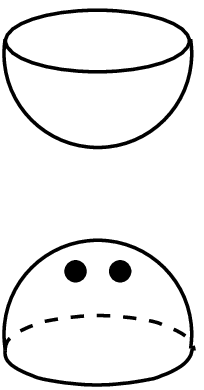}
\tag{NC}\label{eq:cn}
\\[1.5ex]\displaybreak[0]
\figins{-8}{0.3}{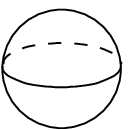}=
\figins{-8}{0.3}{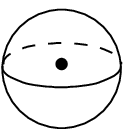}=0,\quad
\figins{-8}{0.3}{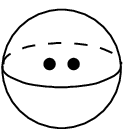}=-1
\tag{S}
\\[1.5ex]\displaybreak[0]
\labellist
\small\hair 2pt
\pinlabel $\alpha$ at 3 33
\pinlabel $\beta$ at -3 17
\pinlabel $\delta$ at 3 5
\endlabellist
\centering
\figins{-10}{0.4}{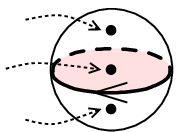} 
=\begin{cases}
\ \ 1, & (\alpha,\beta,\delta)=(1,2,0)\text{ or a cyclic permutation}, \\ 
-1, & (\alpha,\beta,\delta)=(2,1,0)\text{ or a cyclic permutation}, \\ 
\ \ 0, & \text{else}.
\end{cases}
\tag{$\Theta$}\label{eq:theta}
\end{gather*}
\begin{quote}
The \textit{closure relation}, i.e. any $\mathbb{C}$-linear combination of foams with the same boundary,
is equal to zero iff any way of capping off these foams with a 
common foam yields a $\mathbb{C}$-linear combination of closed foams 
whose evaluation is zero. 
\end{quote}

The relations in $\ell$ imply some identities (for detailed proofs see~\cite{kh3}). We only recall a few here that we need later on. More can be found in~\cite{kh3}.
\begin{align}
\figins{-20}{0.6}{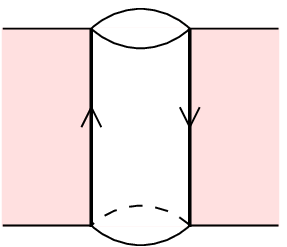}
& = 
\figins{-26}{0.75}{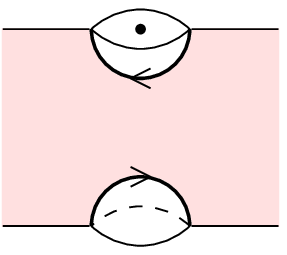}-
\figins{-26}{0.75}{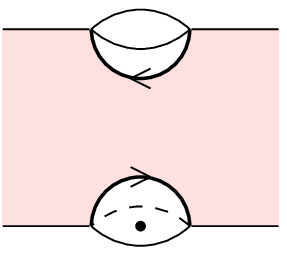}
\tag{DR}\label{eq:dr}
\\[1.2ex]\displaybreak[0]
\figins{-28}{0.8}{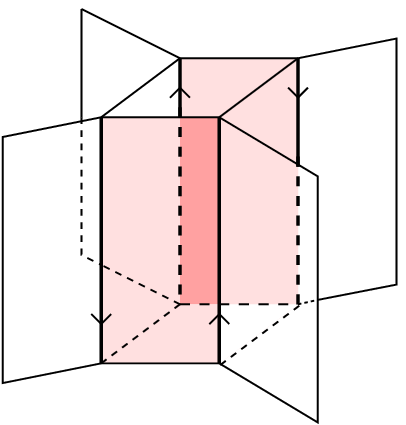}
&=
-\ \figins{-28}{0.8}{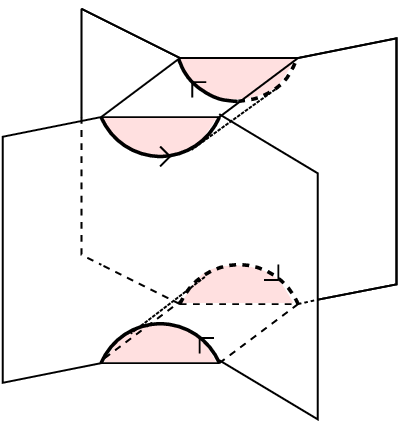}
-\figins{-28}{0.8}{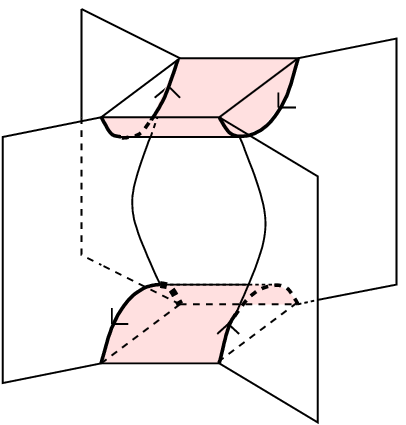}
\tag{SqR}\label{eq:sqr}
\end{align}
\begin{equation}\tag{Dot Migration}\label{eq:dotm}
\begin{split}
\figins{-22}{0.6}{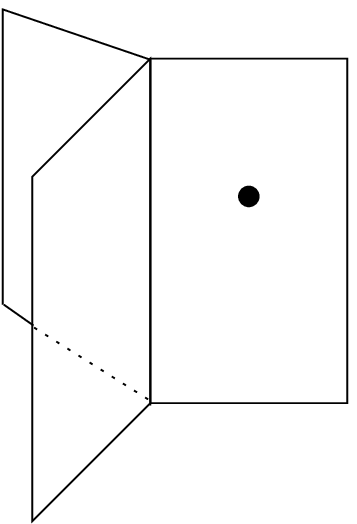}
\,+\,
\figins{-22}{0.6}{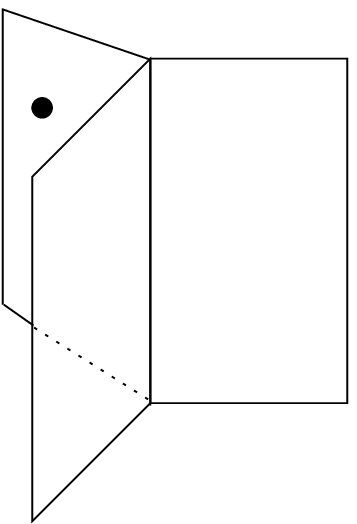}
\,+\,
\figins{-22}{0.6}{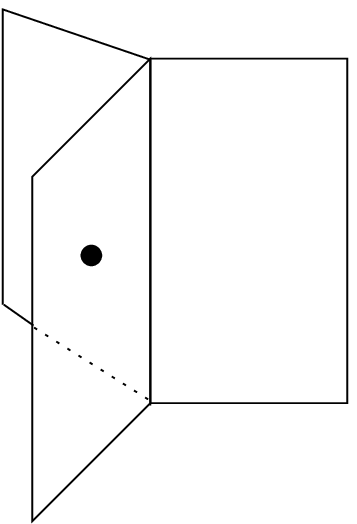}
\,= \figins{-22}{0.6}{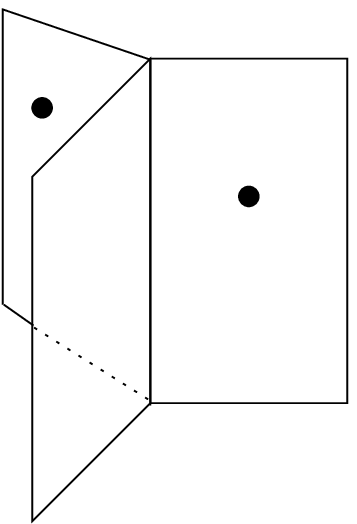}
\,+\,
\figins{-22}{0.6}{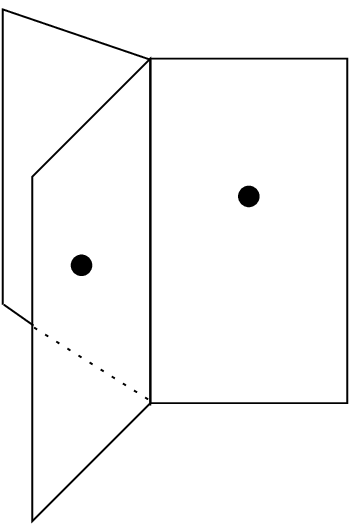}
\,+\,
\figins{-22}{0.6}{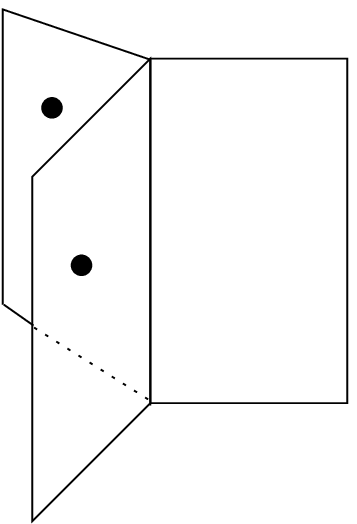}
\, = \figins{-22}{0.6}{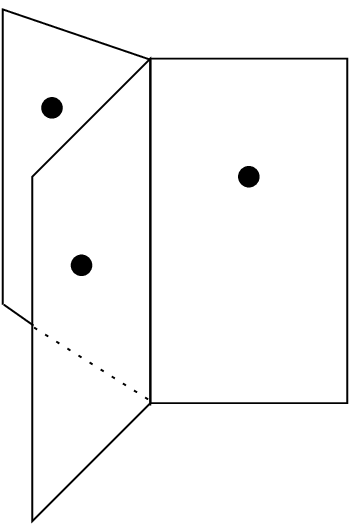}\, =0
\end{split}
\end{equation}
\begin{defn}
Let $\foamt$ be the category whose 
objects are webs $\Gamma$ lying inside a horizontal strip in 
$\mathbb{R}^2$, which is bounded by the lines $y=0,1$ containing the boundary points of $\Gamma$. 
The morphisms of $\foamt$ are $\mathbb{C}$-linear combinations of foams lying inside the horizontal strip bounded by $y=0,1$ times the unit interval. We require that the vertical boundary of each foam is a set (possibly empty) of vertical lines.
\end{defn}
The \textit{$q$-degree} of a foam $F$ is defined as 
\[
\mathrm{deg}_q(F) = \chi(\partial F) - 2\chi(F) + 2d+b,
\] 
where $\chi$ denotes the Euler characteristic, $d$ is 
the number of dots and $b$ is the number of vertical boundary 
components. This makes $\foamt$ into a graded category. We show later, using $q$-skew Howe duality, that the $q$-degree $\mathrm{deg}_q$ is Brundan, Kleshchev and Wang's degree $\mathrm{deg}_{BKW}$.
\begin{defn}\cite{kh3} \label{foamhom} \textbf{(Foam homology)} Given a web 
$w$ the \textit{foam homology} of $w$ is the complex vector space, 
$\F(w)$, spanned by all foams  
\[U\colon\emptyset \to w\]
in $\foamt$.
\end{defn}
\begin{rem}\label{rem-dim}
The complex vector space $\F(w)$ is graded by 
the $q$-degree on foams and has $q$-rank 
$\langle w\rangle$, where $\langle w\rangle $ is the \textit{Kuperberg bracket}. See~\cite{kh3} or~\cite{mv1} for details.
\end{rem}
\subsection{The $\mathfrak{sl}_3$-web algebra}\label{sec-webalg}
\setcounter{subsubsection}{1}
We are going to recall the definition of the \textit{$\mathfrak{sl}_3$-web algebra} $K_S$ as given in~\cite{mpt}. For the rest of the section let $S$ denote a fixed sign string of length $n$. 
\begin{defn}
\textbf{($\mathfrak{sl}_3$-web algebra)}\label{defn:webalg} For $u,v\in B_S$, we define 
\[
{}_{u}K_{v}=\F(u^*v)\{n\},
\] 
where $\{n\}$ denotes a grading shift upwards in degree by $n$.

The \textit{$\mathfrak{sl}_3$-web algebra} $K_S$ is defined by 
\[
K_S=\bigoplus_{u,v\in B_S}{}_{u}K_{v}.
\]
 
The multiplication on $K_S$ is defined by taking 
\[
{}_{u}K_{v_1}\otimes {}_{v_2}K_{w} \to 
{}_{u}K_{w}
\]
to be zero, if $v_1\ne v_2$, and by the map to be defined as in~\cite{mpt} 
if $v_1=v_2=v$.
\end{defn}
We use the following lemma throughout the whole paper.
\begin{lem}
\label{lem:webalgaltern}
For any $u,v\in B_S$, we have a 
grading preserving isomorphism 
\[
\foamt(u,v)\cong{}_{u}K_{v}.
\] 

Using this isomorphism, the multiplication  
\[
{}_{u}K_{v}\otimes {}_{v^{\prime}}K_{w}\to 
{}_{u}K_{w}
\]
corresponds to the composition 
\[
\foamt(u,v)\otimes\foamt(v^{\prime},w)\to\foamt(u,w),
\]
if $v=v^{\prime}$, and is zero otherwise. 
\end{lem}
\begin{proof}
See~\cite{mpt}.
\end{proof}
\begin{rem}
Lemma~\ref{lem:webalgaltern} shows that $K_S$ is an associative, unital algebra. Moreover, completely similar as in~\cite{mpt}, the algebra $K_S$ is a graded Frobenius algebra of Gorenstein parameter $n$. We do not need this fact in this paper, so the interested reader can find the details in~\cite{mpt}.
 
Note that for any $u\in B_S$, the identity $1_u\in\foamt(u,u)$ defines an idempotent. 
We have 
\[
1=\sum_{u\in B_S} 1_u\in K_S.
\]
Alternatively, one can see $K_S$ as a category whose 
objects are the elements in $B_S$ such that the module of morphisms 
between $u\in B_S$ and $v\in B_S$ is given by $\foamt(u,v)$. 

Moreover, there is a graded, linear involution on $K_S$ denoted
\begin{equation}\label{eq-invo}
{}^*\colon K_S\to K_S
\end{equation}
that reflects the foams along the xy-plane and reorient the edges afterwards. For example
\[
\left(\xy(0,1.25)*{\includegraphics[scale=0.75]{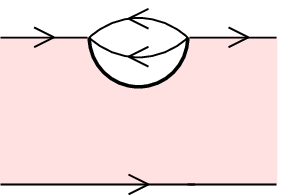}}\endxy\right)^*=\xy(0,-1.75)*{\reflectbox{\includegraphics[scale=0.75,angle=180]{res/figs/basicsb/digonem.eps}}}\endxy
\]
\end{rem}
\begin{rem}\label{rem-basis}
The algebra has a basis given by the so-called \textit{face removing algorithm} as explained in~\cite{tub3}. We do not know much about this basis except that it is easy to compute. We do not need it here and do not recall it, but it is worth noting that it is possible to prove Theorem~\ref{thm-foambasis} ``by hand'', i.e. use induction on the number of faces and show that the corresponding moves are ``almost'' face removing moves.
\end{rem}
\vspace*{0.15cm}

We will now briefly recall an instance of $q$-skew Howe duality, the so-called \textit{foamation functor}. We will refer to the application of this functor on the level of webs and foams, depending on the context, by saying \textit{``using foamation''} or by abuse of language \textit{``by $q$-skew Howe duality''}. The hope that the reader does not get confused.

We note that such a functor was independently studied by Lauda, Queffelec and Rose~\cite{lqr} with a slightly different convention than the one we recall from~\cite{mpt}.
\vspace*{0.25cm}

Let us denote by $K_S\text{-}\mathrm{(p)\textbf{Mod}}_{\mathrm{gr}}$ 
the category of all finite dimensional, (projective), unitary, 
graded $K_S$-modules. Recall that we need a slightly generalization of a sign string called \textit{enhanced sign string}. With a slight abuse of notation, we use $S$ for enhanced sign strings. Fix $d=3\ell$. 
\begin{defn} An \textit{enhanced sign string} is 
a sequence denoted by $S=(s_1,\ldots, s_n)$ with entries $s_i\in\{\circ,-1,+1,\times\}$, for 
all $i=1,\ldots n$. The corresponding weight $\mu=\mu_S\in\Lambda(n,d)$ is 
given by the rules
\[
\mu_i=
\begin{cases}
0,&\quad\text{if}\;s_i=\circ,\\
1,&\quad\text{if}\;s_i=1,\\
2,&\quad\text{if}\;s_i=-1,\\
3,&\quad\text{if}\;s_i=\times.
\end{cases}
\] 
Let $\Lambda(n,d)_3\subset \Lambda(n,d)$ be again the subset of weights 
with entries between $0$ and $3$.
\end{defn}
Define  
\[
K_{(3^{\ell})}=\bigoplus_{\mu_S\in\Lambda(n,n)_3} K_S
\]
and
\begin{align*}
{\mathcal W}^{(p)}_{(3^{\ell})}&=K_{(3^{\ell})}\text{-}\mathrm{(p)\textbf{Mod}}_{\mathrm{gr}}\\
&\cong 
\bigoplus_{\mu_S\in\Lambda(n,n)_3} 
K_S\text{-}\mathrm{(p)\textbf{Mod}}_{\mathrm{gr}}.
\end{align*}
\vspace*{0.005cm}

Recall that there exists a categorical $\Ucat$-action on ${\mathcal W}^{(p)}_{(3^{\ell})}$, called \textit{foamation}.
\begin{defn}\label{defn-foamation}(\textbf{$\mathfrak{sl}_3$-foamation})
We define a $2$-functor
\[
\Psi\colon\Ucat\to{\mathcal W}^{(p)}_{(3^{\ell})}
\]
called \textit{foamation}, in the following way. Recall that we read $\Ucat$ diagrams from right to left - this will correspond to reading webs from bottom to top.
\vspace*{0.13cm}

\textbf{On objects:} The functor is defined by sending an $\mathfrak{sl}_n$-weight $\lambda=(\lambda_1,\dots,\lambda_{n-1})$ to an object $\Psi(\lambda)$ of ${\mathcal W}_{(3^{\ell})}$ by
\[
\Psi(\lambda)=S,\; S=(a_1,\dots,a_n),\,a_i\in\{0,1,2,3\},\;\lambda_i=a_{i+1}-a_i,\;\sum_{i=1}^na_i=3\ell,
\]
and to zero if no such enhanced sign string $S$ exists. Note that, if such an $S$ exists, then it is unique.
\vspace*{0.13cm}

\textbf{On morphisms:} The functor on morphisms is defined to be the algebra homomorphism of $\bC$-algebras defined by glueing the following webs on top of the $\mathfrak{sl}_3$-webs in $W_{(3^{\ell})}$.
\begin{align*}
\onel&\mapsto
\;\xy
(0,0)*{\includegraphics[width=60px]{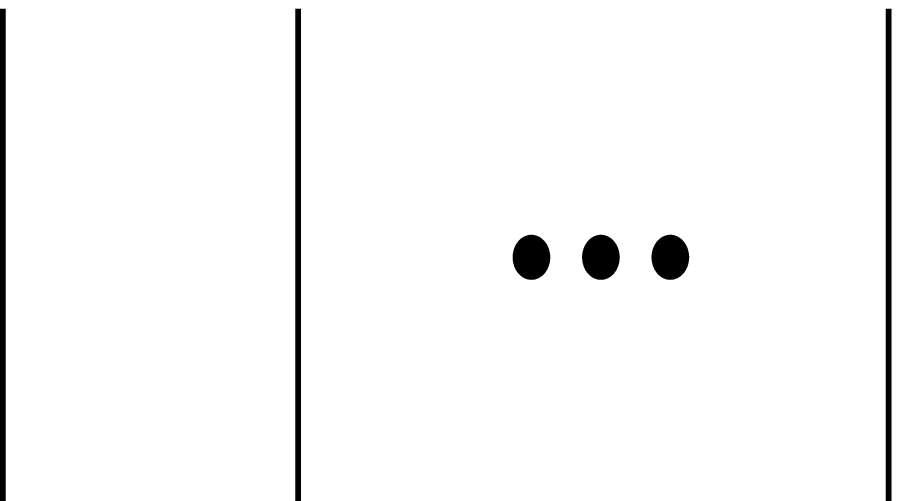}};
(-10,-7.3)*{\scriptstyle\lambda_1};
(-3,-7.3)*{\scriptstyle\lambda_{2}};
(10,-7.3)*{\scriptstyle\lambda_n};
\endxy
\\[0.5ex]
\mathcal E_{i}\onel&\mapsto
\;\xy
(-0,0)*{\includegraphics[width=150px]{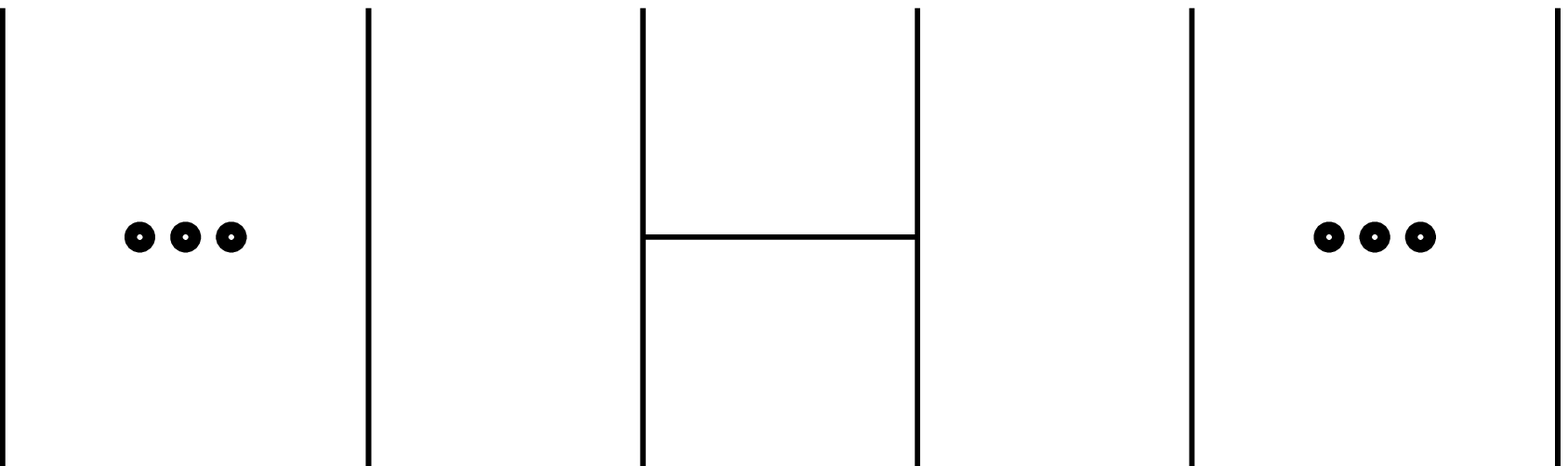}};
(-27,-9.3)*{\scriptstyle\lambda_1};
(-14,-9.3)*{\scriptstyle\lambda_{i-1}};
(-5,-9.3)*{\scriptstyle\lambda_{i}};
(5,-9.3)*{\scriptstyle\lambda_{i+1}};
(-5,9.3)*{\scriptstyle\lambda_{i}+1};
(5,9.3)*{\scriptstyle\lambda_{i+1}-1};
(14,-9.3)*{\scriptstyle\lambda_{i+2}};
(27,-9.3)*{\scriptstyle\lambda_n};
\endxy
\\[0.5ex]
\mathcal F_{i}\onel&\mapsto
\;\xy
(-0,0)*{\includegraphics[width=150px]{res/figs/basicsb/Hweb.eps}};
(-27,-9.3)*{\scriptstyle\lambda_1};
(-14,-9.3)*{\scriptstyle\lambda_{i-1}};
(-5,-9.3)*{\scriptstyle\lambda_{i}};
(5,-9.3)*{\scriptstyle\lambda_{i+1}};
(-5,9.3)*{\scriptstyle\lambda_{i}-1};
(5,9.3)*{\scriptstyle\lambda_{i+1}+1};
(14,-9.3)*{\scriptstyle\lambda_{i+2}};
(27,-9.3)*{\scriptstyle\lambda_n};
\endxy
\end{align*}
If the boundary values $a_i$ are $a_i\notin\{0,1,2,3\}$, then we send the morphism to zero. We use the convention that vertical edges labeled 1 are oriented upwards, 
vertical edges labeled 2 are oriented downwards and edges labeled 0 or 3 
are erased. Note that the orientation of the horizontal edges is uniquely determined by 
the orientation of the vertical edges, e.g. a H-move from an $E$ is always orientated from right to left and a H-move from a $F$ from left to right.
\vspace*{0.155cm}

\textbf{On $2$-cells:} Note our conventions. We only draw the most important part of the foams, 
omitting partial identity foams. We only mention the most important ones here. See~\cite{mpt} for the full list. Note that our drawing conventions in this paper are slightly different from those in~\cite{mpt}, i.e. we have rotated the foams by $\frac{\pi}{2}$ clockwise around the y-axis.
\begin{itemize}
\item[(1)] The bottom (top) part of the string diagram in $\Ucat$ is represented by the web at the bottom (top) of the foam.
\item[(2)] Glueing a $\Ucat$ diagram on top of another corresponds to gluing a foam on top of the other.
\item[(3)] In our string diagram conventions an upwards pointing arrow represents an $E$. Hence, on the level of webs, such H-moves have to point to the bottom left.
\item[(4)] We only draw some of the orientations in the pictures. The rest are fixed by the ones drawn following the conventions in~\cite{mpt}. They are not so important here, so we will not recall them.
\item[(5)] A facet is labeled 0 or 3 if and only if its boundary has edges 
labeled 0 or 3.
\item[(6)] All facets labeled 0 or 3 in the images below have to be erased, 
in order to get real foams.
\item[(7)] For any $\lambda>(3^{\ell})$, the image of 
the elementary morphisms below is taken to be zero, by convention.
\end{itemize}
In the list below, we always assume that $i<j$. 

{\allowdisplaybreaks
\begin{align*}
\xy
(4,1.5)*{\includegraphics[width=9px]{res/figs/basicsc/upsimple}};
(7,-4)*{{\scriptstyle i,\lambda}};
\endxy &\mapsto 
\;\xy
(0,0)*{\includegraphics[scale=0.5]{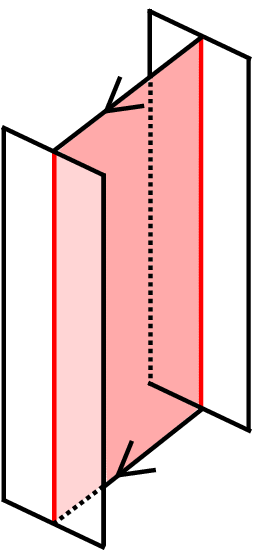}};
(0.5,5.5)*{\scriptstyle\lambda_i};
(10,11)*{\scriptstyle\lambda_{i+1}};
\endxy
\;\;\hspace*{0.05cm}\;\;
\xy
(4,1.5)*{\includegraphics[width=9px]{res/figs/basicsc/downsimple}};
(7,-4)*{{\scriptstyle i,\lambda}};
\endxy \mapsto 
\;\xy
(0,0)*{\includegraphics[scale=0.5]{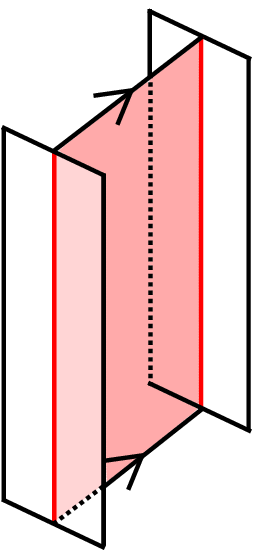}};
(0.5,5.5)*{\scriptstyle\lambda_i};
(10,11)*{\scriptstyle\lambda_{i+1}};
\endxy
\;\;\hspace*{0.05cm}\;\;
\xy
(4,1.5)*{\includegraphics[width=9px]{res/figs/basicsc/upsimpledot}};
(7,-4)*{{\scriptstyle i,\lambda}};
\endxy \mapsto 
\;\xy
(0,0)*{\includegraphics[scale=0.5]{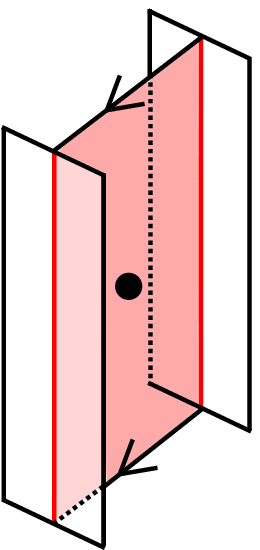}};
(0.5,5.5)*{\scriptstyle\lambda_i};
(10,11)*{\scriptstyle\lambda_{i+1}};
\endxy
\;\;\hspace*{0.05cm}\;\;
\xy
(4,1.5)*{\includegraphics[width=9px]{res/figs/basicsc/downsimpledot}};
(7,-4)*{{\scriptstyle i,\lambda}};
\endxy \mapsto 
\;\xy
(0,0)*{\includegraphics[scale=0.5]{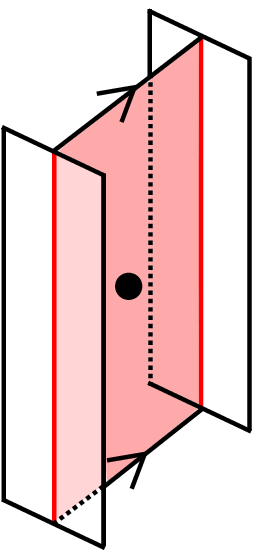}};
(0.5,5.5)*{\scriptstyle\lambda_i};
(10,11)*{\scriptstyle\lambda_{i+1}};
\endxy
\\
&\hspace*{3.5cm}\phantom{.}
\xy
(6,1.5)*{\includegraphics[width=20px]{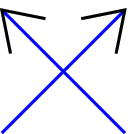}};
(12.5,-4)*{{\scriptstyle i,i,\lambda}};
\endxy  \mapsto
\;\;-\;\;
\xy
(0,0)*{\includegraphics[scale=0.5]{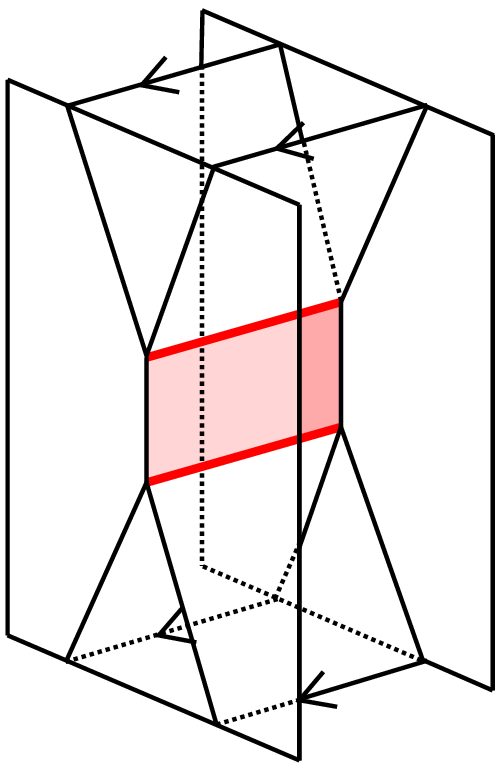}};
(4.5,9.5)*{\scriptstyle\lambda_i};
(16,13)*{\scriptstyle\lambda_{i+1}};
\endxy
\\
\xy
(6,1.5)*{\includegraphics[width=20px]{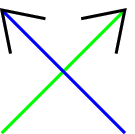}};
(12.5,-4)*{{\scriptstyle i,i+1,\lambda}};
\endxy &\mapsto
\;\;(-1)^{\lambda_{i+1}}\;\;
\xy
(0,0)*{\includegraphics[scale=0.5]{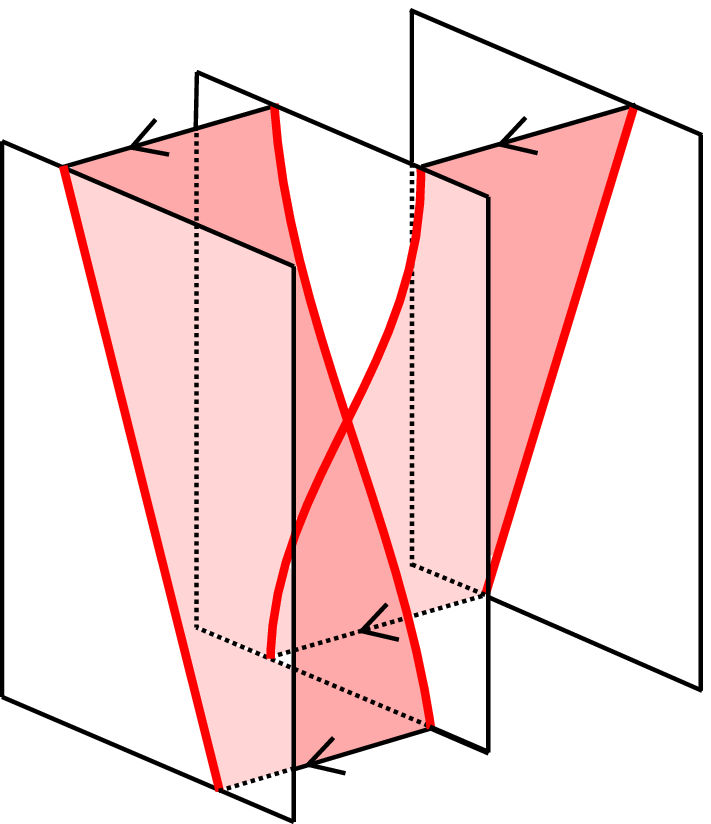}};
(-1,8)*{\scriptstyle\lambda_i};
(10,11.5)*{\scriptstyle\lambda_{i+1}};
(21,15)*{\scriptstyle\lambda_{i+2}};
\endxy
\;\;\hspace*{0.05cm}\;\;
\xy
(6,1.5)*{\includegraphics[width=20px]{res/figs/basicsc/upcross}};
(12.5,-4)*{{\scriptstyle i+1,i,\lambda}};
\endxy \mapsto
\;
\xy
(0,0)*{\includegraphics[scale=0.5]{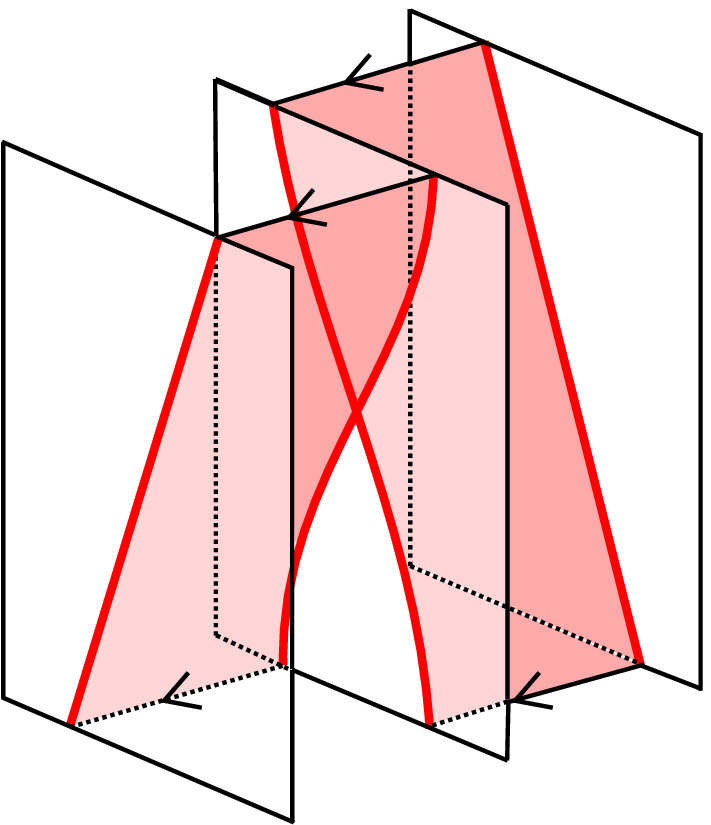}};
(-1.5,8)*{\scriptstyle\lambda_i};
(11,11)*{\scriptstyle\lambda_{i+1}};
(21,14.5)*{\scriptstyle\lambda_{i+2}};
\endxy
\\
\xy
(6,1.5)*{\includegraphics[width=20px]{res/figs/basicsc/upcross}};
(12.5,-4)*{{\scriptstyle i,j,\lambda}};
\endxy &\mapsto
\;
\xy
(0,0)*{\includegraphics[scale=0.5]{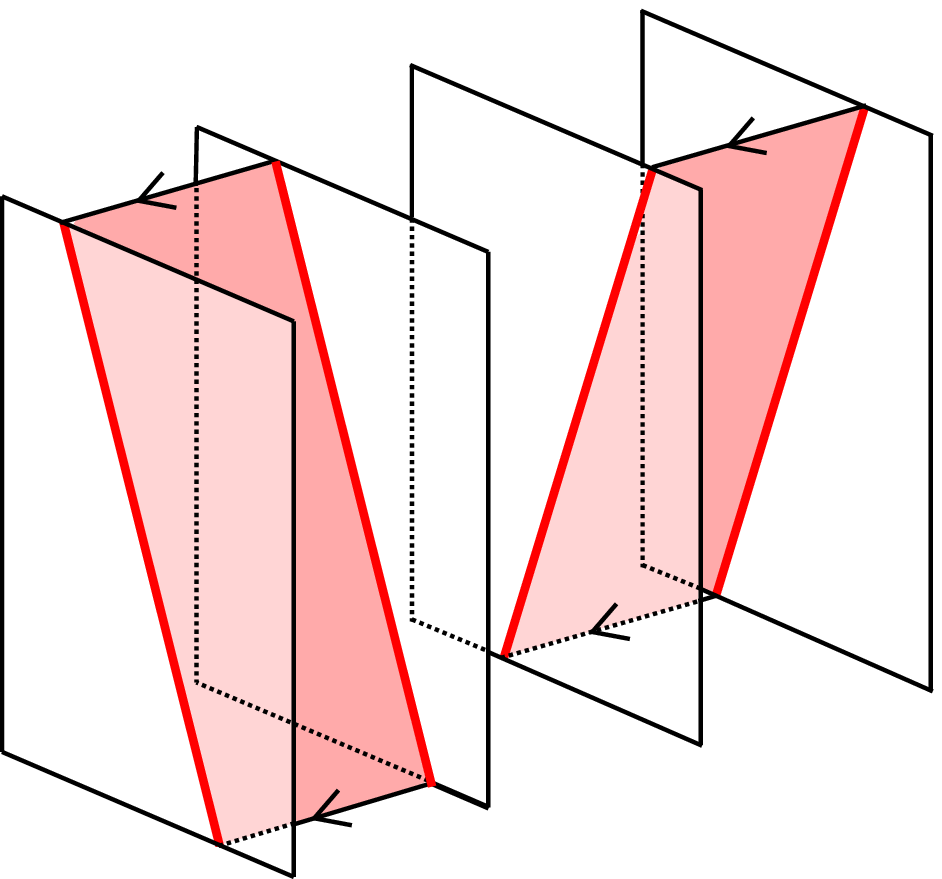}};
(-7.5,6.5)*{\scriptstyle\lambda_{j}};
(4,10)*{\scriptstyle\lambda_{j+1}};
(13.5,13)*{\scriptstyle\lambda_{i}};
(27,15.5)*{\scriptstyle\lambda_{i+1}};
\endxy
\;\;\hspace*{0.05cm}\;\;
\xy
(6,1.5)*{\includegraphics[width=20px]{res/figs/basicsc/upcross}};
(12.5,-4)*{{\scriptstyle j,i,\lambda}};
\endxy \mapsto
\;
\xy
(0,0)*{\includegraphics[scale=0.5]{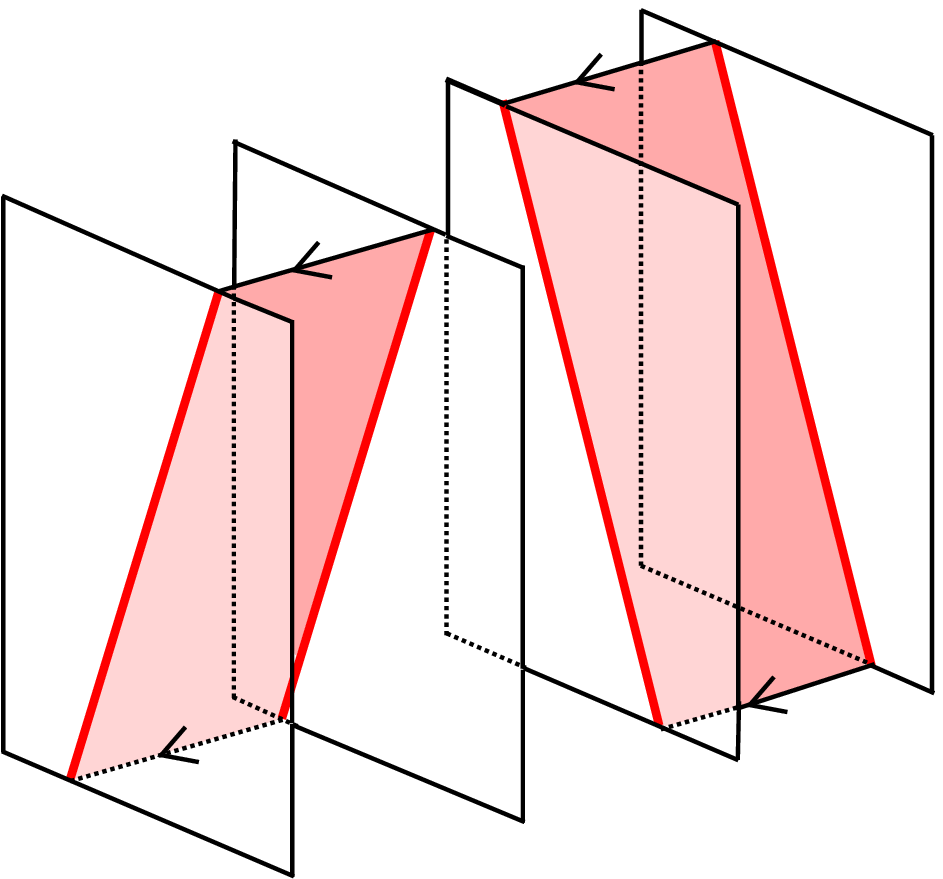}};
(-7.5,6.5)*{\scriptstyle\lambda_i};
(6,9)*{\scriptstyle\lambda_{i+1}};
(15.5,12)*{\scriptstyle\lambda_{j}};
(27,15.5)*{\scriptstyle\lambda_{j+1}};
\endxy
\\
\xy
(8,0)*{\includegraphics[width=25px]{res/figs/basicsc/rightcup}};
(12.5,-4)*{{\scriptstyle i,\lambda}};
\endxy &\mapsto
\;\xy
(0,0)*{\includegraphics[scale=0.5]{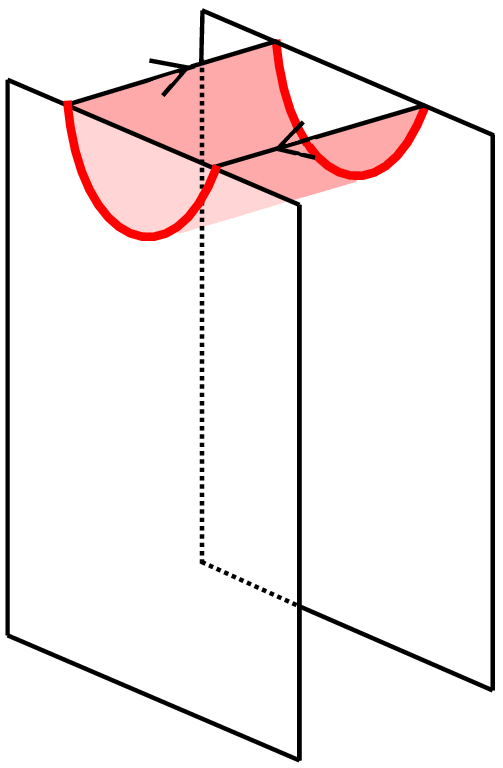}};
(4.5,9.5)*{\scriptstyle\lambda_i};
(16,13)*{\scriptstyle\lambda_{i+1}};
\endxy
\;\;\hspace*{0.05cm}\;\;
\xy
(8,0)*{\includegraphics[width=25px]{res/figs/basicsc/leftcup}};
(12.5,-4)*{{\scriptstyle i,\lambda}};
\endxy \mapsto
(-1)^{\lfloor\frac{\lambda_i}{2}\rfloor+\lceil\frac{\lambda_{i+1}}{2}\rceil}\;\xy
(0,0)*{\includegraphics[scale=0.5]{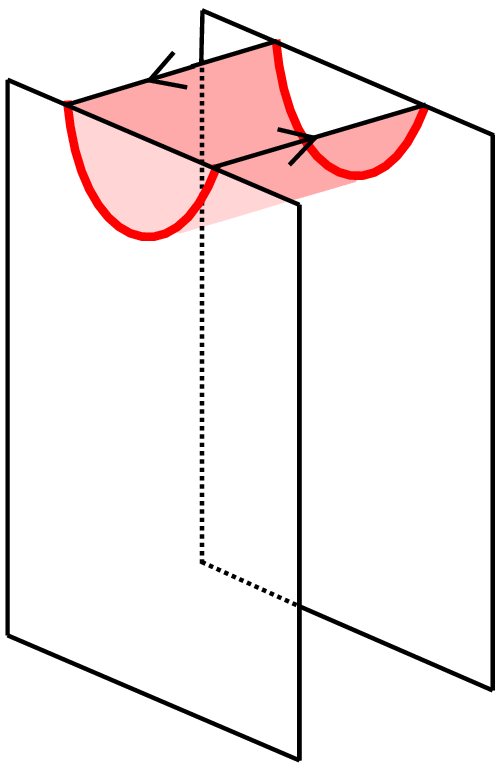}};
(4.5,9.5)*{\scriptstyle\lambda_i};
(16,13)*{\scriptstyle\lambda_{i+1}};
\endxy
\\
\xy
(8,0)*{\includegraphics[width=25px]{res/figs/basicsc/leftcap}};
(12.5,-4)*{{\scriptstyle i,\lambda}};
\endxy &\mapsto 
(-1)^{\lceil\frac{\lambda_i}{2}\rceil+\lfloor\frac{\lambda_{i+1}}{2}\rfloor}\;\xy
(0,0)*{\includegraphics[scale=0.5]{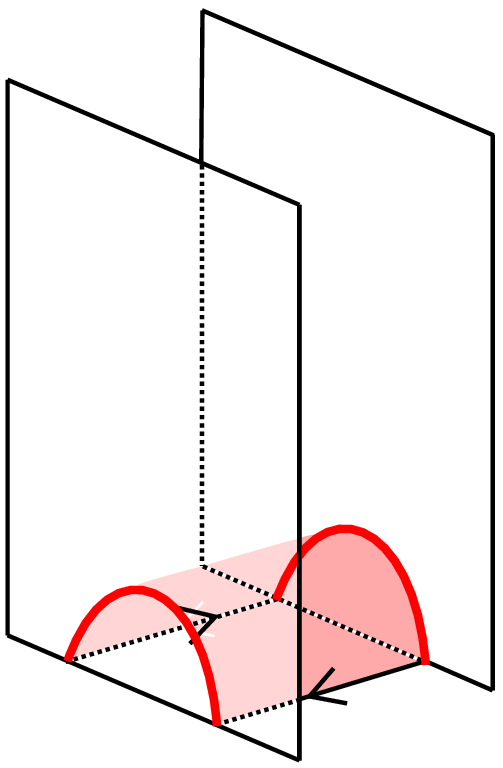}};
(4.5,9.5)*{\scriptstyle\lambda_i};
(16,13)*{\scriptstyle\lambda_{i+1}};
\endxy
\;\;\hspace*{0.05cm}\;\;
\xy
(8,0)*{\includegraphics[width=25px]{res/figs/basicsc/rightcap}};
(12.5,-4)*{{\scriptstyle i,\lambda}};
\endxy \mapsto
\;\xy
(0,0)*{\includegraphics[scale=0.5]{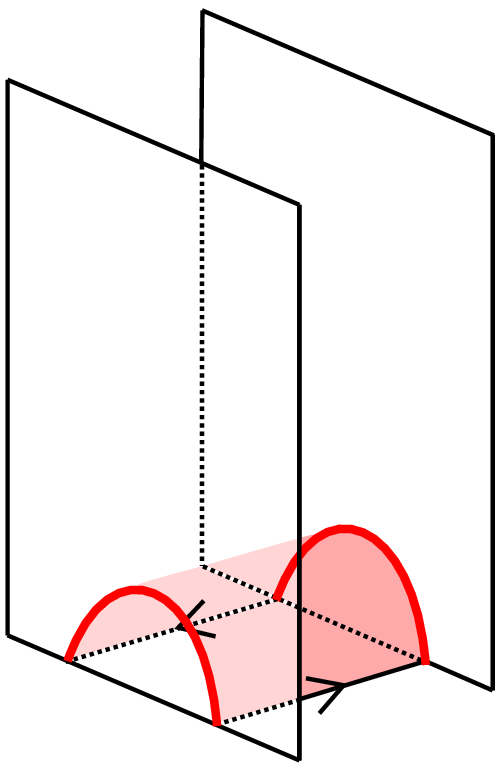}};
(4.5,9.5)*{\scriptstyle\lambda_i};
(16,13)*{\scriptstyle\lambda_{i+1}};
\endxy
\end{align*} 
}
\end{defn}
\begin{prop}\label{prop-actwelldef}
The foamation
\[
\Psi\colon\Ucat\to{\mathcal W}^{(p)}_{(3^{\ell})}
\]
is a well-defined strong $2$-representation in the sense of~\cite{cala}. 
\end{prop}
\begin{proof}
See~\cite{mpt} or with a slightly different setting~\cite{lqr}.
\end{proof}
It is worth noting that the foamation functor descends to the KL-R algebra by taking only downwards pointing arrows in a good way, see Cautis and Lauda~\cite{cala} or Lauda, Queffelec and Rose~\cite{lqr} for more details. There is also a version for the extended graphical calculus from~\cite{klms} and we conjecture that using it could make some of the later arguments easier.
\vspace*{0.25cm}

Moreover, for $N>3$ Mackaay and Yonezawa give a version of the foamation using matrix factorizations~\cite{my}. Conjecturally, there should also be a foamation using $N>3$-foams (see~\cite{msv}).
\section{The uncategorified picture}\label{sec-uncat}
\subsection{Web bases and intermediate crystal bases}\label{sec-compare}We are going to show in this section that Kuperberg's web basis $B_S$ is in fact an intermediate crystal basis, i.e. we show that it can be obtained from a Leclerc-Toffin like\footnote{The reader should compare our definition with the one of Leclerc-Toffin explained in Section~\ref{sec-quantum}.} algorithm using $q$-skew Howe duality, i.e. use the foamation $2$-functor $\Psi$ from Definition~\ref{defn-foamation} on the level of webs and let the divided powers act on the highest weight vector $v_{3^{\ell}}$.
\vspace*{0.25cm}

Throughout the whole section we assume that all partitions $\lambda$ are partitions of $3\ell$ for some $\ell$, i.e. $\lambda\in\Lambda^+(n,3\ell)$, and $S$ denotes a sign string of length $n$ such that $n_++2n_-=3\ell$, where $n_+,n_-$ denote the number of $+,-$ of $S$.
\begin{defn}\label{defn-LTgen}
Let $T\in\mathrm{Std}^s(\lambda)\subset\mathrm{Col}(\lambda)$ denote a semi-standard tableau. We associate to each such $T$ a string of divided powers of $F$ which we call the \textit{LT-generators} of $T$ and we denote it by
\[
\mathrm{LT}(T)=\prod_k F_{i_k}^{(j_k)}.
\]
The rule to obtain the string is as follows. We have two different kinds of rules, i.e. the \textit{usual} and the \textit{extraordinary} rules. We generate the string of $F$'s inductively using these rules.

Start with the empty product $\mathrm{LT}(T)_0=1$ and set $T_0=T$. Assume that we are in the $k$-th step. Search for the lowest entry of $T_k$ which is below its own row and not in a \textit{forbidden position}, i.e. a position such that lowering this entry by one does not lead to a semi-standard tableau any more.

Note that we use for simplicity $2$ for this entry in the following definition and we only picture the important parts of the tableau. As long as the tableau $T_k$ is not of the form
\[
T_k=\xy
 (0,0)*{
\begin{Young}
1&2& \cr
2&& \cr
3'&& \cr
\end{Young}
}\endxy\;\;\text{or}\;\;
T_k=\xy
 (0,0)*{
\begin{Young}
1&2& 3 \cr
2&& \cr
\end{Young}
}\endxy
\]
where the $3^{\prime}$ should indicate that there is at least one $3$ in the first column, we use the \textit{usual rules}, that is, replace all $j_k$ occurrences of $2$ below its own row by $1$ and obtain a new tableau
\[
T_{k+1}=T_k(2\mapsto 1\text{ for all }2\text{ below row }2)
\]
and a longer string $\mathrm{LT}(T)_{k+1}=\mathrm{LT}(T)_kF_{1_k}^{(j_k)}$. Otherwise, and this is the extraordinary step, search for the next higher, that is $>2$, value that appears below its own row and that is not only in the first column or in the first and second column and repeat this step with it instead of $2$.

We denote the \textit{length} of
\[
\mathrm{LT}(T)=\prod_k^{\ell(\mathrm{LT}(T))} F_{i_k}^{(j_k)}
\]
by $\ell(\mathrm{LT}(T))$ and the \textit{total length} by $\ell_{\mathrm{t}}(\mathrm{LT}(T))=\sum_k j_k$.
\end{defn}
Note that, since we do not allow infinite tableaux, there can not be an infinite string of extraordinary cases, i.e. a tableau like
\[
\xy
 (0,0)*{
\begin{Young}
1&2& \cr
2&3& \cr
3&4& \cr
4&5& \cr
5&& \cr
\end{Young}
};
(0,-13)*{\vdots}
\endxy
\]
does not appear. Hence, the inductive process terminates.
\begin{defn}\label{defn-LT}(\textbf{LT-algorithm for webs})
Let $T\in\mathrm{Std}^s(\lambda)\subset\mathrm{Col}(\lambda)$ denote a semi-standard tableau. The \textit{Leclerc-Toffin algorithm} for $\mathfrak{sl}_3$-webs is defined by applying the LT-generators to the highest weight vector $v_{3^{\ell}}$ using $q$-skew Howe duality, that is
\[
u(T)=\mathrm{LT}(T)v_{3^{\ell}}\in W_S.
\]

We use the notions \textit{length} and \textit{total length of non-elliptic webs} $\ell_{\mathrm{t}}(u)$ under the identification with their semi-standard tableaux (recall that each non-elliptic web $u$ can be identified with a semi-standard tableau $T_u$ using its unique canonical flow. See~\cite{mpt} for more details).
\end{defn}
\begin{ex}\label{ex-totlength}
The only non-elliptic webs whose total length $\ell_{\mathrm{t}}(u)\leq 3$ are either arcs
\[
\xy(0,0)*{\includegraphics[scale=0.5]{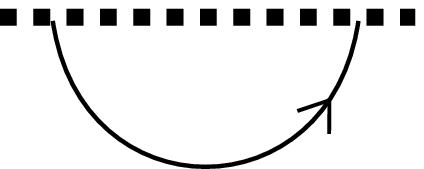}}\endxy\;\; T=\xy
 (0,0)*{
\begin{Young}
1 & 1 & 2\cr
\end{Young}}
\endxy\;\;\text{ and }\;\;\xy(0,0)*{\includegraphics[scale=0.5]{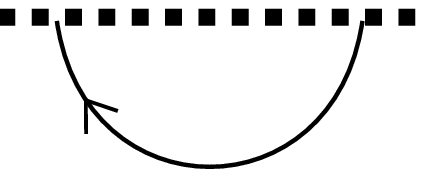}}\endxy\;\; T=\xy
 (0,0)*{
\begin{Young}
1 & 2 & 2\cr
\end{Young}}
\endxy
\]
with LT-generators $F_1$ and $F_1^{(2)}$, or theta webs
\[
\xy(0,0)*{\includegraphics[scale=0.5]{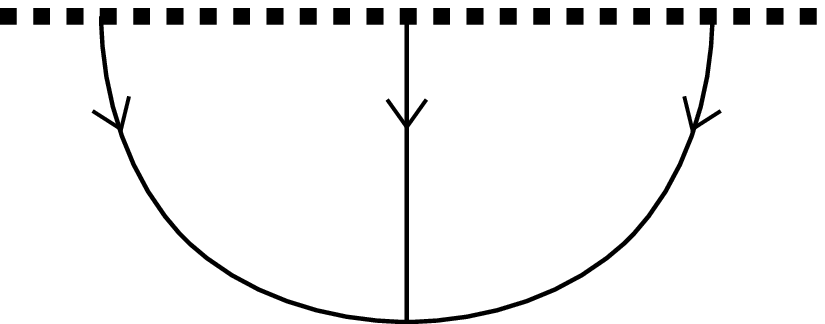}}\endxy\;\; T=\xy
 (0,0)*{
\begin{Young}
1 & 1 & 2\cr
2 & 3 & 3\cr
\end{Young}}
\endxy\;\;\text{ and }\;\;\xy(0,0)*{\includegraphics[scale=0.5]{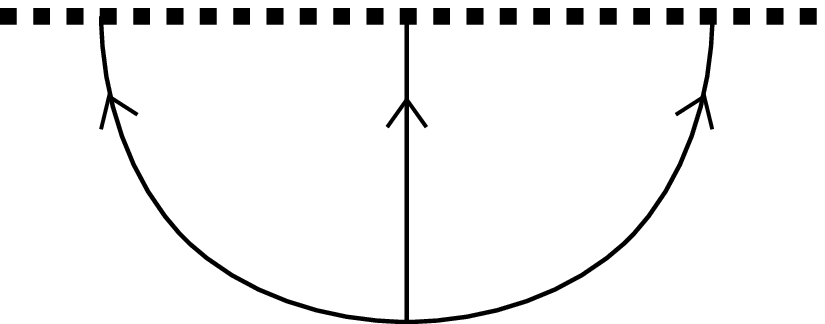}}\endxy\;\;T=\xy
 (0,0)*{
\begin{Young}
1 & 2 & 3\cr
\end{Young}}
\endxy
\]
with LT-generators $F_1F_2^{(2)}$ and $F_1F_2F_1$ respectively. Moreover, the first three are also the only examples of length $\ell(u)\leq 2$. The corresponding procedure under $q$-skew Howe duality for the third web is
\[
\xy
(0,0)*{\includegraphics[scale=0.65]{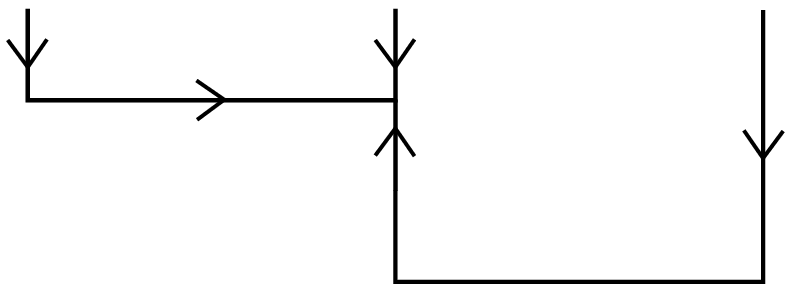}};
(-35,3)*{F_1};
(-35,-6.5)*{F_2^{(2)}};
(-35,-15.75)*{v_{3^2}=};
(-22.5,10)*{2};
(-22.5,-2)*{3};
(-22.5,-15)*{3};
(2,10)*{2};
(2,-2)*{1};
(2,-15)*{3};
(26,10)*{2};
(26,-2)*{2};
(26,-15)*{0};
\endxy
\]
Note that we read form bottom to top when we apply $F_1F_2^{(2)}$ to the highest weight vector $v_{3^2}$ and we label the boundary points from left to right, i.e. an $F_i$ acts between the boundary point number $i$ and $i+1$.
\end{ex}
\begin{prop}
\label{prop-sl3bases}
For any $T\in \mathrm{Std}^s(3^{\ell})$ we have that Khovanov and Kuperberg's growth 
algorithm and Leclerc and Toffin's algorithm give the same 
non-elliptic $\mathfrak{sl}_3$-web. 
\end{prop}
\begin{proof}
In the growth algorithm there are often cases in which 
one can choose between several steps. Khovanov and Kuperberg show 
in Lemma 1 in~\cite{kk} that the final web does not depend on these choices. 
We are going to show that the Leclerc-Toffin algorithm makes 
a particular choice for each step in the growth algorithm, which proves the 
proposition. 

Suppose we are given a semi-standard tableau and that we have 
already done some steps in the Leclerc-Toffin algorithm. The list below 
gives all possibilities for the next step. All cases in which 
one of the entries appears three times are similar, so we have only 
given one such example at the bottom of our list. To simplify matters we have 
called the entries $0,1,2$ and $3$ etc., instead of $k-1,k,k+1$ and $k+2$ etc., and we 
have only drawn the relevant part of the tableau. The nodes whose contents 
is not relevant for our argument are left empty. Moreover, if a number $3$ appears, then we only assume for simplicity that this number has to appear with the same multiplicity in the same columns, but is allowed to be above the shown row.

The usual cases are (beware that in the following pictures we use the numbers $+1,0,-1$ from Khovanov-Kuperberg's growth algorithm, that is they do not indicate the orientation of the webs)   
\[
\begin{array}{lllllll}
\xy
 (0,0)*{
\begin{Young}
1&2& \cr
& & \cr
\end{Young}
};
 (0,-15)*{
\begin{Young}
1&1& \cr
& & \cr
\end{Young}
};
(0,-7.5)*{\downarrow};
(3,-7.5)*{F_1};
\endxy
&

&
\xy
 (0,-7.5)*{\includegraphics[height=.07\textheight]{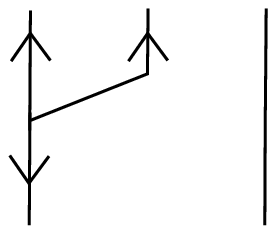}};
 (-7.5,-17)*{1};
 (-7.5,2)*{1};
 (0.75,2)*{0};
\endxy
&
&
\xy
(0,0)*{
\begin{Young}
1&3& \cr
2& & \cr
\end{Young}
};
(0,-15)*{
\begin{Young}
1&2& \cr
2& & \cr
\end{Young}
};
(0,-7.5)*{\downarrow};
(3,-7.5)*{F_2};
\endxy
&

&
\xy
 (0,-7.5)*{\includegraphics[height=.07\textheight]{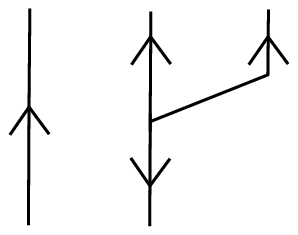}};
 (0,-17)*{1};
 (0,2)*{1};
 (7.75,2)*{0};
\endxy\\[6em]

\xy
 (0,0)*{
\begin{Young}
1&3&3 \cr
2& & \cr
\end{Young}
};
 (0,-15)*{
\begin{Young}
1&2&2 \cr
2& & \cr
\end{Young}
};
(0,-7.5)*{\downarrow};
(4.5,-7.5)*{F_2^{(2)}};
\endxy
&

&
\xy
 (0,-7.5)*{\includegraphics[height=.07\textheight]{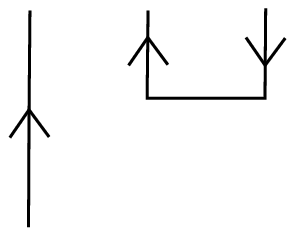}};
 (0,2)*{1};
 (7.75,2)*{-1};
\endxy
&
&
\xy
 (0,0)*{
\begin{Young}
1&2&2 \cr
& & \cr
\end{Young}
};
 (0,-15)*{
\begin{Young}
1&1&1 \cr
& & \cr
\end{Young}
};
(0,-7.5)*{\downarrow};
(8.5,-7.5)*{F_1^{(2)}\phantom{F_2^{(2)}}};
\endxy
&

&
\xy
 (0,-7.5)*{\includegraphics[height=.07\textheight]{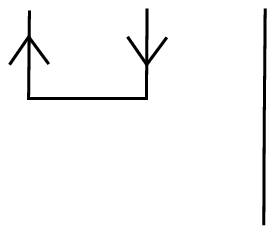}};
 (-7,2)*{1};
 (0.75,2)*{-1};
\endxy
\\[6em]
\end{array}
\]
\[
\begin{array}{lllllll}
\xy
(0,0)*{
\begin{Young}
1&1&2 \cr
& & \cr
\end{Young}
};
(0,-15)*{
\begin{Young}
1&1&1 \cr
& & \cr
\end{Young}
};
(0,-7.5)*{\downarrow};
(3,-7.5)*{F_1};
\endxy
&

&
\xy
 (0,-7.5)*{\includegraphics[height=.07\textheight]{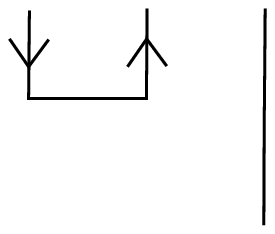}};
  (-7.5,2)*{1};
 (0.75,2)*{-1};
\endxy

&
&
\xy
 (0,0)*{
\begin{Young}
1&1&3 \cr
2& & \cr
\end{Young}
};
 (0,-15)*{
\begin{Young}
1&1&2 \cr
2& & \cr
\end{Young}
};
(0,-7.5)*{\downarrow};
(3,-7.5)*{F_2};
\endxy
&

&
\xy
 (0,-7.5)*{\includegraphics[height=.07\textheight]{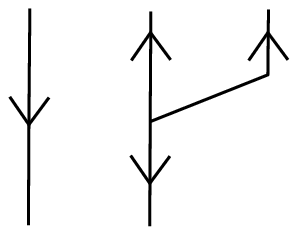}};
 (0,2)*{1};
 (7.75,2)*{-1};
 (0,-17)*{0};
\endxy\\[6em]

\xy
(0,0)*{
\begin{Young}
1&1&3 \cr
2&3 & \cr
\end{Young}
};
(0,-15)*{
\begin{Young}
1&1&2 \cr
2&2 & \cr
\end{Young}
};
(0,-7.5)*{\downarrow};
(4.5,-7.5)*{F_2^{(2)}};
\endxy
&

&
\xy
 (0,-7.5)*{\includegraphics[height=.07\textheight]{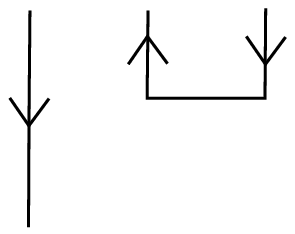}};
  (0,2)*{1};
 (7.75,2)*{-1};
\endxy
&
&
\xy
 (0,0)*{
\begin{Young}
1&1&2 \cr
2& & \cr
\end{Young}
};
 (0,-15)*{
\begin{Young}
1&1&1 \cr
2& & \cr
\end{Young}
};
(0,-7.5)*{\downarrow};
(3,-7.5)*{F_1};
\endxy
&

&
\xy
 (0,-7.5)*{\includegraphics[height=.07\textheight]{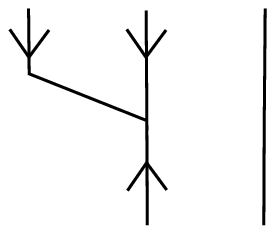}};
 (-7.5,2)*{1};
 (0.75,2)*{0};
 (0.75,-17)*{1};
\endxy
\\[6em]

\xy
(0,0)*{
\begin{Young}
1&1&3 \cr
2&2& \cr
\end{Young}
};
(0,-15)*{
\begin{Young}
1&1&2 \cr
2&2& \cr
\end{Young}
};
(0,-7.5)*{\downarrow};
(3,-7.5)*{F_2};
\endxy
&

&
\xy
 (0,-7.5)*{\includegraphics[height=.07\textheight]{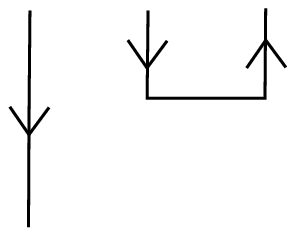}};
  (0,2)*{1};
 (7.75,2)*{-1};
\endxy
&
&
\xy
 (0,0)*{
\begin{Young}
1&2&2 \cr
2&3&3 \cr
\end{Young}
};
 (0,-15)*{
\begin{Young}
1&1&1 \cr
2&2&2 \cr
\end{Young}
};
(0,-7.5)*{\downarrow};
(8.5,-7.5)*{F_1^{(2)}F_2^{(2)}};
\endxy
&

&
\xy
 (0,-7.5)*{\includegraphics[height=.07\textheight]{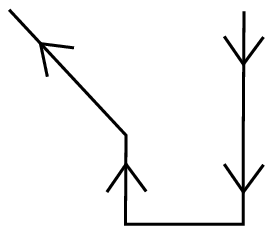}};
 (-8.5,2)*{1};
 (-0.5,2)*{0};
 (7.75,2)*{-1};
\endxy
\end{array}
\]
In the first extraordinary case, that is
\[
\xy
 (0,0)*{
\begin{Young}
1&2& \cr
2&& \cr
\end{Young}
};
 (0,-15)*{
\begin{Young}
1&1& \cr
2&& \cr
\end{Young}
};
(0,-7.5)*{\downarrow};
(3,-7.5)*{F_1};
\endxy
\;\;\;\;\;\;\;\;
\xy
 (0,-7.5)*{\includegraphics[height=.07\textheight]{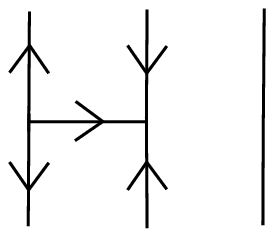}};
 (-7.5,-17)*{?};
 (0.75,-17)*{?};
 (-7.5,2)*{1};
 (0.75,2)*{1};
\endxy
\]
the Leclerc-Toffin algorithm gives the ``wrong'' answer. That is why we had to add an extra rule for this case. One easily checks that the extra rule works out as claimed.

To be more precise, the first extraordinary case, i.e. at least one three in the first column, is related to the question of nested components of webs. The rules in Definition~\ref{defn-LTgen} for this case can be read as ``outside first''. In fact, choosing the smallest entry that is not only in the first column or in the first and second columns ensures that this entry does not have flow $1$, while its left neighbour has flow $1$. In this case, no matter what the orientation of the web is, it is always possible to perform a step in the growth algorithm, i.e. either a Y-move, if the flow is $1$ and $0$ or $1$ and $-1$ and the orientation is the same, a H-move, if the orientation is different and the flow is $1$ and $0$ or an arc-move otherwise.
\vspace*{0.25cm}

In the second extraordinary case (where $3$ is the next to be replaced) we can easily see that the rules forces us to use $F_2$ which is precisely the answer of the growth algorithm, that is
\[
\xy
 (0,0)*{
\begin{Young}
1&2&3 \cr
2&& \cr
\end{Young}
};
 (0,-15)*{
\begin{Young}
1&1&2 \cr
2&& \cr
\end{Young}
};
(0,-7.5)*{\downarrow};
(3,-7.5)*{F_2};
\endxy
\;\;\;\;\;\;\;\;
\xy
 (0,-7.5)*{\includegraphics[height=.07\textheight]{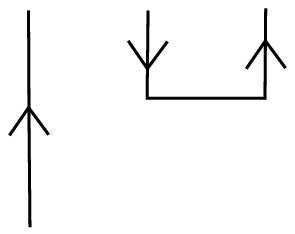}};
 (-8.25,-17)*{1};
 (-8.25,2)*{1};
 (0,2)*{1};
 (6.4,2)*{-1};
\endxy
\]
We should note that the other two cases, where the LT-algorithm gives the ``wrong'' answer, i.e.
\[
\xy
 (0,0)*{
\begin{Young}
1&2& \cr
2&3& \cr
\end{Young}
}
\endxy\;\;\text{ and }\;\;
\xy
 (0,0)*{
\begin{Young}
1&2&3 \cr
2&3& \cr
\end{Young}
}
\endxy
\]
work out up to isotopies. To see this note that in the first case we get first $F_1$ and then $F_2$, while we get $F_2^{(2)}$ after $F_1$ in the second case. Hence, the LT-algorithm gives us
\[
\xy
 (0,0)*{\includegraphics[height=.1\textheight]{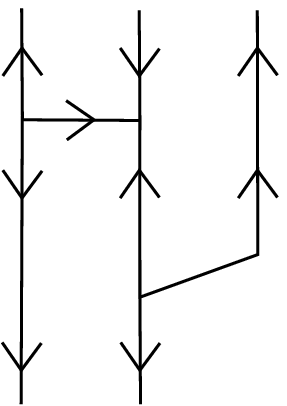}};
 (-6.5,13)*{1};
 (0,13)*{1};
 (6.5,13)*{0};
\endxy\;\;\text{ and }\;\;
\xy
 (0,0)*{\includegraphics[height=.1\textheight]{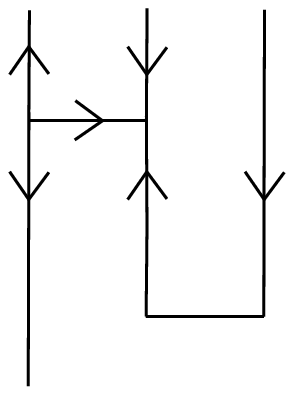}};
 (-7,13)*{1};
 (0,13)*{1};
 (6.2,13)*{-1};
\endxy
\]
while the growth algorithm gives us
\[
\xy
 (0,0)*{\includegraphics[height=.1\textheight]{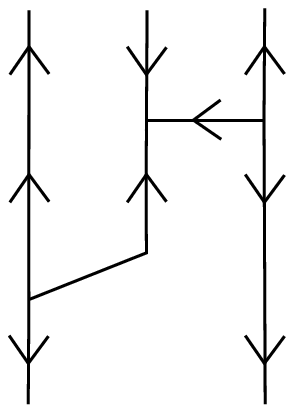}};
 (-6.5,13)*{1};
 (0,13)*{1};
 (6.5,13)*{0};
\endxy\;\;\text{ and }\;\;
\xy
 (0,0)*{\includegraphics[height=.105\textheight]{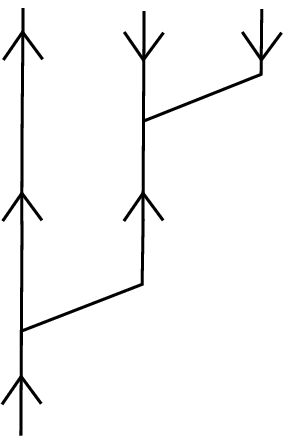}};
 (-7,13)*{1};
 (0,13)*{1};
 (6.2,13)*{-1};
\endxy
\]
\vspace*{0.15cm}

It is worth noting that, under the assumption that the next step in the Leclerc-Toffin algorithm is a step after $i-1$, tableaux of the form
\[
\xy
 (0,0)*{
\begin{Young}
& & \cr
&1& \cr
\end{Young}
};
\endxy
\;\;\;\;
\xy
 (0,0)*{
\begin{Young}
& & \cr
& &1 \cr
\end{Young}
};
\endxy
\;\;\;\;
\xy
 (0,0)*{
\begin{Young}
&1 & \cr
1& & \cr
\end{Young}
};
\endxy
\;\;\;\;
\xy
 (0,0)*{
\begin{Young}
& & 1\cr
1& & \cr
\end{Young}
};
\endxy
\;\;\;\;
\xy
 (0,0)*{
\begin{Young}
& & 1\cr
& 1& \cr
\end{Young}
};
\endxy
\] 
can not appear, since, if we assume that the $1$ does not change in this step, one has to create a tower of non-allowed pairs as illustrated below.
\[
\xy
 (0,0)*{
\begin{Young}
&-2 & \cr
&-1 & \cr
& 0& \cr
&1& \cr
\end{Young}
};
\endxy
\;\;\;\;
\xy
 (0,0)*{
\begin{Young}
& &-2  \cr
& &-1  \cr
& &0 \cr
& &1 \cr
\end{Young}
};
\endxy
\;\;\;\;
\xy
 (0,0)*{
\begin{Young}
-2&-1 & \cr
-1& 0 & \cr
0&1 & \cr
1& & \cr
\end{Young}
};
\endxy
\;\;\;\;
\xy
 (0,0)*{
\begin{Young}
-2 & & -1\cr
-1 & & 0\cr
0& & 1\cr
1& & \cr
\end{Young}
};
\endxy
\;\;\;\;
\xy
 (0,0)*{
\begin{Young}
&-2 & -1\cr
&-1 & 0\cr
& 0& 1\cr
& 1& \cr
\end{Young}
};
\endxy
\]
That is the reason why we do not have to take these case in account, since they will never lead to legal tableaux constructions.
\end{proof}
\begin{ex}\label{ex-sl3growth}
It should be noted that the Leclerc-Toffin algorithm makes a choice of how to apply Khovanov and Kuperberg's growth algorithm that avoid doing H-move with an horizontal arrow pointing to the left, since they would correspond to $E$'s and not $F$'s. For example, instead of doing a move like
\[
\xy
 (0,0)*{\includegraphics[height=.1\textheight]{res/figs/LT/case12a.eps}};
 (-6.5,13)*{0};
 (0,13)*{0};
 (6.2,13)*{-1};
\endxy
\]
the Leclerc-Toffin algorithm makes the following move.
\[
\begin{array}{llll}
\xy
 (0,11)*{
\begin{Young}
&&2 \cr
&1&3 \cr
2&& \cr
\end{Young}
};
 (0,-9)*{
\begin{Young}
& &1 \cr
&1&2 \cr
2&& \cr
\end{Young}
};
(0,1)*{\downarrow};
(8.5,1)*{F_1\cdots F_2};
\endxy
&

&
\xy
 (0,0)*{\includegraphics[height=.1\textheight]{res/figs/LT/case12b.eps}};
 (-6.5,13)*{0};
 (0,13)*{0};
 (6.2,13)*{-1};
\endxy
&
\end{array}
\]
The dots between $F_1$ and $F_2$ should indicate that there could be $F_i$ in between, but for all those $F_i$ one has $i<1$.
\end{ex}
We obtain directly from Proposition~\ref{prop-sl3bases} that the LT-algorithm gives the $\mathfrak{sl}_3$-web basis $B_S$, since the set of semi-standard tableaux enumerates the web basis.
\begin{cor}\label{cor-ltalgo}
The LT-algorithm gives the full $\mathfrak{sl}_3$-web basis $B_S$.\qed
\end{cor}
We note that this gives another proof of Khovanov and Kuperberg's result.
\begin{proof}(\textbf{Theorem}~\ref{thm:upptriang})
This follows from the LT-algorithm, i.e. an equation like~\ref{eq:LT2} can be verified (by using the ``LT-like'' algorithm from above) as in~\cite{leto}, Proposition~\ref{prop-sl3bases} and the substitution $v=-q^{-1}$. Note that the underlying space is changed under $q$-skew Howe duality (and therefore the corresponding elementary tensors).
\end{proof}
\begin{cor}\label{cor-uptri}
Kuperberg's web basis $B_S$ and the dual canonical basis are related by a change of base matrix which is unitriangular.
\end{cor}
\begin{proof}
Note again that our small change of the rules does not affect the arguments in~\cite{leto}.

By Proposition~\ref{prop-sl3bases} the LT-algorithm gives the basis $B_S$. Moreover, since Leclerc and Toffin showed~\cite{leto} that their basis is related to the canonical basis by an unitriangular change of base matrix, see~\ref{eq:LT2}, and since $q$-skew Howe duality turns the canonical into the dual canonical basis (as for example explained in~\cite{mack1}), we obtain the statement.
\end{proof}
To conclude this section, we note the following simple, but nevertheless important observation that we need repeatedly.
\begin{lem}\label{lem-ltalgo}
Let $u,v\in B_S$ be two non-elliptic webs with the same boundary and let $\mathrm{LT}(u),\mathrm{LT}(v)$ denote their LT-generators. Then $\ell_{\mathrm{t}}(\mathrm{LT}(u))=\ell_{\mathrm{t}}(\mathrm{LT}(v))$.
\end{lem}
\begin{proof}
Note that the multiplicities of all the entries in the associated semi-standard tableaux $T_u,T_v$ for $u,v$ are equal. Therefore, since the LT-generators are obtained by reducing these semi-standard tableaux $T_u,T_v$ to the one associated to the highest weight vector, the total length is the same. 
\end{proof}
It should be noted that the length can be different in general, that is $\ell(\mathrm{LT}(u))\neq\ell(\mathrm{LT}(v))$, due to Serre relations like
\[
F_iF_{i+1}F_i=F_i^{(2)}F_{i+1}+F_{i+1}F_i^{(2)}.
\]
\subsection{Multipartitions and webs with flow}\label{sec-flow}
We are going to show in this and the next section that non-elliptic webs with flows $u_f$ can be obtained from fillings of $3$-multitableaux via an \textit{extended growth algorithm}. In this section we discuss how we can turn a web with flow into a filling of a $3$-multipartition, while we give the extended growth algorithm in the next section.
\vspace*{0.25cm}

The main observation is that there are way more ways to fill a $3$-multipartition than there are webs with flows, because, roughly speaking, a filling of a $3$-multipartition does not see isotopies. This corresponds to the fact that two Morita equivalent algebras can have different dimensions, e.g. the $\mathfrak{sl}_3$-web algebra has a way smaller dimension than the corresponding cyclotomic Hecke algebra. \textit{Very roughly}, the cyclotomic Hecke algebra does not see isotopies. We therefore stress that, if we work in $B^J_S$ or $W^J_S$ as defined below, then we \textit{do not} take isotopies into account.
\vspace*{0.25cm}

The way we show that the extended growth algorithm is really an algorithm to obtain every flow on any web is that we construct an injection $\iota$ from webs with flow to fillings of $3$-multipartitions and we show that $\iota$ is a left-inverse of the extended growth algorithm (which we discuss in the next section).
\vspace*{0.25cm}

We start by extending the Proposition~\ref{prop-tableauxflows}, that is we give another bijection to a set of so-called \textit{$3$-multipartitions $\vec{\lambda}=(\lambda^1,\lambda^2,\lambda^3)$} from Definition~\ref{defn-tabcomb}. Recall that $\mathrm{Col}(\lambda)$ below means that the entry $k$ in the column-strict tableaux of $\mathrm{Col}(\lambda)$ appears with multiplicity $\lambda_k$.
\begin{prop}\label{prop-tableauxflows2}
There is a bijection between $\mathrm{Col}(\lambda)$, the 
set of state strings $J$ such that there exists a web $w\in B_S$ and 
a flow $f$ on $w$ which extends $J$, denoted by $B_S^J$, and $\Lambda^+(n,d,3)$, i.e. the set of $3$-multipartitions of $d=3\ell$. 
\end{prop}
\begin{proof}
The first bijection was proven in Proposition~\ref{prop-tableauxflows}.

We give the bijection from $\mathrm{Col}(\lambda)$ to $3$-multipartitions $\Lambda^+(n,d,3)$ with $d=3\ell$. First we define a standard filling of a diagram with $i$-rows and $j$-columns, denoted by $T^i_j$, by filling all nodes in the $i^{\prime}$-row with $i^{\prime}$ for all $1\leq i^{\prime}\leq i$, e.g.
\[
T^5_3=\xy
 (0,0)*{
\begin{Young}
1&1&1 \cr
2&2&2 \cr
3&3&3 \cr
4&4&4 \cr
5&5&5 \cr
\end{Young}
}
\endxy
\]
Now, given a $T\in\mathrm{Col}(\lambda)$, the tableau $T-T^i_j$ (for suitable $i,j$) has non-negative entries due to the column-strictness of $T$. We can now associate for each column $1\leq j^{\prime}\leq j$ the column word $c_{j^{\prime}}$ by reading the entries of $T-T^i_j$ from bottom to top. This gives a partition $\lambda_{j^{\prime}}$. The reader can easily verify that this is a bijection.
\end{proof}
Note that the second bijection of Proposition~\ref{prop-tableauxflows2} is not restricted to tableaux with only three columns. In fact the whole bijection is not new, i.e. it is well-known. See~\cite{bk3} for example.
\begin{ex}\label{ex-tabl}
We have the following correspondence.
\[
\xy
 (0,0)*{
\begin{Young}
2&1&2 \cr
4&3&4 \cr
6&5&6 \cr
\end{Young}
}
\endxy\leftrightarrow
\vec{\lambda}=
((3,2,1),(2,1),(3,2,1))=\left(\;
\xy
 (0,0)*{
\begin{Young}
&& \cr
& \cr
 \cr
\end{Young}}
\endxy\;,
\;\xy
 (0,0)*{
\begin{Young}
 &\cr
 \cr
\end{Young}}
\endxy\;,\;
\xy
 (0,0)*{
\begin{Young}
&& \cr
& \cr
 \cr
\end{Young}}
\endxy\;
\right).
\]
\end{ex}
\vspace*{0.15cm}

Seeing a column-strict tableau $T$ as a $3$-multipartition $\vec{\lambda}$ has some advantages. For example we can fill such a $3$-multipartition $\vec{\lambda}$ with non-negative numbers. As in Definition~\ref{defn-tabcomb}, we denote the set of all $3$-multitableaux $\vec{T}$ of $\vec{\lambda}$ with a \textit{standard fillings} by $\mathrm{Std}(\vec{\lambda})$. For example the following tableau is in $\mathrm{Std}(\vec{\lambda}=((3,2,1),(2,1),(3,2,1)))$.
\[
\left(\;
\xy
 (0,0)*{
\begin{Young}
1&2&3 \cr
4&7 \cr
8 \cr
\end{Young}}
\endxy\;,
\;\xy
 (0,0)*{
\begin{Young}
5 & 9\cr
6\cr
\end{Young}}
\endxy\;,\;
\xy
 (0,0)*{
\begin{Young}
1&2&3 \cr
6&10 \cr
11 \cr
\end{Young}}
\endxy\;
\right)
\]
\vspace*{0.15cm}

In order to give a \textit{growth algorithm for webs with flow lines} $u_f\in B_S^J$ (that is the set of all non-elliptic webs $\partial u=S$ with a fixed flow $f$ that extends $J$), we first define an injection of the set of all non-elliptic webs with flows associated to a certain pair $(S,J)$ (or equivalent a column-strict tableau or a $3$-multipartition $\vec{\lambda}$) into $\mathrm{Std}(\vec{\lambda})$. Moreover, we denote the set \textit{all} webs with boundary datum $(S,J)$ and a chosen flow by $W_S^J$.
\vspace*{0.25cm}

Note again that there are in general much more elements in $\mathrm{Std}(\vec{\lambda})$ than in $B_S^J$. After that we give a method to obtain from an element of $\mathrm{Std}(\vec{\lambda})$ a web with flow and we show that this map is in fact the left-inverse of the embedding $\iota\colon B_S^J\to\mathrm{Std}(\vec{\lambda})$ . Hence, we can call this process an \textit{extended growth algorithm}.
\begin{defn}\label{defn-fareastwest}
Given a partition $\lambda$ which is partially filled with numbers and a fixed number $j$. Then the \textit{highest not filled node $N$ of residue $j$} is the unique (if it exists), highest node of $\lambda$ without filling to the north east such that a filling of this node still gives a legal tableau and $r(N)=j$.
\end{defn}
\begin{ex}\label{ex-fareastwest}
Assume that $\lambda$ is the following partition (note that $m=3$, i.e. the residue of the nodes is shifted by $3$) with the following partially filling.
\[
\lambda=\xy
 (0,0)*{
\begin{Young}
1&2&5 \cr
6&$\times$ & \cr
$\circ$  & \cr
\end{Young}}
\endxy
\]
We have marked the highest node $N_{\times}$ with residue $r(N_{\times})=3$ by $\times$ and the highest node $N_{\circ}$ with residue $r(N_{\circ})=1$ by $\circ$. Moreover, there is no possible node of residue $2$ to be filled. 
\end{ex}
\vspace*{0.15cm}

We start now by defining the map $\iota$. We give it inductively using an inductive algorithm.

It should be noted that the following algorithm looks very complicated, since it has many different rules one has to follow. But the main idea is simple and straightforward, i.e. look at the corresponding pictures and read of the tableau ``locally'' at the top and the bottom. Then the $k$-th step in the algorithm should add nodes labeled $k$, whose number depends on the divided power of the $F^{(j_k)}_{i_k}$'s, at the corresponding given positions with residue $i_k$.
\begin{defn}\label{defn-webtotab}(\textbf{Flows to fillings})
Given a pair of a sign string and a state string $(S,J)$ and a web $u_f\in B_S^J$, we associate to it a standard filling $\iota(u_f)\in\mathrm{Std}(\vec{\lambda})$ inductively by going backwards the LT-growth algorithm from Proposition~\ref{ex-sl3growth}, i.e. use the canonical flow on $w$ to determine the LT-generators $LT(w)=F^{(j_n)}_{i_n}\cdots F^{(j_1)}_{i_1}$ for $u$ (note that we read backwards now).
\begin{enumerate}
\item At the initial stage set $\vec{T}_0=(\emptyset,\emptyset,\emptyset)$.
\item At the $k$-th step use $F^{(j_{k})}_{i_{k}}$ to determine the \textit{type and color} of the operation performed on $\vec{T}_{k-1}$. We give a full list of all possible types and colors below together with the operation
\[
\mathbf{k}\colon\vec{T}_{k-1}\mapsto\vec{T}_{k}.
\]
\item Repeat until $k=n$. Then set $\iota(u_f)=\vec{T}_n$.
\end{enumerate}
The full list of possible types (depending on the orientation of the web) and colors (depending on the flow line) is the following.
\begin{itemize}
\item \textit{Arc moves} come in two different types, called arc-type a and b, each with three different colors, called $1,0,-1$. Note that type a corresponds to an $F^{(2)}$, while type b corresponds to an $F^{(1)}$.
\[
\xy
(0,5)*{\text{Type a:}};
(0,-5)*{\text{Type b:}}
\endxy\;\;\;\;
\xy
(0,12)*{\text{Color } 1};
(0,-12)*{\text{Color } 1};
(0,5)*{\includegraphics[scale=.5]{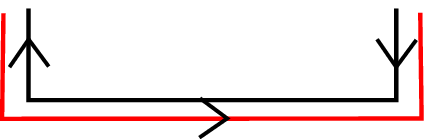}};
(0,-5)*{\includegraphics[scale=.5]{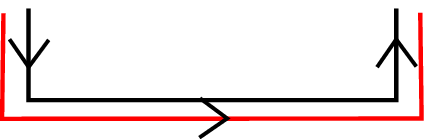}}
\endxy\;\;\;\;
\xy
(0,12)*{\text{Color } 0};
(0,-12)*{\text{Color } 0};
(0,5)*{\includegraphics[scale=.5]{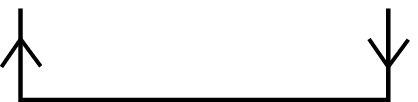}};
(0,-5)*{\includegraphics[scale=.5]{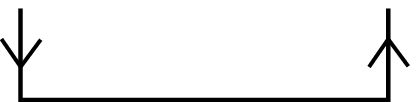}}
\endxy\;\;\;\;
\xy
(0,12)*{\text{Color } -1};
(0,-12)*{\text{Color } -1};
(0,5)*{\includegraphics[scale=.5]{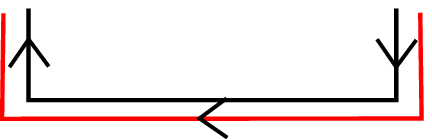}};
(0,-5)*{\includegraphics[scale=.5]{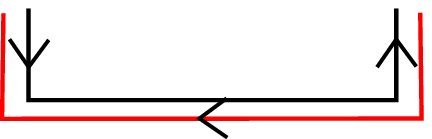}}
\endxy
\]
\item \textit{Y-moves} come in two different types, called a and b, with six different colors, called $1,1^{\prime},0,0^{\prime}$ and $-1,-1^{\prime}$. Note that all Y-moves belong to a $F^{(1)}$. The Y-moves of type a are
\[
\xy
(0,0)*{\text{Type a:}};
\endxy\;\;\;\;
\xy
(0,12)*{\text{Color } 1};
(0,-12)*{\text{Color } 1^{\prime}};
(0,5)*{\includegraphics[scale=.5]{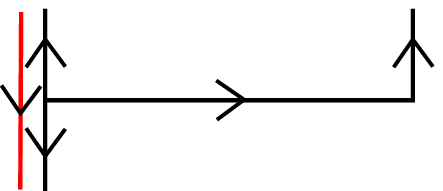}};
(0,-5)*{\includegraphics[scale=.5]{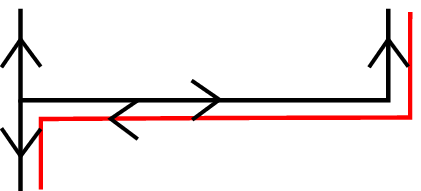}}
\endxy\;\;\;\;
\xy
(0,12)*{\text{Color } 0};
(0,-12)*{\text{Color } 0^{\prime}};
(0,5)*{\includegraphics[scale=.5]{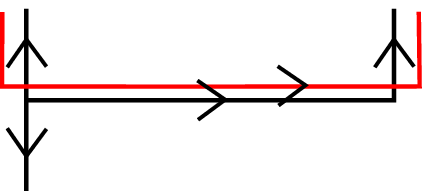}};
(0,-5)*{\includegraphics[scale=.5]{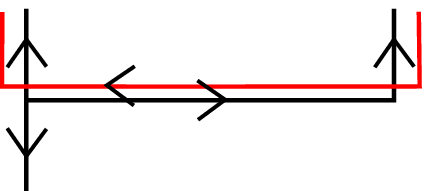}}
\endxy\;\;\;\;
\xy
(0,12)*{\text{Color } -1};
(0,-12)*{\text{Color } -1^{\prime}};
(0,5)*{\includegraphics[scale=.5]{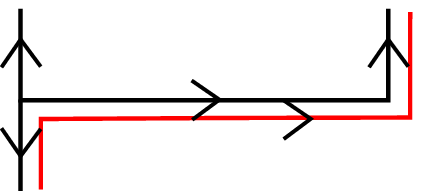}};
(0,-5)*{\includegraphics[scale=.5]{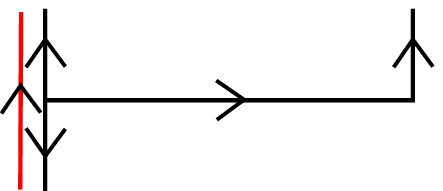}}
\endxy
\]
and the Y-moves of type b are
\[
\xy
(0,0)*{\text{Type b:}};
\endxy\;\;\;\;
\xy
(0,12)*{\text{Color } 1};
(0,-12)*{\text{Color } 1^{\prime}};
(0,5)*{\includegraphics[scale=.5]{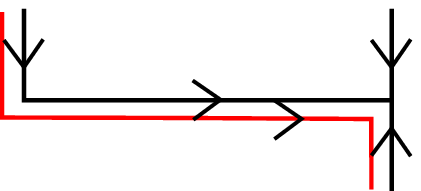}};
(0,-5)*{\includegraphics[scale=.5]{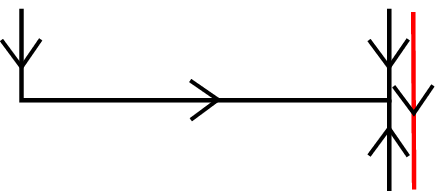}}
\endxy\;\;\;\;
\xy
(0,12)*{\text{Color } 0};
(0,-12)*{\text{Color } 0^{\prime}};
(0,5)*{\includegraphics[scale=.5]{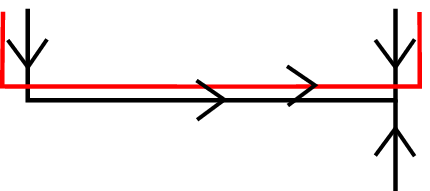}};
(0,-5)*{\includegraphics[scale=.5]{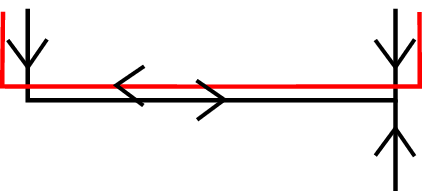}}
\endxy\;\;\;\;
\xy
(0,12)*{\text{Color } -1};
(0,-12)*{\text{Color } -1^{\prime}};
(0,5)*{\includegraphics[scale=.5]{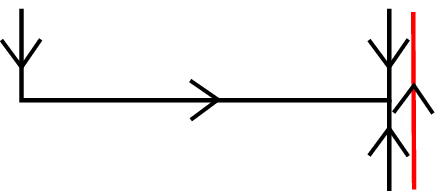}};
(0,-5)*{\includegraphics[scale=.5]{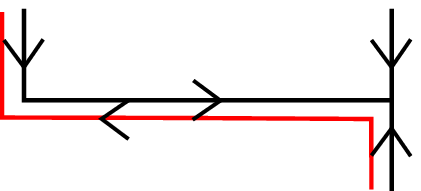}}
\endxy
\]
\item \textit{H-moves} only have one type, since the other type does not appear in the LT-algorithm (it would use a $E^{(1)}$ instead of a $F^{(1)}$). But to make the argument later more convenient, we use two types, again a and b, for them again with six different colors, again $1,1^{\prime},0,0^{\prime}$ and $-1,-1^{\prime}$. The H-moves of type a are
\[
\xy
(0,0)*{\text{Type a:}};
\endxy\;\;\;\;
\xy
(0,12)*{\text{Color } 1};
(0,-12)*{\text{Color } 1^{\prime}};
(0,5)*{\includegraphics[scale=.5]{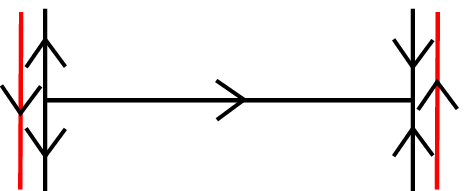}};
(0,-5)*{\includegraphics[scale=.5]{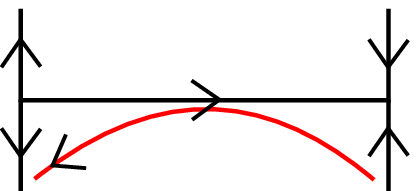}}
\endxy\;\;\;\;
\xy
(0,12)*{\text{Color } 0};
(0,-12)*{\text{Color } 0^{\prime}};
(0,5)*{\includegraphics[scale=.5]{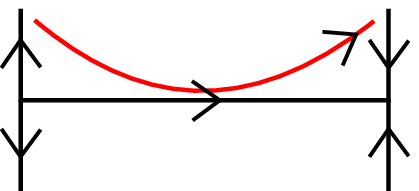}};
(0,-5)*{\includegraphics[scale=.5]{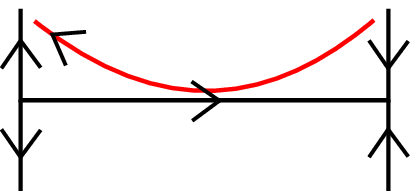}}
\endxy\;\;\;\;
\xy
(0,12)*{\text{Color } -1};
(0,-12)*{\text{Color } -1^{\prime}};
(0,5)*{\includegraphics[scale=.5]{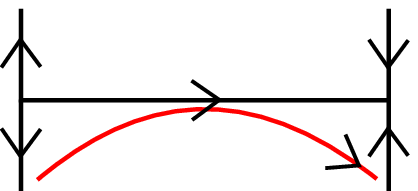}};
(0,-5)*{\includegraphics[scale=.5]{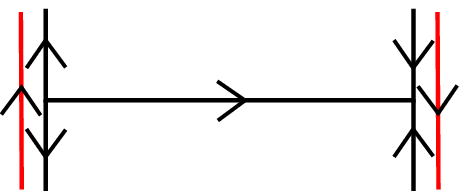}}
\endxy
\]
while the H-moves of type b are
\[
\xy
(0,0)*{\text{Type b:}};
\endxy\;\;\;\;
\xy
(0,12)*{\text{Color } 1};
(0,-12)*{\text{Color } 1^{\prime}};
(0,5)*{\includegraphics[scale=.5]{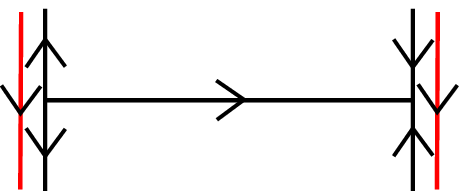}};
(0,-5)*{\includegraphics[scale=.5]{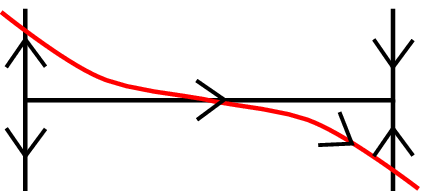}}
\endxy\;\;\;\;
\xy
(0,12)*{\text{Color } 0};
(0,-12)*{\text{Color } 0^{\prime}};
(0,5)*{\includegraphics[scale=.5]{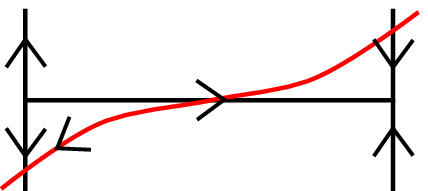}};
(0,-5)*{\includegraphics[scale=.5]{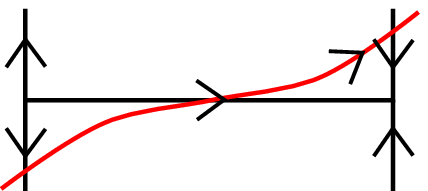}}
\endxy\;\;\;\;
\xy
(0,12)*{\text{Color } -1};
(0,-12)*{\text{Color } -1^{\prime}};
(0,5)*{\includegraphics[scale=.5]{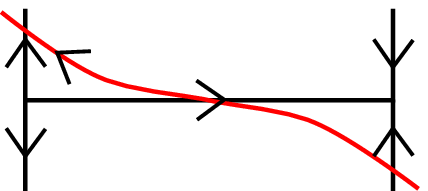}};
(0,-5)*{\includegraphics[scale=.5]{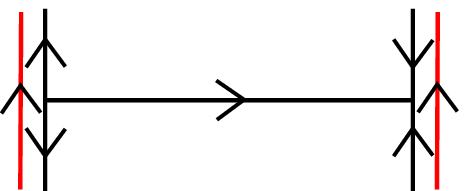}}
\endxy
\]
\item \textit{Right shifts} come again in two types a and b with three colors, again $1,0,-1$. Note that type a corresponds to $F^{(1)}$, while type b corresponds to $F^{(2)}$.
\[
\xy
(0,5)*{\text{Type a:}};
(0,-5)*{\text{Type b:}};
\endxy\;\;\;\;
\xy
(0,12)*{\text{Color } 1};
(0,-12)*{\text{Color } 1};
(0,5)*{\includegraphics[scale=.5]{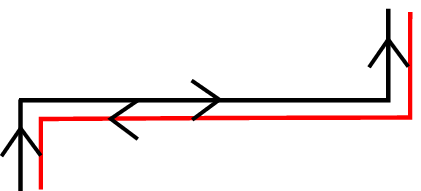}};
(0,-5)*{\includegraphics[scale=.5]{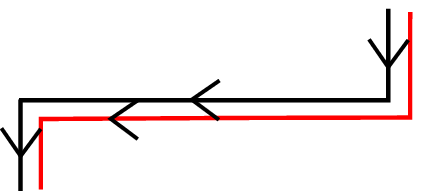}}
\endxy\;\;\;\;
\xy
(0,12)*{\text{Color } 0};
(0,-12)*{\text{Color } 0};
(0,5)*{\includegraphics[scale=.5]{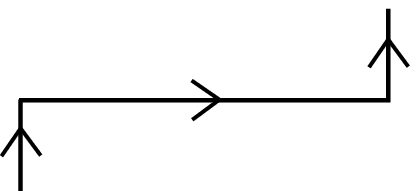}};
(0,-5)*{\includegraphics[scale=.5]{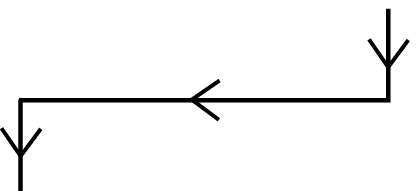}}
\endxy\;\;\;\;
\xy
(0,12)*{\text{Color } -1};
(0,-12)*{\text{Color } -1};
(0,5)*{\includegraphics[scale=.5]{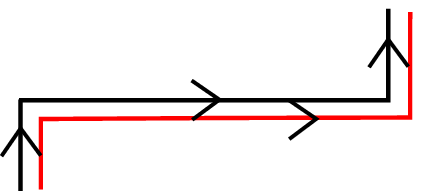}};
(0,-5)*{\includegraphics[scale=.5]{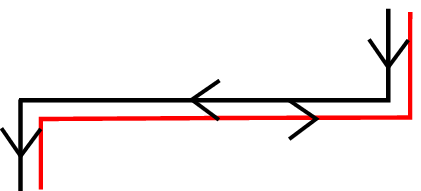}}
\endxy
\]
\item \textit{Left shifts} come again in two types a and b with three colors, again $1,0,-1$. Note that type a corresponds to $F^{(2)}$, while type b corresponds to $F^{(1)}$.
\[
\xy
(0,5)*{\text{Type a:}};
(0,-5)*{\text{Type b:}};
\endxy\;\;\;\;
\xy
(0,12)*{\text{Color } 1};
(0,-12)*{\text{Color } 1};
(0,5)*{\includegraphics[scale=.5]{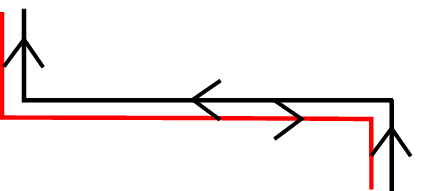}};
(0,-5)*{\includegraphics[scale=.5]{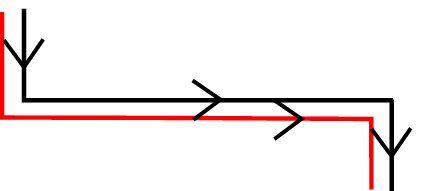}}
\endxy\;\;\;\;
\xy
(0,12)*{\text{Color } 0};
(0,-12)*{\text{Color } 0};
(0,5)*{\includegraphics[scale=.5]{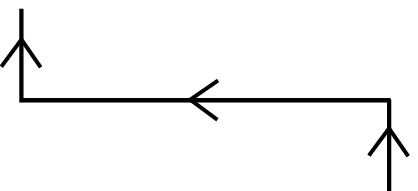}};
(0,-5)*{\includegraphics[scale=.5]{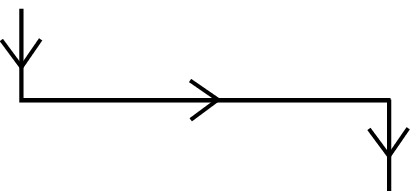}}
\endxy\;\;\;\;
\xy
(0,12)*{\text{Color } -1};
(0,-12)*{\text{Color } -1};
(0,5)*{\includegraphics[scale=.5]{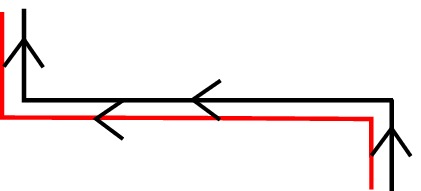}};
(0,-5)*{\includegraphics[scale=.5]{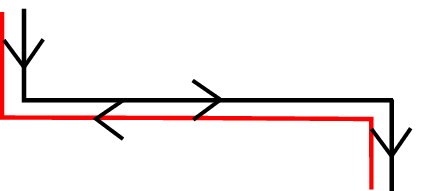}}
\endxy
\]
\item The unique \textit{empty shift}. Note that type a corresponds to $F^{(3)}$.
\[
\xy
(9,7)*{\times};
(-9,-7)*{\times};
(-9,7)*{\circ};
(9,-7)*{\circ};
\endxy
\]
\end{itemize}
Note that the pair $(S,J)$ gives rise to a 3-multipartition $\vec{\lambda}=(\lambda^1,\lambda^2,\lambda^3)$. The tableau $\vec{T}_n$ should be a $3$-multitableau of shape $\vec{\lambda}$. We build it inductively.

The following list of operators should add the corresponding nodes at the \textit{highest} free place of the tableau $T^{1,2,3}$ of residue $i_k$ (if the step is for $F_{i_k}^{(j_k)}$) (as in in Definition~\ref{defn-fareastwest}). We use the symbol $+$ for this procedure and denote the tableaux of $\vec{T}_{k-1}$ by $T_{k-1}^{1,2,3}$. The operators for the $k$-th step are the following.
\begin{itemize}
\item For the arc-moves of type a we use
\[
\mathbf{k}\colon(T_{k-1}^{1},T_{k-1}^{2},T_{k-1}^{3})\mapsto
\begin{cases}
(T_{k-1}^{1},T_{k-1}^{2}+\xy(0,0)*{\begin{Young}k\cr\end{Young}}\endxy,T_{k-1}^{3}+\xy(0,0)*{\begin{Young}k\cr\end{Young}}\endxy),  & \text{if the color is }1,\\
(T_{k-1}^{1}+\xy(0,0)*{\begin{Young}k\cr\end{Young}}\endxy,T_{k-1}^{2},T_{k-1}^{3}+\xy(0,0)*{\begin{Young}k\cr\end{Young}}\endxy), & \text{if the color is }0,\\
(T_{k-1}^{1}+\xy(0,0)*{\begin{Young}k\cr\end{Young}}\endxy,T_{k-1}^{2}+\xy(0,0)*{\begin{Young}k\cr\end{Young}}\endxy,T_{k-1}^{3}),  & \text{if the color is }-1,
\end{cases}
\]
and for the arc-moves of type b we use
\[
\mathbf{k}\colon(T_{k-1}^{1},T_{k-1}^{2},T_{k-1}^{3})\mapsto
\begin{cases}
(T_{k-1}^{1},T_{k-1}^{2},T_{k-1}^{3}+\xy(0,0)*{\begin{Young}k\cr\end{Young}}\endxy),  & \text{if the color is }1,\\
(T_{k-1}^{1},T_{k-1}^{2}+\xy(0,0)*{\begin{Young}k\cr\end{Young}}\endxy,T_{k-1}^{3}), & \text{if the color is }0,\\
(T_{k-1}^{1}+\xy(0,0)*{\begin{Young}k\cr\end{Young}}\endxy,T_{k-1}^{2},T_{k-1}^{3}),  & \text{if the color is }-1.
\end{cases}
\]
\item For the Y-moves of type a and colors $1,1^{\prime}$ we use
\[
\mathbf{k}\colon(T_{k-1}^{1},T_{k-1}^{2},T_{k-1}^{3})\mapsto
\begin{cases}
(T_{k-1}^{1},T_{k-1}^{2}+\xy(0,0)*{\begin{Young}k\cr\end{Young}}\endxy,T_{k-1}^{3}),  & \text{if the color is }1,\\
(T_{k-1}^{1}+\xy(0,0)*{\begin{Young}k\cr\end{Young}}\endxy,T_{k-1}^{2},T_{k-1}^{3}),  & \text{if the color is }1^{\prime},
\end{cases}
\]
and or the Y-moves of type a and colors $0,0^{\prime}$ we use
\[
\mathbf{k}\colon(T_{k-1}^{1},T_{k-1}^{2},T_{k-1}^{3})\mapsto
\begin{cases}
(T_{k-1}^{1},T_{k-1}^{2},T_{k-1}^{3}+\xy(0,0)*{\begin{Young}k\cr\end{Young}}\endxy),  & \text{if the color is }0,\\
(T_{k-1}^{1}+\xy(0,0)*{\begin{Young}k\cr\end{Young}}\endxy,T_{k-1}^{2},T_{k-1}^{3}),  & \text{if the color is }0^{\prime},
\end{cases}
\]
and or the Y-moves of type a and colors $-1,-1^{\prime}$ we use
\[
\mathbf{k}\colon(T_{k-1}^{1},T_{k-1}^{2},T_{k-1}^{3})\mapsto
\begin{cases}
(T_{k-1}^{1},T_{k-1}^{2},T_{k-1}^{3}+\xy(0,0)*{\begin{Young}k\cr\end{Young}}\endxy),  & \text{if the color is }-1,\\
(T_{k-1}^{1},T_{k-1}^{2}+\xy(0,0)*{\begin{Young}k\cr\end{Young}}\endxy,T_{k-1}^{3}),  & \text{if the color is }-1^{\prime}.
\end{cases}
\]
\item For the Y-moves of type b and colors $1,1^{\prime}$ we use similar rules, i.e.
\[
\mathbf{k}\colon(T_{k-1}^{1},T_{k-1}^{2},T_{k-1}^{3})\mapsto
\begin{cases}
(T_{k-1}^{1},T_{k-1}^{2},T_{k-1}^{3}+\xy(0,0)*{\begin{Young}k\cr\end{Young}}\endxy),  & \text{if the color is }1,\\
(T_{k-1}^{1},T_{k-1}^{2}+\xy(0,0)*{\begin{Young}k\cr\end{Young}}\endxy,T_{k-1}^{3}),  & \text{if the color is }1^{\prime},
\end{cases}
\]
and or the Y-moves of type b and colors $0,0^{\prime}$ we use
\[
\mathbf{k}\colon(T_{k-1}^{1},T_{k-1}^{2},T_{k-1}^{3})\mapsto
\begin{cases}
(T_{k-1}^{1},T_{k-1}^{2},T_{k-1}^{3}+\xy(0,0)*{\begin{Young}k\cr\end{Young}}\endxy),  & \text{if the color is }0,\\
(T_{k-1}^{1}+\xy(0,0)*{\begin{Young}k\cr\end{Young}}\endxy,T_{k-1}^{2},T_{k-1}^{3}),  & \text{if the color is }0^{\prime},
\end{cases}
\]
and or the Y-moves of type b and colors $-1,-1^{\prime}$ we use
\[
\mathbf{k}\colon(T_{k-1}^{1},T_{k-1}^{2},T_{k-1}^{3})\mapsto
\begin{cases}
(T_{k-1}^{1},T_{k-1}^{2}+\xy(0,0)*{\begin{Young}k\cr\end{Young}}\endxy,T_{k-1}^{3}),  & \text{if the color is }-1,\\
(T_{k-1}^{1}+\xy(0,0)*{\begin{Young}k\cr\end{Young}}\endxy,T_{k-1}^{2},T_{k-1}^{3}),  & \text{if the color is }-1^{\prime}.
\end{cases}
\]
\item For the H-moves of type a and colors $1,1^{\prime}$ we use
\[
\mathbf{k}\colon(T_{k-1}^{1},T_{k-1}^{2},T_{k-1}^{3})\mapsto
\begin{cases}
(T_{k-1}^{1},T_{k-1}^{2}+\xy(0,0)*{\begin{Young}k\cr\end{Young}}\endxy,T_{k-1}^{3}),  & \text{if the color is }1,\\
(T_{k-1}^{1}+\xy(0,0)*{\begin{Young}k\cr\end{Young}}\endxy,T_{k-1}^{2},T_{k-1}^{3}),  & \text{if the color is }1^{\prime},
\end{cases}
\]
and or the H-moves of type a and colors $0,0^{\prime}$ we use
\[
\mathbf{k}\colon(T_{k-1}^{1},T_{k-1}^{2},T_{k-1}^{3})\mapsto
\begin{cases}
(T_{k-1}^{1},T_{k-1}^{2},T_{k-1}^{3}+\xy(0,0)*{\begin{Young}k\cr\end{Young}}\endxy),  & \text{if the color is }0,\\
(T_{k-1}^{1}+\xy(0,0)*{\begin{Young}k\cr\end{Young}}\endxy,T_{k-1}^{2},T_{k-1}^{3}),  & \text{if the color is }0^{\prime},
\end{cases}
\]
and or the H-moves of type a and colors $-1,-1^{\prime}$ we use
\[
\mathbf{k}\colon(T_{k-1}^{1},T_{k-1}^{2},T_{k-1}^{3})\mapsto
\begin{cases}
(T_{k-1}^{1},T_{k-1}^{2},T_{k-1}^{3}+\xy(0,0)*{\begin{Young}k\cr\end{Young}}\endxy),  & \text{if the color is }-1,\\
(T_{k-1}^{1},T_{k-1}^{2}+\xy(0,0)*{\begin{Young}k\cr\end{Young}}\endxy,T_{k-1}^{3}),  & \text{if the color is }-1^{\prime}.
\end{cases}
\]
\item For the H-moves of type b and colors $1,1^{\prime}$ we use
\[
\mathbf{k}\colon(T_{k-1}^{1},T_{k-1}^{2},T_{k-1}^{3})\mapsto
\begin{cases}
(T_{k-1}^{1},T_{k-1}^{2}+\xy(0,0)*{\begin{Young}k\cr\end{Young}}\endxy,T_{k-1}^{3}),  & \text{if the color is }1,\\
(T_{k-1}^{1},T_{k-1}^{2},T_{k-1}^{3}+\xy(0,0)*{\begin{Young}k\cr\end{Young}}\endxy),  & \text{if the color is }1^{\prime},
\end{cases}
\]
and or the H-moves of type b and colors $0,0^{\prime}$ we use
\[
\mathbf{k}\colon(T_{k-1}^{1},T_{k-1}^{2},T_{k-1}^{3})\mapsto
\begin{cases}
(T_{k-1}^{1}+\xy(0,0)*{\begin{Young}k\cr\end{Young}}\endxy,T_{k-1}^{2},T_{k-1}^{3}),  & \text{if the color is }0,\\
(T_{k-1}^{1},T_{k-1}^{2},T_{k-1}^{3}+\xy(0,0)*{\begin{Young}k\cr\end{Young}}\endxy),  & \text{if the color is }0^{\prime},
\end{cases}
\]
and or the H-moves of type b and colors $-1,-1^{\prime}$ we use
\[
\mathbf{k}\colon(T_{k-1}^{1},T_{k-1}^{2},T_{k-1}^{3})\mapsto
\begin{cases}
(T_{k-1}^{1}+\xy(0,0)*{\begin{Young}k\cr\end{Young}}\endxy,T_{k-1}^{2},T_{k-1}^{3}),  & \text{if the color is }-1,\\
(T_{k-1}^{1},T_{k-1}^{2}+\xy(0,0)*{\begin{Young}k\cr\end{Young}}\endxy,T_{k-1}^{3}),  & \text{if the color is }-1^{\prime}.
\end{cases}
\]
\item For the right shifts of type a and colors $1,0,-1$ we use
\[
\mathbf{k}\colon(T_{k-1}^{1},T_{k-1}^{2},T_{k-1}^{3})\mapsto
\begin{cases}
(T_{k-1}^{1}+\xy(0,0)*{\begin{Young}k\cr\end{Young}}\endxy,T_{k-1}^{2},T_{k-1}^{3}),  & \text{if the color is }1,\\
(T_{k-1}^{1},T_{k-1}^{2}+\xy(0,0)*{\begin{Young}k\cr\end{Young}}\endxy,T_{k-1}^{3}),  & \text{if the color is }0,\\
(T_{k-1}^{1},T_{k-1}^{2},T_{k-1}^{3}+\xy(0,0)*{\begin{Young}k\cr\end{Young}}\endxy),  & \text{if the color is }-1,
\end{cases}
\]
and for right shifts of type b and colors $1,0,-1$ we use
\[
\mathbf{k}\colon(T_{k-1}^{1},T_{k-1}^{2},T_{k-1}^{3})\mapsto
\begin{cases}
(T_{k-1}^{1}+\xy(0,0)*{\begin{Young}k\cr\end{Young}}\endxy,T_{k-1}^{2}+\xy(0,0)*{\begin{Young}k\cr\end{Young}}\endxy,T_{k-1}^{3}),  & \text{if the color is }1,\\
(T_{k-1}^{1}+\xy(0,0)*{\begin{Young}k\cr\end{Young}}\endxy,T_{k-1}^{2},T_{k-1}^{3}+\xy(0,0)*{\begin{Young}k\cr\end{Young}}\endxy),  & \text{if the color is }0,\\
(T_{k-1}^{1},T_{k-1}^{2}+\xy(0,0)*{\begin{Young}k\cr\end{Young}}\endxy,T_{k-1}^{3}+\xy(0,0)*{\begin{Young}k\cr\end{Young}}\endxy),  & \text{if the color is }-1,
\end{cases}
\]
\item And finally, for the left shifts of type a and colors $1,0,-1$ we use
\[
\mathbf{k}\colon(T_{k-1}^{1},T_{k-1}^{2},T_{k-1}^{3})\mapsto
\begin{cases}
(T_{k-1}^{1},T_{k-1}^{2}+\xy(0,0)*{\begin{Young}k\cr\end{Young}}\endxy,T_{k-1}^{3}+\xy(0,0)*{\begin{Young}k\cr\end{Young}}\endxy),  & \text{if the color is }1,\\
(T_{k-1}^{1}+\xy(0,0)*{\begin{Young}k\cr\end{Young}}\endxy,T_{k-1}^{2},T_{k-1}^{3}+\xy(0,0)*{\begin{Young}k\cr\end{Young}}\endxy),  & \text{if the color is }0,\\
(T_{k-1}^{1}+\xy(0,0)*{\begin{Young}k\cr\end{Young}}\endxy,T_{k-1}^{2}+\xy(0,0)*{\begin{Young}k\cr\end{Young}}\endxy,T_{k-1}^{3}),  & \text{if the color is }-1,
\end{cases}
\]
and for left shifts of type b and colors $1,0,-1$ we use
\[
\mathbf{k}\colon(T_{k-1}^{1},T_{k-1}^{2},T_{k-1}^{3})\mapsto
\begin{cases}
(T_{k-1}^{1},T_{k-1}^{2},T_{k-1}^{3}+\xy(0,0)*{\begin{Young}k\cr\end{Young}}\endxy),  & \text{if the color is }1,\\
(T_{k-1}^{1},T_{k-1}^{2}+\xy(0,0)*{\begin{Young}k\cr\end{Young}}\endxy,T_{k-1}^{3}),  & \text{if the color is }0,\\
(T_{k-1}^{1}+\xy(0,0)*{\begin{Young}k\cr\end{Young}}\endxy,T_{k-1}^{2},T_{k-1}^{3}),  & \text{if the color is }-1,
\end{cases}
\]
\item For the unique empty shift we use
\[
\mathbf{k}\colon(T_{k-1}^{1},T_{k-1}^{2},T_{k-1}^{3})\mapsto
(T_{k-1}^{1}+\xy(0,0)*{\begin{Young}k\cr\end{Young}}\endxy,T_{k-1}^{2}+\xy(0,0)*{\begin{Young}k\cr\end{Young}}\endxy,T_{k-1}^{3}+\xy(0,0)*{\begin{Young}k\cr\end{Young}}\endxy).
\]
\end{itemize}
Thus, we have a map
\[
\iota\colon B_S^J\to \mathrm{Std}(\vec{\lambda}).
\]
\end{defn}
\vspace*{0.15cm}

We have to show that the algorithm is well-defined, i.e. a priori we could run into ambiguities if the nodes in one of the three tableaux are already filled with numbers or that there are no suitable free nodes of residue $i_k$. But the following lemma ensures that this will never happen. Moreover, the lemma shows that the map $\iota$ is an injection. But we give an example before we state and prove the lemma.
\begin{ex}\label{ex-tabl2}
\begin{itemize}
\item[(a)] For example for $S=(-,-,-)$ and $J=(1,-1,0)$ we have the half-theta web with the following flow and tableau, $3$-multipartition and LT-generators.
\[
\xy
(0,0)*{\includegraphics[scale=0.85]{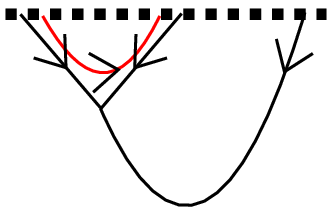}}
\endxy\;\;\;\;
T=\xy(0,0)*{\begin{Young}1 & 1 & 2\cr 3 & 2 & 3\cr\end{Young}}\endxy\;\;\;\;
\vec{\lambda}=\left(\;\xy(0,0)*{\begin{Young}\cr\end{Young}}\endxy\;,\;\emptyset\;,\;\xy(0,0)*{\begin{Young}\cr \cr\end{Young}}\endxy\;\right)\;\;\;\; LT=F_1F_2^{(2)}
\]
Hence, the algorithm from Definition~\ref{defn-webtotab} has three steps, i.e. the initial and two ``honest'' steps. One gets the following sequence
\[
\vec{T}_0=(\emptyset,\emptyset,\emptyset)\mapsto \vec{T}_1=\left(\;\xy(0,0)*{\begin{Young}1\cr\end{Young}}\endxy\;,\;\emptyset\;,\;\xy(0,0)*{\begin{Young}1 \cr\end{Young}}\endxy\;\right)\mapsto \vec{T}_2=\left(\;\xy(0,0)*{\begin{Young}1\cr\end{Young}}\endxy\;,\;\emptyset\;,\;\xy(0,0)*{\begin{Young}1\cr 2\cr\end{Young}}\endxy\;\right),
\]
which can be easily verified from the following picture
\[
\xy
(0,0)*{\includegraphics[scale=0.7]{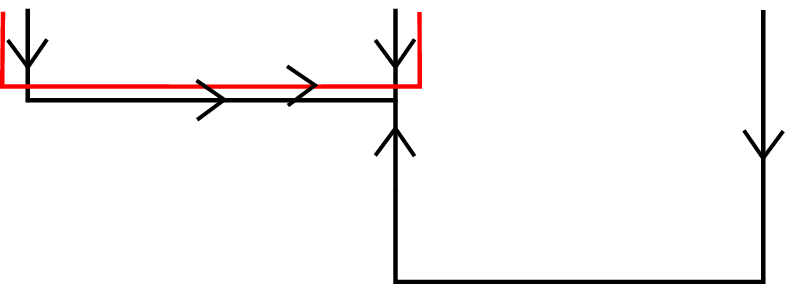}};
(13,-13.5)*{\text{Arc(a,0)}};
(-13,6.5)*{\text{Y(b,0)}};
\endxy
\]
\item[(b)] We have two webs with flows for the pair $S=(+,-,+,-)$ and $J=(0,0,0,0)$, i.e. two either nested or non-nested circles without flow
\[
\xy
(-30,-1.89)*{\includegraphics[scale=0.4]{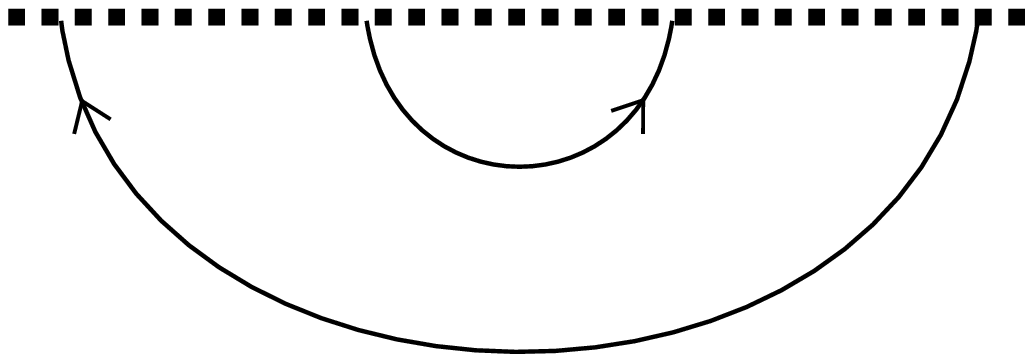}};
(30,1.89)*{\includegraphics[scale=0.4]{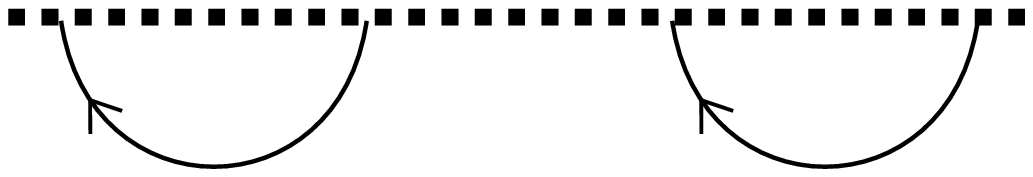}};
\endxy
\]
In this case the tableaux and the $3$-multipartitions are the same, but the LT-generators are different, that is
\[
T=\xy(0,0)*{\begin{Young}2 & 1 & 2\cr 4 & 3 & 4\cr\end{Young}}\endxy\;\;
\vec{\lambda}=\left(\;\xy(0,0)*{\begin{Young}&\cr\cr\end{Young}}\endxy\;,\;\xy(0,0)*{\begin{Young}\cr\end{Young}}\endxy\;,\;\xy(0,0)*{\begin{Young}&\cr\cr\end{Young}}\endxy\;\right)\;\; LT_{\mathrm{left}}=F_2F_1^{(2)}F_3^{(2)}F_2^{(2)}\;\; LT_{\mathrm{right}}=F_1^{(2)}F_2F_3^{(2)}F_2^{(2)}
\]
For the left case one gets the following sequence
\begin{align*}
\vec{T}_0=(\emptyset,\emptyset,\emptyset)&\mapsto \vec{T}_1=\left(\;\xy(0,0)*{\begin{Young}1\cr\end{Young}}\endxy\;,\;\emptyset\;,\;\xy(0,0)*{\begin{Young}1 \cr\end{Young}}\endxy\;\right)\mapsto \vec{T}_2=\left(\;\xy(0,0)*{\begin{Young}1 & 2\cr\end{Young}}\endxy\;,\;\emptyset\;,\;\xy(0,0)*{\begin{Young}1 & 2\cr\end{Young}}\endxy\;\right)\\&\mapsto \vec{T}_3=\left(\;\xy(0,0)*{\begin{Young}1 & 2\cr 3\cr\end{Young}}\endxy\;,\;\emptyset\;,\;\xy(0,0)*{\begin{Young}1 & 2\cr 3\cr\end{Young}}\endxy\;\right)\mapsto \vec{T}_4=\left(\;\xy(0,0)*{\begin{Young}1 & 2\cr 3\cr\end{Young}}\endxy\;,\;\xy(0,0)*{\begin{Young}4\cr\end{Young}}\endxy\;,\;\xy(0,0)*{\begin{Young}1 & 2\cr 3\cr\end{Young}}\endxy\;\right),
\end{align*}
while in the right case we get
\begin{align*}
\vec{T}_0=(\emptyset,\emptyset,\emptyset)&\mapsto \vec{T}_1=\left(\;\xy(0,0)*{\begin{Young}1\cr\end{Young}}\endxy\;,\;\emptyset\;,\;\xy(0,0)*{\begin{Young}1 \cr\end{Young}}\endxy\;\right)\mapsto \vec{T}_2=\left(\;\xy(0,0)*{\begin{Young}1 & 2\cr\end{Young}}\endxy\;,\;\emptyset\;,\;\xy(0,0)*{\begin{Young}1 & 2\cr\end{Young}}\endxy\;\right)\\&\mapsto \vec{T}_3=\left(\;\xy(0,0)*{\begin{Young}1 & 2\cr\end{Young}}\endxy\;,\;\xy(0,0)*{\begin{Young}3\cr\end{Young}}\endxy\;,\;\xy(0,0)*{\begin{Young}1 & 2\cr\end{Young}}\endxy\;\right)\mapsto \vec{T}_4=\left(\;\xy(0,0)*{\begin{Young}1 & 2\cr 4\cr\end{Young}}\endxy\;,\;\xy(0,0)*{\begin{Young}3\cr\end{Young}}\endxy\;,\;\xy(0,0)*{\begin{Young}1 & 2\cr 4\cr\end{Young}}\endxy\;\right),
\end{align*}
which can be easily verified from the following picture
\[
\xy
(-30,0)*{\includegraphics[scale=0.4]{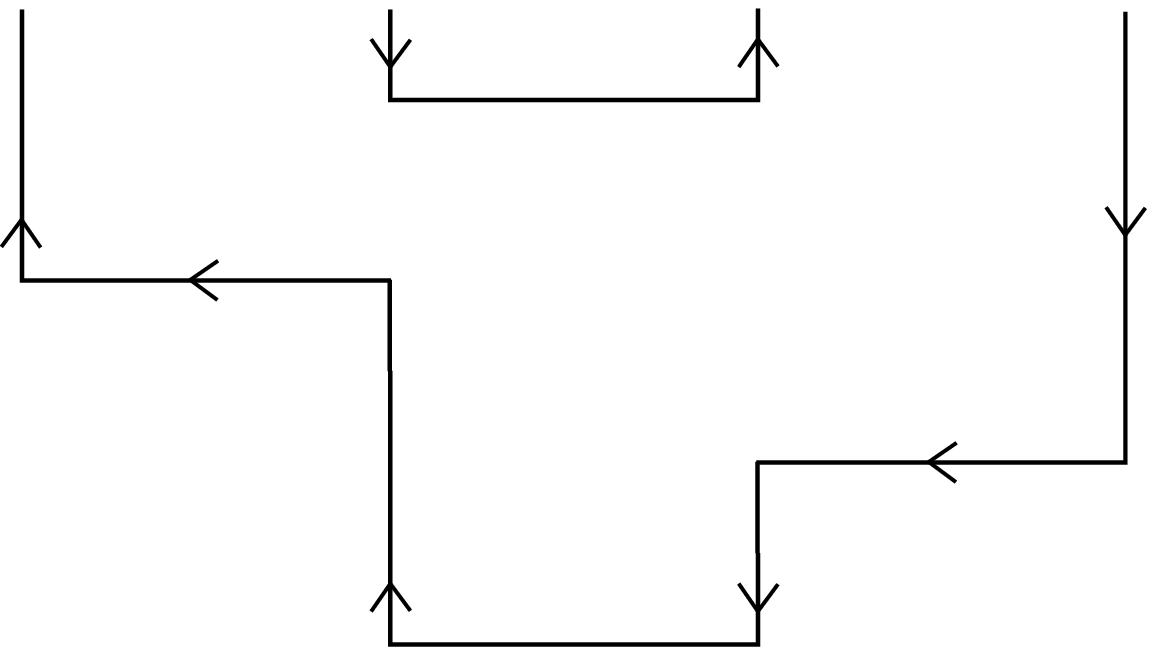}};
(-30,-16.5)*{\text{Arc(a,0)}};
(-30,12.5)*{\text{Arc(b,0)}};
(-14,-8.5)*{\text{right(b,0)}};
(-45,-1)*{\text{left(a,0)}};
(30,0)*{\includegraphics[scale=0.4]{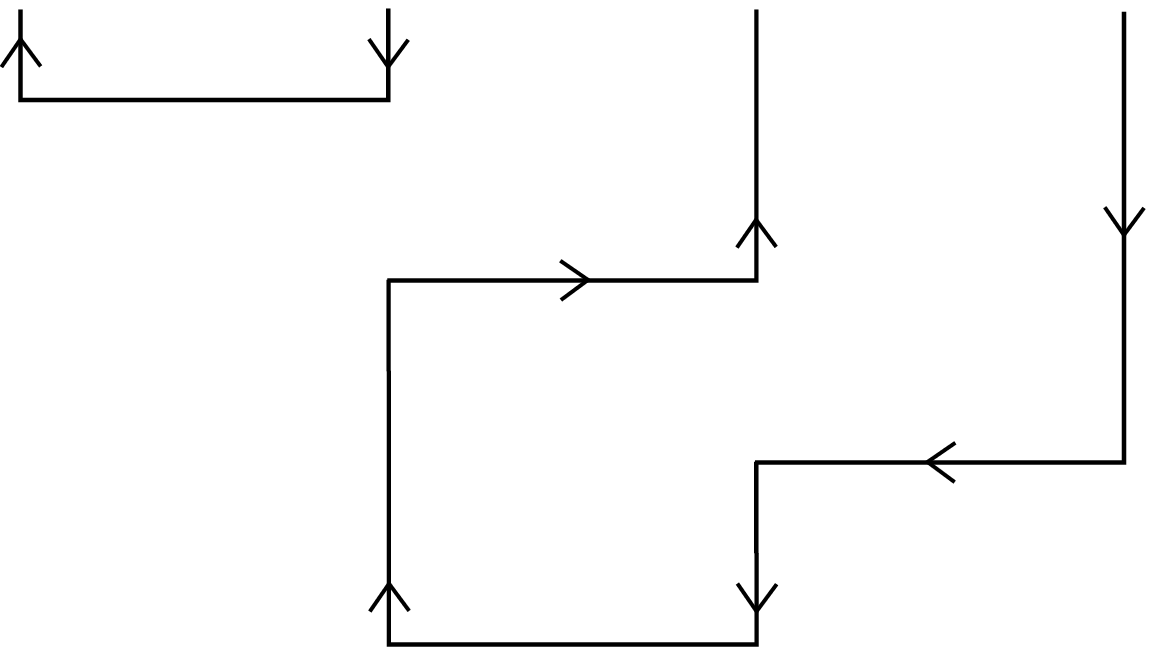}};
(30,-16.5)*{\text{Arc(a,0)}};
(15,12.5)*{\text{Arc(a,0)}};
(46,-8.5)*{\text{right(b,0)}};
(31,-1)*{\text{right(a,0)}};
\endxy
\]
It is worth noting that the residues of the nodes match exactly with the corresponding positions.
\end{itemize}
\end{ex}
\begin{lem}\label{lem-welldef}
The algorithm in Definition~\ref{defn-webtotab} is well-defined and two webs with flow $u_f,v_{f^{\prime}}$ satisfy
\[
\iota(u_f)=\iota(v_{f^{\prime}})\Leftrightarrow u=v\text{ and }f=f^{\prime}.
\]
\end{lem}
\begin{proof}
To show that the algorithm is well-defined observe that Lemma~\ref{lem-ltalgo} ensures that the total number of nodes of the $3$-multipartition $\vec{\lambda}_u$ associated to $u$ is exactly the number of nodes added in total by the algorithm from Definition~\ref{defn-webtotab}. Hence, we only need to ensure that the placement of the nodes works as claimed, that is, we fill the nodes without running into a step where there are either no nodes to fill at the corresponding position or the tableau can not be filled in a suitable standard way any more.
\vspace*{0.25cm}

To show this we use induction on the total length $\ell_{\mathrm{t}}(u)$. One can easily verify that the placement works for all cases with $\ell_{\mathrm{t}}(u)\leq 3$, i.e. the arcs and theta webs from Example~\ref{ex-totlength} with all possible flows.

By induction, given a web with a flow $u_f$ and LT-generators
\[
\mathrm{LT}(u)=\prod_k^{\ell(u)} F_{i_k}^{(j_k)},
\]
we know that the placement works for the web with flow $u^{<}_{f^<}$ that is obtained from $u_f$ by removing the last step. Note that $u^{<}_{f^<}$ has in general a different pair of sign string $S$ and state string $J$ than $u_f$. Moreover, the total length can drop by maximal three.
\vspace{0.25cm}

A case-by-case check considering all possible steps from Definition~\ref{defn-webtotab} (since all of them could be the last step) then shows the statement. They all work the same so we only illustrate two here, i.e. first an H-move of type a with color 0
\[
\xy(0,0)*{\includegraphics[scale=.5]{res/figs/webflow/hrulee.eps}};
(0,10)*{\begin{Young}1 &2 & 2\cr\end{Young}};
(0,-10)*{\begin{Young}1 &2 & 1\cr\end{Young}};
(-12.5,10)*{u_f};
(-13,-10)*{u^{<}_{f^<}};
\endxy
\]
where we have illustrated the change for the tableau going from $u^{<}_{f^<}$ (bottom) to $u_f$ (top). Note that we for simplicity use the entries $1,2$ again, but this does not affect the argument.

Hence, the difference between the associated bottom and top $3$-multipartition is that latter has one extra node in the last entry. The rules from Definition~\ref{defn-webtotab} now tell us to perform the following move for the H of type a and color $0$.
\[
\mathbf{k}\colon(T_{k-1}^{1},T_{k-1}^{2},T_{k-1}^{3})\mapsto
(T_{k-1}^{1},T_{k-1}^{2},T_{k-1}^{3}+\xy(0,0)*{\begin{Young}k\cr\end{Young}}\endxy),
\]
with $k=\ell(u)$. That the residue of this extra node is exactly $1$ is due to induction, since in one of the steps before (that is for the web $u^{<}_{f^<}$) we see that the $2$ in the middle was filled with a node that, by the map from Proposition~\ref{prop-tableauxflows2}, has the same residue as the new node.

Second, we consider an right shift of type b and color 0, that is
\[
\xy
(0,0)*{\includegraphics[scale=.5]{res/figs/webflow/shiftrightb2.eps}};
(0,10)*{\begin{Young}2 & & 2\cr\end{Young}};
(0,-10)*{\begin{Young}1 & & 1\cr\end{Young}};
(-12.5,10)*{u_f};
(-13,-10)*{u^{<}_{f^<}};
\endxy
\]
where we have again illustrated the change for the tableau going from $u^{<}_{f^<}$ (bottom) to $u_f$ (top). Thus, the top has two more nodes in the first and third partition. The rules tell us to put two nodes exactly in these partitions, that is 
\[
\mathbf{k}\colon(T_{k-1}^{1},T_{k-1}^{2},T_{k-1}^{3})\mapsto
(T_{k-1}^{1}+\xy(0,0)*{\begin{Young}k\cr\end{Young}}\endxy,T_{k-1}^{2},T_{k-1}^{3}+\xy(0,0)*{\begin{Young}k\cr\end{Young}}\endxy),
\]
with $k=\ell(u)$. That these two nodes have the right residue follows by induction again and the fact that this move corresponds to a $F_1^{(2)}$ and the two extra nodes will be in a row that starts with residue $1$ (by our shift in the residue from Definition~\ref{defn-tabcomb}) and was empty for $u^{<}_{f^<}$ (again by Proposition~\ref{prop-tableauxflows2}).

Hence, the placement works out as claimed. Moreover, it follows directly that the corresponding tableau is still standard.
\vspace*{0.25cm}

To show the second statement we use induction on the total length again. First we consider the case $u=v$ and $f\neq f^{\prime}$. Moreover, we can freely assume that the sign and state string pair is the same for both flows, since otherwise the shape of the corresponding $3$-multitableaux $\vec{T}_{u_f},\vec{T}_{u_{f^{\prime}}}$ is already different.

One can check the statement again by hand for the four webs from Example~\ref{ex-totlength}. As before, we consider the difference between $u^<_f$ and $u^<_{f^{\prime}}$ by removing the last step of the LT-algorithm and do a case-by-case check.

It should be noted that the number of cases we have to check is small, since the assumption $u=v$ fixes the type of the last performed move, i.e. only the color could be different, and the assumption that they have the same state string reduces the question to H-moves of type a. To be more precise, we need to ensure that the H-moves of type a and colors $1/0$, $1^{\prime}/-1$ and $0^{\prime}/-1^{\prime}$ perform a different last step, which is indeed the case. Hence, using induction, we see that
\[
\iota(u_f)=\iota(u_{f^{\prime}})\Leftrightarrow f= f^{\prime}.
\]

The case $u\neq v$ works in a slightly different vein, i.e. we are going to show that the two tableaux $\iota(u_f)$ and $\iota(v_{f^{\prime}})$ already differ at the first step, following the LT-algorithm for both webs, where the two webs differ.

In order to show this, we note that Proposition~\ref{prop-tableauxflows} shows that different webs $u$ and $v$ with different sign string $S$ or state string $J$ have different column-strict tableaux and therefore we see that the shapes of $\iota(u_f)$ and $\iota(v_{f^{\prime}})$ are different.

So let us consider the first step in the LT-algorithm of $u$ that is different form the one from $v$. Denote the $F$ obtained in this step for $u$ by $F^{(j_k)}_{i_k}$ and the one for $v$ by $F^{(j^{\prime}_k)}_{i^{\prime}_k}$.

Since we assume that they differ at this position, we see that $i_k\neq i^{\prime}_k$ or $j_k\neq j^{\prime}_k$. In the first case the residue of the new node with entry $k$ is different by our construction, since the residue corresponds to the $i_k$ and the $i^{\prime}_k$. In the second case the total number of new boxes with entry $k$ are different, since by our construction they correspond to $j_k$ and the $j^{\prime}_k$. In both case we see that the tableaux are different.

Hence, we see that
\[
\iota(u_f)=\iota(v_{f^{\prime}})\Leftrightarrow u=v\text{ and }f=f^{\prime}.
\]
This proves the lemma. 
\end{proof}
\subsection{Extended growth algorithm}\label{sec-exgrowth}
In this section we give the extended growth algorithm
\[
\mathrm{g}\colon\mathrm{Std}(\vec{\lambda})\to W_S^J\supset B_S^J,
\]
i.e. we give a method to obtain from an element of $\mathrm{Std}(\vec{\lambda})$ a web with a flow $u_f\in W_S^J$ (recall that $W_S^J$ means the set of \textit{all} possible webs with boundary pair $(S,J)$). In general, this web can be elliptic. We start with some preliminary combinatorial definitions. It is worth noting that the following definitions can be seen as an extended version of the definitions in the $\mathfrak{sl}_2$ case, i.e. the case of the level two Hecke algebras, see for example~\cite{bs1},~\cite{bs2} and~\cite{bs3}.
\begin{defn}\label{defn-combinatorics1}
Let $\vec{T}\in\mathrm{Std}(\vec{\lambda})$ be a $3$-multitableau $\vec{T}=(T^1,T^2,T^3)$ such that $T^i$ is a standard tableau and the numbers are from a fixed set $\{1,\dots,k\}$. Then we associate to $\vec{T}$ a \textit{sequence of $3$-multitableaux} $(\vec{T}^j)$ for each $j\in\{0,1,\dots,k\}$ where $\vec{T}^j=(T_j^1,T_j^2,T_j^3)$ and $T_j^{1,2,3}$ is obtained from $T^{1,2,3}$ by deleting all nodes with numbers strictly bigger than $j$.

Moreover, we associate to it a \textit{sequence of $3$-multipartitions} $(\vec{\lambda}^j)$ by removing the entries of the nodes of $(\vec{T}^j)$.
\end{defn}
\begin{ex}\label{ex-combinatorics1}
Given the following $3$-multitableau
\[
\vec{T}=\left(\;
\xy
 (0,0)*{
\begin{Young}
1&2 \cr
3 \cr
\end{Young}}
\endxy\;,
\;\xy
 (0,0)*{\begin{Young}
3 \cr
\end{Young}}\endxy\;,\;
\xy
 (0,0)*{
\begin{Young}
1&4 \cr
\end{Young}}
\endxy
\;\right),
\]
one obtains the following sequence. First note that, by definition, $\vec{T}^0=(\emptyset,\emptyset,\emptyset)$ and the rest is
\[
\vec{T}^1=\left(\;
\xy
 (0,0)*{
\begin{Young}
1 \cr
\end{Young}}
\endxy\;,
\;\emptyset\;,\;
\xy
 (0,0)*{
\begin{Young}
1 \cr
\end{Young}}
\endxy
\;\right),\;\;\vec{T}^2=\left(\;
\xy
 (0,0)*{
\begin{Young}
1&2 \cr
\end{Young}}
\endxy\;,
\;\emptyset\;,\;
\xy
 (0,0)*{
\begin{Young}
1 \cr
\end{Young}}
\endxy
\;\right),\;\;\vec{T}^3=\left(\;
\xy
 (0,0)*{
\begin{Young}
1&2 \cr
3 \cr
\end{Young}}
\endxy\;,
\;\xy
 (0,0)*{
\begin{Young}
3 \cr
\end{Young}}\endxy\;,\;
\xy
 (0,0)*{
\begin{Young}
1 \cr
\end{Young}}
\endxy
\;\right),\;\;\vec{T}^4=\vec{T}.
\]
\end{ex}
In order to make connection to webs with flows recall that a flow can be seen as either $+1$, if it is pointing downwards, $-1$, if it is pointing upwards or $0$, if there is no flow at all. We use two extra symbols $\circ,\times$ to illustrate the case when there is no web at all.
\begin{defn}\label{defn-combinatorics2}
A \textit{$\mathfrak{sl}_3$-weight diagram} $\Lambda$ is a $\bZ$-graded tuple with entries $\Lambda_i$ of symbols from the set $\{\circ,-1,0,+1,\times\}$.

A $\mathfrak{sl}_3$-weight diagram $\Lambda$ is said to be \textit{trivial}, if $\Lambda_{\bN_{\leq 0}}\in\{\times\}$ and $\Lambda_{\bN_{>0}}\in\{\circ\}$.
\end{defn}
\begin{ex}\label{ex-combinatorics2}
An example of a $\mathfrak{sl}_3$-weight diagram $\Lambda$ is the following.
\[
\xy(-52.5,0)*{\dots};
(-45,0)*{\times};
(-37.5,0)*{\times};
(-30,0)*{+1};
(-22.5,0)*{-1};
(-15,0)*{0};
(-7.5,0)*{0};
(0,0)*{\times};
(0,-5)*{\bullet};
(7.5,0)*{\circ};
(15,0)*{\times};
(22.5,0)*{+1};
(30,0)*{0};
(37.5,0)*{\circ};
(45,0)*{\circ};
(52.5,0)*{\dots};
\endxy
\]
Note that the $\bullet$ indicates the position of $\Lambda_0$. And the entries $\Lambda_i$ with $i>0$ are right of $\Lambda_0$, while the entries $\Lambda_i$ with $i<0$ are left of $\Lambda_0$.
\end{ex}
\begin{defn}\label{defn-combinatorics3}
A \textit{tower of $\mathfrak{sl}_3$-weight diagrams} $(\Lambda^i)_{i\in I}$ is a sequence of $\bZ$-graded tuples with entries $\Lambda_i$ of symbols from the set $\{\circ,-1,0,+1,\times\}$ for a finite index set $I=\{0,1,\dots,k\}$.

We draw the sequence as a tower starting from the bottom with $\Lambda^0$. 
\end{defn}
\begin{ex}\label{ex-combinatorics3}
An example of such a tower is the following.
\[
\xy(-52.5,0)*{\dots};
(-45,0)*{\times};
(-37.5,0)*{\times};
(-30,0)*{+1};
(-22.5,0)*{-1};
(-15,0)*{0};
(-7.5,0)*{0};
(0,0)*{\times};
(0,-5)*{\bullet};
(7.5,0)*{\circ};
(15,0)*{\times};
(22.5,0)*{+1};
(30,0)*{0};
(37.5,0)*{\circ};
(45,0)*{\circ};
(52.5,0)*{\dots};
(-52.5,10)*{\dots};
(-45,10)*{\times};
(-37.5,10)*{\times};
(-30,10)*{+1};
(-22.5,10)*{-1};
(-15,10)*{0};
(-7.5,10)*{0};
(0,10)*{\times};
(7.5,10)*{\circ};
(15,10)*{\times};
(22.5,10)*{+1};
(30,10)*{0};
(37.5,10)*{\circ};
(45,10)*{\circ};
(52.5,10)*{\dots};
(-52.5,20)*{\dots};
(-45,20)*{\times};
(-37.5,20)*{\times};
(-30,20)*{-1};
(-22.5,20)*{0};
(-15,20)*{-1};
(-7.5,20)*{-1};
(0,20)*{\circ};
(7.5,20)*{\circ};
(15,20)*{0};
(22.5,20)*{-1};
(30,20)*{\times};
(37.5,20)*{\circ};
(45,20)*{\circ};
(52.5,20)*{\dots};
(-52.5,30)*{\dots};
(-45,30)*{\times};
(-37.5,30)*{\times};
(-30,30)*{0};
(-22.5,30)*{-1};
(-15,30)*{\circ};
(-7.5,30)*{\times};
(0,30)*{\times};
(7.5,30)*{+1};
(15,30)*{\circ};
(22.5,30)*{0};
(30,30)*{-1};
(37.5,30)*{\circ};
(45,30)*{\circ};
(52.5,30)*{\dots};
\endxy
\]
\end{ex}
\begin{defn}\label{defn-combinatorics4}
Let $\vec{T}\in\mathrm{Std}(\vec{\lambda})$ be a $3$-multitableau and let $(\vec{\lambda}^j)$ denote its associated sequence of $3$-multipartitions. We denote the number of nodes in the $1,2,3$-th entry and the $l$-th row of $T_j^{1,2,3}$ by $|T_j^{1,2,3}|_l$. We can associate to it a tower of $\mathfrak{sl}_3$-weight diagrams $(\Lambda^i)_{i\in I}$ with $I=\{0,\dots,k\}$, where $k$ is the maximal entry of $\vec{T}$ in the following way.
\begin{itemize}
\item First define for each $(\vec{\lambda}^j)$ three $\bN$-graded tuples $\vec{k}^j_{1,2,3}$ by
\[
\vec{k}^j_{1,2,3}=(|T_j^{1,2,3}|_1,|T_j^{1,2,3}|_2-1,|T_j^{1,2,3}|_3-2,\dots).
\]
\item Set $\Lambda^0$ to be the trivial $\mathfrak{sl}_3$-weight diagram.
\item For $i=1,\dots,k$ we use the convention
\[
\Lambda^i_l=\begin{cases}
\times,  & \text{if }l\text{ is an entry of } \vec{k}^j_{1,2,3},\\
+1,  & \text{if }l\text{ is only an entry of } \vec{k}^j_{1,2}\text{ or only of } \vec{k}^j_{1},\\
0,  & \text{if }l\text{ is only an entry of } \vec{k}^j_{1,3}\text{ or only of } \vec{k}^j_{2},\\
-1,  & \text{if }l\text{ is only an entry of } \vec{k}^j_{2,3}\text{ or only of } \vec{k}^j_{3},\\
\circ,  & \text{otherwise}.
\end{cases}
\]
\end{itemize}
It should be noted that the convention agrees with the initial convention to set $\Lambda^0$ to be the trivial, since the corresponding $3$-multitableau $\vec{T}^0$ will only contain empty sets.
\end{defn}
\begin{ex}\label{ex-combinatorics4}
Suppose one has the following filling of a $3$-multipartition $\vec{\lambda}$, which is the one from Example~\ref{ex-tabl2} part (a),
\[
\vec{T}=\left(\;\xy(0,0)*{\begin{Young}1\cr\end{Young}}\endxy\;,\;\emptyset\;,\;\xy(0,0)*{\begin{Young}1\cr 2\cr\end{Young}}\endxy \;\right).
\]
Then we only have to consider three steps, i.e. the initial plus two non-trivial steps. We illustrate the resulting vectors $\vec{k}^j_{1,2,3}$ below with $j$ increasing from left to right and $\vec{k}^j_{1}$ in the top row, $\vec{k}^j_{2}$ in the middle and $\vec{k}^j_{3}$ at the bottom.
\[
\xy
(-40,5)*{(0,-1,-2,\dots)};
(-40,0)*{(0,-1,-2,\dots)};
(-40,-5)*{(0,-1,-2,\dots)};
(-40,-10)*{j=0};
(0,5)*{(1,-1,-2,\dots)};
(0,0)*{(0,-1,-2,\dots)};
(0,-5)*{(1,-1,-2,\dots)};
(0,-10)*{j=1};
(40,5)*{(1,-1,-2,\dots)};
(40,0)*{(0,-1,-2,\dots)};
(40,-5)*{(1,0,-2,\dots)};
(40,-10)*{j=2};
\endxy
\]
The associated tower is illustrated below. Note that it is no coincidence that we need exactly as many steps as in the LT-algorithm for the half-theta web. In fact, the reader is encourage to draw the half-theta web with the flow from Example~\ref{ex-tabl2} (a) inside.
\[
\xy(-22.5,0)*{\dots};
(-15,0)*{\times};
(-7.5,0)*{\times};
(0,0)*{\times};
(0,-5)*{\bullet};
(7.5,0)*{\circ};
(15,0)*{\circ};
(22.5,0)*{\dots};
(-22.5,10)*{\dots};
(-15,10)*{\times};
(-7.5,10)*{\times};
(0,10)*{0};
(7.5,10)*{0};
(15,10)*{\circ};
(22.5,10)*{\dots};
(-22.5,20)*{\dots};
(-15,20)*{\times};
(-7.5,20)*{1};
(0,19.8)*{-1};
(7.5,20)*{0};
(15,20)*{\circ};
(22.5,20)*{\dots};
\endxy
\]
\end{ex}
\begin{rem}\label{rem-comb}
Usually it is more convenient to extend the notions from Definition~\ref{defn-combinatorics4} slightly by allowing the entries $+1,0,-1$ to be marked with a $*$. For the webs this has the advantage that the tower constructed in~\ref{defn-combinatorics4} suffices to recover the web completely by the convention that a $*$-marked entry is a downwards pointing arrow of the web and an entry $+1,0,-1$ without such a marked is a upwards pointing arrow. We use these notions in the following.

In order to define $\Lambda^i_l$ with $*$-markers one just have to refine the convention and also count the number of occurrences in the vectors $\vec{k}^j_{1,2,3}$, i.e. use a $*$-marker iff the corresponding entry appears twice. Then the example from above would be
\[
\xy(-22.5,0)*{\dots};
(-15,0)*{\times};
(-7.5,0)*{\times};
(0,0)*{\times};
(0,-5)*{\bullet};
(7.5,0)*{\circ};
(15,0)*{\circ};
(22.5,0)*{\dots};
(-22.5,10)*{\dots};
(-15,10)*{\times};
(-7.5,10)*{\times};
(0,10)*{0};
(7.5,10)*{0^*};
(15,10)*{\circ};
(22.5,10)*{\dots};
(-22.5,20)*{\dots};
(-15,20)*{\times};
(-7.5,20)*{1^*};
(0,19.8)*{-1^*};
(7.5,20)*{0^*};
(15,20)*{\circ};
(22.5,20)*{\dots};
\endxy
\]
\end{rem}
\vspace*{0.15cm}

We are now ready to define the \textit{extended growth algorithm} for $\mathfrak{sl}_3$-webs.
\begin{defn}\label{defn-exgrowth}(\textbf{Extended growth algorithm})
The \textit{extended growth algorithm} is a map
\[
\mathrm{g}\colon\mathrm{Std}(\vec{\lambda})\to W_S^J\supset B_S^J,
\]
i.e. it assigns to every $3$-multitableau $\vec{T}\in\mathrm{Std}(\vec{\lambda})$ a (possible elliptic) web with a flow $\mathrm{g}(\vec{T})$.

The $\mathrm{g}(\vec{T})$ is defined to be the web that one obtains by connecting the entries of the tower $(\Lambda^i)_{i\in I}$ of Definition~\ref{defn-combinatorics4} in the way explained in Remark~\ref{rem-comb}.
\end{defn}
\begin{thm}\label{thm-fine}
The extended growth algorithm is well-defined. Moreover, we have
\[
\mathrm{g}\circ\iota=\mathrm{id}_{B^J_S},
\]
that is the extended growth algorithm can be used to obtain any non-elliptic web with any flow and if one start with the unique $3$-multitableau $\iota(\cdot)$, then the resulting web is also unique.
\end{thm}
\vspace*{0.15cm}
\begin{ex}\label{ex-big}
Before we prove the theorem, let us give an more sophisticated example. We take the web $u_f$ with one internal hexagon and with the following flow.
\[
\xy
(0,0)*{\includegraphics[height=.075\textheight]{res/figs/exgrowth/hexanewflowa2.eps}}
\endxy
\]
As one can easily check (see also Example~\ref{ex-flowline}) the corresponding column-strict tableau is
\[
T=\xy
 (0,0)*{
\begin{Young}
2&1&2 \cr
4&3&4 \cr
6&5&6 \cr
\end{Young}}
\endxy\]
and the corresponding $3$-multipartition is
\[
\vec{\lambda}=
\left(\;
\xy
 (0,0)*{
\begin{Young}
&& \cr
& \cr
 \cr
\end{Young}}
\endxy\;,
\; \xy
 (0,0)*{
\begin{Young}
 &\cr
 \cr
\end{Young}}
\endxy\;,\;
\xy
 (0,0)*{
\begin{Young}
&& \cr
& \cr
 \cr
\end{Young}}
\endxy
\;\right).
\]
Moreover, one easily checks that the sequence of LT-generators is
\[
LT(u_f)=F_1F_2F^{(2)}_3F_2F_1F_4F_3F_2F^{(2)}_5F^{(2)}_4F^{(2)}_3,
\]
which can be verified as we explained above in Section~\ref{sec-compare}.

Everything is of course also the same for the web
\[
\xy
(0,0)*{\includegraphics[height=.075\textheight]{res/figs/exgrowth/hexanewflowa.eps}}
\endxy
\]
since the data on the boundary is the same in both cases.

We have to fill the $3$-multitableaux with numbers $\{1,\dots,11\}$ to distinguish the two and the algorithm given in Definition~\ref{defn-webtotab} for the first makes the following twelve steps.
\begin{align*}
\vec{T}_0=(\emptyset,\emptyset,\emptyset)&\mapsto \vec{T}_1=\left(\;\xy(0,0)*{\begin{Young}1\cr\end{Young}}\endxy\;,\;\emptyset\;,\;\xy(0,0)*{\begin{Young}1 \cr\end{Young}}\endxy \;\right)\mapsto \vec{T}_2=\left(\;\xy(0,0)*{\begin{Young}1 & 2\cr\end{Young}}\endxy\;,\;\emptyset\;,\;\xy(0,0)*{\begin{Young}1 & 2\cr\end{Young}}\endxy \;\right)\\
&\mapsto \vec{T}_3=\left(\;\xy(0,0)*{\begin{Young}1 & 2 & 3\cr\end{Young}}\endxy\;,\;\emptyset\;,\;\xy(0,0)*{\begin{Young}1 & 2 & 3\cr\end{Young}}\endxy \;\right)\mapsto \vec{T}_4=\left(\;\xy(0,0)*{\begin{Young}1 & 2 & 3\cr 4\cr\end{Young}}\endxy\;,\;\emptyset\;,\;\xy(0,0)*{\begin{Young}1 & 2 & 3\cr\end{Young}}\endxy \;\right)\\
&\mapsto \vec{T}_5=\left(\;\xy(0,0)*{\begin{Young}1 & 2 & 3\cr 4\cr\end{Young}}\endxy\;,\;\xy(0,0)*{\begin{Young}5\cr\end{Young}}\endxy\;,\;\xy(0,0)*{\begin{Young}1 & 2 & 3\cr\end{Young}}\endxy \;\right)\mapsto \vec{T}_6=\left(\;\xy(0,0)*{\begin{Young}1 & 2 & 3\cr 4\cr\end{Young}}\endxy\;,\;\xy(0,0)*{\begin{Young}5 & 6\cr\end{Young}}\endxy\;,\;\xy(0,0)*{\begin{Young}1 & 2 & 3\cr\end{Young}}\endxy \;\right)\\
&\mapsto \vec{T}_7=\left(\;\xy(0,0)*{\begin{Young}1 & 2 & 3\cr 4 \cr 7\cr\end{Young}}\endxy\;,\;\xy(0,0)*{\begin{Young}5 & 6\cr\end{Young}}\endxy\;,\;\xy(0,0)*{\begin{Young}1 & 2 & 3\cr\end{Young}}\endxy \;\right)\mapsto \vec{T}_8=\left(\;\xy(0,0)*{\begin{Young}1 & 2 & 3\cr 4 \cr 7\cr\end{Young}}\endxy\;,\;\xy(0,0)*{\begin{Young}5 & 6\cr\end{Young}}\endxy\;,\;\xy(0,0)*{\begin{Young}1 & 2 & 3\cr 8\cr\end{Young}}\endxy \;\right)\\
&\mapsto \vec{T}_9=\left(\;\xy(0,0)*{\begin{Young}1 & 2 & 3\cr 4 & 9\cr 7\cr\end{Young}}\endxy\;,\;\xy(0,0)*{\begin{Young}5 & 6\cr\end{Young}}\endxy\;,\;\xy(0,0)*{\begin{Young}1 & 2 & 3\cr 8 & 9\cr\end{Young}}\endxy \;\right)\mapsto \vec{T}_{10}=\left(\;\xy(0,0)*{\begin{Young}1 & 2 & 3\cr 4 & 9\cr 7\cr\end{Young}}\endxy\;,\;\xy(0,0)*{\begin{Young}5 & 6\cr 10\cr\end{Young}}\endxy\;,\;\xy(0,0)*{\begin{Young}1 & 2 & 3\cr 8 & 9\cr\end{Young}}\endxy\right)\\
&\mapsto \vec{T}_{11}=\left(\;\xy(0,0)*{\begin{Young}1 & 2 & 3\cr 4 & 9\cr 7\cr\end{Young}}\endxy\;,\;\xy(0,0)*{\begin{Young}5 & 6\cr 10\cr\end{Young}}\endxy\;,\;\xy(0,0)*{\begin{Young}1 & 2 & 3\cr 8 & 9\cr 11\cr\end{Young}}\endxy \;\right).
\end{align*}
It is worth noting that the node in the seventh step, due to our convention for the residue, is placed the way it is. The reader is encourage to try to do the example
\[
\xy
(0,0)*{\includegraphics[height=.075\textheight]{res/figs/exgrowth/hexanewflowa.eps}}
\endxy
\]
that has the same boundary datum as before (and the same LT-generators). The resulting filling in this case after the last step is
\[
\vec{T}_{11}=\left(\;\xy(0,0)*{\begin{Young}1 & 2 & 3\cr 8 & 9\cr 11\cr\end{Young}}\endxy\;,\;\xy(0,0)*{\begin{Young}5 & 6\cr 10\cr\end{Young}}\endxy\;,\;\xy(0,0)*{\begin{Young}1 & 2 & 3\cr 4 & 9\cr 7\cr\end{Young}}\endxy \;\right)
\]
Following the extended growth algorithm one obtains the original web with flow back. Again we encourage the reader to do it for the case with the reversed flow.
\[
\xy
(-9,0)*{\includegraphics[scale=0.405]{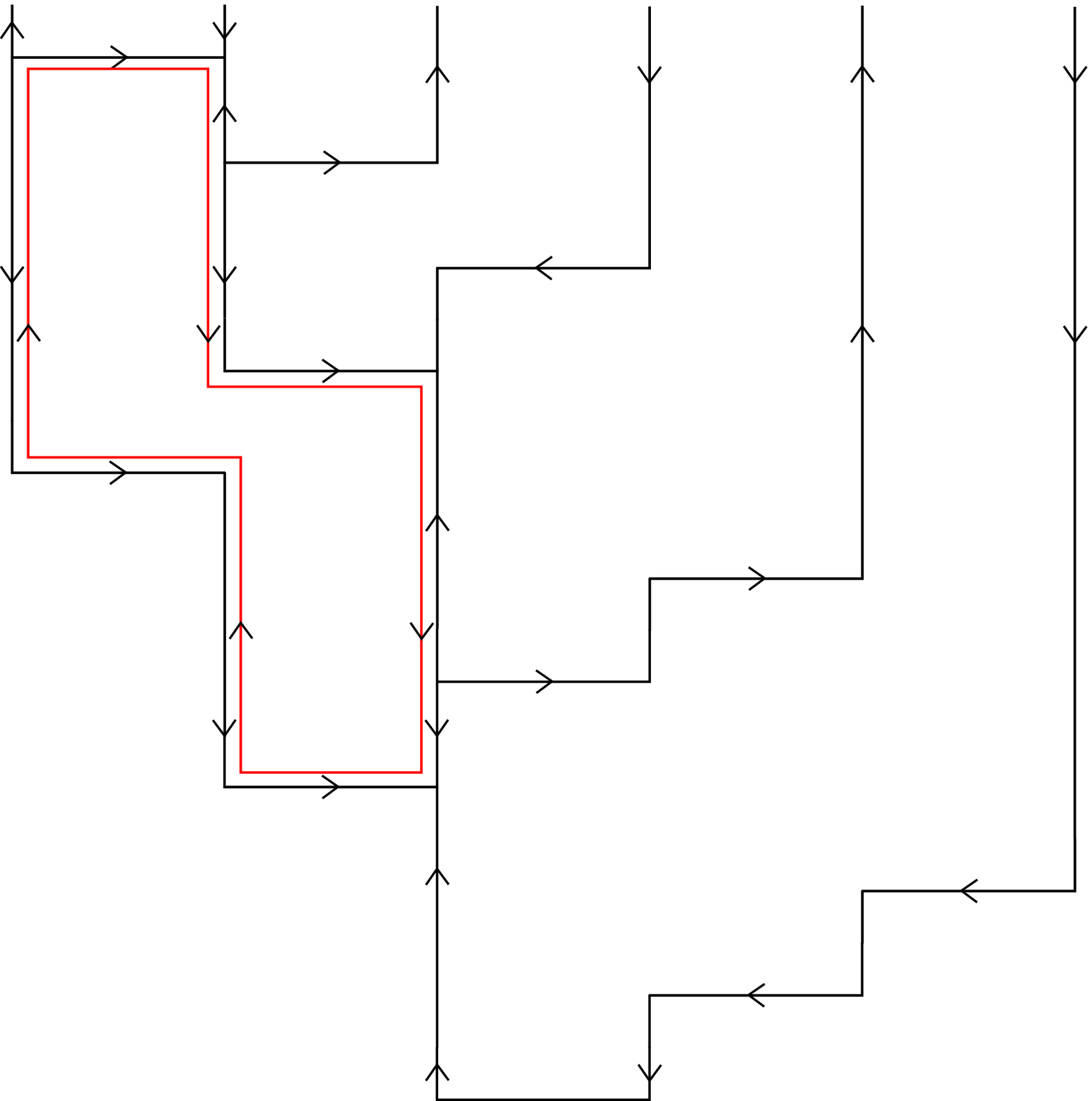}};
(-75,-45)*{\dots};
(-60,-45)*{\times};
(-45,-45)*{\times};
(-30,-45)*{\times};
(-15,-45)*{\times};
(0,-45)*{\circ};
(-15,-50)*{\bullet};
(15,-45)*{\circ};
(30,-45)*{\circ};
(45,-45)*{\circ};
(60,-45)*{\circ};
(75,-45)*{\dots};
(-75,-37.5)*{\dots};
(-60,-37.5)*{\times};
(-45,-37.5)*{\times};
(-30,-37.5)*{\times};
(-14.8,-37.2)*{0};
(0.8,-37.2)*{0^*};
(15,-37.5)*{\circ};
(30,-37.5)*{\circ};
(45,-37.5)*{\circ};
(60,-37.5)*{\circ};
(75,-37.5)*{\dots};
(-75,-30)*{\dots};
(-60,-30)*{\times};
(-45,-30)*{\times};
(-30,-30)*{\times};
(-14.8,-30)*{0};
(0,-30)*{\circ};
(15.8,-29.9)*{0^*};
(30,-30)*{\circ};
(45,-30)*{\circ};
(60,-30)*{\circ};
(75,-30)*{\dots};
(-75,-22.5)*{\dots};
(-60,-22.5)*{\times};
(-45,-22.5)*{\times};
(-30,-22.5)*{\times};
(-14.8,-22.5)*{0};
(0,-22.5)*{\circ};
(15,-22.5)*{\circ};
(30.8,-22.4)*{0^*};
(45,-22.5)*{\circ};
(60,-22.5)*{\circ};
(75,-22.5)*{\dots};
(-75,-15)*{\dots};
(-60,-15)*{\times};
(-45,-15)*{\times};
(-28.5,-15)*{-1^*};
(-14.5,-15)*{1^*};
(0,-15)*{\circ};
(15,-15)*{\circ};
(30.8,-14.9)*{0^*};
(45,-15)*{\circ};
(60,-15)*{\circ};
(75,-15)*{\dots};
(-75,-7.5)*{\dots};
(-60,-7.5)*{\times};
(-45,-7.5)*{\times};
(-28.5,-7.5)*{-1^*};
(-14.9,-7.5)*{1};
(0.5,-7.5)*{0};
(15,-7.5)*{\circ};
(30.8,-7.4)*{0^*};
(45,-7.5)*{\circ};
(60,-7.5)*{\circ};
(75,-7.5)*{\dots};
(-75,0)*{\dots};
(-60,0)*{\times};
(-45,0)*{\times};
(-28.5,0)*{-1^*};
(-14.9,0)*{1};
(0,0)*{\circ};
(15,0)*{0};
(30.8,0)*{0^*};
(45,0)*{\circ};
(60,0)*{\circ};
(75,0)*{\dots};
(-75,7.5)*{\dots};
(-60,7.5)*{\times};
(-43.8,7.5)*{-1^*};
(-30,7.5)*{\times};
(-14.9,7.5)*{1};
(0,7.5)*{\circ};
(15,7.5)*{0};
(30.8,7.5)*{0^*};
(45,7.5)*{\circ};
(60,7.5)*{\circ};
(75,7.5)*{\dots};
(-75,15)*{\dots};
(-60,15)*{\times};
(-43.8,15)*{-1^*};
(-29.8,15)*{1^*};
(-14.6,15)*{0^*};
(0,15)*{\circ};
(15,15)*{0};
(30.8,15)*{0^*};
(45,15)*{\circ};
(60,15)*{\circ};
(75,15)*{\dots};
(-75,22.5)*{\dots};
(-60,22.5)*{\times};
(-43.8,22.5)*{-1^*};
(-29.8,22.5)*{1^*};
(-15,22.5)*{\circ};
(0.8,22.5)*{0^*};
(15,22.5)*{0};
(30.8,22.5)*{0^*};
(45,22.5)*{\circ};
(60,22.5)*{\circ};
(75,22.5)*{\dots};
(-75,30)*{\dots};
(-60,30)*{\times};
(-43.8,30)*{-1^*};
(-29.8,30)*{1};
(-14.6,30)*{0};
(0.8,30)*{0^*};
(15,30)*{0};
(30.8,30)*{0^*};
(45,30)*{\circ};
(60,30)*{\circ};
(75,30)*{\dots};
(-75,37.5)*{\dots};
(-60,37.5)*{\times};
(-44.8,37.5)*{0};
(-29.1,37.5)*{0^*};
(-14.6,37.5)*{0};
(0.8,37.5)*{0^*};
(15,37.5)*{0};
(30.8,37.5)*{0^*};
(45,37.5)*{\circ};
(60,37.5)*{\circ};
(75,37.5)*{\dots};
\endxy
\]
\end{ex}
\begin{proof}
The main part of this proof is the Lemma~\ref{lem-welldef}, i.e. $\iota\colon B^J_S\to\mathrm{Std}(\vec{\lambda})$ is a well-defined inclusion of webs with flows $u_f$ to filled $3$-multitableaux $\iota(u_f)$. In fact the only thing we have to see is that the process which we described in Definition~\ref{defn-exgrowth} is given by reading the process from Definition~\ref{defn-webtotab} backwards and place the web with flow $u_f$ on a suitable chosen grid.
\vspace*{0.25cm}

That the algorithm is well-defined follows by its deterministic nature.
\vspace*{0.25cm}

To see that the extended growth algorithm is a left-inverse for the embedding $\iota\colon B^J_S\to\mathrm{Std}(\vec{\lambda})$, we note that, by Proposition~\ref{prop-sl3bases}, we only have to consider non-elliptic webs. That the process of Definition~\ref{defn-webtotab} is literally undone by the process of Definition~\ref{defn-exgrowth} follows by comparing the two procedures as follows.
\vspace*{0.25cm}

We prove the statement again by induction on the total length. It is easy to verify all the small cases from Example~\ref{ex-totlength}, so assume that for all $1\leq i\leq j$ the statement is true and consider a web whose total length is $j=\ell_{\mathrm{t}}(u)$.

Hence, we have to check all the possible last moves. Since they all work in the same fashion, we only do two here and leave the rest to the reader. The first move we discuss is a Y-move of type b and color $0^{\prime}$, i.e.
\[
\xy
(0,0)*{\includegraphics[scale=.5]{res/figs/webflow/yrulee2.eps}};
(-9.25,-7)*{\times};
(9.25,-7)*{0};
(-9.24,-12)*{l^{\prime}};
(9.24,-12)*{l^{\prime}+1};
(18,-7)*{\cdots};
(-18,-7)*{\cdots};
(-9,7)*{-1^*};
(9.6,7)*{1^*};
(18,7)*{\cdots};
(-18,7)*{\cdots};
\endxy
\]
We assume that this is the $j$-th step. The rules from Definition~\ref{defn-webtotab} tells us to add one extra node labeled $j$ to the first tableau $T^1$, say in the $l$-th row. Looking at the information at the bottom, we see that the $\bN$-graded tuples $\vec{k}^{j-1}_{1,2,3}$ are of the form
\[
\xy
(-70,5)*{\vec{k}_1:};
(-70,0)*{\vec{k}_2:};
(-70,-5)*{\vec{k}_3:};
(-40,5)*{(\dots,l^{\prime},>l^{\prime}+1,\dots)};
(-40,0)*{(\dots,l^{\prime},l^{\prime}+1,\dots)};
(-40,-5)*{(\dots,l^{\prime},>l^{\prime}+1,\dots)};
(-40,-10)*{j-1};
(0,5)*{(\dots,<l^{\prime},l^{\prime}+1,\dots)};
(0,0)*{(\dots,l^{\prime},l^{\prime}+1,\dots)};
(0,-5)*{(\dots,l^{\prime},>l^{\prime}+1,\dots)};
(0,-10)*{j};
\endxy
\]
where we have only illustrated the important parts of the tuples. Here $l^{\prime}=|T^{1}_{j-1}|_l-l+1$.

This is true, because we can use the induction hypothesis and an easy calculation checking the corresponding residues. It should be noted that the grid is chosen in such a way that for $d=3\ell$ a $F_{\ell}^{(j_{\ell})}$ acts between the $0$-entry and the $1$-entry of the tower.

To see that the step from $j-1$ to $j$ happens as indicated above, we note that adding a node only in the first tableau $T^1$ only effects the tuple $\vec{k}^{j-1}_1$. Moreover, since we assume to add the node in the $l$-row, the only change is as indicated above because $|T^{1}_{j}|_l=|T^{1}_{j-1}|_l+1$.
\vspace*{0.25cm}

The second move we discuss is an arc-move of type a and color $-1$, i.e.
\[
\xy
(0,0)*{\includegraphics[scale=.5]{res/figs/webflow/arcc1.eps}};
(-9.25,-7)*{\times};
(9.25,-7)*{\circ};
(-9.24,-12)*{l^{\prime}};
(9.24,-12)*{l^{\prime}+1};
(18,-7)*{\cdots};
(-18,-7)*{\cdots};
(-9,7)*{-1};
(9.6,7)*{1^*};
(18,7)*{\cdots};
(-18,7)*{\cdots};
\endxy
\]
We use the same notation as before. Again, looking at the information at the bottom, we see that the $\bN$-graded tuples $\vec{k}^{j-1}_{1,2,3}$ are of the form
\[
\xy
(-70,5)*{\vec{k}_1:};
(-70,0)*{\vec{k}_2:};
(-70,-5)*{\vec{k}_3:};
(-40,5)*{(\dots,l^{\prime},>l^{\prime}+1,\dots)};
(-40,0)*{(\dots,l^{\prime},>l^{\prime}+1,\dots)};
(-40,-5)*{(\dots,l^{\prime},>l^{\prime}+1,\dots)};
(-40,-10)*{j-1};
(0,5)*{(\dots,<l^{\prime},l^{\prime}+1,\dots)};
(0,0)*{(\dots,<l^{\prime},l^{\prime}+1,\dots)};
(0,-5)*{(\dots,l^{\prime},>l^{\prime}+1,\dots)};
(0,-10)*{j};
\endxy
\]
Here $l^{\prime}=|T^{1,2}_{j-1}|_{l_{1,2}}-l_{1,2}+1$ by the same argument as above. Note that this time the rules from Definition~\ref{defn-webtotab} tell us to add a $j$ labeled node to the first $T^{1}$ and second $T^{2}$ tableau. Note that both nodes will be added in different rows $l_{1,2}$, but since they have the same residue, we see that $l^{\prime}=l^{\prime}_1=l^{\prime}_2$. Therefore the corresponding numbers for the LT-algorithm have to be in the same row, because their nodes were filled with the highest appearing numbers at the beginning. Hence, the change is as indicated above because $|T^{1,2}_{j}|_{l_{1,2}}=|T^{1,2}_{j-1}|_{l_{1,2}}+1$.

All other case work similar, hence this proves the statement.
\end{proof}
\subsection{Degree of tableaux and weight of flows}\label{sec-degflow}
We are going to recall Brundan, Kleshchev and Wang's definition of the degree of a tableau given in~\cite{bkw} in this section. We show that the embedding $\iota\colon B_S^J\to\mathrm{Std}(\vec{\lambda})$ is a degree preserving map, where we consider minus the weight of a web with flow $u_f\in B_S^J$ as the degree $\mathrm{deg}_{\mathrm{wt}}$ for $B_S^J$ and Brundan, Kleshchev and Wang's degree $\mathrm{deg}_{\mathrm{BKW}}$ as the degree for a $3$-multitableau $\vec{T}\in\mathrm{Std}(\vec{\lambda})$.

Note that this shows that $\mathrm{deg}_{\mathrm{BKW}}$ is an isotopy invariant and that Brundan, Kleshchev and Wang's degree has a very natural interpretation on the level of webs, i.e. it is minus the weight of a flow line and is therefore directly connected (to be more precise: to the slightly changed version) to the LT-coefficients in~\ref{eq:LT}, since Khovanov and Kuperberg showed in Section 4 of~\cite{kk} that the weight of flows gives the coefficients of the web basis $B_S$ with respect to the elementary tensors and we showed in Proposition~\ref{prop-sl3bases} that $B_S$ is in fact an intermediate crystal basis.
\vspace*{0.25cm}

Recall the following combinatorial definitions. We only recall the case of a $3$-multitableaux, since we do not need the more general case here.

The reader should be careful that we use a slightly different definition than Brundan, Kleshchev and Wang, since we are working with divided powers (and therefore we have in general more than one node with entry $j\in\{1,\dots,k\}$). This is the reason why we have some shifts in our definition. For example, an arc-move created by a divided power differs from the one without the divided power by a digon
\[
F^{(2)}\colon\;\xy(0,0)*{\includegraphics[scale=.5]{res/figs/webflow/arcb1.eps}}\endxy\;\;\text{ and }\;\;F^2\colon\;\xy(0,0)*{\includegraphics[scale=0.45]{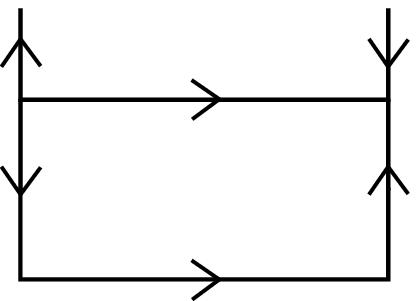}}\endxy
\]
and the empty-shift differs from the one without divided powers by a theta web
\[
F^{(3)}\colon\;
\xy
(9,7)*{\times};
(-9,-7)*{\times};
(-9,7)*{\circ};
(9,-7)*{\circ};
\endxy
\;\;\text{ and }\;\;F^3\colon\;\xy(0,0)*{\includegraphics[scale=0.45]{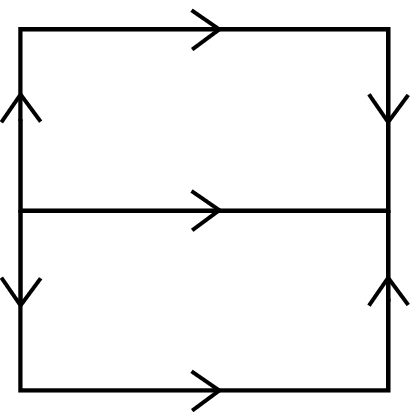}}\endxy
\]
\begin{defn}\label{defn-degpart1}(\textbf{Brundan, Kleshchev and Wang: Degree of a tableau})
Let $\vec{T}\in\mathrm{Std}(\vec{\lambda})$ be a (filled with numbers from $\{1,\dots,k\}$) $3$-multitableau $\vec{T}=(T^1,T^2,T^3)$ as in Definition~\ref{defn-combinatorics1}. For $j\in\{1,\dots,k\}$ let $N^j$ denote the set of all nodes that are filled with the number $j$ and let $\vec{T}^j$ denote the $3$-multitableau obtained from $\vec{T}$ by removing all nodes with entries $>j$.

The \textit{degree of $\vec{T}^j$}, denoted by $\mathrm{deg}(\vec{T}^j)$, is defined to be
\[
\mathrm{deg}(\vec{T}^j)=|\mathsf{A}^{k\succ N}(\vec{T}^j)|-|\mathsf{R}^{k\succ N}(\vec{T}^j)|-a\;\;\text{ with }\;\;a=\begin{cases}0, &\text{if }|N^j|=1,\\1, &\text{if }|N^j|=2,\\3, &\text{if }|N^j|=3,\end{cases}
\]
where we use the convention to count all nodes with the same number step by step starting from the leftmost.

The \textit{degree} of the $3$-multitableau $\vec{T}=(T^1,T^2,T^3)$, denoted by $\mathrm{deg}_{\mathrm{BKW}}(\vec{T})$, is then defined by
\[
\mathrm{deg}_{\mathrm{BKW}}(\vec{T})=\sum_{j=1}^k \mathrm{deg}(\vec{T}^j).
\]
\end{defn}
\begin{ex}\label{ex-degreea}
The following three standard $3$-multitableaux have all degree zero.
\[
\vec{T}_1 =\left(\emptyset\;,\;\emptyset\;,\;\xy(0,0)*{\begin{Young} 1\cr\end{Young}}\endxy \right)\;,\; \vec{T}_2=\left(\emptyset\;,\;\xy(0,0)*{\begin{Young} 1\cr\end{Young}}\endxy\;,\;\xy(0,0)*{\begin{Young} 1\cr\end{Young}}\endxy \right)\;,\; \vec{T}_3=\left(\xy(0,0)*{\begin{Young} 1\cr\end{Young}}\endxy\;,\;\xy(0,0)*{\begin{Young} 1\cr\end{Young}}\endxy\;,\;\xy(0,0)*{\begin{Young} 1\cr\end{Young}}\endxy \right).
\]
To see this, we note that in the first case there is no node after $\succ$ the unique node $N^1$. Hence, $\mathrm{deg}(\vec{T}_1)=0$. In the second case we have to calculate two steps. In the first step, i.e.
\[
\left(\emptyset\;,\;\xy(0,0)*{\begin{Young} 1\cr\end{Young}}\endxy\;,\;\xy(0,0)*{\begin{Young} $\cdot$\cr\end{Young}}\endxy \right),
\]
we count one addable node of the same residue which we have marked with a $\cdot$, but the second step there is again no node after $\succ$ the last node anymore. Hence, $\mathrm{deg}(\vec{T}_2)=0$, since we have to take the divided power into account. For the last case we have to calculate three steps, i.e. the first and the second are
\[
\left(\xy(0,0)*{\begin{Young} 1\cr\end{Young}}\endxy\;,\;\xy(0,0)*{\begin{Young} $\cdot$\cr\end{Young}}\endxy\;,\;\xy(0,0)*{\begin{Young} $\cdot$\cr\end{Young}}\endxy \right)\;\;\text{ and }\;\;\left(\xy(0,0)*{\begin{Young} 1\cr\end{Young}}\endxy\;,\;\xy(0,0)*{\begin{Young} 1\cr\end{Young}}\endxy\;,\;\xy(0,0)*{\begin{Young} $\cdot$\cr\end{Young}}\endxy \right),
\]
where we have again indicated the addable nodes of the same residue with a $\cdot$. The third step is again as before. Hence, $\mathrm{deg}(\vec{T}_3)=0$, because the divided power is $3$.
\vspace*{0.25cm}

It should be noted that it is possible that the degree (total or local) is negative. For example the last step of
\[
\vec{T}_4=\left(\xy(0,0)*{\begin{Young} 1 & 2 &3\cr 8 & 9\cr\end{Young}}\endxy\;,\;\xy(0,0)*{\begin{Young} 5 & 6\cr 10\cr 11\cr\end{Young}}\endxy\;,\;\xy(0,0)*{\begin{Young} 1 & 2 & 3\cr 4& 9\cr 7\cr\end{Young}}\endxy \right)
\]
has no addable nodes after $\succ$ the node $N^{11}$ with the same residue, but one removable, namely the left node filled with the entry $7$. Hence, $\mathrm{deg}(\vec{T}_4^{11})=-1$. The total degree in this case is
\[
\mathrm{deg}_{\mathrm{BKW}}(\vec{T}_4)=1+0+0+0+1+0+0+1+0+1-1=3.
\]
\end{ex}
\begin{ex}\label{ex-degreeb}
Given a column-strict tableau $T$, a string of LT-generators $\mathrm{LT}(T)$ and a filling, i.e. flow, for its corresponding $3$-multitableau $\vec{T}$ like
\[
T=\xy(0,0)*{\begin{Young} 2&  1&  3\cr  4& 3&4\cr\end{Young}}\endxy\;\;\text{ and }\;\;\mathrm{LT}(T)=F_1F_2^{(2)}F_1F_3^{(2)}F_2^{(2)}\;\;\text{ and }\;\;\vec{T}=\left(\xy(0,0)*{\begin{Young} 1& 2 \cr 5 \cr\end{Young}}\endxy\;,\;\xy(0,0)*{\begin{Young} 4\cr\end{Young}}\endxy\;,\;\xy(0,0)*{\begin{Young}1&2\cr 3&4 \cr\end{Young}}\endxy\right),
\]
one can easily draw a web with flow $u_f$ from this data
\[
u_f=\xy
(0,0)*{\includegraphics[scale=0.45]{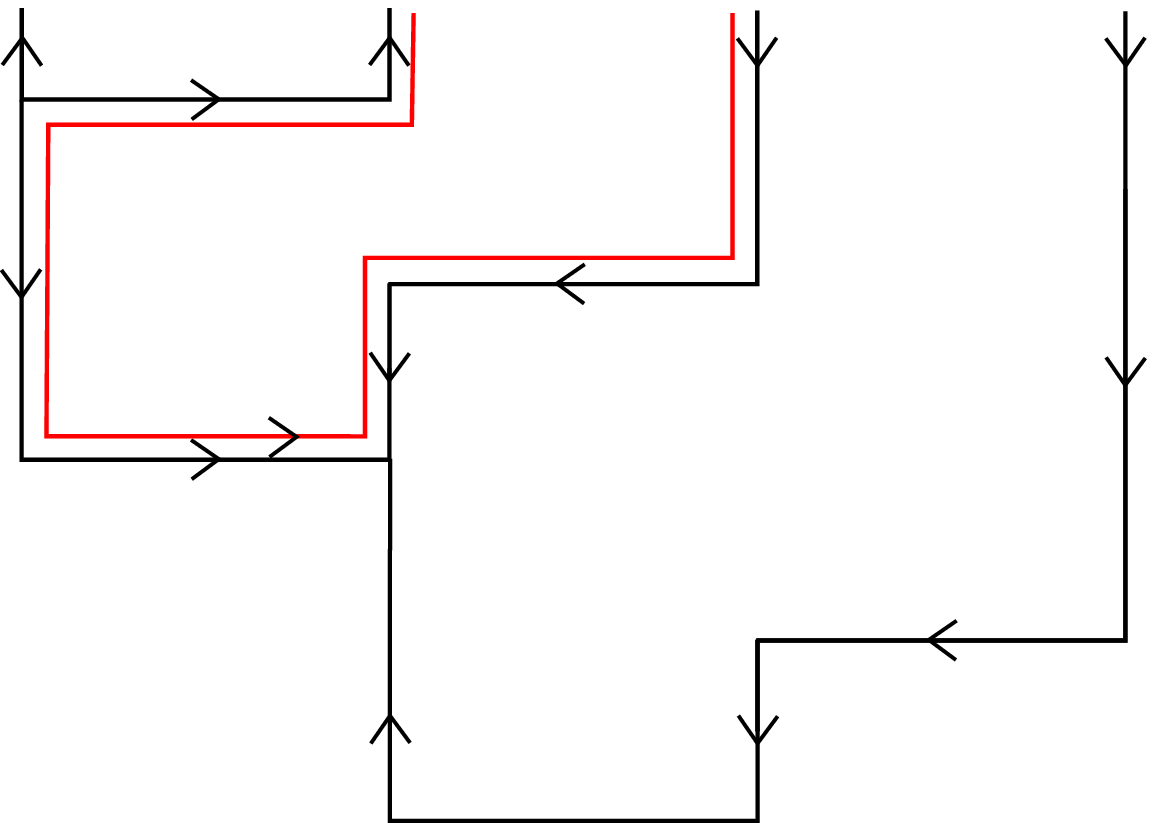}};
(0,-16.5)*{1};
(0,9.5)*{0};
(-17,1.25)*{0};
(-17,16.75)*{1};
(17,-7.5)*{0};
\endxy
\]
with total degree $\mathrm{deg}_{\mathrm{wt}}(u_f)=-\mathrm{wt}(u_f)=2$ and local degrees as indicated in the figure. The degree $\mathrm{deg}_{\mathrm{BKW}}(\vec{T})$ of $\vec{T}$ is given inductively by $(\vec{T}^j)^5_{j=1}$ with
\begin{align*}
\vec{T}^1 &=\left(\xy(0,0)*{\begin{Young}1\cr\end{Young}}\endxy\;,\;\emptyset\;,\;\xy(0,0)*{\begin{Young} 1\cr\end{Young}}\endxy \right)\;,\; \vec{T}^2=\left(\xy(0,0)*{\begin{Young} 1& 2\cr\end{Young}}\endxy\;,\;\emptyset\;,\;\xy(0,0)*{\begin{Young} 1& 2\cr\end{Young}}\endxy \right)\;,\; \vec{T}^3=\left(\xy(0,0)*{\begin{Young} 1& 2\cr\end{Young}}\endxy\;,\;\emptyset\;,\;\xy(0,0)*{\begin{Young} 1& 2\cr 3\cr\end{Young}}\endxy \right)\;,\\
\vec{T}^4&=\left(\xy(0,0)*{\begin{Young} 1 & 2\cr\end{Young}}\endxy\;,\;\xy(0,0)*{\begin{Young}4\cr\end{Young}}\endxy\;,\;\xy(0,0)*{\begin{Young} 1& 2\cr 3&4 \cr\end{Young}}\endxy \right)\;,\;
\vec{T}^5=\left(\xy(0,0)*{\begin{Young} 1& 2\cr 5\cr\end{Young}}\endxy\;,\;\xy(0,0)*{\begin{Young}4\cr\end{Young}}\endxy\;,\;\xy(0,0)*{\begin{Young} 1& 2\cr 3 &4 \cr\end{Young}}\endxy \right).
\end{align*}
Hence, we get
\[
\mathrm{deg}(\vec{T}^1)=1,\;\;\mathrm{deg}(\vec{T}^2)=0,\;\;\mathrm{deg}(\vec{T}^3)=0,\;\;\mathrm{deg}(\vec{T}^4)=0,\;\;\mathrm{deg}(\vec{T}^5)=1.
\]
Therefore we see that the total degree is
\[
\mathrm{deg}_{\mathrm{BKW}}(\vec{T})=1+0+0+0+1=2=\mathrm{deg}_{\mathrm{wt}}(u_f)
\]
and the local degrees are exactly the ones we have seen in the figure above.
\end{ex}
\vspace*{0.15cm}

Recall that the sets $B_S^J$ and $W_S^J$ are graded with $\mathrm{deg}_{\mathrm{wt}}(u_f)=-\mathrm{wt}(u_f)$, while the set $\mathrm{Std}(\vec{\lambda})$ is graded with the degree from Definition~\ref{defn-degpart1}. In fact the observation from Example~\ref{ex-degreeb} that the degree of a web with flow $u_f\in B^J_S$ is exactly the degree of $\iota(u_f)\in\mathrm{Std}(\vec{\lambda})$ is no coincidence, i.e. we note that following interesting Proposition.
\begin{prop}\label{prop-degree}
The map $\iota\colon B^J_S\to \mathrm{Std}(\vec{\lambda})$ and $g\colon \mathrm{Std}(\vec{\lambda})\to W^J_S$ are of degree $0$.
\end{prop}
\begin{proof}
The proof is again an inductive case-by-case check. It is easy to verify the statement for the case where the total length $\ell_{\mathrm{t}}(u)\leq 3$, i.e. for the four webs given in Example~\ref{ex-totlength}.

So assume that $\ell_{\mathrm{t}}(u)> 3$ and that the statement is true for all $i<\ell_{\mathrm{t}}(u)=k$, i.e. after removing the last LT-generator from $u_f$, which results in a web with flow denoted by $u^<_{<f}$, we have
\[
\mathrm{deg}_{\mathrm{wt}}(u^<_{<f})=\mathrm{deg}_{\mathrm{BKW}}(\iota(u^<_{<f})).
\]
Thus, we only have to show that the local changes of minus weight and $\mathrm{deg}_{\mathrm{BKW}}(\iota(u_f)^k)$ coincide.

We have to check all the possible cases again. We skip most of them and leave them to the reader, since they all work in the same vein. We do three as an example, i.e. an H-move of type a and color $1$, an H-move of type a and color $-1^{\prime}$, and a right shift of type b and color $0$.

The two H-moves of type a are almost the same, i.e. we have
\[
\xy
(-17.5,0)*{\includegraphics[scale=.5]{res/figs/webflow/hrulea.eps}};
(-17.5,3)*{-1};
(17.5,0)*{\includegraphics[scale=.5]{res/figs/webflow/hruled.eps}};
(17.5,3)*{+1};
\endxy\;\;\text{ and for both }\;\;
\mathbf{k}\colon
(\lambda^{k-1}_{1},\lambda^{k-1}_{2}+\xy(0,0)*{\begin{Young}k\cr\end{Young}}\endxy,\lambda^{k-1}_{3}).
\]
That is the H-move of type a and color $1$ lowers the degree by one, while the H-move of type a and color $-1^{\prime}$ raises the degree by $1$, but both have the same rule to place the node. 

To see why the degree is still different we note that they have a different state string at the top. In fact, if we use $1$ and $2$ for simplicity again, we have
\[
\xy
(-25,0)*{\includegraphics[scale=.5]{res/figs/webflow/hrulea.eps}};
(-25,10)*{\begin{Young}1 & 2 & 2\cr \end{Young}};
(-25,-10)*{\begin{Young}2 & 2 & 1\cr \end{Young}};
(25,0)*{\includegraphics[scale=.5]{res/figs/webflow/hruled.eps}};
(25,10)*{\begin{Young}2 & 2 & 1\cr \end{Young}};
(25,-10)*{\begin{Young}1 & 2 & 2\cr \end{Young}};
\endxy
\]
Hence, the H-move of color $1$ has a removable node of the same residue above the new node iff the H-move of color $-1^{\prime}$ has an addable node of the same residue above the new node. To see that such a removable (or addable) node exists and is unique note that this follows from the fact that there will be a last move before the H-moves that gives rise to the flow at the left side of the H-moves. The node that corresponds to this move will be, by construction, the removable (or addable) node.

The shift of type b and color $0$, that is
\[
\xy
(0,0)*{\includegraphics[scale=.5]{res/figs/webflow/shiftrightb2.eps}};
(0,3)*{0};
\endxy\;\;\text{ and }\;\;
\mathbf{k}\colon
(\lambda^{k-1}_{1}+\xy(0,0)*{\begin{Young}k\cr\end{Young}}\endxy,\lambda^{k-1}_{2},\lambda^{k-1}_{3}+\xy(0,0)*{\begin{Young}k\cr\end{Young}}\endxy),
\]
can be done in a similar vein. Again we have to take the state string in account. We have
\[
\xy
(0,0)*{\includegraphics[scale=.5]{res/figs/webflow/shiftrightb2.eps}};
(0,10)*{\begin{Young}2 &  & 2\cr \end{Young}};
(0,-10)*{\begin{Young}1 &  & 1\cr \end{Young}};
\endxy
\]
Thus, the local change in degree is zero, since, by construction, the left new node will have exactly one addable node of the same residue and the second none. Since this shift corresponds to a divided power we have to subtract $1$ from the degree. Hence, the local change is zero.

We leave the verification that the map $g$ is also of degree $0$ to the reader, since it follows by similar arguments as given above.
\end{proof}
\begin{ex}\label{ex-degreec}
For the filled $3$-multitableau $\vec{T}_4$ from Example~\ref{ex-degreea} we see that the corresponding web with flow $u_f$ is
\[
u_f=
\xy
(0,0)*{\includegraphics[scale=0.405]{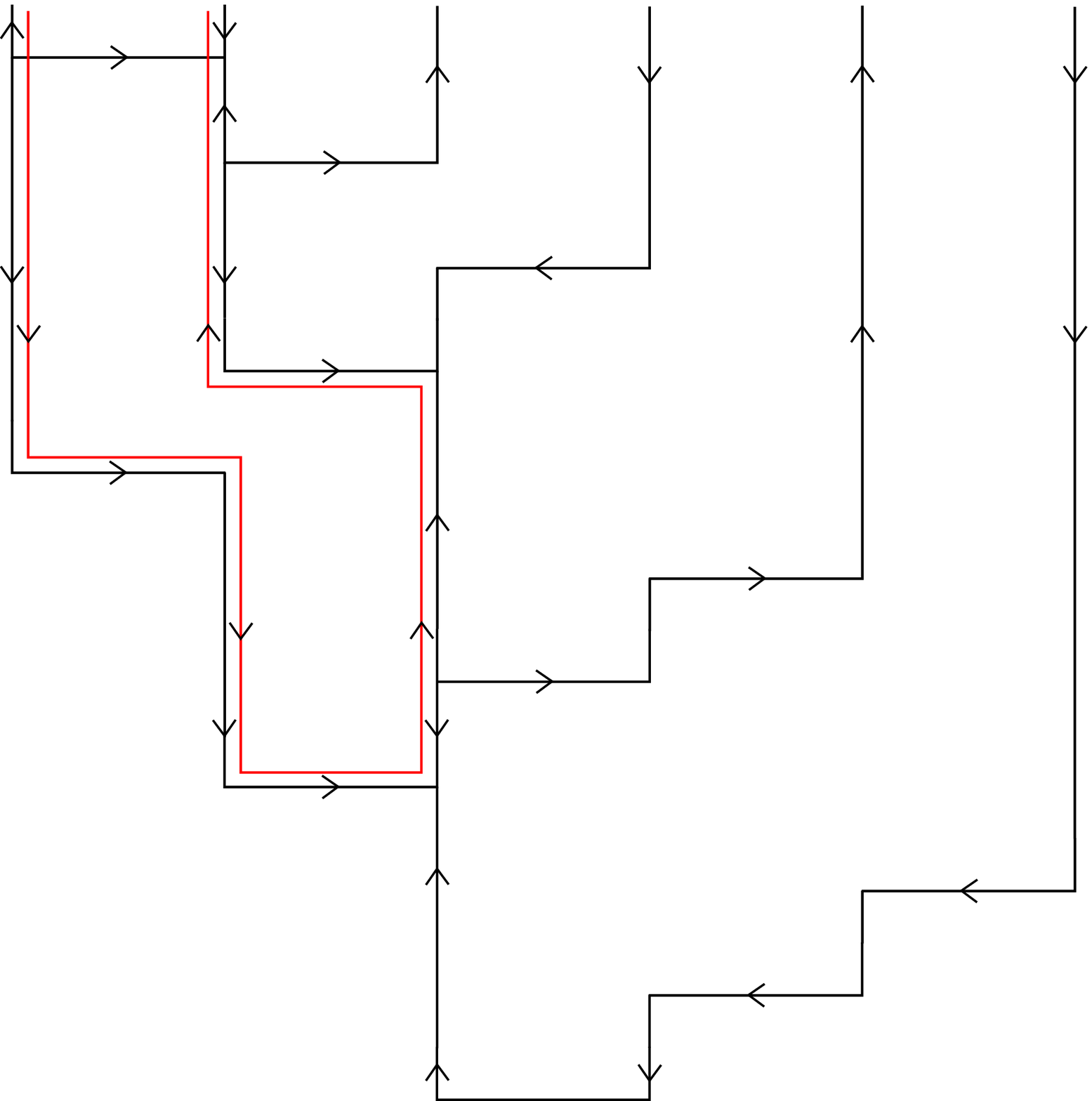}};
(-31,37.7)*{-1};
(-30,9)*{0};
(-15.1,30.7)*{1};
(-15.1,15.5)*{1};
(-15.1,-13.25)*{0};
(0,22.8)*{0};
(0,-6.7)*{1};
(0,-36.25)*{1};
(15.2,0.75)*{0};
(15.2,-28.8)*{0};
(30.5,-21.6)*{0};
\endxy
\]
where we have indicated the local degree changes in the figure above (the weight of the flow has a different sign). This corresponds exactly to the calculation from Example~\ref{ex-degreea}.
\end{ex}
Since Khovanov and Kuperberg~\cite{kk} showed that the weight is invariant under isotopies, we get the following two corollaries.
\begin{cor}\label{cor-degreea}
Let $u_f\in B^J_S$ be a web with flow. Then $\mathrm{deg}_{\mathrm{BKW}}(\iota(u_f))$ is invariant under isotopies of the web $u_f$.\qed
\end{cor}
\begin{cor}\label{cor-degreeb}
Let $\vec{T}_1,\vec{T}_2\in\mathrm{Std}(\vec{\lambda})$ such that $\mathrm{g}(\vec{T}_1)$ and $\mathrm{g}(\vec{T}_2)$ differ only by isotopies. Then we have $\mathrm{deg}_{\mathrm{BKW}}(\vec{T}_1)=\mathrm{deg}_{\mathrm{BKW}}(\vec{T}_2)$.\qed
\end{cor}
Moreover, we have the following natural interpretation of Brundan, Kleshchev and Wang's degree on the level of webs. Note that, by our more general construction, the proposition works not just for non-elliptic webs $u\in B_S$, but for all webs.
\begin{prop}\label{prop-degreenat}
Let $u\in B_S$ be a non-elliptic web. Moreover, let $F_J$ denote the possible empty set of all flows on $u$ with boundary $J$. Denote by $e^S_J=e^{s_1}_{j_1}\otimes\cdots\otimes e^{s_n}_{j_n}$ the corresponding element in the tensor basis. Then
\[
u=\sum_{J}c(S,J)e^S_J,
\]
where the coefficients $c(S,J)$ (recall $v=-q^{-1}$) are given by
\[
c(S,J)=\sum_{f\in F_J}v^{\mathrm{wt}(u_f)}=\sum_{f\in F_J}\pm q^{-\mathrm{deg}_{\mathrm{wt}}(u_f)}=\sum_{f\in F_J}\pm q^{-\mathrm{deg}_{\mathrm{BKW}}(\iota(u_f))}.
\]
\end{prop}
\begin{proof}
This is a direct consequence of Khovanov and Kuperberg's results from Section 4 of~\cite{kk} together with the two Propositions~\ref{prop-sl3bases} and~\ref{prop-degree}.
\end{proof}
\section{The categorified picture}\label{sec-cat}
\subsection{A growth algorithm for foams}\label{sec-catpart1}
In this section we are going to categorify the results from before. That is, we explain how the extended growth algorithm can be used as a growth algorithm for foams.
\vspace*{0.25cm}

It should be noted that we freely use the results from the uncategorified framework that we discussed in the sections before, e.g. we identify the flows on webs with their $3$-multitableaux.
\vspace*{0.25cm}

The main idea can be explained as follows. To define a foam $\mathcal F\colon u\to v$ with $u,v\in B_S$ we use the two different types of data, i.e. the pair $(S,J)$ of the sign and state string on the line and the two flows on $u,v$. The first completely determines the dot placement, while the second gives rise to the topology. It should be noted that the dot placement, i.e. the combinatorics, is rather mysterious in the foam framework, but rather ``straightforward'' in the cyclotomic Hecke algebra framework. The topology on the other hand is in the foam framework only given by zipping certain edges away. Recall that we have $\mathcal F(u,v)\cong \mathcal F(u^*v)$. We tend to use the former in the following.
\vspace*{0.25cm}

We start and give the definition of the basic idempotent, denoted by $e(\vec{\lambda})$. It is worth noting that $e(\vec{\lambda})$ will in general be a foam between elliptic webs, since we do not use divided powers.
\begin{defn}\label{defn-idem}(\textbf{Idempotent associated to $\vec{\lambda}$})
Given a $3$-multipartition $\vec{\lambda}$ with $k$ nodes, we can associate to it a certain \textit{idempotent foam}, denoted by $e(\vec{\lambda})$, using the following rules. Define a sequence of LT-generators for $\vec{\lambda}$ by (with $r(\vec{\lambda})$ as in Definition~\ref{defn-rsequence})
\begin{equation}\label{eq-idem}
\mathrm{LT}(\vec{\lambda})=\prod_k F_{r(\vec{\lambda})_k}=F_{r(\vec{\lambda})_{c(S)}}\cdot\ldots\cdot F_{r(\vec{\lambda})_1}.
\end{equation}
Define a web $w(\vec{\lambda})$ to be the web generated by applying $\mathrm{LT}(\vec{\lambda})$ to a highest weight vector $v_{3^{\ell}}$ and use $q$-skew Howe duality. Then
\[
e(\vec{\lambda})=\mathrm{Id}\in\mathcal F(w(\vec{\lambda}),w(\vec{\lambda})),
\]
that is the identity foam $\mathrm{Id}\colon w(\vec{\lambda})\to w(\vec{\lambda})$.
\end{defn}
Recall that there is an $\bC$-linear involution ${}^*\colon\mathcal F\to \mathcal F$ of the foam space
\[
\mathcal F_S=\bigoplus_{\partial u,\partial v = S}\mathcal F(u,v)
\]
that turns the foams around (and changes some internal boundary orientations). For example
\[
\left(\xy(0,1.25)*{\includegraphics[scale=0.75]{res/figs/basicsb/digonem.eps}}\endxy\right)^*=\xy(0,-1.75)*{\reflectbox{\includegraphics[scale=0.75,angle=180]{res/figs/basicsb/digonem.eps}}}\endxy
\]
\begin{lem}\label{lem-welldefidem}
The idempotent $e(\vec{\lambda})$ is well-defined, that is it is not zero and an idempotent. Moreover, for all $3$-multipartitions $\vec{\lambda},\vec{\mu}$ we have
\[
e(\vec{\lambda})e(\vec{\mu})=e(\vec{\mu})e(\vec{\lambda})=\delta_{\vec{\lambda},\vec{\mu}}e(\vec{\lambda})=\delta_{\vec{\lambda},\vec{\mu}}e(\vec{\mu}),\delta_{\vec{\lambda},\vec{\mu}}=\begin{cases}1, &\text{if }r(\vec{\lambda})=r(\vec{\mu}),\\0, &\text{if }r(\vec{\lambda})\neq r(\vec{\mu}).\end{cases}
\]
Moreover, we have
\[
e(\vec{\lambda})^*=e(\vec{\lambda}),
\]
that is $e(\vec{\lambda})$ is fixed by the involution ${}^*$.
\end{lem}
\begin{proof}
To see that it is well-defined, i.e. that $\mathrm{LT}(\vec{\lambda})$ does not kill the highest weight vector, we have to use Theorem~\ref{thm-fine}. That is the Theorem~\ref{thm-fine} ensures that there is a web with a flow that has exactly the corresponding multitableaux $T_{\vec{\lambda}}$. Thus, the idempotent can not be zero.

That the element is an idempotent follows from the fact that it is just the identity foam on a certain web $w(\vec{\lambda})$.

The other statements follow directly from the definition of the multiplication, that is different  $3$-multipartition $\vec{\lambda},\vec{\mu}$ give rise to different webs $w(\vec{\lambda})$ and $w(\vec{\mu})$ iff $r(\vec{\lambda})\neq r(\vec{\mu})$, since the convention to obtain the webs from the multitableaux only depends on the residue sequence, see~\ref{eq-idem}. Moreover, the composition is zero if the webs are not the same. by convention.

That the idempotents are fixed by the involution is clear by the definition of ${}^*$, i.e. it just turns the foams around (and inverts some arrows).
\end{proof}
\begin{ex}\label{ex-idemfoam}
If the $3$-multipartition is the one from Example~\ref{ex-tabl2} part (b), that is
\[
\vec{\lambda}=\left(\xy(0,0)*{\begin{Young}&\cr\cr\end{Young}}\endxy\;,\;\xy(0,0)*{\begin{Young}\cr\end{Young}}\endxy\;,\;\xy(0,0)*{\begin{Young}&\cr\cr\end{Young}}\endxy \right),
\]
then we get $\mathrm{LT}(\vec{\lambda})=F_1F_3F_2F_2F_1F_3F_2$ and therefore $w(\vec{\lambda})$ will be
\[
w(\vec{\lambda})=\xy(0,0)*{\includegraphics[scale=0.5]{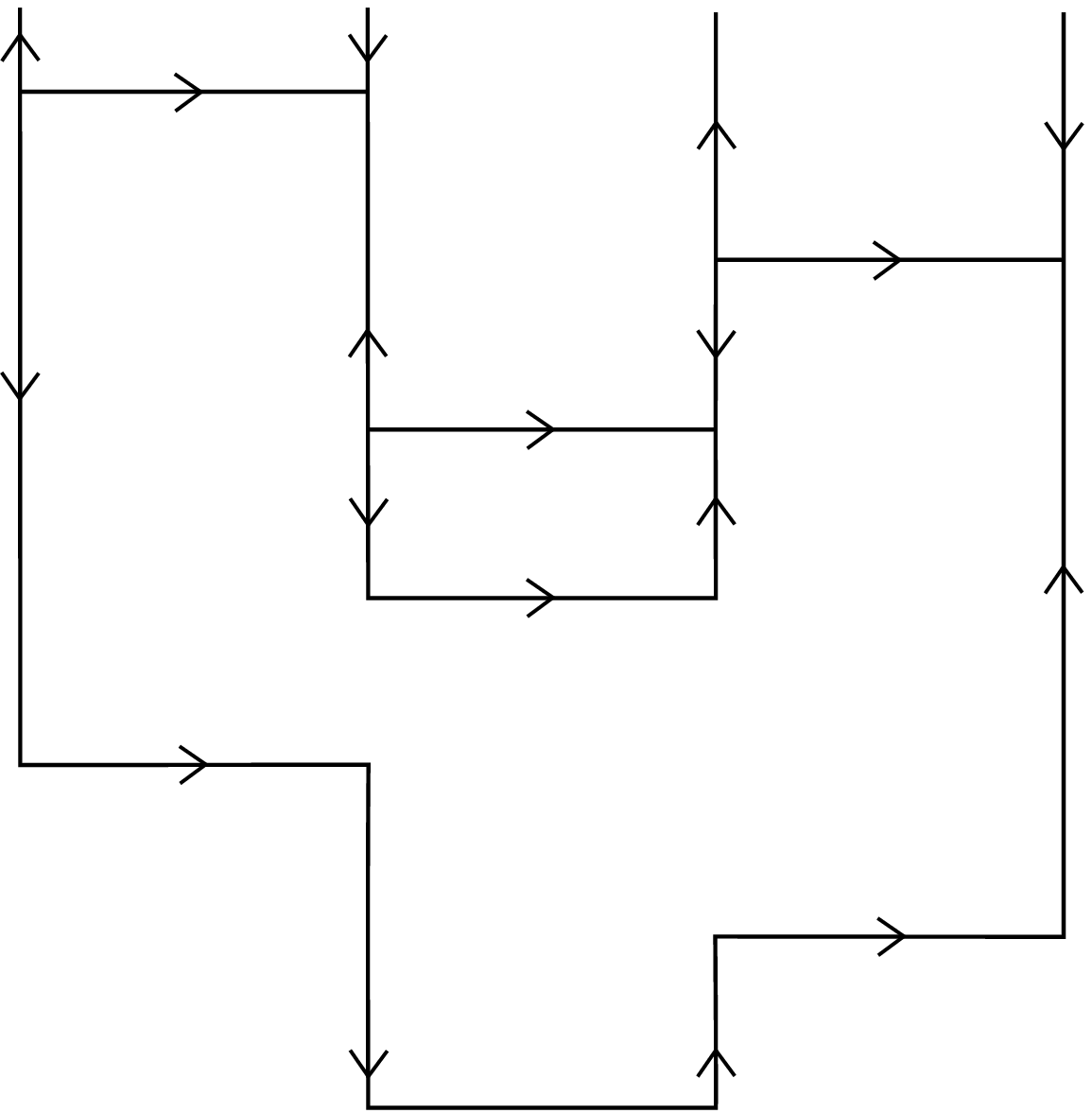}}\endxy
\]
Hence, the idempotent (we have skipped the orientations below) for $\vec{\lambda}$ is
\[
e(\vec{\lambda})=\xy(0,0)*{\includegraphics[scale=0.5]{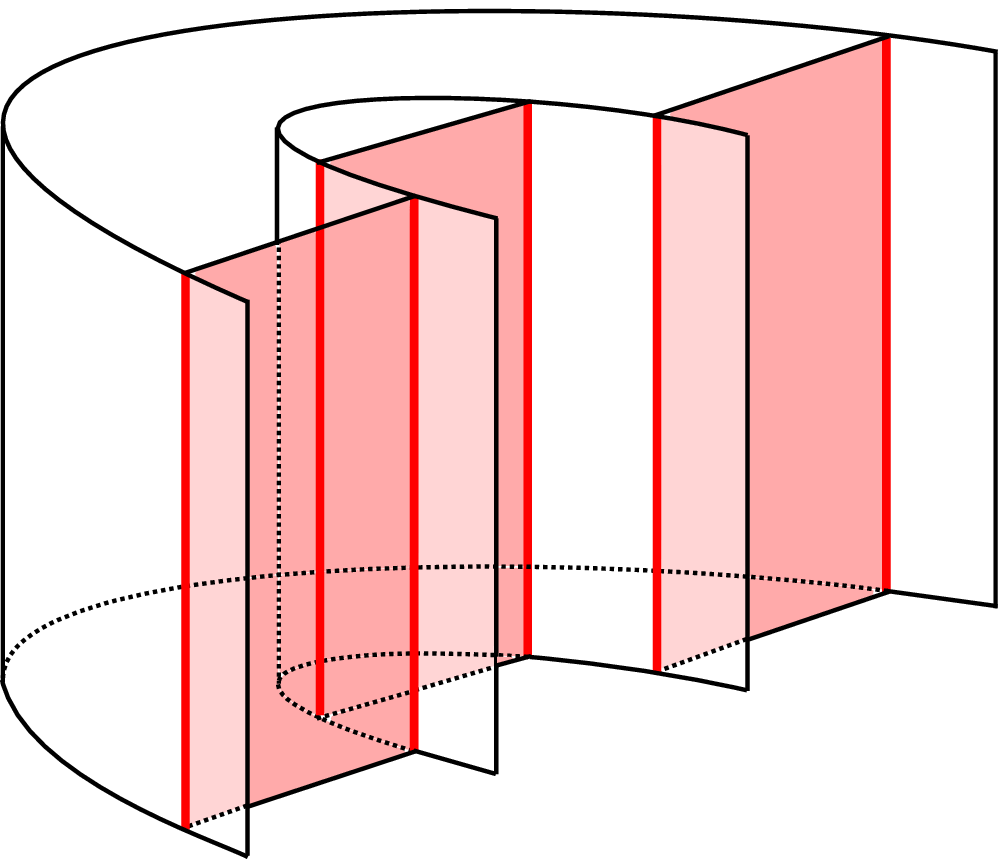}}\endxy
\]
An important property of this idempotent is the square (if we ignore the digon) in the back, since it allows use to change to both webs that can be associated to $\vec{\lambda}$ for ``free'', i.e. without changing the degree. 
\end{ex}
\begin{defn}\label{defn-idemdots}(\textbf{Dot placement associated to $\vec{\lambda}$})
Given a $3$-multipartition $\vec{\lambda}$ as in Definition~\ref{defn-idem} together with its associated idempotent $e(\vec{\lambda})$ and its LT-generators
\[
\mathrm{LT}(\vec{\lambda})=\prod_k F_{r(\vec{\lambda})_k}=F_{r(\vec{\lambda})_{c(S)}}\cdot\ldots\cdot F_{r(\vec{\lambda})_1}.
\]
Recall that each of the $F_{r(\vec{\lambda})_k}$ corresponds to a ladder-move (with some possible deleted edges) on the level of webs. Therefore, each $F_{r(\vec{\lambda})_k}$ corresponds to a
\[
\phantom{(m_k)}F_{r(\vec{\lambda})_k}\rightsquigarrow
\xy
(0,3)*{\includegraphics[scale=0.7]{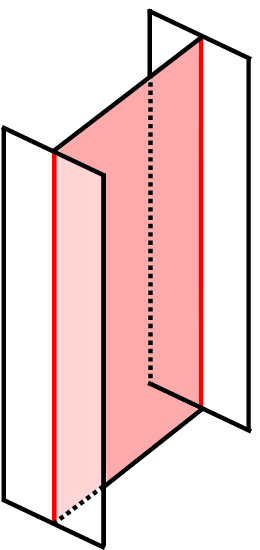}}
\endxy
\]
for the foam $e(\vec{\lambda})$. We place \textit{$m_k$ dots on $F_{r(\vec{\lambda})_k}$ on the middle facet}
\[
F_{r(\vec{\lambda})_k}(m_k)\rightsquigarrow
\xy
(0,3.5)*{\includegraphics[scale=0.7]{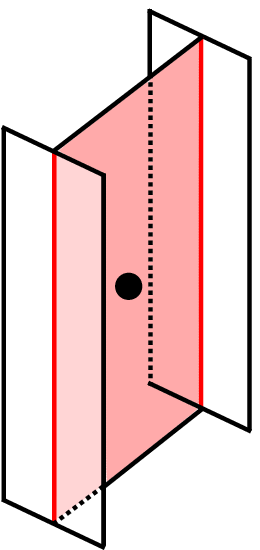}};
(2,5.5)*{m_k}
\endxy
\]
where $m_k=\mathsf{A}^{^{k\succ N}}(T_{\vec{\lambda}^k})$, i.e. is the number of \textit{addable nodes} after the node $N$ with entry $k$ in $T_{\vec{\lambda}^k}$ with the same residue as the node $N$. We denote the dotted idempotent associated to $\vec{\lambda}$ by $e(\vec{\lambda})d(\vec{\lambda})$, where
\[
d(\vec{\lambda})=\prod_k d_k^{m_k}
\]
should indicate the number of dots on the different facets.
\end{defn}
It should be noted that it is not clear that $e(\vec{\lambda})d(\vec{\lambda})$ is well-defined. We show this fact in a lemma below, but first we give an example.
\begin{ex}\label{ex-dots}
We consider the $3$-multipartition from Example~\ref{ex-tabl2} part (b) again, i.e.
\[
\vec{\lambda}=\left(\xy(0,0)*{\begin{Young}&\cr\cr\end{Young}}\endxy\;,\;\xy(0,0)*{\begin{Young}\cr\end{Young}}\endxy\;,\;\xy(0,0)*{\begin{Young}&\cr\cr\end{Young}}\endxy \right)\;.
\]
As before, we have $\mathrm{LT}(\vec{\lambda})=F_1F_3F_2F_2F_1F_3F_2$ (that is $k=1,\dots,7=c((+,-,+,-))$) and 
\[
T_{\vec{\lambda}}=\left(\xy(0,0)*{\begin{Young}1&2\cr3\cr\end{Young}}\endxy\;,\;\xy(0,0)*{\begin{Young}4\cr\end{Young}}\endxy\;,\;\xy(0,0)*{\begin{Young}5&6\cr7\cr\end{Young}}\endxy \right)\;.
\]
Thus, we see that $m_k=0$ unless $k=1$ or $k=4$, where $m_1=2$ and $m_4=1$. Therefore, we have
\[
e(\vec{\lambda})d(\vec{\lambda})=
\xy(0,0)*{\includegraphics[scale=0.5]{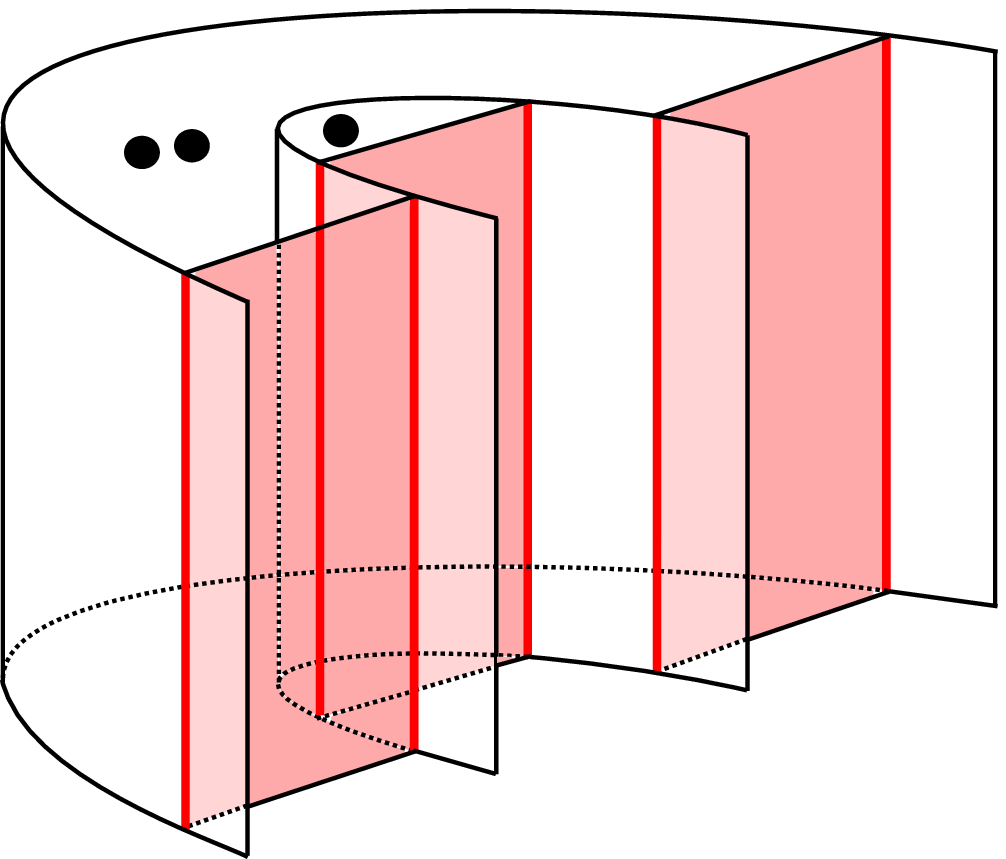}}\endxy
\]
\end{ex}
\begin{lem}\label{lem-welldefdots}
The dot placement of Definition~\ref{defn-idemdots} is well-defined. Moreover, we have
\[
e(\vec{\lambda})d(\vec{\lambda})=d(\vec{\lambda})e(\vec{\lambda})=(e(\vec{\lambda})d(\vec{\lambda}))^*.
\]
\end{lem}
\begin{proof}
To see that the dot placement is well-defined it suffices to show that the middle facet always exists, i.e. it is not removed because its corresponding number is $0$ or $3$. That this is indeed the case follows from the fact that all possible ladder-moves in our convention from Definition~\ref{defn-webtotab}.

The second statement follows directly from the definition of the involution ${}^*$, the Lemma~\ref{lem-welldefidem} and the fact that the corresponding foams are just identity foams.
\end{proof}
\vspace*{0.15cm}

It should be noted that $S_{k-1}$ acts on the set of strings of $F$'s of length $k$ with a fixed number of occurrences of the $F$'s by defining the action of the $j$-th transposition $\tau_j$ by exchanging the neighboring entries $j-1$ and $j$ reading from right to left. We denote a transposition $\tau_j$ that exchanges a $F_{a}$ and a $F_{b}$ by $\tau_j(a,b)$, i.e.
\[
\tau_j(a,b)\cdot(F_k \cdots \underbrace{F_{b}F_{a}}_{\text{pos. }j} \cdots F_1)=F_k\cdots \underbrace{F_{a}F_{b}}_{\text{pos. }j} \cdots F_1.
\]
\begin{defn}\label{defn-foamLT}(\textbf{Foam between LT-generators}) Given two strings of LT-generators
\[
\mathrm{LT}_1=\prod_k F_{i_k}\;\;\text{ and }\;\;\mathrm{LT}_2=\prod_k F_{i^{\prime}_k}
\]
such that $\mathrm{LT}_1$ and $\mathrm{LT}_2$ differ only by a permutation $\sigma\in S_{k-1}$ of their $F$'s. Let $w_1=\mathrm{LT}_1v_{3^{\ell}}$ and $w_2=\mathrm{LT}_2v_{3^{\ell}}$ be the corresponding webs obtained by $q$-skew Howe duality. We assume that $\sigma\in S_{k-1}$ is already decomposed into a string of transpositions
\[
\sigma=\tau_{i_1}\dots\tau_{i_l},
\]
such that $\sigma \mathrm{LT}_1=\mathrm{LT}_2$. Then we associate to $\mathrm{LT}_1$, $\mathrm{LT}_2$ and $\sigma$ a \textit{foam}
\[
\mathcal F_{\sigma}(\mathrm{LT}_1,\mathrm{LT}_2)\colon w_1\to w_2,\;\; \mathcal F_{\sigma}(\mathrm{LT}_1,\mathrm{LT}_2)=\mathcal F(\tau_{i_1}(a_{i_1},b_{i_1}))\circ\cdots \circ\mathcal F(\tau_{i_l}(a_{i_l},b_{i_l}))\circ\mathrm{Id}_{w_1}
\]
by composing the identity $\mathrm{Id}_{w_1}$ on $w_1$ from the bottom with the foams
\[
\xy
(-40,22.5)*{\mathcal F(\tau_{i}(a_{i},b_{i}))=};
(-40,0)*{\includegraphics[scale=0.5]{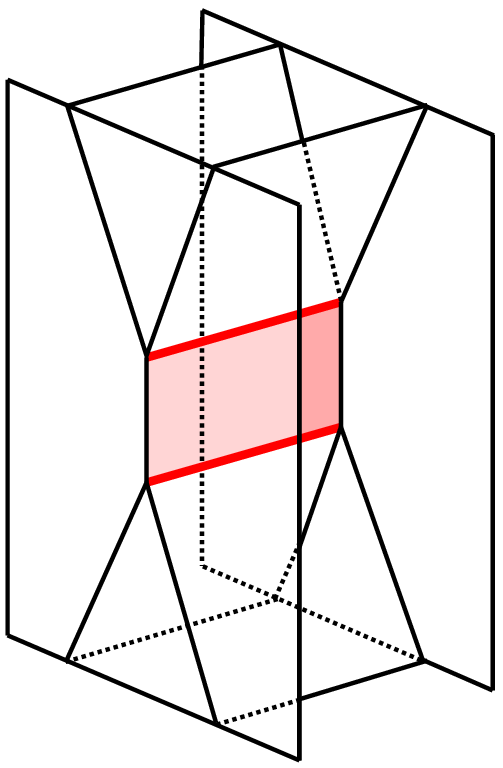}};
(-47.5,18)*{a_i};
(-40,-22.5)*{a_i=b_i};
(-55,0)*{-};
(0,22.5)*{\mathcal F(\tau_{i}(a_{i},b_{i}))=};
(0,0)*{\includegraphics[scale=0.5]{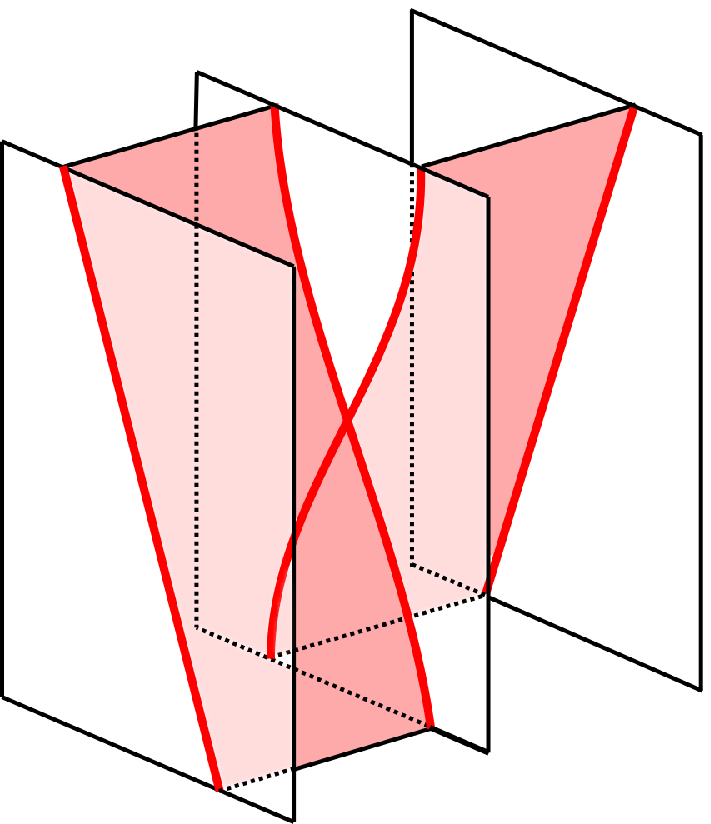}};
(-13.5,16)*{b_i};
(-1,18)*{a_i};
(0,-22.5)*{a_i=b_i+1};
(-20,0)*{\varepsilon};
(40,22.5)*{\mathcal F(\tau_{i}(a_{i},b_{i}))=};
(40,0)*{\includegraphics[scale=0.5]{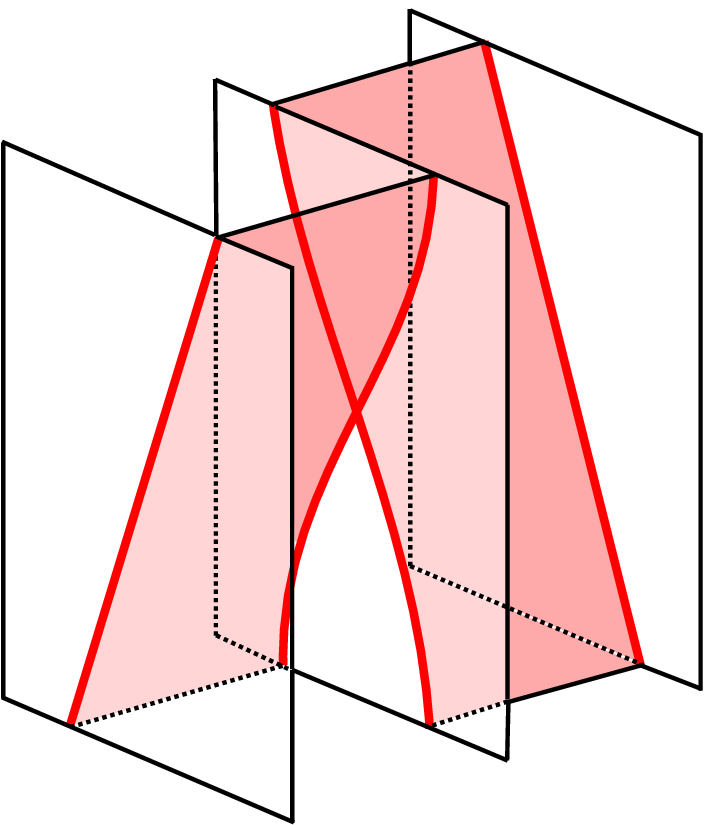}};
(26.5,16)*{a_i};
(38.5,19.25)*{b_i};
(40,-22.5)*{a_i+1=b_i};
\endxy
\]
and
\[
\xy
(-30,25)*{\mathcal F(\tau_{i}(a_{i},b_{i}))=};
(-30,0)*{\includegraphics[scale=0.5]{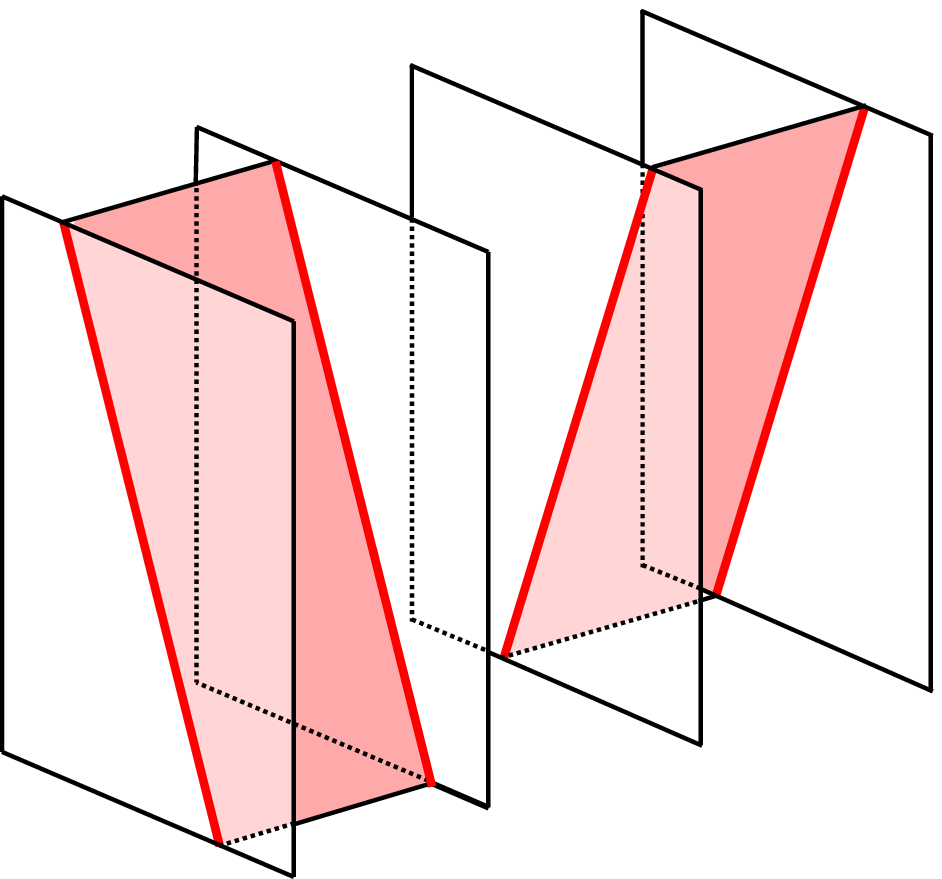}};
(-47.5,15)*{b_i};
(-25,18.5)*{a_i};
(-30,-25)*{a_i-b_i>1};
(30,25)*{\mathcal F(\tau_{i}(a_{i},b_{i}))=};
(30,0)*{\includegraphics[scale=0.5]{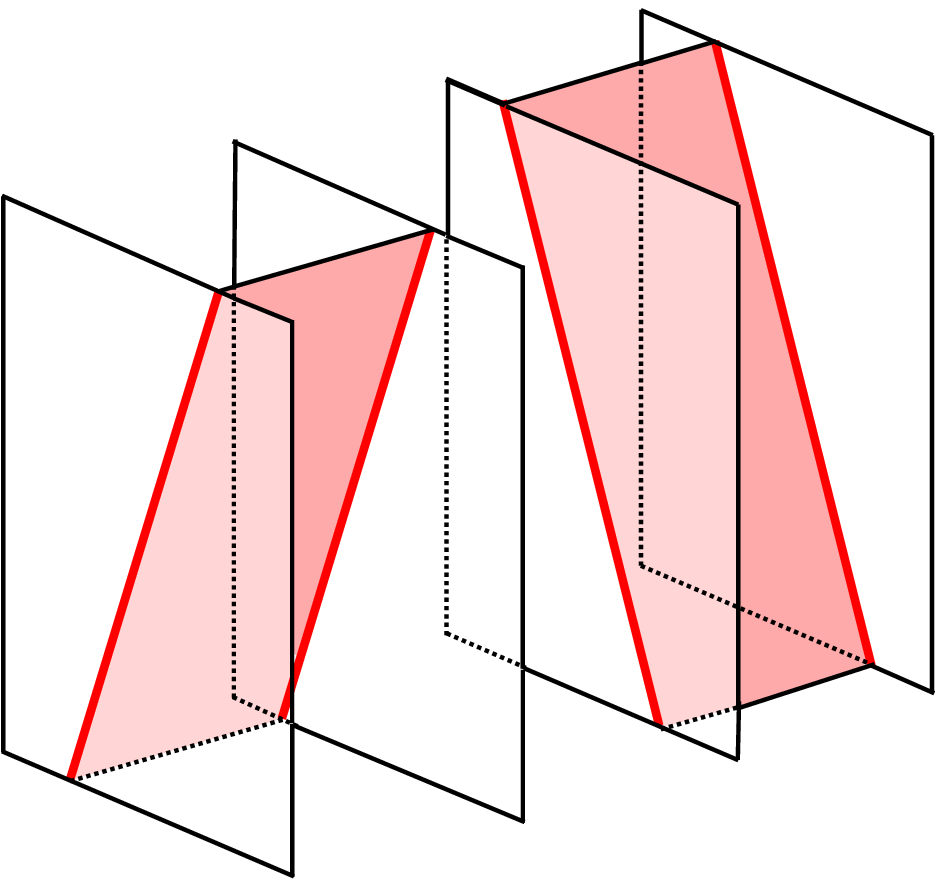}};
(12,13)*{a_i};
(35,20.75)*{b_i};
(30,-25)*{a_i-b_i<-1};
\endxy
\]
where the not pictured parts should be the identity and we use the usual rules to read these foams and the usual signs $\varepsilon\in\{+,-\}$, i.e. the ones given by categorified $q$-skew Howe duality. Note that, by our conventions, this procedure really gives a foam from $w_1$ to $w_2$.
\vspace*{0.25cm}

The action defined above can be stated in a different way just on the level of tableaux without multiple entries. That is, the group $S_{k-1}$ acts on the set of all $3$-multitableaux $\mathrm{Std}(\vec{\lambda})\cup\{\ast\}$ by defining the action of the transposition $\tau_j$ by
\[
\tau_j(\vec{T}_1)=\begin{cases}\vec{T}_2=\vec{T}_1(j\mapsto j+1,\;j+1\mapsto j), &\text{if the result is a standard }3\text{-multitableau},\\ \ast, &\text{otherwise}.\end{cases}
\]
\vspace*{0.15cm}

Finally, for a fixed $3$-multipartition $\vec{\lambda}$ and a corresponding $3$-multitableau $\vec{T}\in\mathrm{Std}(\vec{\lambda})$ we choose a fixed permutation $\sigma\in S_{k-1}$ such that 
\[
\sigma^{-1}(\vec{T})=T_{\vec{\lambda}}
\]
by searching for the lowest $j\in\{1,\dots,k\}$ such that the node $N$ with entry $k$ in $\vec{T}$ is not the same as the node $N^{\prime}$ with entry $k$ in $T_{\vec{\lambda}}$. Apply a minimal sequence of transpositions until they match and repeat the process until $\sigma^{-1}(\vec{T})=T_{\vec{\lambda}}$. By construction, the permutation $\sigma\in S_{k-1}$ will be of minimal length with respect to the property $\sigma^{-1}(\vec{T})=T_{\vec{\lambda}}$.

It is worth noting that we put a $\sigma^{-1}$ here because we usually want a foam with a web associated to $T_{\vec{\lambda}}$ at the top. One easily verifies that
\[
\sigma(\vec{T}_1)=\vec{T}_2
\Leftrightarrow\sigma^{-1}(\vec{T}_2)=\vec{T}_1\;\;\text{for all }\vec{T}_1,\vec{T}_2\in\mathrm{Std}(\vec{\lambda}),\;\sigma\in S_{k-1}.
\]
We denote the foam associated to this permutation $\sigma$ by $\mathcal F_{\sigma}$. 
\end{defn}
\begin{lem}\label{lem-actweldef}
The definition of the foam $\mathcal F_{\sigma}(\mathrm{LT}_1,\mathrm{LT}_2)$ is well-defined, i.e. one does not run into ambiguities and the resulting foam is a foam between $w_1$ and $w_2$. Moreover, the definition via the LT-generators and via actions on $3$-multitableaux agree.
\end{lem}
\begin{proof}
To see that the definition is well-defined, note that, using Lemma~\ref{lem-ltalgo} and the observation that, without divided powers, the total length and the length agree, each step of the action of $\mathcal F_{\sigma}(\mathrm{LT}_1,\mathrm{LT}_2)$ on the top web changes the boundary of the web to the bottom boundary as indicated in the pictures above. Thus, each step is a legal move between a web at the top and a web at the bottom. That is, starting from $w_1$ at the top, one builds inductively a foam going stepwise from $w_1$ to $w_2$, i.e.
\[
w_1\mapsto w_{i_l}\mapsto w_{i_{l-1}}\mapsto\cdots\mapsto w_{i_2}\mapsto w_{i_1}\mapsto w_2,
\]
where each step is well-defined and one of the foams pictured above. Hence, we get a well-defined foam
\[
\mathcal F_{\sigma}(\mathrm{LT}_1,\mathrm{LT}_2)\colon w_1\to w_2.
\]
\vspace*{0.15cm}

To see that the two actions agree we note again that the residue $r(N)$ of the node $N$ with entry $k$ is exactly the $F_{r(N)}$. Thus, reading the two residue sequences for $\vec{T}$ and $T_{\vec{\lambda}}$, we see that they give rise to a web on the bottom and the top and the action on $3$-multitableaux is exactly defined such that the result of the action of $\tau_j$ depends in the same way on the residue of the two nodes $N_j$ and $N_{j+1}$ as the action via foams depends on the difference $a_i-b_i$. Then one repeats the inductive argument from before, that is, a direct comparison of the LT-generators in each step shows that each step is a legal permutation and the end result are exactly the LT-generators associated to $w_2$.

It should be noted that this also shows that our rules to define $\sigma$ ensure that the action of the transpositions of $\sigma$ on the tableaux never produce a non-standard tableau.
\end{proof}
\begin{defn}\label{defn-foamLT2}(\textbf{Foam for $u_f\in B^J_S$})
Given a non-elliptic web with a flow $u_f\in B^J_S$, we associate to it a foam
\[
\mathcal F_{u_f}\colon u_f\to w(\vec{\lambda}),
\]
where $\vec{\lambda}$ is the corresponding boundary datum and $w(\vec{\lambda})$ is as in Definition~\ref{defn-idem}, in the following way. Change the $3$-multitableau $\iota(u_f)$ associated to $u_f$ by replacing the lowest multiple entry $j$ of multiplicity $j^{\prime}$ of $\iota(u_f)$ decreasing from left to right with consecutive numbers $j,\dots,j+j^{\prime}$ and shift all other entries by $j^{\prime}$. Repeat until no multiple entries occur and obtain $\iota(u_f)^{\prime}$, see~\ref{eq-divtab}. Set
\[
\mathcal F_{u_f}=\mathcal F_{\mathrm{R}}\circ \mathcal F_{\sigma}(\iota(u_f)^{\prime},T_{\vec{\lambda}})\colon u_f\to w(\vec{\lambda}),
\]
with $\mathcal F_{\sigma}(\iota(u_f)^{\prime},T_{\vec{\lambda}})$ for the LT-generators $\mathrm{LT}_{1,2}$ corresponding to $\iota(u_f)^{\prime}$ and $T_{\vec{\lambda}}$ respectively. The foam $\mathcal F_{\mathrm{R}}$ is given by removing the generated internal digons (pictured below~\ref{eq-invo}) with a dot-free digon removal and the internal theta-foams with a dot-free identity. It should be noted that these are the only possible differences for the webs associated to $g(\iota(u_f))$ and $g(\iota(u_f)^{\prime})$.
\end{defn}
\begin{ex}\label{ex-action}
If we consider Example~\ref{ex-tabl2} part (b) again, that is
\[
\vec{\lambda}=\left(\xy(0,0)*{\begin{Young}&\cr\cr\end{Young}}\endxy\;,\;\xy(0,0)*{\begin{Young}\cr\end{Young}}\endxy\;,\;\xy(0,0)*{\begin{Young}&\cr\cr\end{Young}}\endxy \right),
\]
and consider the right example $u_f$, then we have
\[
T_{\vec{\lambda}}=\left(\;\xy(0,0)*{\begin{Young}1&2\cr3\cr\end{Young}}\endxy\;,\;\xy(0,0)*{\begin{Young}4\cr\end{Young}}\endxy\;,\;\xy(0,0)*{\begin{Young}5&6\cr7\cr\end{Young}}\endxy\;\right)\,\iota(u_f)=\left(\;\xy(0,0)*{\begin{Young}1 & 2\cr 4\cr\end{Young}}\endxy\;,\;\xy(0,0)*{\begin{Young}3\cr\end{Young}}\endxy\;,\;\xy(0,0)*{\begin{Young}1 & 2\cr 4\cr\end{Young}}\endxy\;\right)\,\iota(u_f)^{\prime}=\left(\;\xy(0,0)*{\begin{Young}1 & 3\cr 6\cr\end{Young}}\endxy\;,\;\xy(0,0)*{\begin{Young}5\cr\end{Young}}\endxy\;,\;\xy(0,0)*{\begin{Young}2 & 4\cr 7\cr\end{Young}}\endxy\;\right).
\]
Therefore, $\sigma\in S_6$ is given by $\sigma=\tau_4\tau_5\tau_3\tau_4\tau_5\tau_2$, since
\begin{align*}
\iota(u_f)^{\prime}=\left(\;\xy(0,0)*{\begin{Young}1&3\cr 6\cr\end{Young}}\endxy\;,\;\xy(0,0)*{\begin{Young}5\cr\end{Young}}\endxy\;,\;\xy(0,0)*{\begin{Young}2&4\cr 7\cr\end{Young}}\endxy\;\right)&\stackrel{\tau_2}{\mapsto}\left(\;\xy(0,0)*{\begin{Young}1&2\cr 6\cr\end{Young}}\endxy\;,\;\xy(0,0)*{\begin{Young}5\cr\end{Young}}\endxy\;,\;\xy(0,0)*{\begin{Young}3&4\cr 7\cr\end{Young}}\endxy\;\right)\stackrel{\tau_5}{\mapsto}\left(\;\xy(0,0)*{\begin{Young}1&2\cr 5\cr\end{Young}}\endxy\;,\;\xy(0,0)*{\begin{Young}6\cr\end{Young}}\endxy\;,\;\xy(0,0)*{\begin{Young}3&4\cr 7\cr\end{Young}}\endxy\;\right)\\
&\stackrel{\tau_4}{\mapsto}\left(\;\xy(0,0)*{\begin{Young}1&2\cr 4\cr\end{Young}}\endxy\;,\;\xy(0,0)*{\begin{Young}6\cr\end{Young}}\endxy\;,\;\xy(0,0)*{\begin{Young}3&5\cr 7\cr\end{Young}}\endxy\;\right)\stackrel{\tau_3}{\mapsto}\left(\;\xy(0,0)*{\begin{Young}1&2\cr 3\cr\end{Young}}\endxy\;,\;\xy(0,0)*{\begin{Young}6\cr\end{Young}}\endxy\;,\;\xy(0,0)*{\begin{Young}4&5\cr 7\cr\end{Young}}\endxy\;\right)\\
&\stackrel{\tau_5}{\mapsto}\left(\;\xy(0,0)*{\begin{Young}1&2\cr 3\cr\end{Young}}\endxy\;,\;\xy(0,0)*{\begin{Young}5\cr\end{Young}}\endxy\;,\;\xy(0,0)*{\begin{Young}4&6\cr 7\cr\end{Young}}\endxy\;\right)\stackrel{\tau_4}{\mapsto}\left(\;\xy(0,0)*{\begin{Young}1&2\cr 3\cr\end{Young}}\endxy\;,\;\xy(0,0)*{\begin{Young}4\cr\end{Young}}\endxy\;,\;\xy(0,0)*{\begin{Young}5&6\cr 7\cr\end{Young}}\endxy\;\right)=T_{\vec{\lambda}}.
\end{align*}
This can be reinterpreted as an action on the two LT-strings, that is (reading backwards now)
\begin{align*}
\mathrm{LT}(T_{\vec{\lambda}})=F_1F_3F_2F_2F_1F_3F_2&\stackrel{\tau_4}{\mapsto}F_1F_3F_2F_2F_1F_3F_2\stackrel{\tau_5}{\mapsto}F_1F_2F_3F_2F_1F_3F_2\\
&\stackrel{\tau_3}{\mapsto}F_1F_2F_3F_1F_2F_3F_2\stackrel{\tau_4}{\mapsto}F_1F_2F_1F_3F_2F_3F_2\\
&\stackrel{\tau_5}{\mapsto}F_1F_1F_2F_3F_2F_3F_2\stackrel{\tau_2}{\mapsto}F_1F_1F_2F_3F_3F_2F_2=\mathrm{LT}(\iota(u_f)^{\prime}).
\end{align*}
When we re-read the strings of $F$'s using $q$-skew Howe duality as webs, then we see how this procedure gives rise to a foam between the web associated to $\mathrm{LT}(T_{\vec{\lambda}})$ (as shown in Example~\ref{ex-idemfoam}) and the web associated to $\mathrm{LT}(\iota(u_f)^{\prime})$ as shown below. Each of the $\tau_j$ corresponds to one of the basic generators pictured in Definition~\ref{defn-foamLT} - depending on the difference of the residues of the exchanged nodes. For example, the foam associated to the last step $\tau_2$ is locally
\[
\tau_2\rightsquigarrow\xy(0,0)*{\includegraphics[scale=0.5]{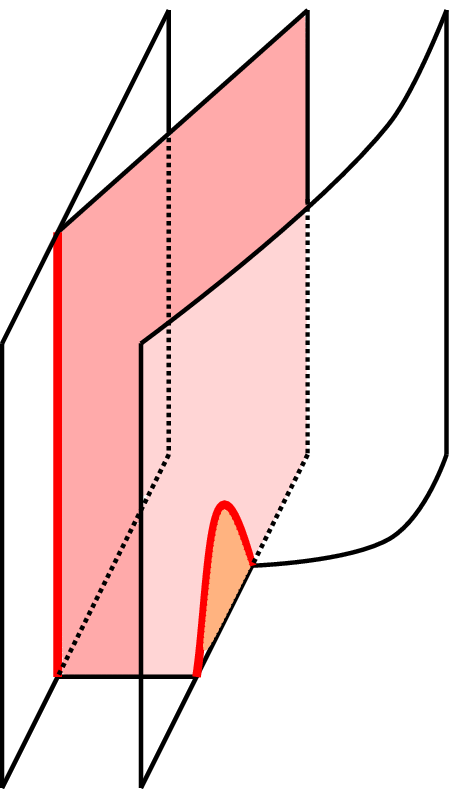}}\endxy
\]
that is a zip, and the identity outside of this local picture. The bottom of this local picture is the part of the web associated to $w_f$ marked with a square on the left and the top is shown on the right
\[
\xy
(0,0)*{\includegraphics[scale=0.5]{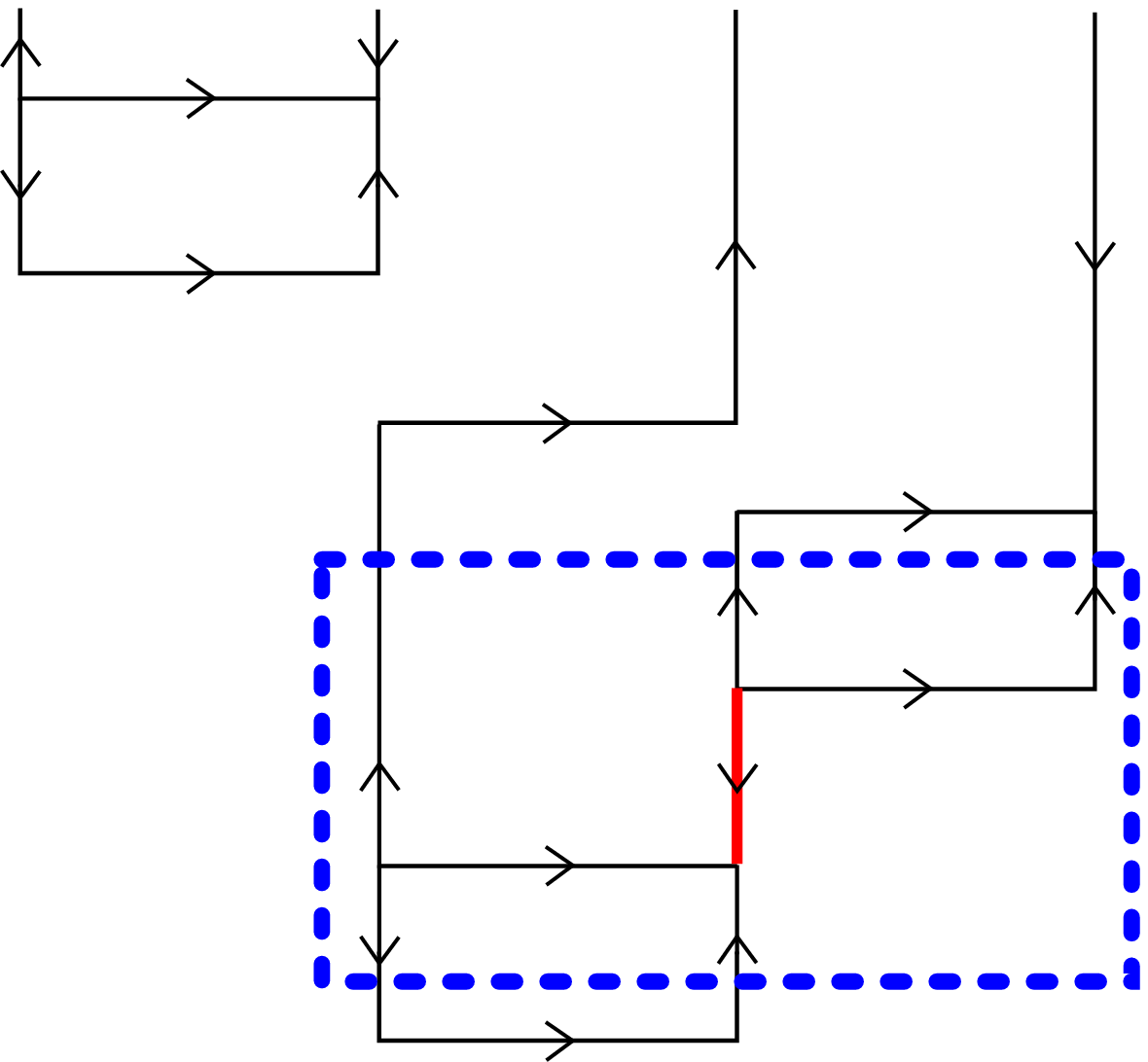}};
\endxy\hspace*{0.25cm}\mapsto\hspace*{0.25cm}
\xy
(0,0)*{\includegraphics[scale=0.5]{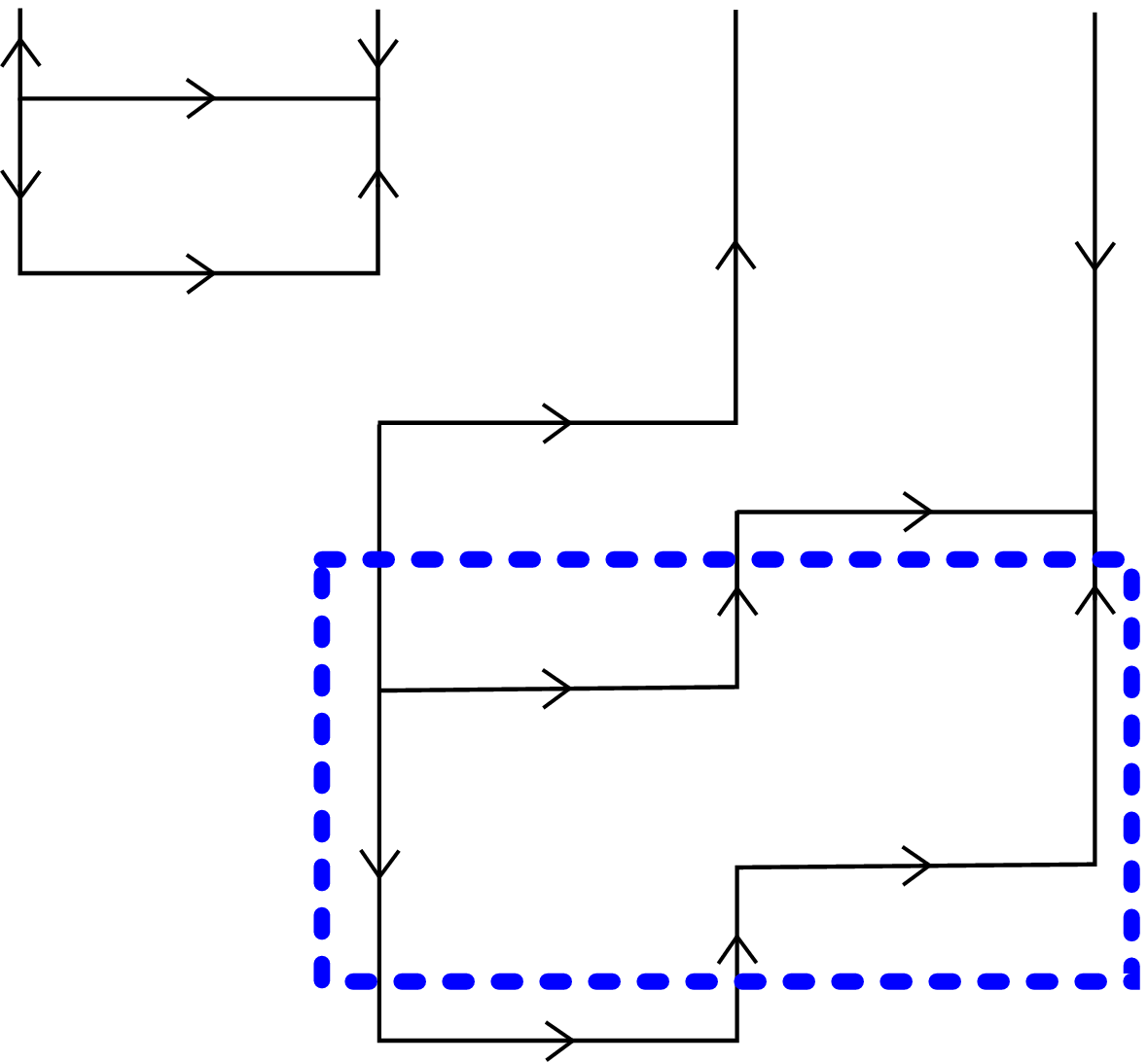}};
\endxy
\]
and unzips the marked edge. The foam $\mathcal F_{\mathrm{R}}$ then just removes the three digons and we obtain the original web from Example~\ref{ex-tabl2} part (b) back.
\end{ex}
\vspace*{0.15cm}

We are now able go give a \textit{growth algorithm for foams}. We will see later that this gives rise to a graded cellular basis in the sense of Hu and Mathas~\cite{hm}. Hence, in the spirit of Proposition~\ref{prop-sl3bases}, we can call this a \textit{categorification} of the intermediate crystal bases.
\begin{defn}\label{defn-growthfoam}(\textbf{Growth algorithm for foams}) Given a pair of a sign string and a state string $(S,J)$, the corresponding $3$-multipartition $\vec{\lambda}$ and two Kuperberg webs $u,v\in B_S$ that extend $J$ to $u_f$ and $v_{f^\prime}$ receptively.

We define a foam following Definition~\ref{defn-foamLT2}
\[
\mathcal F^{\vec{\lambda}}_{\iota(u_f),\iota(v_{f^{\prime}})}\colon u\to v
\]
by
\[
\mathcal F^{\vec{\lambda}}_{\iota(u_f),\iota(v_{f^{\prime}})}=\mathcal F_{u_f}e(\vec{\lambda})d(\vec{\lambda})\mathcal F_{v_{f^{\prime}}}^*.
\]
\end{defn}
\begin{ex}\label{ex-bigfoam}
Let us consider the following example. The web is the right theta-foam from Example~\ref{ex-totlength}. Note that the Kuperberg bracket gives $[2][3]=q^{-3}+2q^{-1}+2q+q^3$. All the six different flows are indicated below.
\[
\xy
 (0,0)*{\includegraphics[width=80px]{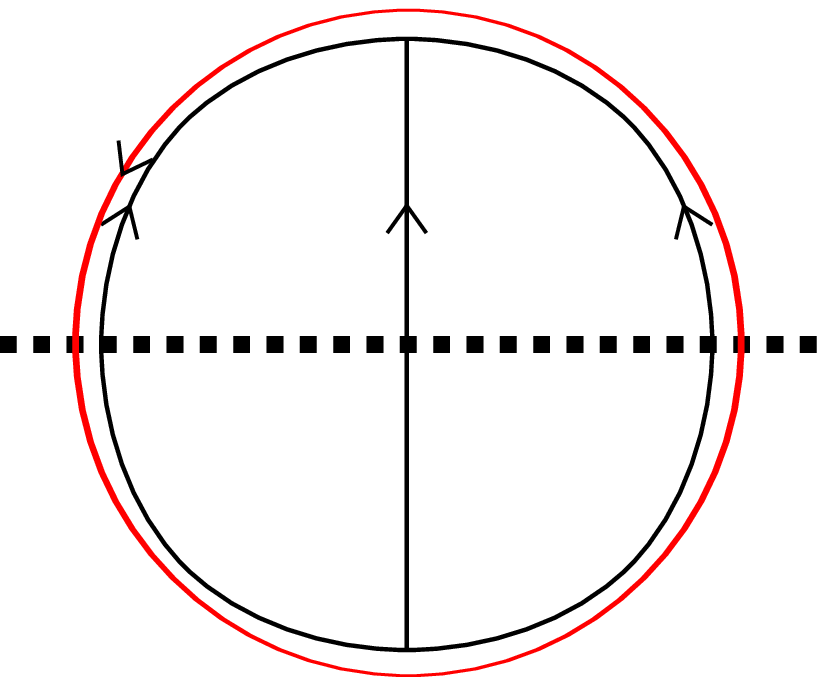}};
 (0,-13.5)*{\scriptstyle \mathrm{wt} = 3};
 (0,-26)*{\includegraphics[width=80px]{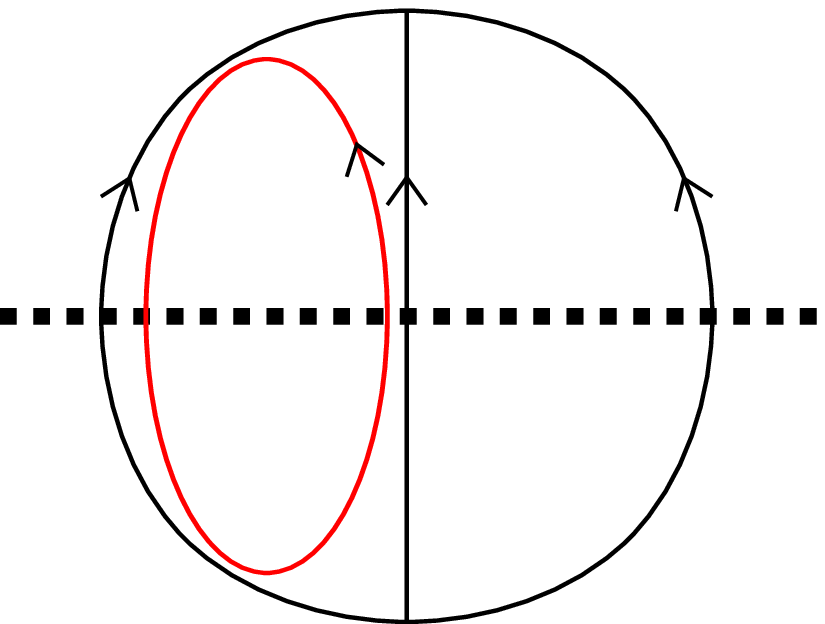}};
 (0,-39)*{\scriptstyle \mathrm{wt} = 1};
 (0,-52)*{\includegraphics[width=80px]{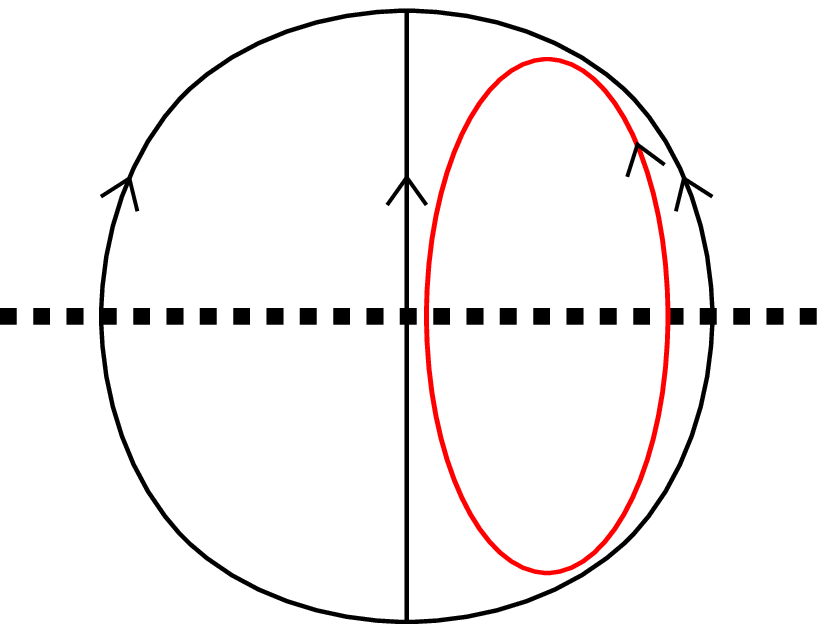}};
 (0,-64.5)*{\scriptstyle \mathrm{wt} = 1};
\endxy\;\;\;\;
\xy
 (0,0)*{\includegraphics[width=80px]{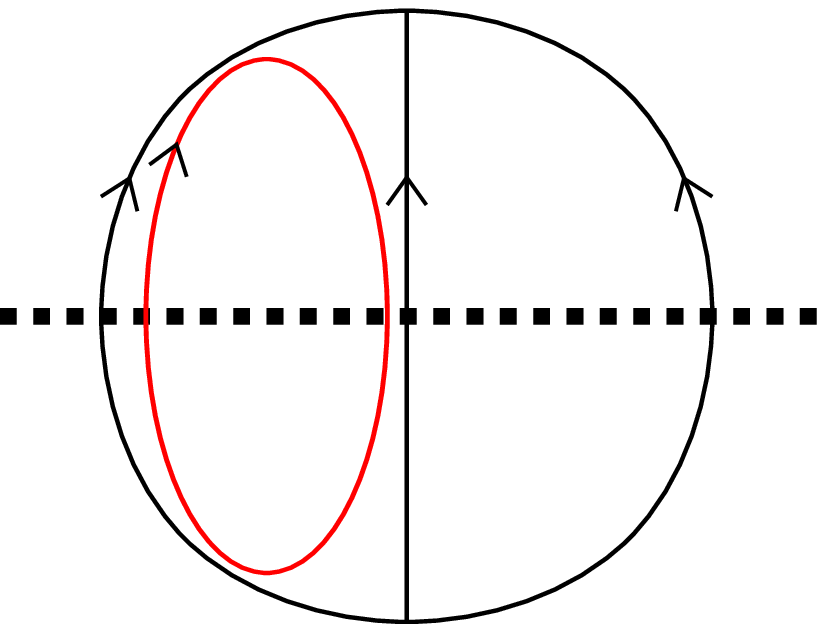}};
 (0,-13.5)*{\scriptstyle \mathrm{wt} = -1};
 (0,-26)*{\includegraphics[width=80px]{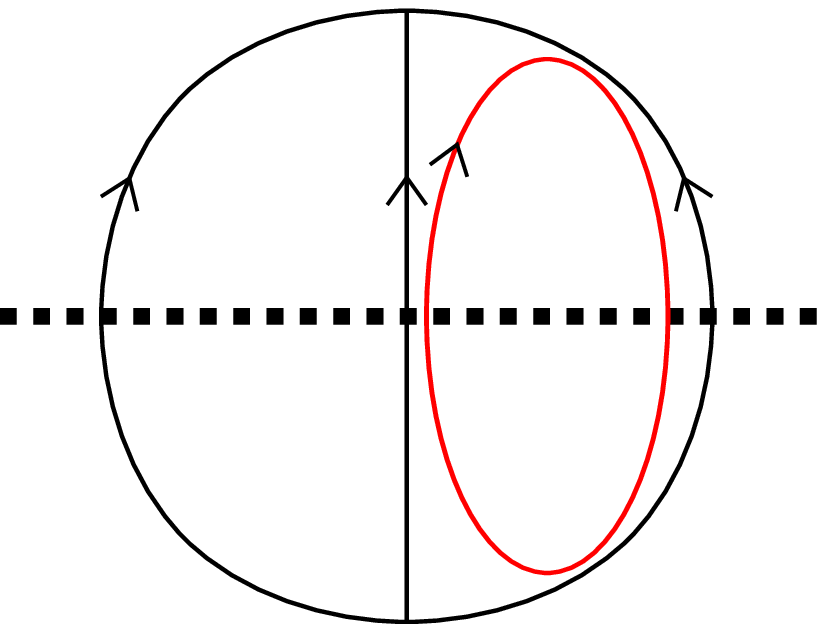}};
 (0,-39)*{\scriptstyle \mathrm{wt} = -1};
 (0,-52)*{\includegraphics[width=80px]{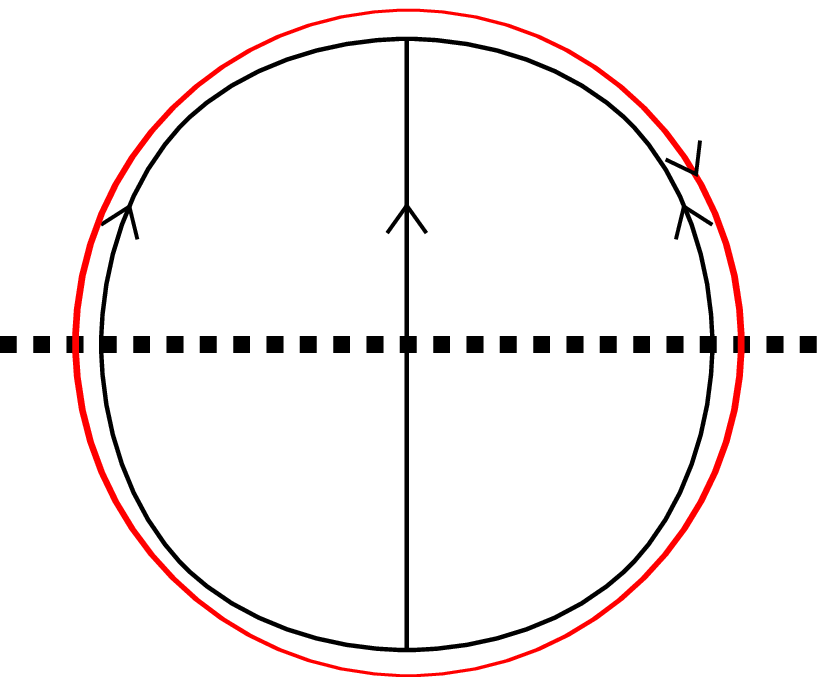}};
 (0,-64.5)*{\scriptstyle \mathrm{wt} = -3};
\endxy\;\;\;\;
\xy
 (0,0)*{\includegraphics[width=80px]{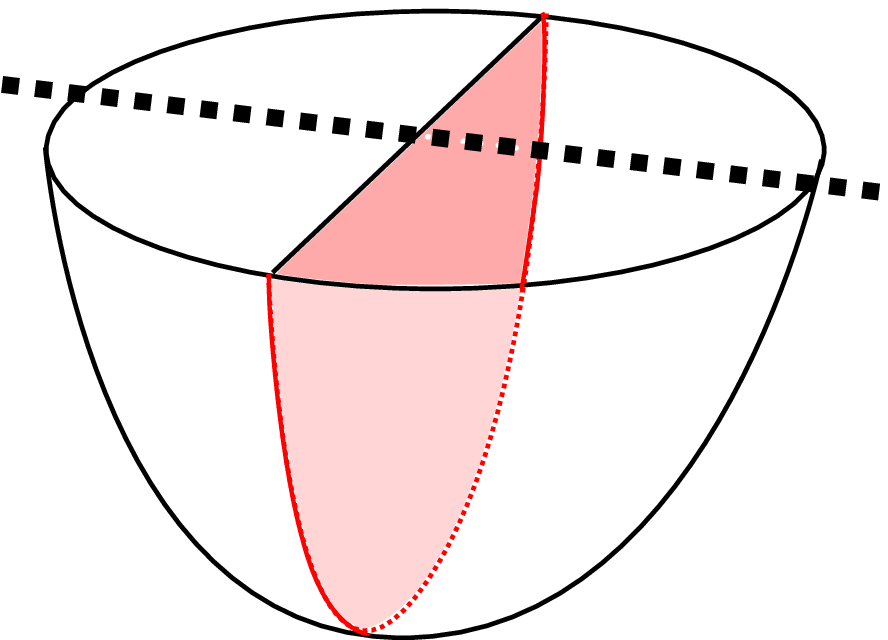}};
 (0,-13.5)*{\scriptstyle \scriptstyle \deg_q = 0};
 (0,-26)*{\includegraphics[width=80px]{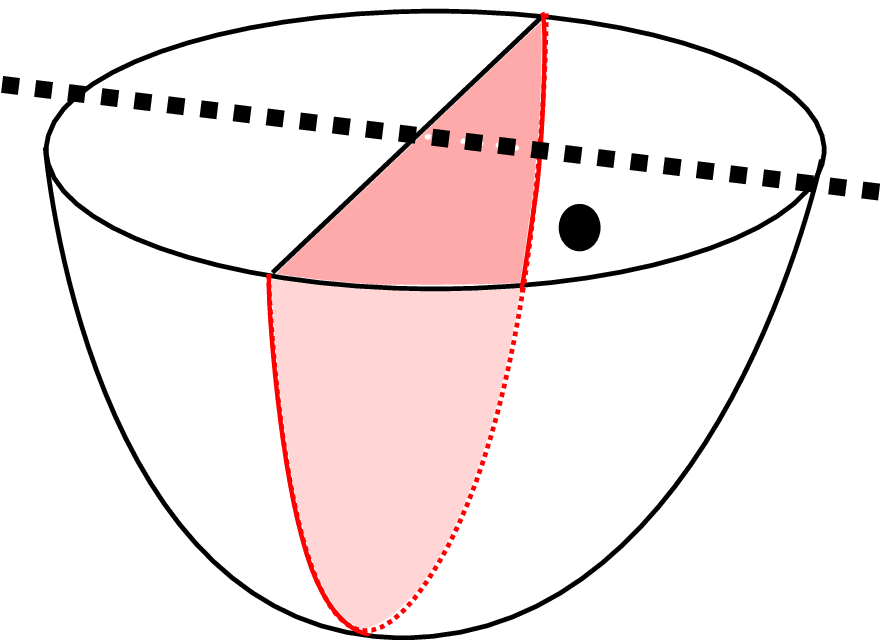}};
 (0,-39)*{\scriptstyle \scriptstyle \deg_q = 2};
 (0,-52)*{\includegraphics[width=80px]{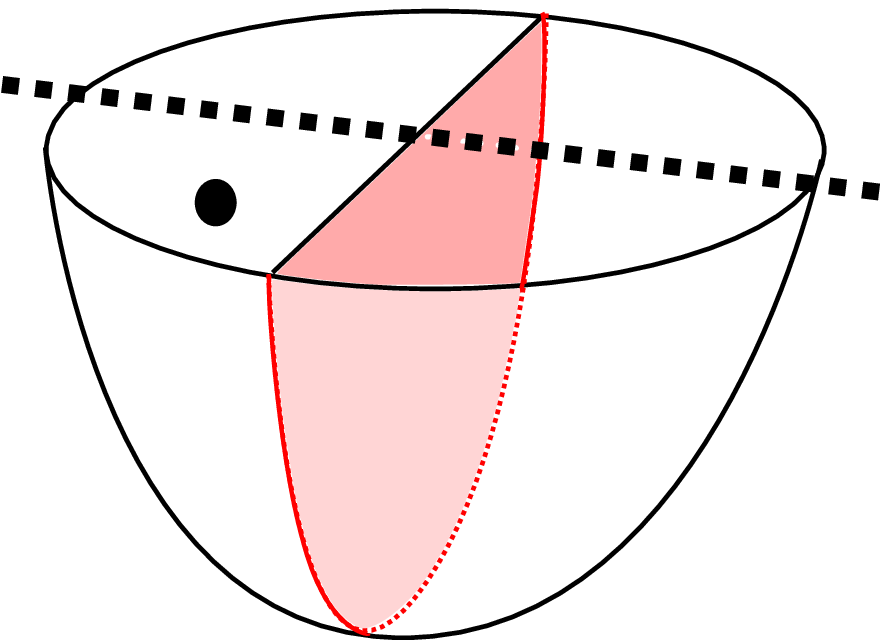}};
 (0,-64.5)*{\scriptstyle \scriptstyle \deg_q = 2};
\endxy\;\;\;\;
\xy
 (0,0)*{\includegraphics[width=80px]{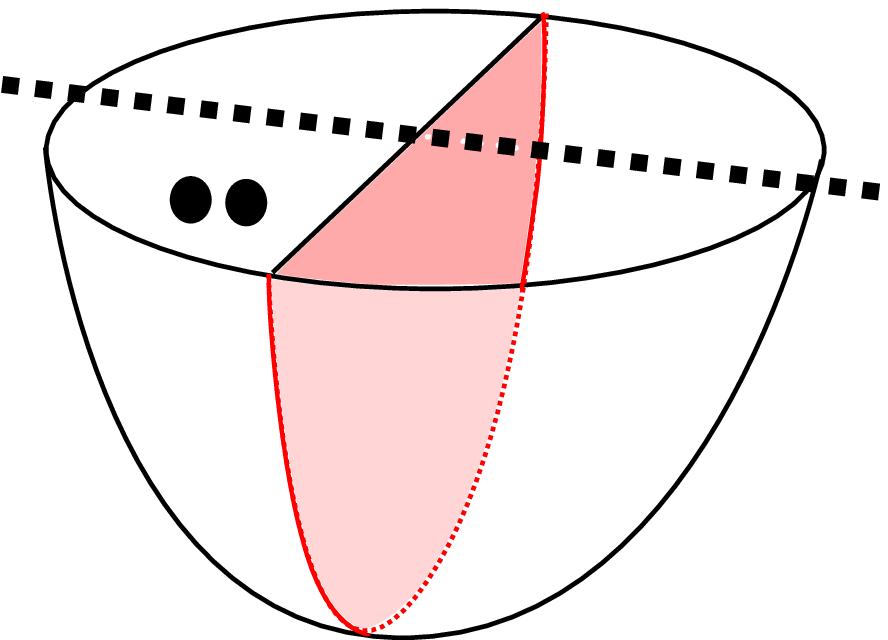}};
 (0,-13.5)*{\scriptstyle \scriptstyle \deg_q = 4};
 (0,-26)*{\includegraphics[width=80px]{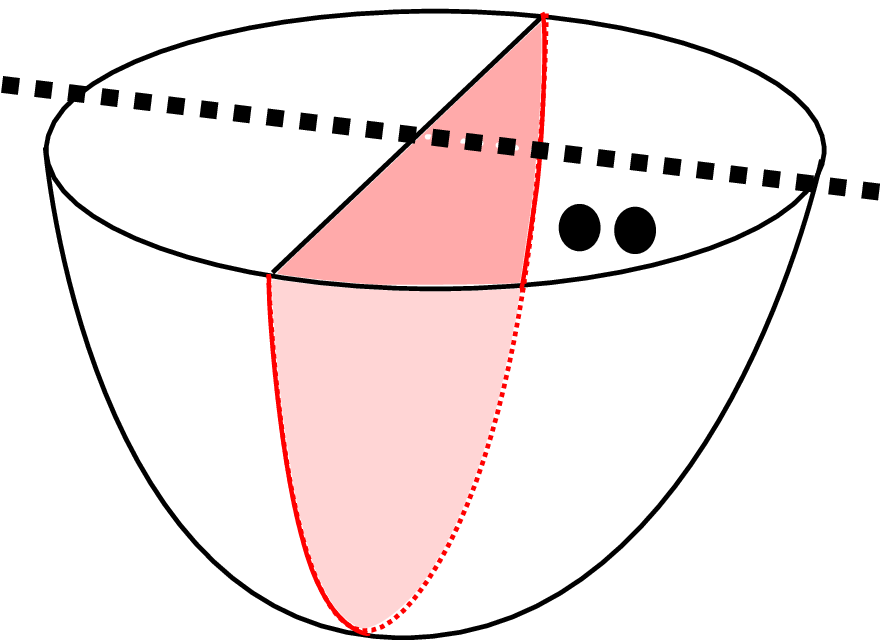}};
 (0,-39)*{\scriptstyle \scriptstyle \deg_q = 4};
 (0,-52)*{\includegraphics[width=80px]{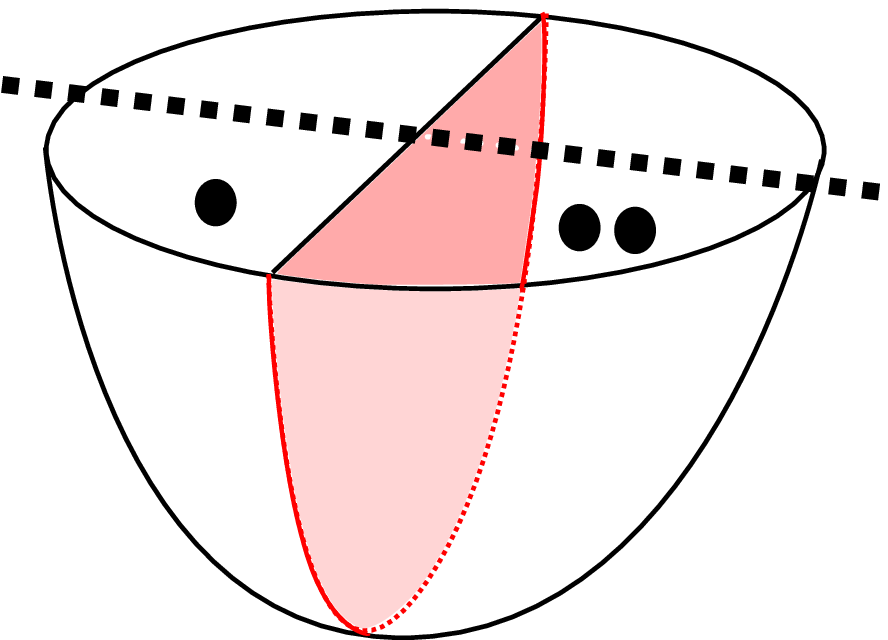}};
 (0,-64.5)*{\scriptstyle \scriptstyle \deg_q = 6};
\endxy
\]
To illustrate why the growth algorithm for foams gives (up to a sign) the result above, let us do the top left $u_c$ in more detail. We have
\[
T_{\vec{\lambda}}=\left(\;\xy(0,0)*{\emptyset}\endxy\;,\;\xy(0,0)*{\begin{Young}1\cr\end{Young}}\endxy\;,\;\xy(0,0)*{\begin{Young}2&3\cr\end{Young}}\endxy\;\right)\;\;\text{and}\;\;\iota(u_c)=\iota(u_c)^{\prime}=\left(\;\xy(0,0)*{\emptyset}\endxy\;,\;\xy(0,0)*{\begin{Young}3\cr\end{Young}}\endxy\;,\;\xy(0,0)*{\begin{Young}1&2\cr\end{Young}}\endxy\;\right).
\]
Thus, the foam $e(\vec{\lambda})d(\vec{\lambda})$ will be a dotted (the dot comes from the node labeled $1$) identity for the web $F_2F_1F_1v_{3^1}$, which is a theta-foam with an extra digon. It is worth noting that this dot is unexpected from the foam framework. But the permutation from $\iota(u_c)$ to $T_{\vec{\lambda}}$ is given by $\sigma=\tau_2\tau_1$ and thus we see that
\[
F_2F_1F_1\stackrel{\tau_1}{\mapsto}F_2F_1F_1\stackrel{\tau_2}{\mapsto}F_2F_1F_2,
\]
where $u=F_1F_2F_1v_{3^1}$. The corresponding foam $\mathcal F_{u_c}$ will be a composite of the top left foam from Definition~\ref{defn-foamLT} and an unzip. Therefore, the final foam $\mathcal F^{\vec{\lambda}}_{\iota(u_c),\iota(u_c)}\colon u\to u$ contains a so-called singular bubble with a dot. As Khovanov explains in~\cite{kh3}, this bubble can be bursted which removes the bubble and the dot. The reader is encouraged to do the other five, but we note that two of them will also create ``two many'' dots, which will be bursted.

Note that it is no coincidence that there are always two basis elements that are \textit{``dual''} to each other, i.e. cupping them of with its ``dual'' gives $\pm 1$ while cupping it of with all the others gives $0$. This is related to the ``dual'' cellular basis given in~\cite{hm}. We note that this could be useful to evaluate ``foams'' for $N>3$.
\end{ex}
\begin{lem}\label{lem-welldefall}
The growth algorithm for foams from Definition~\ref{defn-growthfoam} is well-defined.
\end{lem}
\begin{proof}
That the algorithm is well-defined follows from the Lemmata~\ref{lem-welldefidem},~\ref{lem-welldefdots} and~\ref{lem-actweldef} above.
\end{proof}
\subsection{Foam basis}\label{sec-basis}
We are going to show in this section that the growth algorithm for foams given in Definition~\ref{defn-growthfoam} suffices to produce all foams, i.e. it can be used to generate a basis of the foam space. Moreover, we show that this basis is homogeneous. We show in the next section that it is indeed a graded cellular basis in the sense of Hu and Mathas~\cite{hm}.
\vspace*{0.25cm}

The main observation is that the foams $\mathcal F(\tau_{i}(a_{i},b_{i}))$ defined in Definition~\ref{defn-foamLT} are either zip or unzip foams (in the case $|a_i-b_i|=1$), two digon removals (in the case $a_i=b_i$) or just shifts (all the other cases). Since this is not trivial, we state it as a lemma. Recall that $\mathrm{deg}_q$ denotes the $q$-degree of the foam space.
\vspace*{0.25cm}

We start with some useful Lemmata.
\begin{lem}\label{lem-zipping}
Given the same data as in Definition~\ref{defn-foamLT}, we have that for all $\tau_{i}\in S_{k-1}$ the foams $\mathcal F(\tau_{i}(a_{i},b_{i}))$ split into three disjoint cases.
\begin{itemize}
\item[(a)] $\mathcal F(\tau_{i}(a_{i},b_{i}))\colon w_1\to w_2$ is a zip or a unzip iff $|a_i-b_i|=1$. We have $\mathrm{deg}_q(\mathcal F(\tau_{i}(a_{i},b_{i})))=1$ in this case.
\item[(b)] $\mathcal F(\tau_{i}(a_{i},b_{i}))\colon w_1\to w_2$ removes two digons iff $a_i=b_i$. We have $\mathrm{deg}_q(\mathcal F(\tau_{i}(a_{i},b_{i})))=-2$ in this case.
\item[(c)] $\mathcal F(\tau_{i}(a_{i},b_{i}))\colon w_1\to w_2$ is just a shift iff $|a_i-b_i|>1$. We have $\mathrm{deg}_q(\mathcal F(\tau_{i}(a_{i},b_{i})))=0$ in this case.
\end{itemize}
\end{lem}
\begin{proof}
(a): A case by case check of all the possibilities how the local picture of the webs $w_1$ at the top and $w_2$ at the bottom could look like. Note that both webs are generated by a sequence of LT-generators, thus not all combinations are possible. Since all cases work in the same vein we only illustrate four here. So five in total, i.e. counting the one from Example~\ref{ex-action}. In fact the Example~\ref{ex-action} is a blueprint of all cases. The four other cases are the following.
\[
\xy
(-30,0)*{\includegraphics[scale=0.5]{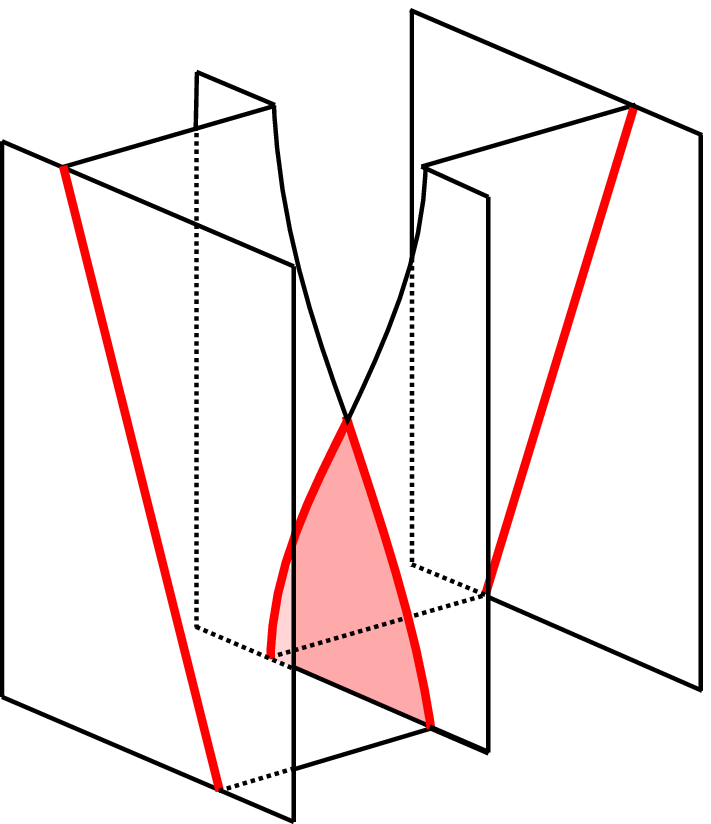}};
(30,0)*{\includegraphics[scale=0.5]{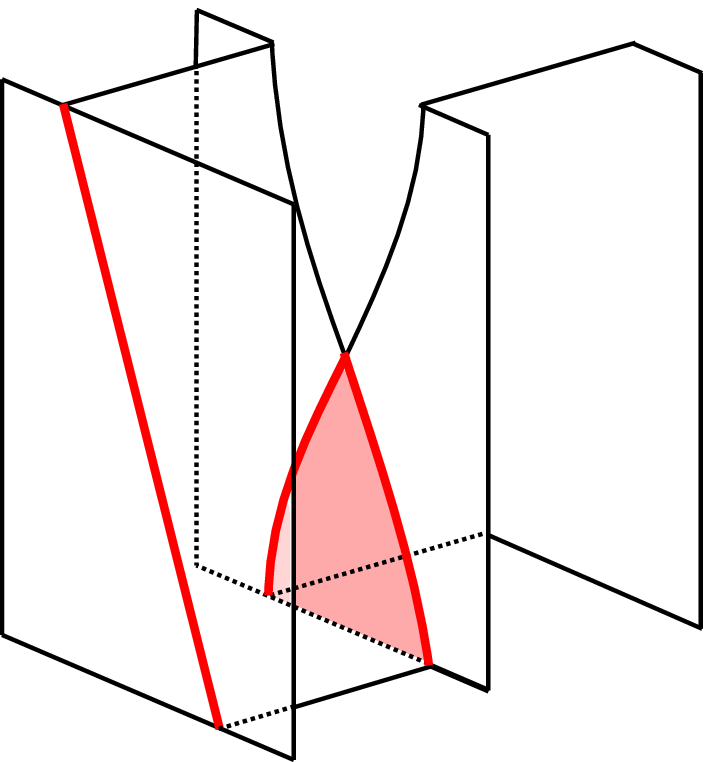}};
\endxy
\]
and
\[
\xy
(-30,0)*{\includegraphics[scale=0.5]{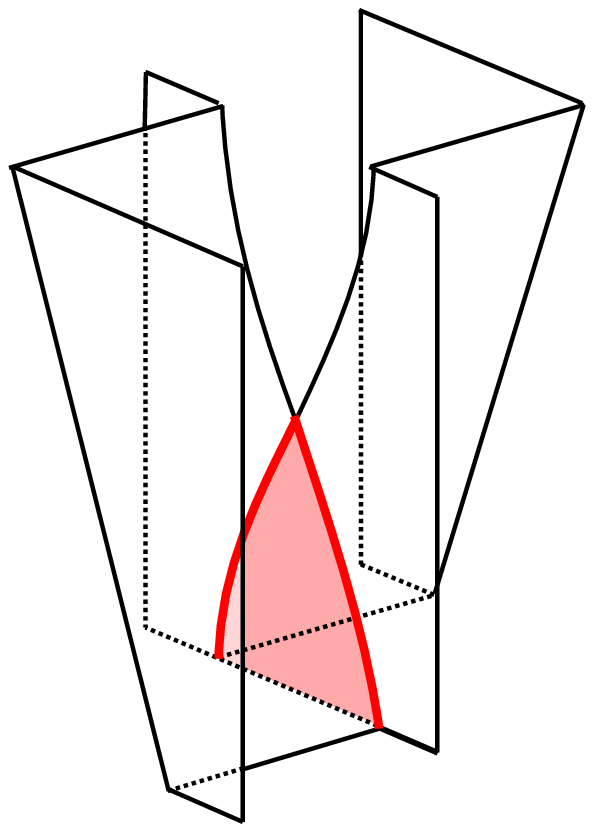}};
(30,0)*{\includegraphics[scale=0.5]{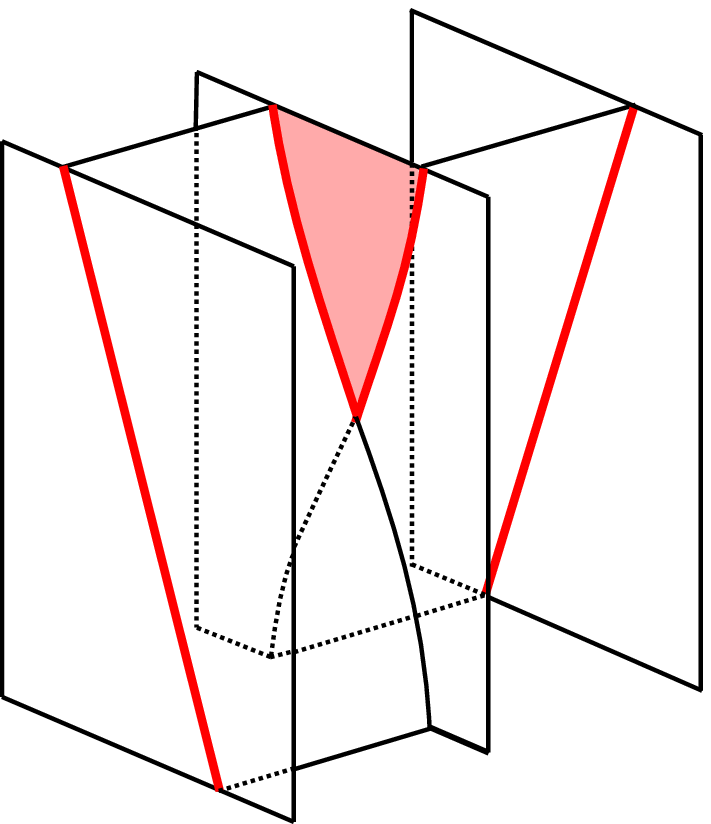}};
\endxy
\]
We have indicated the important face for our argument with color in all the pictures. To be more precise, the foam in the top left is a unzip between a Y-move and a flipped Y-move at the top and two H-moves at the bottom, the foam in the bottom left is between two shifts and a Y-move and a flipped Y-move, the foam in the top right is between a Y-move and arc pair and a H-move and a flipped Y-move at the bottom and the lower right is between two H-moves and a Y-move and a flipped Y-move.

All other cases are similar. It should be noted that the number of case-by-case checks reduces if one takes rotational symmetries into account, e.g. the foam in the top left picture and the one in the right bottom are essentially the same.

Moreover, it is worth noting that the ambiguous case, i.e. two similar moves at the bottom and top, does not occur by construction, since the action of $F_iF_{i+1}$ will be different than the action of $F_{i+1}F_i$ and will therefore always give two different local webs.  
\vspace*{0.25cm}

(b): As before a case-by-case check. In fact the number of cases is small, that is only two can occur, since a H-move after a H-move is not possible because at least one of the H-moves would correspond to an $E$. Thus, we only the following two cases.
\begin{align}\label{eq-digonremoval}
\xy
(0,0)*{\includegraphics[scale=0.5]{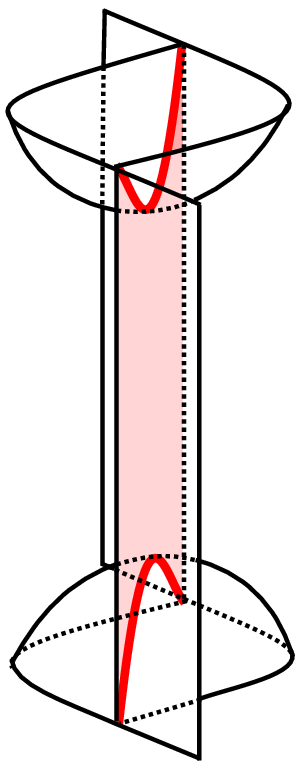}};
\endxy\;\;\;\text{ and }\;\;\;
\xy
(0,0)*{\includegraphics[scale=0.5]{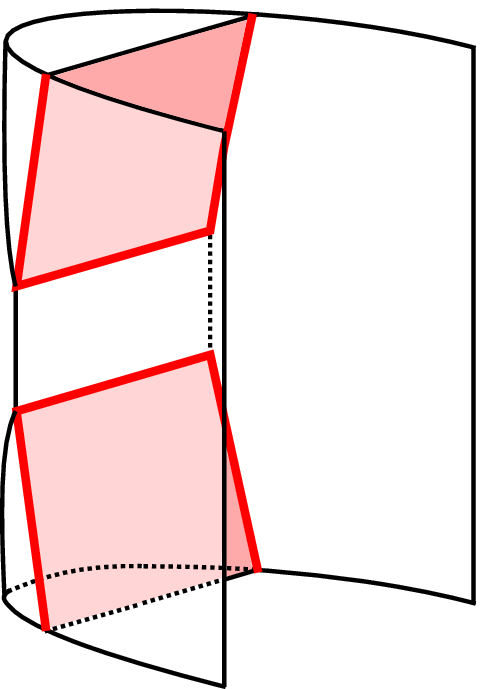}};
\endxy
\end{align}
We have smoothen the first picture a little bit to make it more visible. We have indicated the important face with color again.
\vspace*{0.25cm}

(c): This is clear by the definition of the local changes of the foam pictured in Definition~\ref{defn-foamLT}, because it will be the identity foam up to isotopy.
\vspace*{0.25cm}

Since the $q$-degree of a zip or unzip is $1$, the $q$-degree of a digon removal is $-1$ and the $q$-degree of the identity is zero, we conclude that the degrees work out as stated which finishes the proof.
\end{proof}
\begin{lem}\label{lem-degfoam}
Let the foam
\[
\mathcal F^{\vec{\lambda}}_{\iota(u_f),\iota(v_{f^{\prime}})}=\mathcal F_{u_{f}}e(\vec{\lambda})d(\vec{\lambda})\mathcal F_{v_{f^{\prime}}}^*\colon u\to v
\]
be as in Definition~\ref{defn-growthfoam}. Then
\[
\mathrm{deg}_q(\mathcal F^{\vec{\lambda}}_{\iota(u_f),\iota(v_{f^{\prime}})})=\mathrm{deg}_{\mathrm{wt}}(u_{f})+\mathrm{deg}_{\mathrm{wt}}(v_{f^{\prime}})=\mathrm{deg}_{\mathrm{BKW}}(\iota(u_f))+\mathrm{deg}_{\mathrm{BKW}}(\iota(v_{f^{\prime}})).
\]
\end{lem}
\begin{proof}
We only have to show the first equality, since the second already follows as a consequence of the Proposition~\ref{prop-degree}.
\vspace*{0.25cm}

We start by noting that
\[
\mathrm{deg}_q(e(\vec{\lambda})d(\vec{\lambda}))=\mathrm{deg}_q(d(\vec{\lambda}))=2n_d,
\]
where $n_d$ is the number of dots. This is true, because both are topological just the identity and a dot has $q$-degree $2$. This implies, since we spread dots using a convention of adding dots per addable nodes and $T_{\vec{\lambda}}$ does not have removable nodes, that
\[
\mathrm{deg}_q(e(\vec{\lambda})d(\vec{\lambda}))=2\mathrm{deg}_{\mathrm{BKW}}(T_{\vec{\lambda}}).
\]
It should be noted that the extra removing of the internal faces from Definition~\ref{defn-foamLT2} corresponds exactly to the change in degree from $\iota(u_f)$ to $\iota(u_f)^{\prime}$. So to simplify our notation, we can freely assume that $\iota(u_f)=\iota(u_f)^{\prime}$ and $\iota(v_{f^{\prime}})=\iota(v_{f^{\prime}})^{\prime}$.

Hence, the results of the two Lemmata~\ref{lem-actweldef} and~\ref{lem-zipping} imply now that we only have to check that the degree works out on the level of actions on tableaux, because the action defined via foams agrees with the one on the tableaux (part two of Lemma~\ref{lem-actweldef}) and, as a consequence of Lemma~\ref{lem-zipping}, we see that it has exactly the same degree as the action defined on tableaux.

Thus, we can use a result from the side of the cyclotomic Hecke algebra, i.e. Corollary 3.14 of Brundan, Kleshchev and Wang~\cite{bkw}, and we get
\[
\mathrm{deg}_{\mathrm{BKW}}(\iota(u_f))=\mathrm{deg}_{\mathrm{BKW}}(T_{\vec{\lambda}})+\mathrm{deg}_q(\mathcal F_{u_{f}})\;\;\text{and}\;\;\mathrm{deg}_{\mathrm{BKW}}(\iota(v_{f^{\prime}}))=\mathrm{deg}_{\mathrm{BKW}}(T_{\vec{\lambda}})+\mathrm{deg}_q(\mathcal F_{v_{f^{\prime}}}).
\]
Hence, we have
\begin{align*}
\mathrm{deg}_q(\mathcal F^{\vec{\lambda}}_{\iota(u_f),\iota(v_{f^{\prime}})})&=\mathrm{deg}_q(\mathcal F_{u_{f}}e(\vec{\lambda})d(\vec{\lambda})\mathcal F_{v_{f^{\prime}}}^*)=\mathrm{deg}_q(\mathcal F_{u_{f}})+\mathrm{deg}_q(e(\vec{\lambda})d(\vec{\lambda}))+\mathrm{deg}_q(\mathcal F_{u_{f}})\\
&=\mathrm{deg}_{\mathrm{BKW}}(\iota(u_f))-\mathrm{deg}_{\mathrm{BKW}}(T_{\vec{\lambda}})+2n_d+\mathrm{deg}_{\mathrm{BKW}}(\iota(v_{f^{\prime}}))-\mathrm{deg}_{\mathrm{BKW}}(T_{\vec{\lambda}})\\
&=\mathrm{deg}_{\mathrm{BKW}}(\iota(u_f))-n_d+2n_d+\mathrm{deg}_{\mathrm{BKW}}(\iota(v_{f^{\prime}}))-n_d\\
&=\mathrm{deg}_{\mathrm{BKW}}(\iota(u_f))+\mathrm{deg}_{\mathrm{BKW}}(\iota(v_{f^{\prime}}))=\mathrm{deg}_{\mathrm{wt}}(u_{f})+\mathrm{deg}_{\mathrm{wt}}(v_{f^{\prime}})
\end{align*}
which proves the statement.
\end{proof}
\begin{lem}\label{lem-involution}
Let the foam
\[
\mathcal F^{\vec{\lambda}}_{\iota(u_f),\iota(v_{f^{\prime}})}=\mathcal F_{u_{f}}e(\vec{\lambda})d(\vec{\lambda})\mathcal F_{v_{f^{\prime}}}^*\colon u\to v
\]
be as in Definition~\ref{defn-growthfoam} and let ${}^*$ denote the involution on the foam space. Then
\[
(\mathcal F^{\vec{\lambda}}_{\iota(u_f),\iota(v_{f^{\prime}})})^*=\mathcal F^{\vec{\lambda}}_{\iota(v_{f^{\prime}}),\iota(u_f)}\colon v\to u.
\]
\end{lem}
\begin{proof}
This is a direct consequence of Lemma~\ref{lem-welldefidem}, because we have
\begin{align*}
(\mathcal F^{\vec{\lambda}}_{\iota(u_f),\iota(v_{f^{\prime}})})^* &=(\mathcal F_{u_{f}}e(\vec{\lambda})d(\vec{\lambda})\mathcal F_{v_{f^{\prime}}}^*)^*\\
&=((\mathcal F_{v_{f^{\prime}}})^*)^*(e(\vec{\lambda})d(\vec{\lambda}))^*\mathcal F_{u_{f}}^*\\
&=\mathcal F_{v_{f^{\prime}}}e(\vec{\lambda})d(\vec{\lambda})\mathcal F_{u_{f}}^*,
\end{align*}
where the second equality is due to the definition of the involution ${}^*$.
\end{proof}
The next proposition shows that the growth algorithm for foams is a ``foamy'' version of Hu and Mathas graded cellular basis~\cite{hm}. 
\begin{prop}\label{prop-hmfoam}
Let $S$ be any sign string and let
\[
J_S=\{J\mid (S,J)\text{ is a pair of a sign and state string for some web }w\in B_S\}
\]
and let
\[
B_S^J=\{w_f\mid w\in B_S,\;f\text{ is a flow with boundary }J\text{ extending to }w\}.
\]
Then $(S,J,u_{f},v_{f^{\prime}})$ with $J\in J_S$ and $u_{f},v_{f^{\prime}}\in B_S^J$ gives rise to a foam $\mathcal F^{\vec{\lambda}}_{\iota(u_f),\iota(v_{f^{\prime}})}\in K_S$.
\vspace*{0.2cm}

This foam is obtained from the application of the foamation functor to the HM-basis element with the datum (the multipartition $\vec{\lambda}$ is the one associated to $(S,J)$)
\[
(\vec{\lambda},\iota(u_f)^{\prime}\in\mathrm{Std}(\vec{\lambda}),\iota(v_{f^{\prime}})^{\prime}\in\mathrm{Std}(\vec{\lambda}))
\]
and removing the internal faces using digon removals or theta removals.
\end{prop}
\begin{proof}
The first statement is just a collection of the results of the sections before, i.e. that a pair $(S,J)$ gives rise to a $3$-multitableau $\vec{\lambda}$ is proven in Theorem~\ref{thm-fine}, that the webs with flows $u_{f},v_{f^{\prime}}$ give rise to a unique filling $\iota(u_f),\iota(v_{f^{\prime}})\in\mathrm{Std}(\vec{\lambda})$ of the $3$-multitableau $\vec{\lambda}$ is proven in Lemma~\ref{lem-welldef} and finally the definition of the foam $\mathcal F^{\vec{\lambda}}_{\iota(u_f),\iota(v_{f^{\prime}})}$ is given in Definition~\ref{defn-growthfoam}.
\vspace*{0.25cm}

To see that the second statement is true, we observe that the Section~\ref{sec-compare} and $q$-skew Howe duality imply that we can read the HM-basis elements as foams in the following way. Given a $3$-multipartition $\vec{\lambda}$ with residue sequence $r(\vec{\lambda})=(r_1,r_2,\dots)$, we can read it as a $2$-morphism in $\Ucat$ as indicated below
\[
\xy
(0,0)*{\includegraphics[scale=1.05]{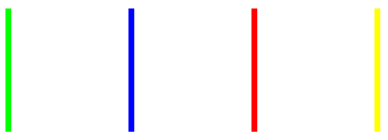}};
(-25,-1)*{\cdots};
(25,-1)*{v_{3^{\ell}}};
(20,-9)*{r_1};
(7.5,-9)*{r_2};
(-6,-9)*{r_3};
(-19,-9)*{r_4};
(20,8)*{r_1};
(7.5,8)*{r_2};
(-6,8)*{r_3};
(-19,8)*{r_4};
(0,-13)*{w=\dots F_{r_4}F_{r_3}F_{r_2}F_{r_1}v_{3^{\ell}}};
(0,12)*{w=\dots F_{r_4}F_{r_3}F_{r_2}F_{r_1}v_{3^{\ell}}};
\endxy
\]
reading from right to left and all arrows pointing downwards. This can be translated to a web for the top and bottom string of residues by applying the corresponding sequence of $F$ to the highest weight vector $v_{3^{\ell}}$ as explained in the Section~\ref{sec-compare} and reading the residues $r_i$ as $F_{r_i}$ as explained in and around the Definition~\ref{defn-exgrowth}.
\vspace*{0.25cm}

Then the second statement follows now from the sequence of Lemmata~\ref{lem-welldefdots},~\ref{lem-actweldef} and~\ref{lem-zipping}, i.e. the Lemma~\ref{lem-welldefdots} shows that the dotted HM-idempotent, after applying the foamation functor, is essentially the same as the dotted idempotent from Definition~\ref{defn-idemdots}, since in both pictures the idempotent $e(\vec{\lambda})$ is obtained from the residue sequence $r(\vec{\lambda})$ and the dots are obtained from the set of addable nodes explained in Definition~\ref{defn-idemdots}. Moreover, the Lemmata~\ref{lem-actweldef} and~\ref{lem-zipping} show that the HM-element $\psi_{\iota(u_f)^{\prime}}$, after applying the foamation functor, is the same as the foam $\mathcal F_{u_{f}}$ (plus the extra internal face removals while passing from $\iota(u_f)^{\prime}$ to $\iota(u_f)$), since we have the correspondence
\[
\xy
(-60,0)*{\includegraphics[scale=1]{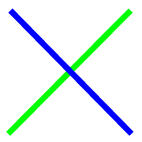}};
(-64.5,-10)*{F_{a_i}};
(-53.5,-10)*{F_{b_i}};
(-45,0)*{\rightsquigarrow};
(-30,0)*{\mathcal F(\tau_{i}(a_{i},b_{i}))};
(-30,-5)*{|a_i-b_i|=1};
(0,0)*{\includegraphics[scale=1]{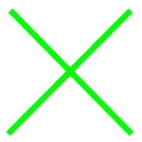}};
(-5,-10)*{F_{a_i}};
(7,-10)*{F_{a_i}};
(10,0)*{\rightsquigarrow};
(25,0)*{\mathcal F(\tau_{i}(a_{i},b_{i}))};
(25,-5)*{|a_i-b_i|=0};
(55,0)*{\includegraphics[scale=1]{res/figs/basis/HM-strings1.eps}};
(50,-10)*{F_{a_i}};
(63,-10)*{F_{b_i}};
(65,0)*{\rightsquigarrow};
(80,0)*{\mathcal F(\tau_{i}(a_{i},b_{i}))};
(80,-5)*{|a_i-b_i|>1};
\endxy
\]
by Lemma~\ref{lem-zipping} and the different local pictures are spread in essentially the same way in both cases by Lemma~\ref{lem-actweldef}. Hence, we conclude that the HM-basis element with the datum
\[
(\vec{\lambda},\iota(u_f)^{\prime}\in\mathrm{Std}(\vec{\lambda}),\iota(v_{f^{\prime}})^{\prime}\in\mathrm{Std}(\vec{\lambda}))
\]
corresponds to the foam $\mathcal F^{\vec{\lambda}}_{\iota(u_f),\iota(v_{f^{\prime}})}\in K_S$ (plus the relevant internal face removals).
\end{proof}
\vspace*{0.15cm}
 
We are now able to proof that the growth algorithm for foams given in Definition~\ref{defn-growthfoam} gives a basis of the $\mathfrak{sl}_3$-web algebra $K_S$.
\begin{thm}\label{thm-foambasis}
The growth algorithm for foams gives a homogeneous basis of $K_S$.
\end{thm}
\begin{proof}
We will show that the growth algorithm gives a linear independent set of foams denoted by
\[
\mathfrak{F}=\{\mathcal F^{\vec{\lambda}}_{\iota(u_f),\iota(v_{f^{\prime}})}\in K_S\mid\,(S,J,u_{f},v_{f^{\prime}}),\,J\in J_S,\,u_{f},v_{f^{\prime}}\in B_S^J\}.
\]
By a counting argument we see that this set forms a basis, since we know that every triple
\[
(\vec{\lambda},\iota(u_f)\in\mathrm{Std}(\vec{\lambda}),\iota(v_{f^{\prime}})\in\mathrm{Std}(\vec{\lambda}))
\]
corresponds exactly to one closed web $w=u^*v$ with a fixed flow by Theorem~\ref{thm-fine}. Thus, because $K_S$ has a basis indexed by closed webs with flows as we have recalled in Remark~\ref{rem-dim}, we conclude that the linear independence of $\mathfrak F$ suffices to show that the set $\mathfrak F$ forms a basis.
\vspace*{0.25cm}

To see linear independence we use Proposition~\ref{prop-hmfoam} and pull $\mathfrak{F}$ back to the cyclotomic Hecke algebra. We denote the pullback (by not removing the internal faces) restricted to the datum
\[
(\vec{\lambda},\iota(u_f)\in\mathrm{Std}(\vec{\lambda}),\iota(v_{f^{\prime}})\in\mathrm{Std}(\vec{\lambda}))
\]
by $\Psi^{-1}(\mathfrak{F})$.

Since Hu and Mathas showed that their definition gives rise to a basis, we conclude that the corresponding set $\Psi^{-1}(\mathfrak{F})$ is linear independent.
\vspace*{0.25cm}

In 4.3.3. of~\cite{lqr} Lauda, Queffelec and Rose showed that all the foam relations in $\F(w)$ follow from certain relations in $\Ucat$. Moreover, they argue that the $\mathfrak{sl}_3$-foam category can be obtained from $\Ucat$ by modding out certain relations involving bubbles of weight $3$. But our convention to only use $3$-multitableaux ensures that the elements of $\Psi^{-1}(\mathfrak{F})$ do not contain any of these relations, while all the other relevant relations are true independent of $n$. Hence, a possible linear independence of $\mathfrak{F}$ can not be created by the foamation functor $\Psi$ (which is an additive functor) nor by modding out by the cyclotomic relation~\ref{eq-cyclorel}.

Thus, the set $\mathfrak F$ must be linear independent, since, as Khovanov showed in Proposition 7 and 8 of~\cite{kh3}, the internal face removals give rise to a isomorphism of graded vector spaces and the foams without the internal face removals form a linear independent set for a version of the $\mathfrak{sl}_3$-web algebra between the webs without the face removals.
\vspace*{0.25cm}

That $\mathfrak{F}$ is homogeneous follows from Lemma~\ref{lem-degfoam}.
\end{proof}
\vspace*{0.05cm}
\begin{rem}\label{rem-foambasis}
Hu and Mathas proof that their set is linear independent relies on the connection to another basis that is known for the cyclotomic Hecke algebra, the so-called \textit{Dipper, James and Mathas standard basis} which can be seen as a ``higher version'' of the classical construction of the basis for Specht modules. The proof that this standard basis is in fact a basis is non-trivial, see Dipper, James and Mathas~\cite{djm}.
\vspace*{0.25cm}

In our framework, since everything is inductively build from small cases, it is possible to proof most statements by an inductive case-by-case check. In fact, there is such an alternative proof of Theorem~\ref{thm-foambasis} in our framework.

The alternative ``by hand'' proof of the statement~\ref{thm-foambasis} is roughly as follows. One does induction on the number of faces of $u^*v$ to conclude that the restriction of $\mathfrak{F}$ to ${}_uK_v$ gives a basis of latter. The small cases can all be verified easily and the induction step is similar to the induction that shows that the basis from Remark~\ref{rem-basis} is in fact a basis. To be more precise, all the three basic moves, called circle removal, digon removal and square removal, have a local inverse. One then verifies ``by hand'' all possible cases how a face in $u^*v$ with flow can look like. If one does so, then one realizes that the growth algorithm for foams has a summand that looks like one of these face removals and all the other summands are killed by the inverse. To see this one does a case-by-case check. It should be noted that Lemma~\ref{lem-zipping} ensures that the growth algorithm ``almost'' remove faces anyway, since the face removing foams are just local zips. One just have to be careful that sometimes not the right edges will be unziped. But one can always use the relations~\ref{eq:dr}, ~\ref{eq:sqr} and~\ref{eq:dotm} to show that at least one of the summands is exactly the desired face removal (and the other is killed by the inverse).

Hence, one can conclude that the set $\mathfrak{F}$ is related by a non-singular change of base matrix to the basis from Remark~\ref{rem-basis}. But since there are a lot of cases to check, we do not do it here.
\end{rem}
We can reprove Brundan and Kleshchev's graded dimension formula, i.e. Theorem 4.20 in~\cite{bk2}, as a direct consequence of Theorem~\ref{thm-foambasis} in a very simple fashion for our framework. Note that we see ${}_{g(T_{\vec{\lambda}})}K{}_{g(T_{\vec{\mu}})}$ in the following theorem by abuse of notation as an analogue of the $\mathfrak{sl}_3$-web algebra between the possible non-elliptic $\mathfrak{sl}_3$-web given by the extended growth algorithm from Definition~\ref{defn-exgrowth}.
\begin{thm}(\textbf{Brundan and Kleshchev: Graded dimension formula})\label{thm-gdf}
For $3$-multipartitions $\vec{\lambda},\vec{\mu}$ of $c(S)$, we have (where $\{n\}$ is the shift by the number of strands)
\[
\mathrm{dim}_q\;e(\vec{\lambda})R_{c(S)}e(\vec{\mu})\{n\}=\mathrm{dim}_q\; {}_{g(T_{\vec{\lambda}})}K{}_{g(T_{\vec{\mu}})}\{n\}=\sum_{\substack{\vec{T}_1\in \mathrm{Std}(\vec{\lambda})\\ \vec{T}_2\in \mathrm{Std}(\vec{\mu})}} q^{\mathrm{deg}_{BKW}(\vec{T}_1)+\mathrm{deg}_{BKW}(\vec{T}_2)}.
\]
\end{thm}
\begin{proof}
The second equality follows from the fact that Brundan, Kleshchev and Wang's degree is given by minus the weight of the corresponding flow, i.e. Proposition~\ref{prop-degree}, and the fact that the $q$-degree of the $\mathfrak{sl}_3$-web algebra is given by the Kuperberg bracket (up to the shift $\{n\}$) - even if the underlying web space consists of non-elliptic $\mathfrak{sl}_3$-webs.

The first equality follows from Proposition~\ref{prop-hmfoam} and Theorem~\ref{thm-foambasis} if we regard ${}_{g(T_{\vec{\lambda}})}K{}_{g(T_{\vec{\mu}})}$ as an analogue of the usual $K_S$ with a possible different underlying $\mathfrak{sl}_3$-web space. The details of this slightly modified version of the $\mathfrak{sl}_3$-web algebra can be verified as in the usual case demonstrated in~\cite{mpt}.
\end{proof}
\subsection{Cellularity}\label{sec-cell}
We are able to show now that $\mathfrak F$ is a cellular basis of $K_S$. First we recall the definition of a graded cellular algebra due to Graham and Lehrer~\cite{grle} in the ungraded setting and Hu and Mathas~\cite{hm} in the graded setting. Note that \textit{graded} always means $\bZ$-graded.
\begin{defn}(\textbf{Graham-Lehrer, Hu-Mathas})\label{defn-cellular}
Suppose $A$ is a graded free algebra over some field $K$ of finite rank. A \textit{graded cell datum} is an ordered quintuple $(\mathfrak{P},\mathcal T,C,\mathrm{i},\mathrm{deg})$, where $(\mathfrak P,\rhd)$ is the \textit{weight poset}, $\mathcal T(\lambda)$ is a finite set for all $\lambda \in\mathfrak P$, $\mathrm{i}$ is an \textit{involution} of $A$ and $C$ is an injection
\[
C\colon\coprod_{\lambda\in\mathfrak P}\mathcal T(\lambda)\times \mathcal T(\lambda)\to A,\;(s,t)\mapsto c^{\lambda}_{st}.
\]
Moreover, the \textit{degree function} deg is given by
\[
\mathrm{deg}\colon\coprod_{\lambda\in\mathfrak P}\mathcal T(\lambda)\to\bZ.
\] 
The whole data should be such that the $c^{\lambda}_{st}$ form a homogeneous $K$-basis of $A$ with $\mathrm{i}(c^{\lambda}_{st})=c^{\lambda}_{ts}$ and $\mathrm{deg}(c^{\lambda}_{st})=\mathrm{deg}(s)+\mathrm{deg}(t)$ for all $\lambda\in\mathfrak{P}$ and $s,t\in\mathcal T(\lambda)$. Moreover, for all $a\in A$
\begin{align}\label{eq-cell1}
ac^{\lambda}_{st}=\sum_{u\in \mathcal T(\lambda)}r_{a}(s,u)c^{\lambda}_{ut}\;(\mathrm{mod}\;A^{\rhd\lambda}).
\end{align}
Here $A^{\rhd\lambda}$ is the $K$-submodule of $A$ spanned by the set $\{c^{\mu}_{st}\mid \mu\rhd\lambda\text{ and }s,t\in\mathcal T(\mu)\}$.

An algebra $A$ with such a quintuple is called a \textit{graded cellular algebra} and the $c^{\lambda}_{st}$ are called a \textit{graded cellular basis} of $A$ (with respect to the involution $\mathrm{i}$).
\end{defn}
We are now able to show that
\[
\mathfrak{F}=\{\mathcal F^{\vec{\lambda}}_{\iota(u_f),\iota(v_{f^{\prime}})}\in K_S\mid\,(S,J,u_{f},v_{f^{\prime}}),\,J\in J_S,\,u_{f},v_{f^{\prime}}\in B_S^J\}.
\]
is a graded cellular basis.
\begin{thm}\label{thm-cellular}(\textbf{Graded cellular basis}) The algebra $K_S$ is a graded cellular algebra in the sense of Definition~\ref{defn-cellular} with the cell datum
\[
(\mathfrak{P}_{c(S)}^3,\iota(B_S^J),\mathfrak F,{}^*,\mathrm{deg}_{\mathrm{BKW}}),
\]
where $\mathfrak{P}_{c(S)}^3$ is the set of all multipartitions of $c(S)$ with at most three non-zero entries that all are gathered in the last three entries ordered by the dominance order $\vartriangleright$ from Definition~\ref{defn-dominnance}, $\iota(B_S^J)$ is the image of the set of all webs given by the LT-algorithm together with flows, the involution ${}^*$ on the foam space and the degree $\mathrm{deg}_{\mathrm{BKW}}$ on the tableaux in $\iota(B_S^J)$\footnote{Note that the formulation can be done alternatively using $\mathfrak{sl}_3$-web notions as described in the sections of~\ref{sec-uncat}.}.
\end{thm}
\begin{proof}
We have to prove four statements. To shorten our notation in the following equations, we write $u_f=\iota(u_f)$, $v_{f^{\prime}}=\iota(v_{f^{\prime}})$ and $B_S^J=\iota(B_S^J)$ as a short hand and hope the reader does not get confused since the left hand side is a flow on a web and the right hand side is a multitableau. Moreover, the scalars below should all only depend on the element on the left side of the multiplication.
\begin{itemize}
\item[(a)] Each $\mathcal F^{\vec{\lambda}}_{u_f,v_{f^{\prime}}}$ is homogeneous of degree
\[
\mathrm{deg}_q(\mathcal F^{\vec{\lambda}}_{u_f,v_{f^{\prime}}})=\mathrm{deg}_{\mathrm{BKW}}(u_f)+\mathrm{deg}_{\mathrm{BKW}}(v_{f^{\prime}}).
\]
\item[(b)] The set $\mathfrak F$ is a basis of the graded $\bC$-algebra $K_S$.
\item[(c)] The involution satisfies
\[
(\mathcal F^{\vec{\lambda}}_{u_f,v_{f^{\prime}}})^*=\mathcal F^{\vec{\lambda}}_{v_{f^{\prime}},u_f}.
\]
\item[(d)] Given $\vec{\lambda},\vec{\mu}\in\mathfrak{P}_{c(S)}^3$, $u_f,v_{f^{\prime}}\in \mathrm{Std}(\vec{\lambda})$ and $\tilde{u}_{\tilde{f}},\tilde{v}_{\tilde{f}^{\prime}}\in \mathrm{Std}(\vec{\mu})$, then there exist scalars $r_{u_f,w_{f^{\prime\prime}}}$ which do not depend on $v_{f^{\prime}}$, such that
\begin{align}\label{eq-cell2}
\mathcal F^{\vec{\mu}}_{\tilde{u}_{\tilde{f}},\tilde{v}_{\tilde{f}^{\prime}}}\mathcal F^{\vec{\lambda}}_{u_f,v_{f^{\prime}}}=\sum_{w_{f^{\prime\prime}}\in B_S^J}r_{u_f,w_{f^{\prime\prime}}}\mathcal F^{\vec{\lambda}}_{w_{f^{\prime\prime}},v_{f^{\prime}}}\;(\mathrm{mod}\;K^{\vartriangleright\lambda}_S).
\end{align}
\end{itemize}
In fact most of the statements follow directly from the work we have already done in the sections before. To be more precise, points (a)+(b) are true because of Lemma~\ref{lem-degfoam} and Theorem~\ref{thm-foambasis}, while the statement (c) is just Lemma~\ref{lem-involution}. Hence, we only have to verify (d), i.e. Equation~\ref{eq-cell2} (it is worth noting that, by linearity, it is sufficient to show Equation~\ref{eq-cell2} to get Equation~\ref{eq-cell1}).
\vspace*{0.25cm}

We want to use the foamation functor $\Psi$ again together with Proposition~\ref{prop-hmfoam} and the fact that Hu and Mathas proved~\cite{hm} that their basis is cellular with almost the same poset.

In fact, the difference between the poset Hu and Mathas use is that they consider all multipartitions of $c(S)$ instead of only the $3$-multipartitions. But our convention how we see the $3$-multipartitions $\vec{\lambda}$ as such multipartitions, i.e. embed into the last three entries, ensures that all multipartitions $\vec{\nu}$ with a node left to the last three entries are bigger in the dominance order, i.e. $\vec{\nu}\vartriangleright\vec{\lambda}$, than all $3$-multipartitions. On the same side the foamation functor and our conventions how to place dots kills all those multipartitions $\vec{\nu}$, since the first (that is leftmost) node in $T_{\vec{\nu}}$ will have three or more addable nodes of the same residue. Thus, the corresponding foam $e(\vec{\nu})d(\vec{\nu})$ will have a face with three or more dots and is therefore killed anyway. Hence, we do not need to consider them in order to prove Equation~\ref{eq-cell2}.
\vspace*{0.25cm}

Another difference is the fact that we have to use additional face removals, because otherwise, as explained before, we would not have the Kuperberg basis as an underlying web basis. But since the corresponding removals are just digon and theta foam removals, it is quite easy to handle them and the crucial case (where we need the underlying combinatorics of the cyclotomic Hecke framework) is in fact the cases $\iota(u_f)=\iota(u_f)^{\prime}$ and $\iota(\tilde{v}_{\tilde{f}^{\prime}})=\iota(\tilde{v}_{\tilde{f}^{\prime}})^{\prime}$ and we can prove the rest by induction on the number $n(\mathrm{face})$ of extra needed face removals in the multiplication of the middle part $\mathcal F_{\tilde{v}_{\tilde{f}^{\prime}}}^*\mathcal F_{u_f}$. Moreover, from now on we consider the case $u=\tilde{v}$, since the other case is, by our multiplication convention, zero anyway.
\vspace*{0.25cm}

\textbf{Case $n(\mathrm{face})=0$:} So given the two elements $\mathcal F^{\vec{\mu}}_{\tilde{u}_{\tilde{f}},\tilde{v}_{\tilde{f}^{\prime}}},\mathcal F^{\vec{\lambda}}_{u_f,v_{f^{\prime}}}\in\mathfrak{F}$, we can pull them back to the HM-basis. We denote the pullbacks by $\psi^{\vec{\mu}}_{\tilde{u}_{\tilde{f}},\tilde{v}_{\tilde{f}^{\prime}}}$ and $\psi^{\vec{\lambda}}_{u_f,v_{f^{\prime}}}$ (and in a similar way all the other HM-basis elements). By Hu and Mathas results we almost obtain the result we want, i.e.
\begin{align}\label{eq-cell3}
\psi^{\vec{\mu}}_{\tilde{u}_{\tilde{f}},\tilde{v}_{\tilde{f}^{\prime}}}\psi^{\vec{\lambda}}_{u_f,v_{f^{\prime}}}=\sum_{\vec{T}\in \mathrm{Std}(\vec{\lambda})}r_{u_f,\vec{T}}\psi^{\vec{\lambda}}_{\vec{T},v_{f^{\prime}}}\;(\mathrm{mod}\;R_{c(S)}^{\vartriangleright\vec{\lambda}}).
\end{align}

Moreover, by Proposition~\ref{prop-tableauxflows2}, we can be sure that all possible $3$-multipartitions $\vec{\lambda}$ have an interpretation as some web $w$ with boundary $S$ together with some flow $f$. Those, by Lemma~\ref{lem-welldefall} and Theorem~\ref{thm-foambasis}, give always rise to a non-zero foam in our basis $\mathfrak F$ given by applying the foamation to a HM-basis element with the same $\vec{\lambda}$ by Proposition~\ref{prop-hmfoam}. That is, there can not be any gaps in the partial ordering $\vartriangleright$ on the side of the foams, i.e. we can restrict our attention to the part of Equation~\ref{eq-cell2} and~\ref{eq-cell3} with $\vec{\lambda}$ as the underlying $3$-multipartition.
\vspace*{0.25cm}

In fact the only problem is the appearance of the full set $\mathrm{Std}(\vec{\lambda})$, because not all of them give rise to a web $w_f\in B_S^J$, since there are in general much more fillings for such tableaux than flows on webs. But it turns out, by locality of the construction, that this can only happen on intermediate layers of the foams (where this is allowed by construction), but not at the top or bottom. To be more precise, after translating Equation~\ref{eq-cell3} to the foam picture, we obtain almost Equation~\ref{eq-cell2}, i.e. we have
\begin{align}\label{eq-cell4}
\mathcal F^{\vec{\mu}}_{\tilde{u}_{\tilde{f}},\tilde{v}_{\tilde{f}^{\prime}}}\mathcal F^{\vec{\lambda}}_{u_f,v_{f^{\prime}}}=\sum_{\vec{T}\in \mathrm{Std}(\vec{\lambda})}r_{u_f,\vec{T}}\mathcal F^{\vec{\lambda}}_{\vec{T},v_{f^{\prime}}}\;(\mathrm{mod}\;K^{\vartriangleright\lambda}_S).
\end{align}
But the $\vec{T}\in \mathrm{Std}(\vec{\lambda})$, which in general does not correspond to a flow on a web in the LT-basis, certainly, by the extended growth algorithm from Definition~\ref{defn-exgrowth} and the result of Theorem~\ref{thm-fine}, corresponds to a web with flow $w_f$ and boundary pair $(S,J)$. But, due to the definition of the multiplication of $K_S$, this web has to be $\tilde u$, since we have
\[
{}_{\tilde u}K_{\tilde v}\otimes {}_{u}K_{v} \to 
{}_{\tilde u}K_{v}
\]
and ${}_{\tilde u}K_{v}$ is closed under addition. Hence, we have
\[
r_{u_f,\vec{T}}=\begin{cases}\tilde{r}_{u_f}, &\text{ if }g(\vec{T})=\tilde{u}_{\tilde{f}},\\0,&\text{ if }g(\vec{T})=f_{w}\text{ with }w\neq \tilde u,\end{cases}
\]
where the scalar $\tilde{r}_{u_f}$ is a certain (we have no further control here) sum of scalars 
$r_{u_f,\tilde{w}_{f^{\prime\prime}}}$ for some webs $\tilde{w}$, but, and that is the crucial observation, all the webs $\tilde w$ are non-elliptic.
\vspace*{0.25cm}

To summarise, we see that Equation~\ref{eq-cell4} can be restricted to
\[
\mathcal F^{\vec{\mu}}_{\tilde{u}_{\tilde{f}},\tilde{v}_{\tilde{f}^{\prime}}}\mathcal F^{\vec{\lambda}}_{u_f,v_{f^{\prime}}}=\sum_{w_{f^{\prime\prime}}\in B_S^J}r_{u_f,w_{f^{\prime\prime}}}\mathcal F^{\vec{\lambda}}_{w_{f^{\prime\prime}},v_{f^{\prime}}}\;(\mathrm{mod}\;K_S^{\vartriangleright\lambda}),
\]
which proves the point (d).

\textbf{Case $n(\mathrm{face})\neq 0$:} Since internal theta removals are completely independent from the rest, we can freely ignore them, since they do not affect point (d) at all.

First note that the digon removal all commute with each other because they exists as digons in the web $\tilde v=u$ and are therefore not neighbors) as we see below.
\begin{align}\label{eq:digcom}
\xy
(0,0)*{\includegraphics[scale=0.5]{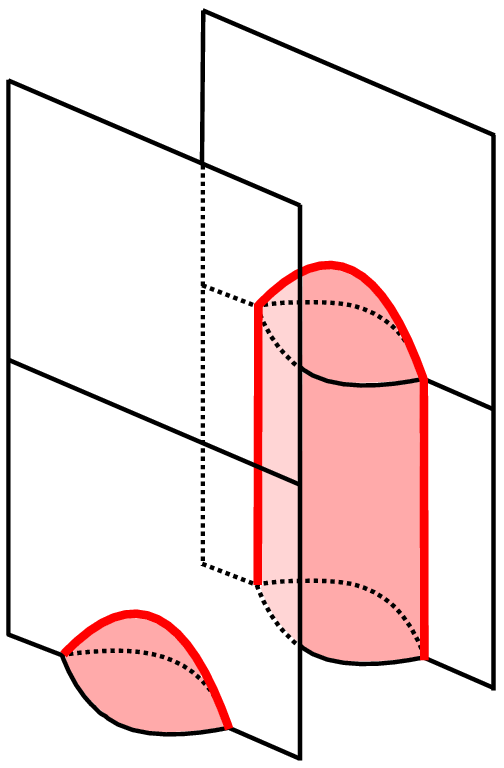}};
\endxy=\xy
(0,0)*{\includegraphics[scale=0.5]{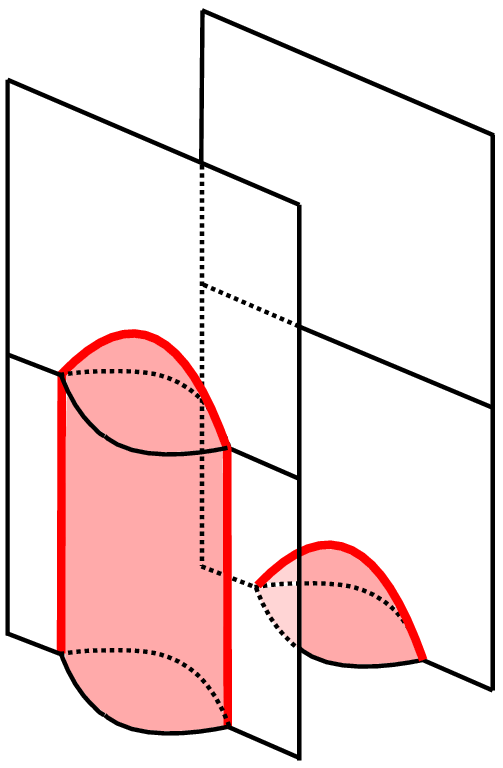}};
\endxy
\end{align}

Note that the case where a certain digon removal on the left side corresponds to a mirror digon removal on the right side follows easy from induction since we would get a local picture as pictured in Equation~\ref{eq-digonremoval}.

With the observation diplayed in~\ref{eq:digcom} above, we can then slide such a pair to the right side (or the left) such that all other digon removals are left to the fixed pair. This local pair can then be regarded as a legal move for some other flow on the middle web $\tilde v=u$ by Lemma~\ref{lem-zipping}. Therefore, we get the requirement (d) without difficulties by induction.

Note that, by construction, all digon removals are always mirrored on the other side, since we assume that $\tilde v=u$. Thus, we obtain the statement by induction.

It should be noted that the scalars are in general, due to our case-by-case argumentation from above, not the same for the cyclotomic Hecke algebra and $K_S$. 
\end{proof}
\begin{rem}\label{rem-byhand}
It is worth noting that Theorem~\ref{thm-cellular} can also be verified (at least in theory) ``by hand'' again, i.e. without using the fact that HM-basis is cellular. Latter relies on the highly non-trivial fact that the standard basis of Dipper, James and Mathas~\cite{djm} is cellular. Thus, on ``higher'' Specht theory.

The cellularity proof ``by hand'' can be done in a similar vein as Brundan and Stroppel did for the $\mathfrak{sl}_2$ case~\cite{bs1}. Of course the number of case one has to check is bigger in our context.

To be slightly more precise, one can check how the local behaviour of the composition
\[
\mathcal F^{\vec{\mu}}_{\tilde{u}_{\tilde{f}},\tilde{v}_{\tilde{f}^{\prime}}}\mathcal F^{\vec{\lambda}}_{u_f,v_{f^{\prime}}}
\]
changes the $3$-multipartition $\vec{\lambda}$. Due to the locality of the construction, we only have to consider compositions of the five cases from Lemma~\ref{lem-zipping} and see that the either increase the number of dots (which corresponds to a higher or equal order in the dominance order $\lhd$). Hence, one only has to check that the one of the same order do not depend on $\tilde{u}_{\tilde{f}}$. But this follows again from the locality of the construction, since we can pull the middle, i.e. $\mathcal F_{\tilde{v}_{\tilde{f}^{\prime}}}^*\mathcal F_{u_f}$, of
\[
\mathcal F^{\vec{\mu}}_{\tilde{u}_{\tilde{f}},\tilde{v}_{\tilde{f}^{\prime}}}\mathcal F^{\vec{\lambda}}_{u_f,v_{f^{\prime}}}=\mathcal F_{\tilde u_{\tilde f}}e(\vec{\mu})d(\vec{\mu})\mathcal F_{\tilde{v}_{\tilde{f}^{\prime}}}^*\mathcal F_{u_f}e(\vec{\lambda})d(\vec{\lambda})\mathcal F_{v_{f^{\prime}}}^*
\]
by shifting it to the left side using
\[
\mathcal F(\tau_{i_k}(a_{i_k},b_{i_k}))e(\vec{\lambda})=e(\tau_{i_k}\cdot\vec{\lambda})\mathcal F(\tau_{i_k}(a_{i_k},b_{i_k})).
\]
This will, by construction, in the end correspond to a linear combination of foams as in~\ref{eq-cell4}, since we do not affect the right side at all and due to the fact that each each such summand will correspond as a consequence of Theorem~\ref{thm-fine} to a basis element.
\end{rem}
As a direct consequence of the cellularity Theorem~\ref{thm-cellular} we construct a complete set of pairwise non-isomorphic, graded, simple $K_S$-modules and a complete set of pairwise non-isomorphic, graded, projective indecomposable $K_S$-modules.

Let us denote the \textit{graded cell modules}\footnote{A definition can be found in~\cite{hm}. We only note that the action is given by the scalars $r_{u_f,w_{f^{\prime\prime}}}$ from Equation~\ref{eq-cell2}.} (also called \textit{Specht modules} in our context), whose existence is guaranteed by Theorem~\ref{thm-cellular}, for $\vec{\lambda}\in\mathfrak{P}_{c(S)}^3$ by $S^{\vec{\lambda}}$ and their \textit{heads or tops} by
\[
D^{\vec{\lambda}}=S^{\vec{\lambda}}/\mathrm{rad}(S^{\vec{\lambda}})
\]
and the \textit{projective cover} of the $D^{\vec{\lambda}}$ by $P^{\vec{\lambda}}$.
Moreover, define $\mathfrak{P}_0^3=\{\vec{\lambda}\in\mathfrak{P}_{c(S)}^3\mid D^{\vec{\lambda}}\neq 0\}$.

By an abstract theorem about graded cellular algebras, due to Graham and Lehrer~\cite{grle} in the ungraded stetting and Hu and Mathas~\cite{hm} is the graded setting, we get the following theorem.
\begin{thm}\label{thm-cellular2}
The sets
\[
\mathcal D=\{D^{\vec{\lambda}}\{k\}\mid \vec{\lambda}\in\mathfrak{P}_0^3,\, k\in\bZ\}\text{  and  }\mathcal P=\{P^{\vec{\lambda}}\{k\}\mid \vec{\lambda}\in\mathfrak{P}_0^3,\, k\in\bZ\}
\]
form a complete set of pairwise non-isomorphic, graded, simple $K_S$-modules and pairwise non-isomorphic, graded, projective indecomposable $K_S$-modules respectively.\qed
\end{thm}
It is worth noting that a direct consequence of Theorem~\ref{thm-cellular2} is that the set $\mathfrak{P}_0^3$ is noting else than the set of all semi-standard tableau $\mathrm{Std}^s(3^{\ell})$ by Corollary 2.11 in~\cite{hm}, since it is well-known that the elements of $B_S$, which can be identified with the elements of $\mathrm{Std}^s(3^{\ell})$, enumerate the set $\underline{\mathcal P}$, i.e. forgetting the grading. These multipartitions are sometimes called \textit{Kleshchev multipartitions} and it is in general highly non-trivial to give an explicit characterization which multipartitions are Kleshchev multipartitions, see~\cite{hm}.
\vspace*{0.15cm}
\begin{rem}\label{rem-basisgrothen}
We note that one can closely follow the approach indicated in~\cite{mack1}, i.e. using the strong $2$-representation (in the sense of Cautis and Lauda~\cite{cala}) on the 
abelian category $\mathcal{W}_{(3^{\ell})}$ (which can be restricted to the additive category 
$\mathcal{W}_{(3^{\ell})}^p$) defined in~\cite{mpt}, where $\mathcal{W}_{(3^{\ell})}$ and $\mathcal{W}_{(3^{\ell})}^p$ are the categories of all finite dimensional, graded, unital $K_{3^{\ell}}$-modules and of all finite dimensional, graded, unital, projective $K_{3^{\ell}}$-modules respectively, to show that the bases $\mathcal D$ and $\mathcal P$ (note that $K_{3^{\ell}}$, as a direct sum of graded cellular algebras, is itself a graded cellular algebra) of the (split) Grothendieck groups
\[
K^{(\oplus)}_0(\mathcal{W}_{(3^{\ell})}^{(p)})_{\bQ(q)}=K^{(\oplus)}_0(\mathcal{W}_{(3^{\ell})}^{(p)})\otimes_{\bZ[q,q^{-1}]}\bQ(q)
\]
corresponds to the canonical basis of the co-standard form of the $\mathfrak{sl}_3$-web space $W^*_{(3^{\ell})}$ and to the dual canonical basis of the $\mathfrak{sl}_3$-web space $W_{(3^{\ell})}$ respectively (up to shifts as explained by Mackaay in~\cite{mack1}). This can be verified almost similar to the approach of Mackaay in~\cite{mack1}, since his arguments are based on Brundan and Kleshchev~\cite{bk2} proof about the graded Specht modules $S_{T}$ of
\[
\mathcal V_{\Lambda}=R_{\Lambda}\text{-}\mathrm{\textbf{Mod}}_{\mathrm{gr}}\text{  and  }\mathcal V^p_{\Lambda}=R_{\Lambda}\text{-}\mathrm{p\textbf{Mod}}_{\mathrm{gr}},
\]
where $R_{\Lambda}$ is the cyclotomic KL-R algebra from Definition~\ref{defn-cyclKLR}, which can be identified with the corresponding ones coming from Hu and Mathas graded cellular basis, due to Corollary 5.10 in~\cite{hm}.
\end{rem}

\end{document}